\documentclass[10pt]{amsart}

\usepackage[dvipsnames]{xcolor}
\usepackage{soul}
\usepackage{cancel}
\usepackage{leftidx}
\usepackage{amsmath}
\usepackage{upgreek}
\usepackage{amscd}
\usepackage{amssymb}
\usepackage{rotating}
\usepackage{amsfonts}
\usepackage{amsthm}
\usepackage{mathrsfs}  
\usepackage{enumerate}
\usepackage{verbatim}
\usepackage{placeins}
\usepackage[normalem]{ulem}
\usepackage[new]{old-arrows}
\usepackage{cases}
\usepackage[top=1in, bottom=1in, left=1.25in, right=1.25in]{geometry}

\theoremstyle{plain}

\usepackage[colorlinks,linkcolor=magenta, anchorcolor=blue, urlcolor=blue, citecolor=blue]{hyperref}
\usepackage{cleveref} 
 \usepackage[hyperpageref]{backref}

\allowdisplaybreaks[4]

\usepackage[all,2cell]{xy}
\UseTwocells

\newcommand{\nn}{\nonumber}

\newcommand{\et}{\mathrm{\acute{e}t}}
\newcommand\scalemath[2]{\scalebox{#1}{\mbox{\ensuremath{\displaystyle #2}}}}

\newcommand{\tsmash}{\wedge}

\makeatletter
\newcommand{\holim@}[2]{%
  \vtop{\m@th\ialign{##\cr
    \hfil$#1\operator@font holim$\hfil\cr
    \noalign{\nointerlineskip\kern1.5\ex@}#2\cr
    \noalign{\nointerlineskip\kern-\ex@}\cr}}%
}
\newcommand{\holim}{%
  \mathop{\mathpalette\holim@{\leftarrowfill@\textstyle}}\nmlimits@
}
\newcommand{\hocolim@}[2]{%
  \vtop{\m@th\ialign{##\cr
    \hfil$#1\operator@font holim$\hfil\cr
    \noalign{\nointerlineskip\kern1.5\ex@}#2\cr
    \noalign{\nointerlineskip\kern-\ex@}\cr}}%
}
\newcommand{\hocolim}{%
  \mathop{\mathpalette\holim@{\rightarrowfill@\textstyle}}\nmlimits@
}
\newcommand{\colim@}[2]{%
  \vtop{\m@th\ialign{##\cr
    \hfil$#1\operator@font colim$\hfil\cr
    \noalign{\nointerlineskip\kern1.5\ex@}#2\cr
    \noalign{\nointerlineskip\kern-\ex@}\cr}}%
}
\newcommand{\colim}{%
  \mathop{\mathpalette\colim@{\rightarrowfill@\textstyle}}\nmlimits@
}
\makeatother

\newtheorem{theorem}{Theorem}[section]
\newtheorem{proposition}[theorem]{Proposition}
\newtheorem{lemma}[theorem]{Lemma}
\newtheorem{corollary}[theorem]{Corollary}
\newtheorem{variant}[theorem]{Variant}

\newtheorem{claim}[theorem]{Claim}
\newtheorem{theorem/definition}[theorem]{Theorem/Definition}

\theoremstyle{definition}
 
\newtheorem{definition}[theorem]{Definition}
\newtheorem{construction}[theorem]{Construction}

\newtheorem{remark}[theorem]{Remark}

\begin{document}

\title{On the algebraic $K$-theory of smooth schemes over truncated Witt vectors}
\author{Xiaowen Hu}
\address{School of Sciences, Great Bay University, Dongguan, China}
\email{huxw06@gmail.com}

\begin{abstract}
We study the algebraic $K$-theory of smooth schemes over $W_n(\Bbbk)$, where $\Bbbk$ is a perfect field of characteristic $p>0$. For a $p$-adic smooth scheme $X_{\centerdot}$ over $W_{\centerdot}(k)$, we introduce complexes $p^{r,m}_{r,n}\Omega^{\bullet}_{X_{\centerdot}}$ and infinitesimal motivic complexes $\mathbb{Z}_{X_n}(r)$, and for 
$0 \leq i \leq p-4$, we  establish a Chern character isomorphism between the sheaf $\mathcal{K}_{X_n,X_{m},i}$ and the direct sum of certain cohomology sheaves of $p^{r,m}_{r,n}\Omega^{\bullet}_{X_{\centerdot}}$ with $1\leq r\leq i$. This leads to a criterion for $K$-theoretic infinitesimal deformations, which is related to Emerton's $p$-adic variational Hodge conjecture. By taking the limit $n \rightarrow \infty$ with $m=1$, we recover a theorem of Bloch, Esnault, and Kerz on continuous relative algebraic $K$-theory.

The proof combines Brun's isomorphism relating $K$-theory to derived cyclic homology, computations of relative cyclic homology over $W(\Bbbk)$, and an analysis of multiplicative structures of the mod $p$ relative $K$-theory.  
\end{abstract}

\maketitle

\tableofcontents

\section{Introduction} 
\label{sec:introduction}

In \cite{Gro66}, Grothendieck proposed  the 
variational Hodge conjecture (VHC). Put in an analytic setting, it states that if $\pi:\mathcal{X} \rightarrow S$ is a smooth projective submersion between connected complex manifolds, and if $\alpha$ is a  flat section of the Gauss-Manin connection and  lies in $F^p\mathbf{R}^{2p}\pi_* \Omega^{\bullet}_{X/S}$ (Hodge obstruction), and if $\alpha$ restricted to one fiber is an algebraic class, then $\alpha$ is algebraic on all fibers. If $S$ is quasi-projective, then by the global invariant cycle theorem of Deligne the Hodge obstruction vanishes and the statement is the usual VHC (see e.g. \cite[\S3]{ChS14}). A major achievement related to VHC is obtained by Bloch \cite{Blo72}, and its generalization by Buchweitz and Flenner \cite{BuF03}  plays a crucial role in the recently announced proof  by Markman \cite{Mar25} of the Hodge conjecture for abelian 4-folds.  
In \cite{Eme97}, Emerton proposed a $p$-adic version of the VHC, where $S$ is replaced with $\operatorname{Spec}W(\Bbbk)$ and $\Bbbk$ is a perfect field of characteristic $p>0$. Maulik and Poonen showed in \cite[Theorem 9.10]{MaP12} that  if the $p$-adic VHC holds for all sufficiently large primes $p$, then the complex VHC holds.

Let $X$ be a smooth projective scheme over $W(\Bbbk)$ of relative dimension $d$, and let $X_n=X\otimes_{W(\Bbbk)}W_n(\Bbbk)$. Via the Chern character map, $K_0(X)_{\mathbb{Q}}$ is isomorphic to $\mathrm{CH}(X_0)_{\mathbb{Q}}$. Since $K$-theory is sensitive to the non-reduced structure, it is natural to study the $p$-adic VHC as a deformation problem in $K$-theory. 
In \cite{BEK14}, under the condition $d<p-6$, Bloch, Esnault, and Kerz showed that an element of $K(X_1)_{\mathbb{Q}}$ lifts to $\varprojlim_n K(X_n)_{\mathbb{Q}}$ if and only if its  Hodge obstruction vanishes. This provides evidence for the $p$-adic VHC and suggests an approach to understand it. Results in a similar flavor are then obtained in \cite{BEK14b}, \cite{Mor14}, and \cite{Mor19b}; see also \cite{Lan18} and \cite{GrL21} for generalizations of the motivic pro-complexes defined in \cite{BEK14}. A crucial ingredient \cite[Claim 11.2]{BEK14} of their proof is a Chern character isomorphism of \emph{étale} pro-sheaves on $X_1$:
\begin{gather}\label{eq:etalePro-iso-BEK}
\mathrm{ch}: \mathcal{K}_{X_{\centerdot},X_1,i} \rightarrow \bigoplus_{r\leq i}\mathcal{H}^{2r-i-1}\big(p(r)\Omega_{X_{\centerdot}}^{\leq r-1}\big)
\end{gather}
for $0\leq i\leq p-3$, where $p(r)\Omega_{X_{\centerdot}}^{\bullet}$ is the pro-complex 
\begin{gather*}
  p^r \mathcal{O}_{X_{\centerdot}}\rightarrow p^{r-1}\Omega_{X_{\centerdot}/W_{\centerdot}(\Bbbk)}^1\rightarrow\dots\rightarrow
  p\Omega_{X_{\centerdot}/W_{\centerdot}(\Bbbk)}^{r-1} 
  \rightarrow \Omega^r_{X_{\centerdot}/W_{\centerdot}}\rightarrow \Omega^{r+1}_{X_{\centerdot}/W_{\centerdot}}\rightarrow \dots
\end{gather*}

Since the functor $K_0$ on schemes is not representable, the Bloch-Esnault-Kerz theorem on formal deformation in algebraic  $K$-theory does not lead to existence to the algebraic deformation. The algebraization problem is still open, but we learn from the classical deformation theory (of vector bundles, etc.) that to study a deformation problem, it is essential and beneficial to study  the relevant infinitesimal deformations.
Our goal in this paper is to find an infinitesimal version of the isomorphism \eqref{eq:etalePro-iso-BEK} and address the problem of infinitesimal deformations  in the algebraic $K$-theory.  
\subsection{Main results} 
\label{sub:main_results}

The main theorem of this paper is
\begin{theorem}[= Theorem \ref{thm:relative-comparison-local}]\label{thm:relative-comparison-local-intro}
Let $\Bbbk$ be a perfect field of characteristic $p$. Let $X_{\centerdot}$ be a $p$-adic smooth scheme separated and of finite type over $W_{\centerdot}(\Bbbk)$. 
For $0\leq i\leq p-4$ (resp. $0\leq i\leq p-3$) and $n> m\geq 1$, there is a canonical Chern character isomorphism (resp. epimorphism)
\begin{gather}\label{eq:relative-comparison-local-intro}
\mathrm{ch}=\sum_{r=1}^i \mathrm{ch}_r:\mathcal{K}_{X_n,X_m,i}\rightarrow  \bigoplus_{r=1}^i \mathcal{H}^{2r-i-1}\big(p^{r,m}_{r,n}\Omega^{\bullet}_{X_{\centerdot}}\big)
\end{gather}
of \emph{Nisnevich} sheaves on $X_1$, where $p^{r,m}_{r,n}\Omega^{\bullet}_{X_{\centerdot}}$ is the complex
\begin{gather*}
p^{rm} \mathcal{O}_{X_{rn}}\rightarrow p^{(r-1)m}\Omega^1_{X_{(r-1)n}/W_{(r-1)n}}\rightarrow
\dots\rightarrow p^m\Omega_{X_{n}/W_{n}}^{r-1}\rightarrow 0\rightarrow \dots
\end{gather*}
\end{theorem}
We expect that this theorem also holds for the Zariski $K$-sheaves, see \S\ref{sub:further_problems}(1).

The complex $p^{r,m}_{r,n}\Omega^{\bullet}_{X_{\centerdot}}$ originates from a computation of relative derived cyclic homology in \S\ref{sec:relative_derived_cyclic_homology_of_smooth_algebras_over_texorpdfstring_w_n_bbbk_} as  explained in \S\ref{sub:defining_the_infinitesimal_motivic_complex}. 
Its appearance is reasonable, as it is the quotient of the pro-complexes  $p^{r,n}\Omega^{\bullet}_{X_{\centerdot}}$ and $p^{r,m}\Omega^{\bullet}_{X_{\centerdot}}$, where  $p^{r,n}\Omega^{\bullet}_{X_{\centerdot}}$ (Definition \ref{def:proComplexes}) is the pro-complex
\begin{gather}\label{eq:motivicProComplex-intro}
p^{rn}\mathcal{O}_{X_{\centerdot}}\rightarrow p^{(r-1)n}\Omega^1_{X_{\centerdot}/W_{\centerdot}}\rightarrow
\dots\rightarrow p^n\Omega_{X_{\centerdot}/W_{\centerdot}}^{r-1}
\rightarrow \Omega^r_{X_{\centerdot}/W_{\centerdot}}\rightarrow \Omega^{r+1}_{X_{\centerdot}/W_{\centerdot}}\rightarrow \dots
\end{gather}
which gives rise to the product structure on the infinitesimal motivic complexes $\mathbb{Z}_{X_n}(r)$, which will be introduced also in \S\ref{sub:defining_the_infinitesimal_motivic_complex}. Moreover, in particular, $p^{r,1}\Omega^{\bullet}_{X_{\centerdot}}$ is the pro-complex  $p(r)\Omega_{X_{\centerdot}}^{\bullet}$ we recalled above.
Theorem \ref{thm:relative-comparison-local-intro} implies the following result on infinitesimal deformations in algebraic $K$-theory, which can be regarded as an infinitesimal version of Emerton's $p$-adic variational Hodge conjecture.
\begin{theorem}[= Proposition \ref{prop:deformationInK-theory}(i)]\label{thm:deformationInK-theory-intro}
Let $\Bbbk$ be a perfect field of characteristic $p$.  Let $X_{\centerdot}$ be a $p$-adic smooth scheme separated and of finite type over $W_{\centerdot}(\Bbbk)$.
Assume $\dim X_1\leq p-6$. Then  a class $\xi\in K_0(X_m)$ lifts to $K_0(X_n)$ if and only if it is sent to 0 by a Hodge obstruction map
  \begin{gather}\label{eq:HodgeOb-intro}
  \mathrm{ob}:
  K_0(X_m) \rightarrow \bigoplus_{r=1}^{d-1}\mathbb{H}^{2r}(X_1,p^{r,m}_{r,n}\Omega^{\bullet}_{X_{\centerdot}}).
  \end{gather}
\end{theorem}
The Hodge obstruction map is induced by a Chern character map
\begin{gather*}
K_0(X_m) \rightarrow \bigoplus_{r=0}^{p-1} \mathbb{H}^{2r}
\big(X_1, \mathbb{Z}_m(r)\big)[\frac{1}{r!}]
\end{gather*}
 and the following exact triangle \eqref{eq:motivicFundamentalTriangle-intro}.
Note that  no assumption of projectivity is needed in this theorem. A similar assertion in higher $K$-theory is given in Proposition \ref{prop:deformationInK-theory}(ii). In view of the definition of the complex $p^{r,m}_{r,n}\Omega^{\bullet}_{X_{\centerdot}}$, the somewhat surprising aspect of Theorem \ref{thm:deformationInK-theory-intro} is that the precise obstruction to deforming a $K$-theoretic element to $X_n$ is naturally expressed by data involving higher order thickenings.
\begin{remark}\label{rem:mainThm-deformation-intro}
\begin{enumerate}[(i)]
  \item Taking limits, the quasi-isomorphism \eqref{eq:relative-comparison-local-intro}  of  Nisnevich sheaves implies the quasi-isomorphism \eqref{eq:etalePro-iso-BEK} of  étale pro-sheaves. Moreover, the limit of the Hodge obstruction map \eqref{eq:HodgeOb-intro} is the composition of the Chern character defined in \cite[\S9]{BEK14} and the Hodge obstruction map defined in \cite[\S 8]{BEK14}.
  \item In \cite{AMMN22} and \cite{Bei14}, the Chern character isomorphism $K_0^{\mathrm{cont}}(X_{\centerdot})_{\mathbb{Q}}\xrightarrow{\cong}\mathbb{H}^{2r}_{\mathrm{cont}}\big(X_{\centerdot},\mathbb{Z}_{X_{\centerdot}}(r)\big)_{\mathbb{Q}}$ of Bloch, Esnault, and Kerz \cite[Theorem 11.1]{BEK14} is reproved by different methods. However, their methods do not yield the relative Chern character isomorphism with integer coefficients, which is necessary for a meaningful statement on the infinitesimal deformations in algebraic $K$-theory, since the relative $K$-groups $K_i(X_n,X_m)$ are $p$-torsion.
  \item By Thomason-Trobaugh localization theorem of algebraic $K$-theory (\cite{ThT90}), Theorem \ref{thm:relative-comparison-local-intro} implies isomorphism of algebraic $K$-theory for a certain range of $i$, see Corollary \ref{cor:smAlg-relKTheory}. In particular, for $2i-1\leq p-5$, we have a Chern character isomorphism
  \begin{gather*}
  \mathrm{ch}: K_{2i-1}\big(W_{n}(\Bbbk),\Bbbk\big)\cong W_{(n-1)i}(\Bbbk)
  \end{gather*}
  and for $2i\leq p-5$, we have $K_{2i}\big(W_{n}(\Bbbk),\Bbbk\big)=0$. In \cite[Corollary 7.4]{Bru01} and \cite[Theorem C]{Ang11} (see also the very recent \cite{AKN24}), vanishing and isomorphism of this type are obtained for  a wider range of $i$. However, even in these cases, it is not obvious that there is a natural Chern character map which induces isomorphisms.
\end{enumerate}
\end{remark}
  Another consequence of Theorem \ref{thm:relative-comparison-local-intro} is the continuity of algebraic $K$-theory.
\begin{theorem}[= Proposition \ref{prop:continuity-alg-K-theory}]\label{thm:continuity-alg-K-theory-intro}
Let $\Bbbk$ be a perfect field of characteristic $p$.
Let $X$ be a smooth projective scheme over $W(\Bbbk)$ and $X_n=X\times_{\operatorname{Spec}W(\Bbbk)}\operatorname{Spec}W_n(\Bbbk)$.  Then for $0\leq i\leq p-5-d$, the canonical map
\begin{align*}
K_{\mathrm{cont},i}(X_{\centerdot})\xrightarrow{\cong} \varprojlim_{n} K_i(X_n)
\end{align*}
is an isomorphism.
\end{theorem}

\subsection{Comparison with the work of Bloch, Esnault, and Kerz} 
\label{sub:comparison_with_BEK14}

Building on the framework of \cite[Theorem 11.1]{BEK14}, we prove Theorem \ref{thm:relative-comparison-local-intro} by developing infinitesimal motivic complexes and a relative Chern character from infinitesimal deformation K-theory. The isomorphism is proved via an analysis of products mod $p$. Here we highlight two novel ingredients in our proof; see \ref{sub:defining_the_infinitesimal_motivic_complex}-\S\ref{sub:study_of_the_mod_p_products} for more details.

The construction of motivic pro-complexes in \cite{BEK14} draws inspiration from \cite{Kat87} and \cite{FoM87}, and its mod $p$ Chern character isomorphism builds upon established results in continuous $K$-theory and topological cyclic homology (see \cite{GeH06} and \cite{HeM04}).  In our infinitesimal setting, however,  naive analogs of these complexes (namely, replacing $X_{\centerdot}/W_{\centerdot}$ in \eqref{eq:motivicProComplex-intro} with $X_n/W_n$ using a uniform $n$ for all terms) turn out to be incorrect. Moreover, there are no results such as \cite{GeH06} and \cite{HeM04} at our disposal in the infinitesimal case. The central novelty of this paper is to define reasonable infinitesimal motivic complexes according to  computations of certain relative derived cyclic homology and a theorem of Brun \cite{Bru01}. Moreover, due to the limitation of the range in Brun's theorem, we have to show the desired Chern character isomorphism by an inductive comparison of the relative Chern character with Brun's isomorphism.

The second novelty in our proof lies in this comparison.
Unlike the continuous $K$-theory (see \cite[Proposition 10.5]{BEK14}),  the mod $p$ relative $K$-sheaves $(\mathcal{K}/p)_{i}$ of the smooth infinitesimal deformations, with odd $i< p-3$, have a direct summand isomorphic to $\mathcal{O}_{X_1}$, which is not generated by the absolute $\mathcal{K}_1$ and the Bott element. We develop a technique to control the ambiguity in the multiplication arising from this summand (detailed in the final paragraph of \S\ref{sub:study_of_the_mod_p_products} and  in \S\ref{sub:functoriality_resolves_ambiguity}). This technique probably applies more widely, as such summands may occur in the relative algebraic $K$-theory of general infinitesimal deformations.


\subsection{Defining the infinitesimal motivic complexes} 
\label{sub:defining_the_infinitesimal_motivic_complex}

Let us first explain the origin of the complex $p^{r,m}_{r,n}\Omega^{\bullet}_{X_{\centerdot}}$. Our starting point is Brun's isomorphism (\cite[Theorem 6.1]{Bru01}): for a simplicial ring  $R$ and an ideal $I$ satisfying $I^m=0$, there is a natural isomorphism (called \emph{Brun maps})
\begin{align}\label{eq:BrunIso-intro}
\mathrm{br}:
K_i(R,I;p)\cong \widetilde{\mathrm{HC}}_{i-1}(R,I;p),
\end{align}
for $0\leq i<\frac{p}{m-1}-2$, and an epimorphism for $0\leq i<\frac{p}{m-1}-1$, where $\widetilde{\mathrm{HC}}$ is the derived cyclic homology, and “$;p$" means the $p$-completion. We outline Brun's proof in \eqref{eq:zigzag-graph-Brun-map}, from which we see that there should be no natural homomorphism from $K_i(R,I;p)$ to $\widetilde{\mathrm{HC}}_{i-1}(R,I;p)$ for $i$ beyond the range $0\leq i<\frac{p}{m-1}-1$, and thus we do not expect to extend Brun's isomorphism to $0\leq i<p-2$ for all $m$ by induction. One can also see this from the case $i=1$, where $K_1(R,I)=1+I$ and $\widetilde{\mathrm{HC}}_{0}(R,I)=I$ (Lemma \ref{lem:relativeHC0}), so a natural map can only be the logarithm, whose existence over $\mathbb{Z}_p$ depends on  $m$. 

Therefore, for a smooth $W(\Bbbk)$-algebra $A$ and $A_n=A/p^n A$, one cannot directly use Brun's isomorphism to obtain  $K(A_n,A_m)$ by computing the relative derived cyclic homology for general $n>m\geq 1$. 
 Nevertheless, we still compute the relative derived cyclic homology. For this, we first show that the derived cyclic homology of $A_n$ as a $\mathbb{Z}$-algebra is equivalent to that of $A_n$ as a $W(\Bbbk)$-algebra (Corollary \ref{cor:reduction-HH-HC-overZ-overW(k)}). Then we take a flat cdga resolution of $A_n$ (\S\ref{sub:a_typical_cdga}), and by using the Hochschild-Kostant-Rosenberg theorem for smooth cdga's (\cite[Prop.~2.5]{BuV88}, \cite[Prop.~5.4.6]{Lod98}), we reduce the computation of $\widetilde{\mathrm{HC}}_{i}(A_n,A_m)$ to computing the homology of a de Rham-like double complex $\mathbf{CC}(A_n,A_m)$ for $i\leq p-2$ (\S\ref{sec:derived_hochschild_homology_of_smooth_algebras} and Proposition \ref{prop:relHC-relBoldHC-bounded-pExp-Isom}).
 Our original computation of the latter homology, which gives Corollary \ref{cor:rel-HC-nonreduced-overZ}, is straightforward and lengthy. Instead, since each summand on the RHS of \eqref{eq:rel-HC-nonreduced-overZ} is a homology of the complex $p^{r,m}_{r,n}\Omega^{\bullet}_{A_{\centerdot}}$ for a certain $r$, we are motivated to define a morphism of complexes $\Psi$ (Lemma \ref{lem:map-Psi-complexes-twist(r)-component}) from such (shifted) complexes to (the total complex of) the above mentioned double complex, and then we show that the direct sum of morphisms $\Psi$ induces a quasi-isomorphism on mod $p$ homology, hence a quasi-isomorphism because both complexes have $p$-torsion homologies of finite exponents. In particular, we obtain canonical isomorphisms
\begin{gather}\label{eq:rel-HC-decomposition-intro}
 \widetilde{\operatorname{HC}}_{i-1}(A_n,A_m)\cong   \bigoplus_{r=1}^i H^{2r-i-1}\big(p^{r,m}_{r,n}\Omega^{\bullet}_{A_{\centerdot}}\big)
 \end{gather}
 and
 \begin{gather}\label{eq:rel-K-decomposition-intro}
 K_i(A_n,A_{n-1})\cong \widetilde{\operatorname{HC}}_{i-1}(A_n,A_{n-1})\cong   \bigoplus_{r=1}^i H^{2r-i-1}\big(p^{r,n-1}_{r,n}\Omega^{\bullet}_{A_{\centerdot}}\big)
 \end{gather}
 for $1\leq i\leq p-3$ and $n>m\geq 1$. 

 The uniformity of \eqref{eq:rel-HC-decomposition-intro} \emph{suggests} introducing the \emph{infinitesimal motivic complex} $\mathbb{Z}_{X_n}(r)$ of Nisnevich sheaves on $X_1$ for $0\leq r<p$ such that there is an exact triangle
 \begin{gather}\label{eq:motivicFundamentalTriangle-intro}
p^{r,m}_{r,n}\Omega^{\bullet}_{X_{\centerdot}}[-1] \rightarrow \mathbb{Z}_{X_n}(r) \rightarrow \mathbb{Z}_{X_m}(r) \rightarrow p^{r,m}_{r,n}\Omega^{\bullet}_{X_{\centerdot}}
\end{gather}
and $ \mathbb{Z}_{X_1}(r)$ is the usual motivic complex of a smooth scheme over the field $\Bbbk$. We do this in \S\ref{sec:infinitesimal_motivic_complexes}, and show that $\mathbb{Z}_{X_n}(r)$ has the expected properties, e.g. the existence of products $\mathbb{Z}_{X_n}(r)\otimes^{\mathbf{L}}_{\mathbb{Z}}\mathbb{Z}_{X_n}(r) \rightarrow \mathbb{Z}_{X_n}(r+s)$,  and $\mathbb{Z}_{X_n}(1)\simeq \mathbb{G}_{m,X_n}[-1]$, and the continuous limit of $\mathbb{Z}_{X_n}(r)$ is the \emph{motivic pro-complex} $\mathbb{Z}_{X_{\centerdot}}(r)$ defined in \cite[\S7]{BEK14}. Then in \S\ref{sec:infinitesimal_chern_classes_and_chern_characters}, by using Gillet's arguments in \cite{Gil81}, and Totaro's theorem \cite[Theorem 9.2]{Tot18} on the Hodge cohomology of the classifying space $\mathrm{BGL}_{\mathbb{Z}}$, and certain arguments from model category, we define  Chern classes and the Chern character valued in the infinitesimal motivic complexes.

\subsection{The structure of the mod \texorpdfstring{$p$}{p} products} 
\label{sub:study_of_the_mod_p_products}

To show Theorem \ref{thm:relative-comparison-local-intro}, by certain standard dévissages, it suffices to show that the mod $p$ relative Chern character
\begin{gather}\label{eq:relative-comparison-local-mod-p-intro}
\mathrm{ch}=\sum_{r=1}^i \mathrm{ch}_r:(\mathcal{K}/p)_{X_n,X_{n-1},i}\rightarrow  \bigoplus_{r=1}^i \mathcal{H}^{2r-i-1}\big(p^{r,n-1}_{r,n}\Omega^{\bullet}_{X_{\centerdot}}\otimes^{\mathbf{L}}_{\mathbb{Z}}\mathbb{Z}/p\big)
\end{gather}
is an isomorphism for $0\leq i\leq p-3$ and $n\geq 2$. The RHS is computed in Lemma \ref{lem:rel-HC-nonreduced-mod-p}. 
So by Brun's isomorphism \eqref{eq:rel-K-decomposition-intro}, we have a decomposition
\begin{gather*}
(\mathcal{K}/p)_{X_n,X_{n-1},i}\cong  \bigoplus_{r=1}^i \mathcal{H}^{2r-i-1}\big(p^{r,n-1}_{r,n}\Omega^{\bullet}_{X_{\centerdot}}\otimes^{\mathbf{L}}_{\mathbb{Z}}\mathbb{Z}/p\big)\cong  \bigoplus_{j=0}^{i-1}\Omega^j_{X_1/\Bbbk}\quad.
\end{gather*}
However, it is not clear whether the Brun map on $(\mathcal{K}/p)_{X_n,X_{n-1},i}$ coincides with the Chern character. 
Instead of showing this, we will deduce the mod $p$ Chern character isomorphism by showing that it differs from the Brun isomorphism by an automorphism of the target, see Claim \ref{claim:relative-comparison-m=n-1-modP-induction}.

To achieve this, we study the multiplicative structure on both sides of \eqref{eq:relative-comparison-local-intro} (resp. \eqref{eq:relative-comparison-local-mod-p-intro}),  as inspired by the arguments in \cite[\S11]{BEK14}. There are multiplications by $\mathcal{K}_{X_n,i}$ (resp. $(\mathcal{K}/p)_{X_n,i}$) on both sides; on the left, it is Loday's product, and on the right, it is given by the Chern character $\mathcal{K}_{X_n,i}\rightarrow \mathcal{HC}^{-}_{X_n,i}$ (\cite[\S8.4]{Lod98}) and the multiplication of $\mathrm{HC}^{-}$ on the (relative) HC.
 There are several problems concerning the products that we have to address:
\begin{enumerate}[(i)]
  \item There are two ways to define products in HH (resp. products in $\mathrm{HC}^{-}$, resp. actions of $\mathrm{HC}^{-}$ on $\mathrm{HC}$). The \emph{algebraic product} in HH is the shuffle product, in HC is one induced by the shuffle  product, which is essentially defined in \cite{HoJ87}, and is explicitly given in the cyclic chain complex in \cite[\S5.1.3]{Lod98}.  The \emph{topological product} in HH is induced by the product of two spaces, and in $\mathrm{HC}^{-}$ (resp. $\mathrm{HC}$) is the one induced by the homotopy fixed points (resp. homotopy orbits) of an  $H$-space. The algebraic product and the topological product in HH coincide, which is essentially a consequence of the Eilenberg-Zilber theorem, and is reviewed as Lemma \ref{lem:smashProduct-shuffleProduct}. Whether the algebraic and the topological products in $\mathrm{HC}^{-}$ (resp., actions of $\mathrm{HC}^{-}$ on $\mathrm{HC}$), coincide, turns out to be a nontrivial issue; see  \cite[Problems 11.8 and 14.8]{GrM95}. In \S\ref{sec:products_on_cyclic_homology}, we give an exposition on these products and the compatibility of the algebraic products (resp. the topological products) in HH and HC with Connes' exact sequence. Then by the coincidence of the algebraic and the topological products in HH, there is a filtration in HC such that the algebraic and the topological products of $\mathrm{HC}^{-}$ on the graded pieces coincide. 
  \item There are two equivalent ways to define the derived cyclic homology of a ring: by a flat simplicial resolution or by a flat (c)dga resolution. Both are needed in this paper and we show the equivalence in \S\ref{sub:products_on_hochschild_homology_of_c_dga_s}-\S\ref{sub:products_in_the_derived_hochschild_and_cyclic_homology_comparison_of_products}, 
  \item In \S\ref{sec:Brun's_isomorphism_and_the_product_structure}, we show that Brun's isomorphism \eqref{eq:BrunIso-intro} is compatible with the multiplication by the $K$-theory and the topological multiplication by $\mathrm{HC}^{-}$ on $\mathrm{HC}$. This turns out (see Lemma \ref{lem:commutativity-mu-1_land_Norm}) to be a consequence of a compatibility property (Lemma \ref{lem:adams-iso-homotopical-functoriality}) of the Adams isomorphism. For the latter, we have to recall the construction of the Adams isomorphism, and reduce the problem to an extension property (Lemma \ref{lem:LMS-II.2.8-enhance}).
\end{enumerate}

To compare the mod $p$ multiplications in the Brun map and the relative Chern character, we need to find explicit lower order $K$-theory elements that generate the $\bigoplus_{i\leq p-3}(\mathcal{K}/p)_{X_n,X_{n-1},i}$   as much as possible. Besides the multiplication on $\bigoplus_{i\leq p-3}(\mathcal{K}/p)_{X_n,X_{n-1},i}$  by  $(\mathcal{K}/p)_{X_n,1}$, we need also the multiplication by a Bott element in $(\mathcal{K}/p)_{X_n,2}$. 
There exists a primitive $p$-th root of unity, e.g. $\zeta_n:=1+p^{n-1}$, in  $W_n(\Bbbk)$, which induces a Bott element. However, $\zeta_n$ maps to 0 in $W_{n-1}(\Bbbk)$, and does not lift to a $p$-th root of unity in $W_{n+1}(\Bbbk)$. This has two consequences:
\begin{itemize}
    \item The multiplication by the corresponding Bott element vanishes on the bottom row of the following Figure \ref{fig:multTable} (see also the proofs of Propositions \ref{prop:BottElement-multiplication-on-rel-HC} and \ref{prop:multiplication-Bott-element}).
    \item There is no primitive $p$-th root of unity in $W(\Bbbk)$. Hence in \cite[\S11]{BEK14}, to apply the theorem of Bloch-Kato \cite{BlK86}, they need to take  a base change to $W(\Bbbk)[\zeta]$, where $\zeta$ is a primitive $p$-th root of unity,   and use a norm argument to conclude the Chern character isomorphism \cite[Claim 11.3]{BEK14}.
    The multiplication by the corresponding Bott element turns out to be similar to the dotted arrow “$\xy (0,0); (6,0) **\crv{~*=<3pt>{\scalemath{0.9}{.}} (3,0)}?>* \dir{>}; \endxy$" on the top row of Figure \ref{fig:multTable}.
\end{itemize}  

With these preparations, we compute the products in \S\ref{sec:mod_p_rel_cyclic_homology_and_the_products} and \S\ref{sub:the_mod_p_product_structure}. Our final results on the products (Propositions \ref{prop:K1-multiplication-on-rel-HC}, \ref{prop:BottElement-multiplication-on-rel-HC}, \ref{prop:K1-multiplication-motivic-complex}, and \ref{prop:multiplication-Bott-element}) can be visualized in the following table
\begin{figure}[ht]
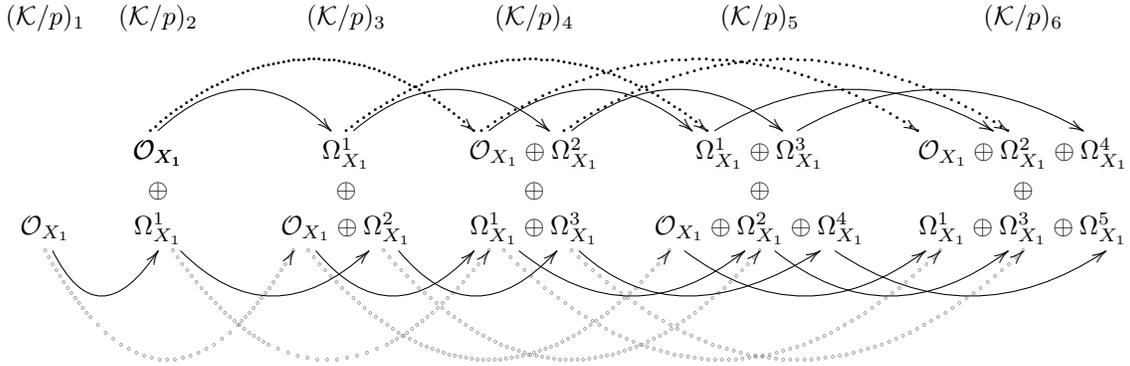

\centering
\begin{gather*}
\xy
(10,8)*+{(\mathcal{K}/p)_1};(25,8)*+{(\mathcal{K}/p)_2};(50,8)*+{(\mathcal{K}/p)_3}; (75,8)*+{(\mathcal{K}/p)_4};
(105,8)*+{(\mathcal{K}/p)_5}; (140,8)*+{(\mathcal{K}/p)_6};
(10,-20)*+{\mathcal{O}_{X_1}};
(25,-10)*+{\mathcal{O}_{X_1}};(25,-15)*+{\oplus}; (25,-20)*+{\Omega^1_{X_1}};
(25,-10)*+{\mathcal{O}_{X_1}};
(50,-10)*+{\Omega^1_{X_1}}; (50,-15)*+{\oplus}; (50,-20)*+{\mathcal{O}_{X_1}\oplus\Omega^2_{X_1}};
(75,-10)*+{\mathcal{O}_{X_1}\oplus\Omega^2_{X_1}}; (75,-15)*+{\oplus}; (75,-20)*+{\Omega^1_{X_1}\oplus\Omega^3_{X_1}};
(105,-10)*+{\Omega^1_{X_1}\oplus\Omega^3_{X_1}}; (105,-15)*+{\oplus}; (105,-20)*+{\mathcal{O}_{X_1}\oplus\Omega^2_{X_1}\oplus\Omega^4_{X_1}}; 
(140,-10)*+{\mathcal{O}_{X_1}\oplus\Omega^2_{X_1}\oplus\Omega^4_{X_1}}; (140,-15)*+{\oplus};  (140,-20)*+{\Omega^1_{X_1}\oplus\Omega^3_{X_1}\oplus\Omega^5_{X_1}};
(25,-7); (48,-7) **\crv{(37,4)}?>*\dir{>};
(51,-7); (77,-7) **\crv{(64,4)}?>*\dir{>};
(80,-7); (108,-7) **\crv{(94,4)}?>*\dir{>};
(110,-7); (148,-7) **\crv{(129,4)}?>*\dir{>};
(69,-7); (96,-7) **\crv{(83,4)}?>*\dir{>};
(99,-7); (136,-7) **\crv{(117,4)}?>*\dir{>};
(11,-23); (25,-23) **\crv{(17,-35)}?>*\dir{>};
(28,-23); (53,-23) **\crv{(40,-35)}?>*\dir{>};
(57,-23); (78,-23) **\crv{(67,-35)}?>*\dir{>};
(81,-23); (113,-23) **\crv{(94,-35)}?>*\dir{>};
(115,-23); (151,-23) **\crv{(133,-35)}?>*\dir{>};
(46,-23); (67,-23) **\crv{(56,-35)}?>*\dir{>};
(73,-23); (103,-23) **\crv{(88,-35)}?>*\dir{>};
(107,-23); (138,-23) **\crv{(122,-35)}?>*\dir{>};
(95,-23); (127,-23) **\crv{(111,-35)}?>*\dir{>};
(24,-7); (67,-7) **\crv{~*=<2pt>{.} (46,12)}?>*\dir{>};
(50,-7); (98,-7) **\crv{~*{.} (75,12)}?>*\dir{>};
(68,-7); (126,-7) **\crv{~*{.} (97,12)}?>*\dir{>};
(79,-7); (138,-7) **\crv{~*{.} (108,12)}?>*\dir{>};
(10,-23); (43,-23) **\crv{~*=<3pt>{\scalemath{0.3}{\circ}} (26,-52)}?>*\dir{>};
(45,-23); (93,-23) **\crv{~*=<3pt>{\scalemath{0.3}{\circ}} (69,-52)}?>*\dir{>};
(27,-23); (69,-23) **\crv{~*=<3pt>{\scalemath{0.3}{\circ}} (47,-52)}?>*\dir{>};
(71,-23); (129,-23) **\crv{~*=<3pt>{\scalemath{0.3}{\circ}} (100,-52)}?>*\dir{>};
(55,-23); (105,-23) **\crv{~*=<3pt>{\scalemath{0.3}{\circ}} (80,-52)}?>*\dir{>};
(80,-23); (140,-23) **\crv{~*=<3pt>{\scalemath{0.3}{\circ}} (110,-52)}?>*\dir{>};
\endxy
\end{gather*}
\caption{Multiplication Table}\label{fig:multTable}
\end{figure}

where
\begin{itemize}
  \item $(\mathcal{K}/p)_i=(\mathcal{K}/p)_{X_n,X_m,i}$, which decomposes into direct summands as the column below it,
  \item $\Omega^i_{X_1}=\Omega_{X_1/\Bbbk}^i$,
  \item a solid arrow “$\xy (0,0); (6,0) **\crv{(3,0)}?>* \dir{>}; \endxy$" denotes the multiplication by  $\pm\operatorname{dlog}x$ for some $x\in \mathcal{O}_{X_1}^{\times}$,
  \item a dotted arrow “$\xy (0,0); (6,0) **\crv{~*=<3pt>{\scalemath{0.9}{.}} (3,0)}?>* \dir{>}; \endxy$" denotes the isomorphism induced by multiplying a Bott element, and
    \item a hollow-dotted arrow “$\xy (0,0); (6,0) **\crv{~*=<3pt>{\scalemath{0.3}{\circ}} (3,0)}?>* \dir{>}; \endxy$"
    denotes the 0 map induced by multiplying a Bott element.
\end{itemize}

From the multiplication table, one sees that the mod $p$ relative $K$-theory in the range $i\leq p-3$ is generated by the summand  $\mathcal{O}_{X_1}\subset (\mathcal{K}/p)_{X_n,X_{n-1},i}$ for odd $i$ and $i=2$ and the multiplication by $(\mathcal{K}/p)_{X_n,1}$ and the Bott element in $(\mathcal{K}/p)_{X_n,2}$. We call these summands the \emph{basic summands}, and their elements the \emph{basic elements} (Definition \ref{def:basicElements}). To show the Chern character  isomorphism \eqref{eq:relative-comparison-local-mod-p-intro}  from the fact that the mod $p$ Brun map is an isomorphism, we develop an ad hoc trick in \S\ref{sub:functoriality_resolves_ambiguity} to resolve the ambiguity of the Chern character map on the basic elements. The idea originates from a simple observation: an endomorphism  of the abelian group $A=\mathbb{Z}/p^n \mathbb{Z}$ induces an automorphism on $A/pA$ if and only if it induces an automorphism on $\leftidx{_p}A$. This observation suffices to show Theorem \ref{thm:relative-comparison-local-intro} in the case $X_{\centerdot}=W_{\centerdot}(\Bbbk)$. In the case of relative dimension $>0$, it is difficult to classify the endomorphisms of the object $p^{r,m}_{r,n}\Omega^{\bullet}_{X_{\centerdot}}$ in the derived category. 
But we can employ the functoriality of the (relative) Chern character to show that the induced maps on mod $p$ cohomology of such endomorphisms have a uniform form which depends only on several universal constants in $\Bbbk$; this is reason that  we have to carefully define the \emph{canonical relative Chern classes} in \S\ref{sec:infinitesimal_chern_classes_and_chern_characters}. Then by studying the case $X_{\centerdot}=\mathbb{A}^1_{W_{\centerdot}(\Bbbk)}$ we determine the universal constants; see the proof of Theorem \ref{prop:functoriality-determine-universal-transform}. Using this and the multiplication table, we complete the proof of Theorem \ref{thm:relative-comparison-local-intro}.


\subsection{Further problems}\label{sub:further_problems} 
\begin{enumerate}[(i)]
  \item We expect that Theorem \ref{thm:relative-comparison-local-intro} holds for the Zariski $K$-sheaves as well. In fact, the only essential use of Nisnevich topology in the proof of Theorem \ref{thm:relative-comparison-local-intro} is in the exactness of \eqref{eq:fundamentalTriangle-S(r)-Nis}, which follows from a theorem of Kato (\cite[Remark 1 on Page 224]{Kat82}). But the syntomic pro-complexes are only used to construct the products in the infinitesimal motivic complexes and show some properties of them. We expect that the products can be constructed directly, and Proposition \ref{prop:properties-Zn(r)} holds also for the associated Zariski sheaves, and then the computations in \S\ref{sub:the_mod_p_product_structure}  work for the Zariski $K$-sheaves.
  \item One advantage of the Chern character isomorphism is that it is possible to have a local explicit expression.  In a future work, we are going to give an explicit formula, and hopefully it will lead  to  explicit constructions of the basic elements. We expect these to help understand the algebraization problem.
  \item Our approach to Theorem \ref{thm:relative-comparison-local-intro} limits us to the range no larger than $i<p$. Incorporating recent techniques on $K$-theory and topological cyclic homology  (see e.g. \cite{AKN24} and the references therein) might give more complete results. However, with the aim at VHC, such limitation of range is not an essential issue: it suffices that the bound on $i$ can be arbitrarily large when varying $p$ and is independent of $n$. Rather than extending this range, we are more concerned with the possibility of relaxing the smoothness condition in Theorem \ref{thm:relative-comparison-local-intro} to the semistable reductions $X \rightarrow \operatorname{Spec}W(\Bbbk)$. 
  This will help generalize VHC to this setting, and then Theorem \ref{thm:deformationInK-theory-intro}, especially the cases $m\geq 2$, due to the presence of vanishing cycles. Such generalizations might lead to new algebraicity results of Hodge cycles.
  \item Concerning the infinitesimal motivic complexes $\mathbb{Z}_{X_n}(r)$, we are aware of the construction of motivic complexes for smooth schemes over a general base in \cite{CiD19}, and the recent constructions of motivic complexes for non-smooth schemes in \cite{ElM23} and \cite{KeS24}. Our set-up seems  complementary to theirs, and comparatively close to \cite{Bou24}. It is interesting to investigate the relationship among these constructions.
  \item Our notation $\mathbb{Z}_{X_n}(r)$ is somewhat misleading: its definition depends on a choice of $X_{\centerdot}$, at least a choice of $X_{nr}$ (see Remark \ref{rem:define-infinitesimalMotivicComp-Xnr}), rather than only $X_n$. For example, if there is a morphism $f:X_{\centerdot} \rightarrow X'_{\centerdot}$ which induces an isomorphism $f/p^n: X_n \xrightarrow{\cong} X'_n$, it is not clear that $f$ induces a quasi-isomorphism $\mathbb{Z}_{X'_n}(r) \rightarrow \mathbb{Z}_{X_n}(r)$. But Theorem \ref{thm:relative-comparison-local-intro} implies a weaker result: the induced map $f^*:\tau^{\geq 2r-p+3}\big(p^{r,m}_{r,n}\Omega^{\bullet}_{X'_{\centerdot}}\big) \rightarrow \tau^{\geq 2r-p+3}\big(p^{r,m}_{r,n}\Omega^{\bullet}_{X_{\centerdot}}\big)$ on canonical truncations is a quasi-isomorphism. So we expect that $\mathbb{Z}_{X_n}(r)$ can be shown to depend only on $X_n$ for all $0\leq r<p$.
\end{enumerate}


\subsection{Organization of the paper} 
\label{sub:organization_of_the_paper}

In \S\ref{sec:products_on_cyclic_homology}, we review the algebraic and topological products in Hochschild and cyclic homology, and show a filtered comparison result used later. 
In \S\ref{sec:Brun's_isomorphism_and_the_product_structure}, we analyze  Brun's isomorphism connecting relative algebraic $K$-theory to derived cyclic homology and prove its compatibility with multiplicative structures, a key step in the proof of the main theorem.
In \S\ref{sec:reduction_to_cyclic_homology_over_Wn(k)}, we reduce the computation of derived cyclic homology for smooth algebras over \(W_n(\Bbbk)\) to computations over the base ring $W(\Bbbk)$. Many of the materials in \S\ref{sec:products_on_cyclic_homology}, and some in \S\ref{sub:products_in_equivariant_stable_homotopy_theory} and \S\ref{sec:reduction_to_cyclic_homology_over_Wn(k)}, are expository; we include them for the reader's convenience and for lack of results in the literature in the form we need.

In \S\ref{sec:derived_hochschild_homology_of_smooth_algebras}, we introduce a specific cdga resolution for a smooth algebra over $W_n(\Bbbk)$ and use the HKR isomorphim for smooth cdga's to construct a complex $\mathbf{CC}_\bullet(A/\ell)$ to compute the derived  cyclic homology of $A/\ell$ where $A$ is a smooth $W(\Bbbk)$-algebra.
In \S\ref{sec:relative_derived_cyclic_homology_of_smooth_algebras_over_texorpdfstring_w_n_bbbk_}, we compute the relative derived cyclic homology \(\widetilde{\mathrm{HC}}_i(A/p^L, A/p^M)\), where $A$ is a smooth $W(\Bbbk)$-algebra, leading to the discovery of the complexes \(p_{r,n}^{r,m}\Omega_X^\bullet\).
In \S\ref{sec:mod_p_rel_cyclic_homology_and_the_products}, we study the mod \(p\) product structures in relative cyclic homology, focusing on the actions of $K_1$ and the Bott element.
In \S\ref{sec:infinitesimal_motivic_complexes}, we define the infinitesimal motivic complexes $\mathbb{Z}_{X_n}(r)$ and their ring structure, which are the targets of the Chern character.
In \S\ref{sec:infinitesimal_chern_classes_and_chern_characters}, we  construct the Chern classes and the Chern character map valued in these infinitesimal motivic complexes and show that they are compatible with the Loday product of algebraic $K$-theory.
In \S\ref{sec:an_infinitesmial_chern_character_isomorphism}, we prove Theorem \ref{thm:relative-comparison-local-intro} by combining the previous computations, analyzing mod $p$ product structures, and resolving ambiguities via functoriality as we sketched above.
In  \S\ref{sec:consequences_in_k_theory_and_deformation_theory}, we show the two main consequences, Theorems \ref{thm:deformationInK-theory-intro} and \ref{thm:continuity-alg-K-theory-intro}.


\subsection{Notations and conventions} 
\label{sub:notations}

\begin{enumerate}[(i)]
  \item  We denote the simplicial $n$-sphere $\Delta^n/\partial \Delta^n$ by $S^n$, and the topological $n$-sphere by $\mathbf{S}^n$. In particular, $\mathbf{S}^1$ is the circle group.
\item We denote $\pi_i(K(X)^{\land}_p)$ by $K_i(X;p)$, and similarly for relative $K$-theory, cyclic homology, and topological cyclic homology. On the other hand, $K_i(X)^{\land}_p$ stands for the $p$-completion of the abelian group $K_i(X)$.
\item We denote \emph{the} cone of a map $f$ between two complexes by $\mathrm{Cone}(f)$. We denote \emph{a} cone of a map $f$ between two objects in a triangulated category by $\mathrm{cone}(f)$. The latter notion is in general only unique up to non-canonical isomorphisms, and thus is only well-defined under some additional conditions (see e.g. \cite[Lemma A.2]{BEK14}).
\item 
Let $\mathbf{A}$ be an additive category. A \emph{complex} of objects in $\mathbf{A}$ means a cochain complex. 
 We regard a complex $\dots\rightarrow C^i\rightarrow C^{i+1}\rightarrow \dots$   as a chain complex $\dots \rightarrow C_{-i}\rightarrow C_{-i-1}\rightarrow \dots$ by letting $C_{i}:=C^{-i}$, and vice versa. Moreover, we define the degree shift of a chain complex by the shift by the same degree of its corresponding cochain complex; thus $(C[n])_i=C_{i-n}$. For a  complex $C^{\bullet}$, denote by $C^{\leq n}$ its $n$-th naive truncation, namely the complex with $(C^{\leq n})^i=C^i$ when $i\leq n$, and $=0$ when $i>n$. In case that $\mathbf{A}$ is an abelian category, denote by $\tau^{\leq n}C$ the $n$-th canonical truncation. We adopt similar notations for canonical truncations of chain complexes. So, for example, for a chain complex $C$, we have $\tau_{\geq n}C=\tau^{\leq -n}C$.
\item For an abelian category $\mathbf{A}$, we let $s \mathbf{A}$ be the category of simplicial objects in $\mathbf{A}$, and $\mathrm{Ch}_+(\mathbf{A})$ the category of nonnegatively graded chain complexes of objects in $\mathbf{A}$. The Dold-Kan correspondence is the equivalence of categories
\begin{gather*}
\mathrm{N}: s \mathbf{A} \rightleftarrows \mathrm{Ch}_+(\mathbf{A}): \Gamma
\end{gather*}
and  $\Gamma$ is left adjoint to $\mathrm{N}$. For a chain complex $C$ and $n\in \mathbb{Z}$, let\footnote{Our notation differs from \cite[Page 212]{Jar15} by a sign, due to the above convention on the degree shifts of chain complexes.}
\begin{gather*}
\mathrm{K}(C,n):=\Gamma \circ\tau_{\geq 0}(C[n]).
\end{gather*}
When $n=0$, we abbreviate this notation by letting $\mathrm{K}(C):=\mathrm{K}(C,0)$.
\item For a field $\Bbbk$ of characteristic $p$, we denote by $\mathrm{Sm}_{W_{\centerdot}(\Bbbk)}$ the category of inductive systems of morphisms of schemes
\begin{gather*}
\xymatrix{
X_1 \ar[r] \ar[d] & X_2 \ar[r] \ar[d] & \dots \ar[r] \ar[d] & X_n \ar[r] \ar[d] & \dots \\
\operatorname{Spec} W_1(\Bbbk) \ar[r] & \operatorname{Spec} W_2(\Bbbk) \ar[r] & \dots \ar[r] & \operatorname{Spec} W_n(\Bbbk) \ar[r] & \dots 
}
\end{gather*}
satisfying that the morphisms $X_n \rightarrow \operatorname{Spec} W_n(\Bbbk)$ are smooth separated and of finite type, and every square in this diagram is cartesian. We denote this object simply by $X_{\centerdot}$ and call it a \emph{$p$-adic smooth scheme over $W_{\centerdot}(\Bbbk)$}. We say that $X_{\centerdot}$ has relative dimension $d$ if $\dim X_1=d$. If $X$ is a smooth scheme separated and of finite type over $W(\Bbbk)$ and $X_n=X\times_{\operatorname{Spec}W(\Bbbk)}\operatorname{Spec}W_n(\Bbbk)$ then we say that $X_{\centerdot}$ is the object of $\mathrm{Sm}_{W_{\centerdot}(\Bbbk)}$ \emph{associated with $X$}. We say that $X_{\centerdot}$ is quasi-projective (resp. projective) if $X_{\centerdot}$ is associated with a smooth quasi-projective (resp.  smooth projective) scheme over $W(\Bbbk)$.
\item For a scheme $X$, $\mathcal{K}^M_{X,r}$ is the improved Milnor $K$-sheaf defined by Kerz \cite{Ker10}. If $X$ is defined over an artinian ring with an infinite residue field, then $\mathcal{K}^M_{X,r}$ coincides with the classical (=naive) Milnor $K$-sheaf.
\end{enumerate}


\subsection{Acknowledgement}
I thank Mao Sheng for initiating a reading seminar on the paper \cite{BEK14} and for many helpful discussions on the VHC. 
I thank  Morten Brun for answering my  questions on his paper \cite{Bru01}, which is the starting point of this paper.
 I thank Guozhen Wang for giving a course on cyclic homology in SCMS in the 2019 autumn.
 I thank Esnault and Kerz for answering my  questions on \cite{BEK14}.
I thank Peng Du, Fangzhou Jin, Heng Xie, and Yigeng Zhao  for their encouragement and many valuable suggestions on this paper. I thank Mao Sheng, Junchao Shentu, Jinxing Xu, Jinbang Yang, and Lei Zhang (USTC) for organizing workshops on Hodge theory in every recent year, and for many helpful discussions on the Hodge conjecture. 
I thank the IWoAT organizers for providing the opportunities, including the very recent IWoAT 2025,  to keep in contact with developments in algebraic topology.
I also thank   Li Cai,  Huai-Liang Chang,  Hanlong Fang, Rixin Fang, John Greenlees, Jingbang Guo, Zhan Li, Sam Raskin, Longke Tang, Junyi Xie, Sen Yang, and Lei Zhang (SYSU),  for helpful discussions on various related problems. Parts of  this paper have been reported since 2024 spring in several workshops, and I thank the invitations by Wenfei Liu, Peng Sun, Yixin Bao, Heer Zhao, Jin Cao, Nanjun Yang, and Chuanhao Wei, as I learned a lot from the comments by the participants. This work is supported by NSFC 12371063, and CAS Project for Young Scientists in Basic Research, Grant No. YSBR-032.

\section{Products in Hochschild homology and cyclic homology} 
\label{sec:products_on_cyclic_homology}

In this section,  we recall several facts about the products in Hochschild homology and (resp. negative)  cyclic homology. These facts should be well-known to experts, and we give an exposition because of a lack of a complete account in the literature. There are two ways to define these products. One is defined on the Hochschild complex and the cyclic complex as given in \cite[Chap.~4, 5]{Lod98}, and another way uses that the realization of the Hochschild complex has a natural $\mathbb{S}^1$-action and the  (resp. negative)  cyclic homology is the homotopy orbit (resp. homotopy fixed points) and thus have induced products. We need the latter product because in \S\ref{sec:Brun's_isomorphism_and_the_product_structure} we need to compare this product with the product in $K$-theory, and we also need the former product because we need to compute the products explicitly in \S\ref{sec:products_on_cyclic_homology}. The problem concerning the coincidence of the two products has been proposed in \cite[Problems 11.8 and 14.8]{GrM95}.
In this section, we show a weak, essentially trivial, comparison result (Propositions \ref{prop:algebraic-and-topological-products-weak-comparison} and \ref{prop:algebraic-and-topological-products-weak-comparison-derived}), which suffices for the proof of Theorem \ref{thm:relative-comparison-local}.

We fix a commutative base ring $k$, and $\otimes$ without subscripts means $\otimes_{k}$. A cyclic $k$-module $M$ (\cite[\S2.5.1]{Lod98}) is a simplicial $k$-module with face operators $d_i$ and degeneracy operators $s_i$, and endowed with maps $t_n:M_n\rightarrow M_n$ satisfying
\begin{gather*}
\begin{cases}
t_n^{n+1}=\mathrm{id},\\
d_it_n=-t_{n-1}d_i\ \mbox{and}\ s_it_n=-t_{n+1}s_{i-1} & \mbox{for}\ 1\leq i\leq n,\\
d_0t_n=(-1)^n d_n\ \mbox{and}\ s_0t_n=(-1)^nt_{n+1}^2s_n.
\end{cases}
\end{gather*}
A $k$-algebra $A$ gives rise to a cyclic $k$-module $C(A)$ with $C(A)_n=A^{\otimes n+1}$ and
\begin{align*}
\mbox{faces}\ & b_i(a_0\otimes \dots\otimes a_n)=\begin{cases}
a_0\otimes \dots\otimes a_ia_{i+1}\otimes \dots a_n,& \mbox{if}\ 0\leq i\leq n-1\\
a_na_0\otimes \dots\otimes a_{n-1},& \mbox{if}\ i=n
\end{cases},\\
\mbox{degeneracies}\ & s_i(a_0\otimes\dots\otimes a_n)=a_0\otimes \dots a_i\otimes 1\otimes a_{i+1}\otimes \dots \otimes a_n,\ 0\leq i\leq n,\\
& t_n(a_0\otimes \dots\otimes a_n)=(-1)^n a_n\otimes a_{0}\otimes \dots a_{n-1}.
\end{align*}
On the graded  $k$-module $C(A)_*=\bigoplus_{n\geq 0} C(A)_n$ there are two operators 
\begin{align*}
b&=\sum_{i=0}^n(-1)^ib_i: C(A)_n\rightarrow C(A)_{n-1},\quad (C(A)_{-1}\ \mbox{is understood to be 0}) \\
B&=(1-t_{n+1})sN: C(A)_{n}\rightarrow C(A)_{n+1},
\end{align*}
where the norm operator $N:C(A)_n\rightarrow C(A)_n$ and the extra degeneracy operator $s:C(A)_n\rightarrow C(A)_{n+1}$ are defined as
\begin{gather*}
N=\sum_{i=0}^i t_n^i,\quad s(a_0\otimes\dots \otimes a_n)=1\otimes a_0\otimes \dots\otimes a_n.
\end{gather*}
The operators $b$ and $B$ satisfy
\begin{gather}\label{eq:b-B-complex}
b^2=0,\ B^2=0,\ bB+Bb=0.
\end{gather}
\begin{definition}\label{def:b-B-complex}
A 
\emph{mixed complex}  
of $k$-modules is a non-negatively graded $k$-module $C$ with  an operator $b$ of degree $-1$ and an operator $B$ of degree 1 satisfying \eqref{eq:b-B-complex}. 
\end{definition}
In the following, we let $\operatorname{HH}(A)$ denote the cyclic module $C(A)$, and let $C(A)$ denote the associated mixed complex. The Hochschild homology group of $A$ is defined to be the homology of the complex $C(A)$ in the $b$-direction, or equivalently to be the homotopy group of the simplicial set $\operatorname{HH}(A)$, namely
\begin{gather*}
\operatorname{HH}_n(A):=\pi_n| \operatorname{HH}(A)|=H_n(C(A),b).
\end{gather*}

Now let $C$ be a mixed complex of $k$-modules. Following \cite[\S5.1.7]{Lod98}, we define the  total cyclic complexes
\begin{subequations}
\begin{align}
&\mathrm{Tot}(\mathcal{B}C^-):=k[u]\widehat{\otimes} C, \label{eq:totComplex-cyclicHomology}\\ 
&\mathrm{Tot}(\mathcal{B}C^{\mathrm{per}}):=k[u,u^{-1}]\widehat{\otimes} C, \label{eq:totComplex-perCyclicHomology}\\
&\mathrm{Tot}(\mathcal{B}C):=k[u,u^{-1}]/uk[u,u^{-1}] \label{eq:totComplex-NegCyclicHomology}
\end{align}
\end{subequations}
with differentials $\partial=uB+b$. 
Here $u$ is endowed with degree $-2$, so that $\partial$  has degree $-1$, and $\widehat{\otimes}$ means the completed tensor product with respect to the $(u)$-topology. 
 The  negative  cyclic homology $\operatorname{HC}_*^-(C)$, the periodic cyclic homology $\operatorname{HC}_*^{\mathrm{per}}(C)$, and the cyclic homology $\operatorname{HC}_*(C)$ of  $C$  are defined to be the homology of these complexes respectively. Then the negative cyclic  (resp. periodic cyclic, resp. cyclic) homology of a $k$-algebra $A$ is defined to be that of $C(A)$. There is  a  map of complexes
\begin{gather*}
h:\mathrm{Tot}(\mathcal{B}C^-)\rightarrow (C,b),\quad
x=\sum_{i}x_{p+2i}u^i\mapsto x_p,
\end{gather*}
and a short exact sequence
\begin{gather*}
0\longrightarrow (C,b)\overset{I}{\longrightarrow} \mathrm{Tot}(\mathcal{B}C)\overset{S}{\longrightarrow} \mathrm{Tot}(\mathcal{B}C)[2]\longrightarrow 0,\\
I(x_p)= x_pu^0,\quad S(\sum_{i\geq 0}x_{p-2i}u^{-i})=\sum_{i\geq 0}x_{p-2-2i}u^{-i}
\end{gather*}
which in the case $C=C(A)$ induces the map
\begin{gather}\label{eq:map-h}
 h: \operatorname{HC}^-_p(A)\rightarrow \operatorname{HH}_p(A)
\end{gather}
and Connes' long exact sequence
\begin{gather}\label{eq:ISB-sequence}
\dots\xrightarrow{B} \operatorname{HH}_p(A)\xrightarrow{I} \operatorname{HC}_p(A)\xrightarrow{S} \operatorname{HC}_{p-2}(A)\xrightarrow{B}
\operatorname{HH}_{p-1}(A)\xrightarrow{I} \dots
\end{gather}

\subsection{The algebraic products} 
\label{sub:the_algebraic_products}

 For two mixed complexes of $k$-modules $C$ and $C'$, we define $C\otimes C'$ to be the mixed complex with
 \begin{gather*}
 (C\otimes C')_n=\bigoplus_{p+q=n}C_p\otimes C'_q,
 \end{gather*}
 and
 \begin{gather*}
 b=b\otimes 1+1\otimes b,\ B=B\otimes 1+1\otimes B.
 \end{gather*}
 Here we adopt the Koszul sign convention (\cite[\S1.0.15]{Lod98}):
 \begin{gather*}
 b(x_p\otimes y_q)=bx_p\otimes y_q+(-1)^p x_p\otimes by_q,\ \mbox{for}\ x_p\in C_p,\ y_q\in C'_q,
 \end{gather*}
 and similar for $B$. 
The external product on $\mathrm{HC}^{-}$ in \cite[\S 5.1.3]{Lod98} is defined as
\begin{gather*}
\times:\mathrm{Tot}(\mathcal{B}C^-)_p\otimes \mathrm{Tot}(\mathcal{B}C^{\prime-})_q\rightarrow \mathrm{Tot}\big(\mathcal{B}(C\otimes C')^-\big)_{p+q},\\
x\otimes_k y\mapsto x\underset{k[u]}{\widehat{\otimes}}y\nn
\end{gather*}
where
\begin{gather*}
x=\sum_{i\geq 0}x_{p+2i}u^{i},\ y=\sum_{j\geq 0}y_{q+2j}u^{j},
\end{gather*}
and $x_k\in C_p$ and $y_k\in C'_k$.
With similar notations, the induced external action of $\mathrm{HC}^-$ on $\mathrm{HC}$ is given by
\begin{gather*}
\times:\mathrm{Tot}(\mathcal{B}C^-)_p\otimes \mathrm{Tot}(\mathcal{B}C')_q\rightarrow \mathrm{Tot}\big(\mathcal{B}(C\otimes C')\big)_{p+q},\\
x\otimes_{k}y\mapsto x\underset{k[u,u^{-1}]}{\widehat{\otimes}}y\mod uk[u]\nn
\end{gather*}
where $y=\sum_{j\geq 0}y_{q-2j}u^{-j}$ is a finite sum by the definition \eqref{eq:totComplex-NegCyclicHomology}. 
The resulting external products on homology coincide with those defined by Hood and Jones \cite[\S2]{HoJ87}; this is a consequence of their homotopical uniqueness theorem (\cite[Theorem 2.3]{HoJ87}).
\begin{definition}[{\cite[\S5.1.13-5.1.15]{Lod98}}]\label{def:algebraicProduct-cyclicHomology}
Let $A$ be a $k$-algebra.
\begin{enumerate}[(i)]
  \item The \emph{algebraic product} in $\operatorname{HC}^-_*(A)$ is the map
\begin{gather*}
\operatorname{HC}^-_p(A)\otimes \operatorname{HC}^-_q(A)\rightarrow \operatorname{HC}^-_{p+q}(A)
\end{gather*}
induced by the composition
\begin{gather*}
\mathrm{Sh}^-\circ \times: \mathrm{Tot}(\mathcal{B}C(A)^-)_p\otimes \mathrm{Tot}(\mathcal{B}C(A)^-)_q\rightarrow \mathrm{Tot}(\mathcal{B}C(A)^-)_{p+q}
\end{gather*}
where
\begin{gather*}
\mathrm{Sh}^-=\mathrm{sh}+\mathrm{sh}^-:  \mathrm{Tot}\big(\mathcal{B}(C(A)\otimes C(A))^-\big)_{p+q}\rightarrow \mathrm{Tot}(\mathcal{B}C(A)^-)_{p+q}
\end{gather*}
and $\mathrm{sh}$ and $\mathrm{sh}^-$ are the shuffle product and the cyclic shuffle product respectively (\cite[\S4.2.3, \S4.3.6]{Lod98}).
  \item  The \emph{algebraic product} of $\operatorname{HC}^-_*(A)$ on $\operatorname{HC}_*(A)$  is the map
\begin{gather*}
\operatorname{HC}^-_p(A)\otimes \operatorname{HC}_q(A)\rightarrow \operatorname{HC}_{p+q}(A)
\end{gather*}
induced by the composition
\begin{gather*}
\mathrm{Sh}^-\circ \times: \mathrm{Tot}(\mathcal{B}C(A)^-)_p\otimes \mathrm{Tot}(\mathcal{B}C(A))_q\rightarrow \mathrm{Tot}(\mathcal{B}C(A))_{p+q}.
\end{gather*}
  \item Using the map $h$ defined in \eqref{eq:map-h}, we define the \emph{algebraic product} of $\operatorname{HC}^-_*(A)$ on $\operatorname{HH}_*(A)$ as
\begin{gather*}
\operatorname{HC}^-_p(A)\otimes \operatorname{HH}_q(A)\xrightarrow{h\otimes \mathrm{id}}\operatorname{HH}_p(A)\otimes \operatorname{HH}_q(A)\xrightarrow{\times}\operatorname{HH}_{p+q}(A)
\end{gather*}
where $\times$ is the product in $\operatorname{HH}_*(A)$ defined via the shuffle products (\cite[\S4.2.1]{Lod98}).
\end{enumerate}
\end{definition}

\begin{lemma}\label{lem:h-I-S-preserve-alg-products}
Let $A$ be a $k$-algebra. Then the maps $h$, $I$, and $S$ preserve products. Namely, the diagrams
\begin{gather*}
\xymatrix{
  \operatorname{HC}^-_p(A)\otimes \operatorname{HC}^-_q(A) \ar[d] \ar[r]^{\mathrm{id}\otimes h} & \operatorname{HC}^-_p(A)\otimes \operatorname{HH}_q(A) \ar[d] \\
  \operatorname{HC}^-_{p+q}(A)  \ar[r]^{h} &  \operatorname{HH}_{p+q}(A)
}
\end{gather*}
and
\begin{gather*}
\xymatrix{
  \operatorname{HC}^-_p(A)\otimes \operatorname{HH}_q(A) \ar[d] \ar[r]^{\mathrm{id}\otimes I} & \operatorname{HC}^-_p(A)\otimes \operatorname{HC}_q(A) \ar[d] \ar[r]^{\mathrm{id}\otimes S} & \operatorname{HC}^-_p(A)\otimes \operatorname{HC}_{q-2}(A) \ar[d] \\
  \operatorname{HH}_{p+q}(A)  \ar[r]^{I} &  \operatorname{HC}_{p+q}(A) \ar[r]^{S} & \operatorname{HC}_{p+q-2}(A)
}
\end{gather*}
are commutative, where the vertical arrows are the algebraic products.
\end{lemma}
\begin{proof}
(See also \cite[Theorem 4.3.4]{McC96}.)
For the first diagram, let $x=\sum_{i\geq 0}x_{p+2i}u^i\in \mathrm{Tot}\big(\mathcal{B}C(A)^-\big)_p$ and $y=\sum_{i\geq 0}y_{q+2i}u^i\in \mathrm{Tot}\big(\mathcal{B}C(A)^-\big)_q$. Then
\begin{gather*}
h(x.y)=h\circ \mathrm{Sh}^-(x\times y)= \mathrm{sh}(x_p\otimes y_q)=x.h(y)\quad.
\end{gather*}
For the left square in the  second diagram, let $x=\sum_{i\geq 0}x_{p+2i}u^i\in \mathrm{Tot}\big(\mathcal{B}C(A)^-\big)_p$ and $y_q\in C(A)_q$. Then
\begin{gather*}
I(x.y_q)=I\circ \mathrm{sh}(x_p\otimes y_q)= \mathrm{Sh}^-(x\times I(y_q)).
\end{gather*}
The proof for the right square in the  second diagram is similar, and we leave it to the reader.
\end{proof}

\subsection{The topological products and a filtered comparison} 
\label{sub:the_topological_products_and_a_filtered_comparison}
It is known that for a cyclic module $C$, the geometric realization of the underlying simplicial set of $C$ has a circle action, and the cyclic homology groups of a cyclic module $C$ can be defined as (see e.g. \cite{BuF86}, \cite{Jon87}, \cite{Goo85}, and \cite[Theorem 6.5.2.1]{DGM13}) 
\begin{gather*}
\mathrm{HC}^-_*(C)=\pi_* |C|^{h \mathbf{S}^1},\ \mathrm{HC}_*(C)=\pi_* |C|_{h\mathbf{S}^1},\ \mathrm{HC}^{\mathrm{per}}_*(C)=\pi_* |C|^{t \mathbf{S}^1},
\end{gather*}
where $|\cdot|^{tS^1}$ is the Tate construction (\cite[Introduction]{GrM95}).

\begin{definition}\label{def:EG}
Let $G$ be a topological group. Without otherwise stated, a $G$-space is a topological space with a \emph{left} $G$-action. We choose a contractible space with a free \emph{right} $G$-action and denote it by $EG$.  By transposing the right action, we also have a left $G$-action on $EG$, namely $g.x=x.g^{-1}$. Moreover, we choose a distinguished point of $EG$ and denote it by 0. (The point 0 should not be confused with the base point of $EG_+$.) If one use the one-sided bar construction $B(G/*)$ or $B(*/G)$ of $EG$, then there is a canonical choice of 0. In the case $G=\mathbf{S}^1$,we take $E \mathbf{S}^1=\varinjlim_n \mathbf{S}(\mathbb{C}^n)$, the colimit of the odd-dimensional spheres, where the transition maps are the inclusions $\mathbf{x}\mapsto (\mathbf{x},0)$, and take $0=(1,0,\dots)$. 
\end{definition}

\begin{definition}\label{def:topological-product}
Let $X$ be a pointed $G$-space. There is a left $G$-action on $\mathrm{Map}_*(EG_+,X)^{G}$ given by $(g.f)(x)=g.f(g^{-1}x)$.   We define the space of homotopy fixed points to be $X^{hG}=\mathrm{Map}_*(EG_+,X)^{G}$. We define  the space  of homotopy orbits to be $X_{hG}=EG_+\underset{G}{\land} X$, namely $EG_+\land X$ modulo the relation $(u,gx)\sim (ug,x)$ for $u\in EG_+$, $x\in X$ and $g\in G$. The point $0\in EG$ induces maps
\begin{gather*}
X^{hG}\rightarrow X,\quad f\mapsto f(0),\\
X\rightarrow X_{hG},\quad x\mapsto (0,x).
\end{gather*}
 Now let  $X$ and $Y$ be pointed $G$-spaces.
\begin{enumerate}[(i)]
  \item We define the external product on homotopy fixed points by 
\begin{gather*}
X^{hG}\land Y^{hG}\rightarrow (X\land Y)^{hG},\\
(f_1,f_2)\mapsto \big(w\mapsto f_1(w)\land f_2(w)\big).
\end{gather*}
If $X$ comes with a $G$-map $\mu: X\land X\rightarrow X$, then there is an induced map $\mu:(X\land Y)^{hG}\rightarrow X^{hG}$, and we define the  product on homotopy fixed points to be the composition 
\begin{gather*}
X^{hG}\land Y^{hG}\rightarrow (X\land Y)^{hG}\xrightarrow{\mu} X^{hG}.
\end{gather*}
  \item We define the external product of homotopy fixed points on homotopy orbits to be  
\begin{gather*}
X^{hG}\land Y_{hG}\rightarrow (X\land Y)_{hG},\\
(f,w,y)\mapsto \big(w, f(w),y\big).
\end{gather*}
If $X$ and $Y$ come with a $G$-map $\mu: X\land Y\rightarrow Y$, then there is an induced map $\mu:(X\land Y)_{hG}\rightarrow Y_{hG}$, and  we define the  product of the homotopy fixed points on the homotopy orbits to be the composition
\begin{gather*}
X^{hG}\land Y_{hG}\rightarrow (X\land Y)_{hG}\xrightarrow{\mu} Y_{hG}.
\end{gather*}
\end{enumerate}
\end{definition}

\begin{construction}\label{cons:tensor-product-cyclic-modules}
Let $M$ and $M'$ be simplicial $k$-modules, and denote their underlying simplicial sets by $\underline{M}$ and $\underline{M}'$, respectively. The product simplicial $k$-module $M\times M'$ has as $\{M_n\otimes M'_n\}_{n\geq 0}$ as the underlying simplicial set (\cite[\S1.6.8]{Lod98}). The obvious map of simplicial sets $\underline{M}\times \underline{M}'\rightarrow M\times M'$ and the canonical homeomorphism of geometric realizations $|\underline{M}\times \underline{M}'|\cong|M|\times |M'|$ induce a map  $|M|\times |M'|\rightarrow |M\times M'|$, which descends to a map 
\begin{gather}\label{eq:tensor-product-simplicial-modules}
|M|\land |M'|\rightarrow |M\times M'|.
\end{gather}
If $M$ and $M'$ are cyclic modules, then the involved maps of simplicial sets preserve the cyclic structures\footnote{One needs to take care of the sign convention on the cyclic modules and cyclic sets, see \cite[\S6.1.2.2]{Lod98}.}, hence this map is an $\mathbf{S}^1$-map and induces a map
\begin{gather*}
  |M|^{h\mathbf{S}^1}\land |M'|^{h\mathbf{S}^1}\rightarrow |M\times M'|^{h\mathbf{S}^1}.
\end{gather*}

\end{construction}

\begin{definition}
Let $A$ be a commutative $k$-algebra. 
There is a map of cyclic modules $\mathrm{HH}(A)\times \mathrm{HH}(A)\rightarrow \mathrm{HH}(A)$, such that
\begin{gather*}
 (a_0\otimes \dots \otimes a_p,b_0\otimes \dots\otimes b_p)\mapsto (a_0b_0\otimes\dots\otimes a_pb_p).
\end{gather*}
Then we apply Definition \ref{def:topological-product} to $X=Y=\mathrm{HH}(A)$ and use Construction \ref{cons:tensor-product-cyclic-modules} to  get maps
\begin{gather*}
|\mathrm{HH}(A)|^{h\mathbf{S}^1}\land |\mathrm{HH}(A)|^{h\mathbf{S}^1}\rightarrow |\mathrm{HH}(A)|^{h\mathbf{S}^1}, \label{eq-NHC-product}\\
|\mathrm{HH}(A)|^{h\mathbf{S}^1}\land |\mathrm{HH}(A)|_{h\mathbf{S}^1}\rightarrow |\mathrm{HH}(A)|_{h\mathbf{S}^1}. \label{eq-HC-action}
\end{gather*}
We call the induced maps
\begin{gather*}
\mathrm{HH}_p(A)\times \mathrm{HH}_q(A)\rightarrow \mathrm{HH}_{p+q}(A),\\
\mathrm{HC}^-_p(A)\times \mathrm{HC}^-_p(A)\rightarrow \mathrm{HC}^-_{p+q}(A),\\
\mathrm{HC}^-_p(A)\times \mathrm{HC}_q(A)\rightarrow \mathrm{HC}_{p+q}(A)
\end{gather*}
the \emph{topological products}. Moreover, the map $|\mathrm{HH}(A)|^{h\mathbf{S}^1}\rightarrow |\mathrm{HH}(A)|$ induces a map 
\begin{gather*}
h': \mathrm{HC}^-_*(A)\rightarrow \mathrm{HH}_*(A).
\end{gather*}
We get the induced topological product
\begin{gather*}
\mathrm{HC}^-_p(A)\times \mathrm{HH}_q(A)\xrightarrow{h'\times \mathrm{id}} \mathrm{HH}_p(A)\times \mathrm{HH}_q(A)\rightarrow \mathrm{HH}_{p+q}(A)
\end{gather*}
where the second map is the topological product in $\mathrm{HH}_*(A)$.
\end{definition}

\begin{lemma}\label{lem:comparison-h-I}
Let $A$ be a commutative $k$-algebra.
\begin{enumerate}[(i)]
  \item The map $h: \mathrm{HC}^-_*(A)\rightarrow \mathrm{HH}_*(A)$ is induced by the map $|\operatorname{HH}(A)|^{h\mathbf{S}^1}\rightarrow |\operatorname{HH}(A)|$.
  \item The map $I: \mathrm{HH}_*(A)\rightarrow \mathrm{HC}_*(A)$ is induced by the map $|\operatorname{HH}(A)|\rightarrow |\operatorname{HH}(A)|_{h\mathbf{S}^1}$.
  \item The map $S: \mathrm{HC}_*(A)\rightarrow \mathrm{HC}_{*-2}(A)$ is the  Thom-Gysin homomorphism (\cite[\S D.6]{Lod98}) induced by the $\mathbf{S}^1$-bundle $E\mathbf{S}^1_+\land X \rightarrow X_{h\mathbf{S}^1}$.
\end{enumerate}
\end{lemma}
This is essentially shown in \cite{BuF86} and \cite{Jon87}. A direct  proof can be extracted  from the proof of \cite[Theorem 6.5.2.1]{DGM13}.

\begin{lemma}\label{lem:I-S-preserves-product}
\begin{enumerate}[(i)]
  \item The map $I: \mathrm{HH}_*(A)\rightarrow \mathrm{HC}_*(A)$ preserves the topological multiplication by $\mathrm{HC}^-(A)$.
  \item  The map $S: \mathrm{HC}_*(A)\rightarrow \mathrm{HC}_{*-2}(A)$  preserves the topological multiplication by $\mathrm{HC}^-(A)$.
\end{enumerate}
\end{lemma}
\begin{proof}
By Lemma \ref{lem:comparison-h-I}(ii), the assertion (i) follows from the commutativity of the diagram
\begin{gather*}
\xymatrix{
  X^{hG}\land Y \ar[r] \ar[d] & X^{hG}\land Y_{hG} \ar[d] \\
  X\land Y \ar[r] & (X\land Y)_{hG}
}
\end{gather*}
where the arrows are the maps defined in Definition \ref{def:topological-product}. This commutativity can be checked directly from the formulas in this definition. The assertion (ii) follows from the naturality of the Thom-Gysin homomorphism and the $G$-map $EG_+\land X\land \mathrm{Map}_*(EG_+,Y)^G \rightarrow EG_+\land X\land Y$, $(a,x,f)\mapsto (a,x,f(a))$.
\end{proof}

\begin{lemma}\label{lem:smashProduct-shuffleProduct}
Via the isomorphism (denoted by $\gamma$) from $\pi_*$ of a simplicial $k$-module to $H_*$ of the associated Moore complex, the smash product in homotopy groups induced by \eqref{eq:tensor-product-simplicial-modules} is computed by the shuffle product. More precisely, let $M$ and $N$ be simplicial $k$-modules, then the diagram
\begin{gather*}
\xymatrix{
\pi_p(M)\times \pi_q(N) \ar[d]_{\cong}^{\gamma} \ar[r]^{\land} & \pi_{p+q}(M\times N) \ar[d]^{\cong}_{\gamma} \\
H_p(M)\times H_q(N) \ar[r]^{\mathrm{sh}} &  H_{p+q}(M\times N)
}
\end{gather*}
commutes. 
\end{lemma}
\begin{proof}
(See also the proof of \cite[Prop.~4.4.3]{McC96} and \cite[(3.18)]{Whi62}.)
Let $x$ be a $p$-simplex of $M$, and $y$ a $q$-simplex of $N$, such that $d_ix=0$ for $0\leq i\leq p$ and $d_jy=0$ for $0\leq j\leq q$. Then $x$ represents a map $\Delta^p\rightarrow M$, and thus induces a map $|\Delta^p|\rightarrow |M|$ which we denote by $|x|$. Similarly, $y$ induces a map $|y|:|\Delta^q|\rightarrow |N|$. Let $\tilde{f}:|\Delta^p|\times |\Delta^q|\rightarrow |M|\times |N|$ be the product map of $|x|$ and $|y|$. Then $\tilde{f}$ descends to a map $f:\mathbf{S}^{p+q}\rightarrow |M|\land |N|$, and the composition of $f$ with 
\eqref{eq:tensor-product-simplicial-modules}  represents $[x]\land [y]\in \pi_{p+q}(M\times N)$.

Recall that (\cite[Page 243]{McL63}) a $(p,q)$-shuffle is a partition $(\mu,\nu)$ of $[0,p+q-1]$ into two disjoint subset $\mu=\{\mu_1<\dots<\mu_p\}$ and $\nu=\{\nu_1<\dots<\nu_q\}$, and one  associates a sequence $(i_l,j_l)_{0\leq l\leq p+q}$ in $\mathbb{Z}^2$ with $(\mu,\nu)$ in the following way: $(i_0,j_0)=(0,0)$, and $(i_{l+1},j_{l+1})=(i_l+1,j_l)$ if $l\in \mu$ and $=(i_l,j_{l}+1)$ otherwise. The map $l\mapsto (i_l,j_l)$ linearly extends to a map $f_{\mu,\nu}:|\Delta^{p+q}|\rightarrow |\Delta^{p}|\times |\Delta^{q}|$, and with $(\mu,\nu)$ running over the $(p,q)$-shuffles, $f_{\mu,\nu}$'s form a simplicial subdivision of $|\Delta^{p}|\times |\Delta^{q}|$. Therefore
\begin{gather}\label{eq:smashProduct-shuffleProduct-1}
\gamma([x]\land[y])=
\gamma([f])=\sum_{(\mu,\nu)}(\pm)[\tilde{f}\circ f_{\mu,\nu}]
\end{gather}
where the sign means the orientation of $|\Delta^{p+q}|$ induced by $f_{\mu,\nu}$
But notice that
\begin{gather*}
\tilde{f}\circ f_{\mu,\nu}=s_{\nu_q}\cdots s_{\nu_1}|x|\times s_{\mu_p}\cdots s_{\mu_1}|y|.
\end{gather*}
Hence \eqref{eq:smashProduct-shuffleProduct-1} is nothing other than the shuffle map.
\end{proof}

\begin{lemma}\label{lem:comparison-product-HH-and-HNonHH}
Let $A$ be a commutative $k$-algebra.
\begin{enumerate}[(i)]
  \item The algebraic and the topological products in $\mathrm{HH}_*(A)$ coincide. 
  \item The algebraic and the topological products of $\mathrm{HC}^-_*(A)$ on $\mathrm{HH}_*(A)$ coincide. 
\end{enumerate}
\end{lemma}
\begin{proof}
Recall that the algebraic product on HH is defined by the shuffle product. Then Part (i) follows from Lemma \ref{lem:smashProduct-shuffleProduct}. Part (ii) follows from (i) and Lemma \ref{lem:comparison-h-I}(i).
\end{proof}

\begin{definition}\label{def:filtration_HC_ungraded}
Let $A$ be a  $k$-algebra.
We define an increasing filtration on $\mathrm{HC}_n(A)$ by 
\begin{gather*}
F_0 \mathrm{HC}_n(A)=0,\quad F_1 \mathrm{HC}_n(A)=\mathrm{Im}\big(\mathrm{HH}_n(A)\xrightarrow{I} \mathrm{HC}_n(A)\big),
\end{gather*}
 and for $i>1$ inductively
\begin{gather*}
F_i \mathrm{HC}_n(A)=S^{-1}F_{i-1} \mathrm{HC}_{n-2}(A)
\end{gather*}
where $S$ is the homomorphism in \eqref{eq:ISB-sequence}.
\end{definition}

\begin{proposition}[Filtered comparison]\label{prop:algebraic-and-topological-products-weak-comparison}
Let $A$ be a commutative $k$-algebra.
The algebraic and the topological multiplications of $\mathrm{HC}^-_*$ on $\operatorname{HC}_*(A)$ preserve the filtrations $F_i \operatorname{HC}_*(A)$
and  coincide on the quotients $F_i \operatorname{HC}_*(A)/F_{i-1}\operatorname{HC}_*(A)$ for $i\geq 1$.
\end{proposition}
\begin{proof}
The assertion for $i=1$ follows from Lemmas \ref{lem:comparison-h-I}(ii), \ref{lem:h-I-S-preserve-alg-products}, \ref{lem:I-S-preserves-product}(i), and  \ref{lem:comparison-product-HH-and-HNonHH}(ii). The map $S$ induces injective maps
\begin{gather*}
\frac{F_i\operatorname{HC}_n(A)}{F_{i-1}\operatorname{HC}_n(A)}\overset{S}{\longhookrightarrow} \frac{F_{i-1}\operatorname{HC}_{n-2}(A)}{F_{i-2}\operatorname{HC}_{n-2}(A)}\quad \mbox{for}\ i\geq 2.
\end{gather*}
By induction on $i$, the assertion  for $i\geq 2$ follows from Lemmas \ref{lem:h-I-S-preserve-alg-products} and \ref{lem:I-S-preserves-product}(ii).
\end{proof}


\subsection{The Hochschild and cyclic homology of cdga's}
\label{sub:the_hochschild_and_cyclic_homology_of_cdga_s} 
We recall the Hochschild and cyclic complexes of (c)dga's and  several related notions  from \cite[\S 5.3-\S5.4]{Lod98}. 

\label{sub:the_hochschild_and_cyclic_homology_of_cdgas}
\begin{definition}
A \emph{differential graded algebra}  over $k$ ($k$-dga for short) is a graded associative $k$-algebra $\mathscr{A}=\bigoplus_{n\geq 0}\mathscr{A}_n$ endowed with a degree $-1$ map $\updelta:\mathscr{A}_n\rightarrow \mathscr{A}_{n-1}$ satisfying
\begin{gather*}
\mathscr{A}_p\cdot \mathscr{A}_q\subset \mathscr{A}_{p+q},\\
\updelta(\mathscr{A}_0)=0,\ 
\updelta^2=0,\ \updelta(ab)=(\updelta a)b+(-1)^{|a|}a(\updelta b),
\end{gather*} 
for all $a,b\in \mathscr{A}$. 
A $k$-dga $\mathscr{A}$  is called a \emph{commutative differential graded algebra} ($k$-cdga for short) if it is \emph{graded commutative}, namely,
\begin{gather*}
ab=(-1)^{|a||b|}ba,\quad \mbox{for all $a,b\in \mathscr{A}$}.
\end{gather*}

\end{definition}
For an element $a\in \mathscr{A}_i$, let $|a|=i$, called the \emph{weight} of $a$. An element 
\begin{gather*}
x=(a_0,\dots,a_n):=a_0\otimes \cdots\otimes a_n\in \mathscr{A}^{\otimes n+1}
\end{gather*}
has \emph{length} $l(x)=n$, and \emph{weight} $|x|=\sum_{i=0}^n |a_i|$ if each $a_i$ has a pure weight.

The Hochschild complex $C(\mathscr{A})$ and cyclic complex $\mathcal{B}C(\mathcal{A})$ are formulated as the ungraded case, modified according to the Koszul sign convention (\cite[\S1.0.15]{Lod98}) according to weights. 
Thus $C(\mathscr{A})_n=\mathscr{A}^{\otimes_k n+1}$, and $b_n: C(\mathscr{A})_n \rightarrow C(\mathscr{A})_{n-1}$ is given by
\begin{gather*}
b_n(a_0,\dots, a_n)=(-1)^{|a_n|(\sum_{j=0}^{n-1}|a_j|)}(a_na_0,a_1,\dots,a_{n-1}),
\end{gather*}
while $b_i$ for $0\leq i\leq n-1$ and the degeneracies $s_i:C(\mathscr{A})_n \rightarrow C(\mathscr{A})_{n+1}$ are defined by the same formula as the ungraded case. 
Similarly, the cyclic operator $t_n: C(\mathscr{A})_n \rightarrow C(\mathscr{A})_{n}$ is given by
\begin{gather*}
t_n(a_0,\dots, a_n)=(-1)^n(-1)^{|a_n|\cdot(|a_0|+\dots+|a_{n-1}|)}(a_n\otimes a_0\otimes \dots\otimes a_{n-1}).
\end{gather*}
With these operators,  $C(\mathscr{A})_{\bullet}$ becomes a cyclic $k$-module.
The operator $\updelta$ extends to $C(\mathscr{A})_n$ by 
\begin{align*}
\updelta(a_0,\dots,a_n)=\sum_{i=0}^n(-1)^{\sum_{j=0}^{i-1}|a_j|}(a_0,\dots,\updelta a_i,\dots,a_n).
\end{align*}
One easily checks that 
\begin{equation}\label{eq:commutativity-delta-b-s-t}
  \begin{gathered}
\updelta\circ b_i=b_i\circ \updelta\ \mbox{and}\ \updelta\circ s_i=s_i\circ \updelta\quad \mbox{for}\ 0\leq i\leq n,\\
\mbox{and}\quad
\updelta\circ t=t\circ \updelta,\quad \updelta\circ s=s\circ \updelta.
\end{gathered}
\end{equation}
It follows that 
\begin{gather*}
\updelta\circ b=b\circ \updelta\quad \mbox{and}\quad \updelta\circ B=B\circ \updelta.
\end{gather*}
As in the ungraded case, we have $b\circ B+B\circ b=0$. By replacing $\updelta$ with an alternating $\pm\updelta=(-1)^{l(-)}\updelta$,  the Hochschild homology of $\mathscr{A}$ is defined  by the total complex of the bicomplex $C(\mathscr{A})$, with differential $b\pm \updelta$ (\cite[(5.2.3.1)]{Lod98}). The cyclic homology of $\mathscr{A}$ is computed by the total complex of a tricomplex $CC(\mathscr{A})$ with differential $b\pm \updelta+B$.

When we need  to emphasize the choice of the operator $\updelta$, we use notations such as $\operatorname{HH}(\mathscr{A},\updelta)$ and $\operatorname{HC}(\mathscr{A},\updelta)$. For example, for a dga $(\mathscr{A},\updelta)$, $(\mathscr{A},0)$ is also a dga, and $\operatorname{HH}(\mathscr{A},0)$ differs from the Hochschild homology of the underlying ungraded algebra $\mathscr{A}$.

\subsection{Products in Hochschild homology of cdga's} 
\label{sub:products_on_hochschild_homology_of_c_dga_s}
For a cdga, the shuffle product is defined by taking into account the Koszul convention according to weights. So for $\sigma\in S_n$, which only permutes $i$ and $i+1$, let 
\begin{gather*}
\sigma.(a_0,a_1,\dots,a_{n}):=(-1)^{|a_i||a_{i+1}|}(a_0,\dots,a_{i-1},a_{i+1},a_i,a_{i+2},\dots,a_n)
\end{gather*}
and for an arbitrary $\sigma\in S_n$ let $\sigma.(a_0,a_1,\dots,a_{n})$ be the composition of the permutations of neighboring indices. Then 
\begin{equation}\label{eq:shuffleProduct-cdga}
\begin{gathered}
(a_0,\dots,a_m)\times (a'_0,\dots,a'_n)\\
:=\sum_{\sigma\in \{(m,n)\text{-shuffles}\}}(-1)^{|a'_0|\sum_{i=1}^m|a_i|}
\mathrm{sgn}(\sigma)\sigma.(a_0a'_0,a_0,\dots,a_m,a'_0,\dots,a'_n).
\end{gathered}
\end{equation}
The cyclic shuffle product (\cite[\S4.3.2, \S4.3.6]{Lod98}) for cdga is defined similarly.
\begin{definition}
We define the \emph{algebraic product} in the cyclic homology of a cdga in the same way as the ungraded case in \S\ref{sec:products_on_cyclic_homology}, using the (cyclic) shuffle product recalled above.  
\end{definition}

\begin{definition}\label{def:topologicalProducts-dga}
Let $(\mathscr{A},\updelta)$ be a $k$-dga.
Since $\updelta$ commutes with the  operators $b_i,s_i,t_i$ by \eqref{eq:commutativity-delta-b-s-t}, the bisimplicial sets $\Gamma \big(C_{\bullet}(\mathscr{A}),\updelta\big)_{\bullet}$ in the first direction is a cyclic set. So  $\Gamma \big(C_{\bullet}(\mathscr{A}),\updelta\big)_{\bullet}$ is a simplicial cyclic set.
Let $(\mathscr{B},\updelta)$ be another  $k$-dga.
By \cite[\S2.3]{ScS03}, there is a natural map for each $i\geq 0$
\begin{gather}\label{eq:schwedeShipleyMap}
\Gamma \big(C_i(\mathscr{A}),\updelta\big)\otimes \Gamma \big(C_i(\mathscr{B}),\updelta\big) \rightarrow \Gamma \big(C_i(\mathscr{A})\otimes C_i(\mathscr{B}) ,\updelta\big)\underset{\tau}{\xrightarrow{\cong}} \Gamma \big(C_i(\mathscr{A}\otimes \mathscr{B}) ,\updelta\big),
\end{gather}
which induces a map of simplicial sets
\begin{gather*}
\operatorname{diag}\Gamma \big(C_{\bullet}(\mathscr{A}),\updelta\big)_{\bullet}\land \operatorname{diag}\Gamma \big(C_{\bullet}(\mathscr{B}),\updelta\big)_{\bullet}  \rightarrow \operatorname{diag}\Gamma \big(C_{\bullet}(\mathscr{A}\otimes \mathscr{B}),\updelta\big)_{\bullet}\quad.
\end{gather*}
Then by the equivalence $\operatorname{HH}(\mathscr{A},\updelta)\cong \operatorname{diag}\Gamma \big(C_{\bullet}(\mathscr{A}),\updelta\big)_{\bullet}$, we define the \emph{exterior topological product}
\begin{gather}\label{eq:topologicalProducts-HH-dga}
|\operatorname{HH}(\mathscr{A},\updelta)|\land |\operatorname{HH}(\mathscr{B},\updelta)| \rightarrow 
|\operatorname{HH}(\mathscr{A}\otimes \mathscr{B},\updelta)|.
\end{gather}
If $\mathscr{A}$ is a $k$-cdga, then the multiplication $\mathscr{A}\otimes_k \mathscr{A} \rightarrow \mathscr{A}$ is a cdga map, and we call the composition 
\begin{gather*}
\big|\operatorname{HH}(\mathscr{A},\updelta)\big|\land \big|\operatorname{HH}(\mathscr{A},\updelta)\big| \rightarrow \big|\operatorname{HH}(\mathscr{A}\otimes \mathscr{A},\updelta)\big| \rightarrow \big|\operatorname{HH}(\mathscr{A},\updelta)\big|,
\end{gather*}
and also the induced map and homotopy groups,  the \emph{topological product} in the Hochschild homology of $\mathscr{A}$.
\end{definition}

\begin{proposition}\label{prop:TopProduct-dga}
For $k$-dga's $\mathscr{A}$ and $\mathscr{B}$, the exterior topological product and the exterior algebraic product in Hochschild homology coincide. For a $k$-cdga $\mathscr{A}$, the  topological product and the algebraic product in Hochschild homology coincide.
\end{proposition}
\begin{proof}
The first assertion follows from Lemma \ref{lem:smashProduct-shuffleProduct}; the sign factors in \eqref{eq:shuffleProduct-cdga} arising from exchanging elements with weights are due to the map $\tau$ in \eqref{eq:schwedeShipleyMap}. 
\end{proof}

\subsection{Products in derived Hochschild homology and cyclic homology
} 
\label{sub:products_in_the_derived_hochschild_and_cyclic_homology_comparison_of_products}

Let $A$ be a  $k$-algebra. 
By definition (e.g. \cite[\S7]{Bru01}), the derived Hochschild homology $\widetilde{\mathrm{HH}}(A)$ (resp.  derived cyclic homology $\widetilde{\mathrm{HC}}(A)$) is defined to be  $\mathrm{HH}(P_{\bullet})$ 
(resp. $\mathrm{HC}(P_{\bullet})$, resp. $\mathrm{HC}^{-}(P_{\bullet})$), where $P_{\bullet}\rightarrow A$ is a  resolution of $A$ by a simplicial  $k$-algebra $P_{\bullet}$ which is degreewise flat over $k$. 
The following is a folklore result.

\begin{proposition}\label{lem:equivalence-Shukla-derivedHH-HC}
Suppose $\mathscr{A} \rightarrow A$ is a flat resolution of $A$ by a $k$-dga.
Then there are canonical equivalences 
\begin{gather}\label{eq:equivalence-Shukla-derivedHH-HC}
\widetilde{\mathrm{HH}}(A)\simeq \mathrm{HH}(\mathscr{A}),\ 
\widetilde{\mathrm{HC}}(A)\simeq \mathrm{HC}(\mathscr{A}),\
\widetilde{\mathrm{HC}}^{-}(A)\simeq \mathrm{HC}^{-}(\mathscr{A}).
\end{gather}
\end{proposition}
\begin{proof}
There is a projective model structure on the category of dga's (see e.g. \cite[\S4.2]{ScS03} and references therein). 
We take an induced Reedy model structure on the category of simplicial dga's. Then one can find a common cofibrant simplicial dga covering both $\mathscr{A}$ and $P_{\bullet}$ by trivial fibrations. The conclusion thus follows from standard homological algebra.
\end{proof}

\begin{definition}
Let $A$ and $B$ be  $k$-algebras, and $P_{\bullet} \rightarrow A$ and  $Q_{\bullet} \rightarrow B$ be  resolutions  by simplicial commutative  $k$-algebras  which are degreewise flat over $k$.
There are natural maps
\begin{gather*}
C_j(P_i)\otimes C_j(Q_i) \rightarrow C_j(P_i\otimes Q_i),
\end{gather*}
which induces a map
\begin{gather*}
\mathrm{diag} \operatorname{HH}(P_{\bullet}) \land \mathrm{diag} \operatorname{HH}(Q_{\bullet}) \rightarrow 
\mathrm{diag} \operatorname{HH}(P_{\bullet}\otimes Q_{\bullet}),
\end{gather*}
and thus the  \emph{exterior topological product}
\begin{gather}\label{eq:exteriorProduct-derivedHH-simpResolution}
\big|\widetilde{\operatorname{HH}}(A)\big|\land \big|\widetilde{\operatorname{HH}}(B)\big| \rightarrow \big|\widetilde{\operatorname{HH}}(A\otimes^{\mathbf{L}}_k B)\big|.
\end{gather}
By the constructions in \S\ref{sub:the_topological_products_and_a_filtered_comparison}, there are also \emph{exterior topological products}
\begin{gather*}
\big|\widetilde{\operatorname{HC}}^{-}(A) \big|\land \big|\widetilde{\operatorname{HC}}^{-}(A)\big| \rightarrow \big|\widetilde{\operatorname{HC}}^{-}(A\otimes^{\mathbf{L}}_k B)\big|
\end{gather*} 
and
\begin{gather*}
\big|\widetilde{\operatorname{HC}}^{-}(A)\big| \land \big|\widetilde{\operatorname{HC}}(A)\big| \rightarrow \big|\widetilde{\operatorname{HC}}(A\otimes^{\mathbf{L}}_k B)\big|.
\end{gather*} 
When $A$ is a commutative $k$-algebra, the compositions of the above maps with the maps induced by $A\otimes^{\mathbf{L}}_k A \rightarrow A$ yield inner \emph{topological products}. 
\end{definition}
\begin{definition}
Let $A$ and $B$ be  $k$-algebras, and $\mathscr{A} \rightarrow A$ and  $\mathscr{B} \rightarrow B$ be  resolutions  by    $k$-dga's  which are degreewise flat over $k$. We define \emph{exterior algebraic products}, and inner \emph{algebraic products} if $A$ is commutative, in the derived Hochschild homology and derived cyclic homology of $A$ and $B$ by those of $\mathscr{A}$ and $\mathscr{B}$ and the isomorphisms \eqref{eq:equivalence-Shukla-derivedHH-HC}.
\end{definition}

For $k$-algebras $A$ and $B$, we can also take flat dga resolutions $\mathscr{A} \rightarrow A$ and $\mathscr{B} \rightarrow B$, and define an exterior product  $|\widetilde{\operatorname{HH}}(A)|\land |\widetilde{\operatorname{HH}}(B)| \rightarrow |\widetilde{\operatorname{HH}}(A\otimes^{\mathbf{L}}_k B)|$ by the equivalence \eqref{eq:equivalence-Shukla-derivedHH-HC} and the topological exterior products \eqref{eq:topologicalProducts-HH-dga}.

\begin{proposition}\label{eq:comparison-topologicalProduct-simplicalRes-vs-dgaRes}
For $k$-algebras $A$ and $B$, the exterior topological product \eqref{eq:exteriorProduct-derivedHH-simpResolution}  in the derived Hochschild homology  is independent of the choice of the simplicial  resolution $P_{\bullet}$ and $Q_{\bullet}$, and  coincides with the topological product defined by flat dga resolutions. The same holds for the (inner) topological product in the derived Hochschild homology of a $k$-algebra.
\end{proposition}
\begin{proof}
One can define the exterior product by a \emph{simplicial dga resolution}, which is compatible with both the exterior product \eqref{eq:exteriorProduct-derivedHH-simpResolution} and the topological product defined by a dga resolution. Then, by taking a common simplicial dga resolution in the  projective model  category of simplicial dga's, the first assertion follows from Proposition \ref{prop:TopProduct-dga}. The second assertion follows from the first one.
\end{proof}
The same argument yields
\begin{proposition}
For $k$-algebras $A$ and $B$, the exterior algebraic products,  and inner algebraic products if $A$ is commutative,  are independent of the choices of the flat (c)dga resolutions.
\end{proposition}

The exact sequence \eqref{eq:ISB-sequence} for $P_i$ yields an exact sequence of derived homologies
\begin{gather*}
\dots\xrightarrow{B} \widetilde{\operatorname{HH}}_p(A)\xrightarrow{I} \widetilde{\operatorname{HC}}_p(A)\xrightarrow{S} \widetilde{\operatorname{HC}}_{p-2}(A)\xrightarrow{B}
\widetilde{\operatorname{HH}}_{p-1}(A)\xrightarrow{I} \dots
\end{gather*}
and we define a filtration on $\widetilde{\operatorname{HC}}_n(A)$ as in Definition \ref{def:filtration_HC_ungraded}.

\begin{proposition}[Filtered comparison]\label{prop:algebraic-and-topological-products-weak-comparison-derived}
Let $A$ be a $k$-algebra. Then the algebraic and the topological multiplications of $\widetilde{\mathrm{HC}}^-_*(A)$ on $\widetilde{\operatorname{HC}}_*(A)$ preserve the filtrations $F_i \widetilde{\operatorname{HC}}_*(A)$  and  coincide on $F_i \widetilde{\operatorname{HC}}_*(A)/F_{i-1}\widetilde{\operatorname{HC}}_*(A)$ for $i\geq 1$.
\end{proposition}
\begin{proof}
This follows from Propositions \ref{eq:comparison-topologicalProduct-simplicalRes-vs-dgaRes} and \ref{prop:TopProduct-dga}.
\end{proof}

\begin{variant}\label{var:generalizations_to_dgas_simplicial_rings_and_homology_with_coefficients}
In this paper, we need to use the algebraic and the topological multiplications on relative Hochschild and cyclic homology, and these homology with coefficients $\mathbb{Z}/p \mathbb{Z}$. One easily deduces these versions of filtered comparison from the arguments in this section. 
\end{variant}



\section{Brun's isomorphism and the multiplicative structure}
\label{sec:Brun's_isomorphism_and_the_product_structure}
From now on, the default base ring of Hochschild and cyclic homology will be $\mathbb{Z}$.
In this section, we recall Brun's isomorphism theorem in \cite{Bru01}, and show some properties of this isomorphism that will be used in the proof of Theorem \ref{thm:relative-comparison-local}.
\begin{theorem}[{\cite[Theorem 6.1]{Bru01}}]\label{thm:Brun}
Let $R$ be a simplicial ring with an ideal $I$ satisfying $I^m=0$. 
Then there are natural isomorphisms and surjections 
\begin{align*}
K_i(R,I;p)&\cong \widetilde{\mathrm{HC}}_{i-1}(R,I;p)\quad \mbox{for}\ 0\leq i\leq \lceil\frac{p}{m-1}\rceil-3,\\
K_i(R,I;p)&\twoheadrightarrow \widetilde{\mathrm{HC}}_{i-1}(R,I;p)\quad \mbox{for}\ i\leq \lceil\frac{p}{m-1}\rceil-2,
\end{align*}
where $\lceil r\rceil:=\min\{s\in \mathbb{Z}| s\geq r\}$ for $r\in \mathbb{R}$.
\end{theorem}

We call the above maps \emph{Brun maps}, and denote them by
\begin{gather*}
\mathrm{br}:  K_i(R,I;p) \rightarrow \widetilde{\mathrm{HC}}_{i-1}(R,I;p).
\end{gather*}

\begin{remark}\label{rem:Brun'sThm-flatness}
(See the introduction of \cite{Bru01} for another explanation.)
In the statement of \cite[Theorem 6.1]{Bru01}, it is assumed that both $R$ and $R/I$ are degreewise flat over $\mathbb{Z}$. We drop this flatness assumption because it is only used in \cite[Lemma 6.3]{Bru01}, and when we replace the underived Hochschild homology with the derived one, and use the spectral sequence $E_{p,q}^2=\widetilde{\mathrm{HH}}\big(R, \operatorname{THH}(\mathbb{Z},R)\big)\Longrightarrow \pi_{p+q}\operatorname{THH}(R)$ (\cite[Remark 4.2(a)]{PiW92}),  this lemma and its proof remain valid without the flatness assumption.
\end{remark}

We refer the reader to \cite[\S 2, \S 5]{Bru01} for the definition and notations for \emph{filtered topological cyclic homology} associated with the pair $(R,I)$; see also \S\ref{sub:a_recap_of_filtered_thh_and_filtered_tc}.
Recall the following graph in the proof of Brun's theorem (see \cite[Proposition 6.6]{Bru01})
\begin{equation}\label{eq:zigzag-graph-Brun-map}
  \begin{gathered}
\xymatrix{
  \mathrm{TC}(R)(-1)\ar[r]^<<<<{e_I} \ar@{=}[d] & \frac{\mathrm{TC}(R)(-1)}{\mathrm{TC}(R)(-p)}\xleftarrow{f_I} S^1\land \Big(\frac{\mathrm{THH}(R)(-1)}{\mathrm{THH}(R)(-p)}\Big)_{hS^1} & S^1\land \mathrm{THH}(R)(-1)_{hS^1} \ar[l]_<<<<{g_I} \ar@{=}[d] \\
  \mathrm{TC}(R,I) & S^1\land\mathrm{HC}(R,I)\cong S^1\land \mathrm{HH}(R,I)_{hS^1}  & S^1\land \mathrm{THH}(R,I)_{hS^1} \ar[l]_<<<<<<{h_I}
  }
\end{gathered}
\end{equation}
where
\begin{itemize}
  \item the $p$-completion of  $e_I$ is ($\lceil \frac{p}{m-1}\rceil -2$)-connected (by \cite[Lemmas 5.2 and 6.5]{Bru01}),
  \item the $p$-completion of $f_I$ is an equivalence (\cite[Lemma 5.3]{Bru01}),
  \item the $p$-completion of  $g_I$ is $(\lceil \frac{p}{m-1}\rceil)$-connected (see Lemma \ref{lem:connectedness-g_I}),
  \item 
  the $p$-completion of  
  $h_I$ is $2p$-connected (by \cite[Lemma 6.3]{Bru01} and that $\mathrm{THH}_{2p-1}(R;p) \twoheadrightarrow \widetilde{\operatorname{HH}}_{2p-1}(R;p)$).
\end{itemize}
\begin{lemma}\label{lem:connectedness-g_I}
The $p$-completion of $g_I$ is $\lceil \frac{p}{m-1}\rceil$-connected.
\end{lemma}
\begin{proof}
By \cite[Lemma 6.5]{Bru01}, $\mathrm{THH}(R)(-p)$ is $(\lceil \frac{p}{m-1}\rceil-2)$-connected. Then by the spectral sequence for homotopy orbits $E_2^{s,t}=H_s(BS^1,\pi_t X)\Longrightarrow \pi_{s+t}(X_{hS^1})$,  $\mathrm{THH}(R)(-p)_{hS^1}$ is also $(\lceil \frac{p}{m-1}\rceil-2)$-connected, and so is its $p$-completion. Smashing with $S^1$, we get the claimed connectedness.
\end{proof}

In \cite{Bru01}, Theorem \ref{thm:Brun} is deduced from the diagram \eqref{eq:zigzag-graph-Brun-map} and McCarthy's theorem \cite{McC97}  (see also \cite{Dun97}): the cyclotomic trace induces a natural equivalence  $K(R,I;p)\simeq \operatorname{TC}(R,I;p)$. By \cite{DGM13}, this holds without passing to $p$-completions.

\begin{remark}\label{rem:BrunIsom-withCoeff}
By the universal coefficient exact sequence for homotopy groups, the function space functor $F(P^m(\mathbb{Z}/\ell),-)$ preserves $n$-connectedness of maps, where $P^m(\mathbb{Z}/\ell)$ is a Moore space. Thus, applying this functor to \eqref{eq:zigzag-graph-Brun-map}, we obtain that Theorem \ref{thm:Brun} holds for homotopy groups with coefficients, and we denote the resulting map still by $\mathrm{br}$. 
\end{remark}

\subsection{On the Brun map for the relative \texorpdfstring{$K_1$}{K1}} 
\label{sub:the_explicit_brun_map_on_relative_k1}
In the proof of our main Theorem \ref{thm:relative-comparison-local}, we need to know about the Brun map in the case $i=1$. Let $R$ be a commutative ring and $I$ be an ideal of $R$ such that $I^m=0$. Then $K_0(R,I)=0$ and $K_1(R,I)=1+I$. By \cite[Theorem 11.1.2 and Proposition 10.1.11]{MaPo12}, if the condition
\begin{equation*}
  K_{i-1}(R,I)\ \mbox{and}\ K_i(R,I)\ \mbox{have bounded $p$-exponents} 
\end{equation*}
is satisfied then $K_i(R,I;p) \cong K_i(R,I)^{\land}_p$, and if the condition
\begin{equation*}
  \widetilde{\mathrm{HC}}_{i-1}(R,I)\ \mbox{and}\ \widetilde{\mathrm{HC}}_i(R,I)\ \mbox{have bounded $p$-exponents} 
\end{equation*}
is satisfied  then $\widetilde{\mathrm{HC}}_i(R,I;p) \cong \widetilde{\mathrm{HC}}_i(R,I)^{\land}_p$.

\begin{lemma}\label{lem:relativeHC0}
Let  $R$ be a commutative ring, and $I$ be a nilpotent ideal of $R$. Then there is a canonical isomorphism $\widetilde{\mathrm{HC}}_0(R,I)\cong I$.
\end{lemma}
\begin{proof}
Let $S$ be a simplicial resolution of $R$ by commutative flat $\mathbb{Z}$-algebras. Then $\widetilde{\mathrm{HC}}_i(R)$ is computed by the double complex 
\begin{gather*}
\xymatrix{
\dots & \dots & \dots & \dots \\
\dots & \dots \ar[r] \ar[d] & \dots \ar[r] \ar[d] & \mathrm{Tot}(\mathcal{B}S_0)_2 \ar[d] \\
\dots & \dots \ar[r] \ar[d] & \mathrm{Tot}(\mathcal{B}S_1)_1 \ar[r] \ar[d] & \mathrm{Tot}(\mathcal{B}S_0)_1 \ar[d] \\
\dots & \mathrm{Tot}(\mathcal{B}S_2)_0 \ar[r] & \mathrm{Tot}(\mathcal{B}S_1)_0 \ar[r] & \mathrm{Tot}(\mathcal{B}S_0)_0
}
\end{gather*}
where the $i$-th row horizontal sequence is the Moore complex associated with the simplicial abelian group $\mathrm{Tot}(\mathcal{B}S_{\centerdot})_i$. Since $\mathrm{Tot}(\mathcal{B}S_i)_0=S_i$, the 0-th row is exact away from $i=0$ and has $H_0=S_0/S_1=R$. Since $S_i$ are commutative, the vertical arrows $\mathrm{Tot}(\mathcal{B}S_i)_1 \rightarrow \mathrm{Tot}(\mathcal{B}S_i)_0$ are 0. Recall that $H_1(\mathrm{Tot}(\mathcal{B}S_i))= \mathrm{HC}_1(S_i)=\Omega_{S_i/\mathbb{Z}}^1/ \mathrm{d} S_i$ (\cite[Proposition 2.1.14]{Lod98}). Then a simple diagram chase yields $\widetilde{\mathrm{HC}}_0(R)\cong R$
and
\begin{gather*}
\widetilde{\mathrm{HC}}_1(R)=\frac{\Omega_{S_0/\mathbb{Z}}^1/ \mathrm{d} S_0}{\Omega_{S_1/\mathbb{Z}}^1/ \mathrm{d} S_1}=\Omega_{R/\mathbb{Z}}^1/ \mathrm{d} R.
\end{gather*}
Applying these results to $R$ and $R/I$ in the  exact sequence 
\begin{gather*}
\dots \rightarrow \widetilde{\mathrm{HC}}_1(R) \rightarrow \widetilde{\mathrm{HC}}_1(R/I) \rightarrow \widetilde{\mathrm{HC}}_0(R,R/I) \rightarrow \widetilde{\mathrm{HC}}_0(R) \rightarrow \widetilde{\mathrm{HC}}_0(R/I),
\end{gather*}
we obtain $\widetilde{\mathrm{HC}}_0(R,R/I)\cong I$.
\end{proof}
Suppose now that $I^2=0$ and $p\geq 5$. If $I$ has bounded $p$-exponents, then $K_1(R,R/I;p)\cong I^{\land}_p$ and $\widetilde{\mathrm{HC}}_0(R,I;p)\cong I^{\land}_p$. We need to compare the Brun map with the obvious identification of $I^{\land}_p$. 
The following partial result turns out to be  enough for us.
\begin{proposition}\label{prop:Brun-map-K1}
There exists $u\in \mathbb{Z}_p^{\times}$ such that for any  commutative ring $R$  and any square zero ideal $I$ of $R$ satisfying that $I$ has bounded $p$-exponents, under the identifications  $K_1(R,R/I)=1+I$ and $\widetilde{\mathrm{HC}}_0(R,R/I)=I$, Brun's isomorphism $K_1(R,R/I;p)\xrightarrow{\sim}\widetilde{\mathrm{HC}}_0(R,I;p)$ is given by $1+x\mapsto u x$. In other words, it is given by $I/p^n \rightarrow I/p^n$, $x\mapsto (u\ \mathrm{mod\ } p^n)\times x$ for all $n\in \mathbb{N}$.
\end{proposition}
\begin{proof}
For any $x\in I$, there is an induced ring map $\mathbb{Z}[X]/(X^2)\rightarrow R$ sending $X$ to $x$. So it suffices to find the image of $1+X$ under Brun's isomorphism. Since an automorphism of the $\mathbb{Z}_p$-module $(X)\otimes \mathbb{Z}_p$ is determined by $u\in \mathbb{Z}_p^{\times}$, we are done.
\end{proof}

\begin{remark}\label{rem:universalUnit-u=1}
We expect that a study of the proof of \cite[Lemma 5.3]{Bru01} and a comparison of the norm map $N:S^1\land \mathrm{HH}_{hS^1}\rightarrow \mathrm{HH}^{hS^1}$ with Connes' $B$-operator will show $u=\pm 1$.
\end{remark}


\subsection{Compatibility of the Brun map with the products} 
\label{sub:compatibility_of_the_brun_map_with_the_products}
In the proof of the main theorem \ref{thm:relative-comparison-local}, we need to build up a large portion of the mod $p$ relative $K$-theory from relative $K_1$ and a certain Bott element, using the product structure. 
The main task of \S \ref{sec:Brun's_isomorphism_and_the_product_structure} is to show the following compatibility result. 
\begin{proposition}\label{prop:brun-iso-product}
Let $R$ be a simplicial ring with an ideal $I$ satisfying $I^m=0$.
Let $i,j\in \mathbb{Z}$ satisfy $0\leq i\leq i+j<\frac{p}{m-1}-1$.
Then the diagram
\begin{gather*}
\xymatrix@C=3pc{
  K_i(R)\times K_j(R,I;p) \ar[d] \ar[r]^<<<<<<{\mathrm{ch}^{-}\times \mathrm{br}} & \widetilde{\operatorname{HC}}^{-}_i(R)\times\widetilde{\operatorname{HC}}_{j-1}(R,I;p)\ar[d]^{\overset{\operatorname{top}}{\times}} \\
  K_{i+j}(R,I;p)  \ar[r]^{\mathrm{br}} & \widetilde{\operatorname{HC}}_{i+j-1}(R,I;p)
}
\end{gather*}
is commutative, where the left vertical arrow is  Loday's products \cite{Lod76} in algebraic $K$-theory  and  $\overset{\operatorname{top}}{\times}$ is the product on completions induced by the topological product in \S\ref{sub:products_in_the_derived_hochschild_and_cyclic_homology_comparison_of_products}, and $\mathrm{ch}^{-}$ is the Chern character (\cite[\S11.4.1]{Lod98}).
\end{proposition}
By \cite[Lemma 6.3.1.1]{DGM13}, there is a natural equivalence 
\begin{gather*}
(\operatorname{THH}(R)^{h \mathbf{S}^1}\big)_p^{\wedge} \xrightarrow{\simeq} \big(\holim_{F}\operatorname{THH}(R)^{h C_{p^n}}\big)_p^{\wedge}
\end{gather*}
and thus a map 
\begin{gather}\label{eq:pCompletionMap-TC-to-HC}
\operatorname{TC}(R)_p^{\wedge} \rightarrow (\operatorname{THH}(R)^{h \mathbf{S}^1}\big)_p^{\wedge} \rightarrow (\operatorname{HH}(R)^{h \mathbf{S}^1}\big)_p^{\wedge}=(\operatorname{HC}(R)\big)_p^{\wedge}\quad .
\end{gather}
Since the cyclotomic trace $\mathrm{trc}$ is a ring spectrum map \cite[Prop.~6.3.1]{GeH99}, the diagram
\begin{gather*}
\xymatrix@C=3pc{
  K_i(R)\times K_j(R,I;p) \ar[d] \ar[r]^<<<<<<{\mathrm{trc}\times \mathrm{trc}} & \operatorname{TC}_i(R)\times\operatorname{TC}_j(R,I)\ar[d] \\
  K_{i+j}(R,I;p)  \ar[r]^{\mathrm{trc}} & \operatorname{TC}_{i+j}(R,I)
}
\end{gather*}
commutes. Therefore it suffices to show that  the diagram
\begin{equation}\label{eq:brun-iso-product-TC-to-HC}
  \begin{gathered}
\xymatrix@C=3pc{
   \operatorname{TC}_i(R;p)\times\operatorname{TC}_j(R,I;p) \ar[d] \ar[r]^<<<<<<{\eqref{eq:pCompletionMap-TC-to-HC}\times \mathrm{br}} & \widetilde{\operatorname{HC}}^{-}_i(R;p)\times\widetilde{\operatorname{HC}}_{j-1}(R,I;p)\ar[d]^{\overset{\operatorname{top}}{\times}} \\
  \operatorname{TC}_{i+j}(R,I;p)  \ar[r]^{\mathrm{br}} & \widetilde{\operatorname{HC}}_{i+j-1}(R,I;p)
}
\end{gathered}
\end{equation}
commutes, where we denote the map $\operatorname{TC}_{i}(R,I;p) \rightarrow \operatorname{HC}_{i-1}(R,I;p)$ induced by the diagram \eqref{eq:zigzag-graph-Brun-map} also by $\mathrm{br}$. The proof of the commutativity of \eqref{eq:brun-iso-product-TC-to-HC}  occupies the rest of this section.

In the following, $\operatorname{TC}$ (resp. $\operatorname{TR}$, resp. $\operatorname{TF}$) means $\operatorname{TC}(\ ; p)$ (resp. $\operatorname{TR}(\ ; p)$, resp. $\operatorname{TF}(\ ; p)$).


\subsection{Products and the Adams isomorphism} 
\label{sub:products_in_equivariant_stable_homotopy_theory}
For an $\mathbf{S}^1$-space $X$, we have the equivalence (\cite[Cor. A.7.6.4]{DGM13})
\begin{gather*}
\big(X^{h \mathbf{S}^1}\big)^{\wedge}_p \xrightarrow{\sim} \big(\holim_n X^{hC_{p^n}}\big)^{\wedge}_p
\end{gather*}
and the equivalence (\cite[Lemma 4.4.9]{Mad94}, \cite[Lemma 4.5]{Bru01})
\begin{gather}\label{eq:transfer-iso-homotopyOrbits}
 (\Sigma X_{h \mathbf{S}^1})^{\wedge}_p \xrightarrow{\sim} (\holim_n X_{hC_{p^n}})^{\wedge}_p\quad .
\end{gather}
The main purpose of this subsection is to show the compatibility of these equivalences and the exterior products, i.e:
\begin{proposition}\label{prop:commutativity-action-hCpn-hS1}
Let $X$ and $Y$ be  $\mathbf{S}^1$-spaces. Then the following diagram is commutative.
\begin{gather*}
\xymatrix{
  \big(\holim_n X^{hC_{p^n}}\big)^{\wedge}_p \land   (\holim_n Y_{hC_{p^n}})^{\wedge}_p \ar[r]  \ar@{}[d]|-{\land}\ar@{}[dr]|{\circlearrowleft}& \big(\holim_n (X\land Y)_{hC_{p^n}}\big)^{\wedge}_p \\
    \big(X^{h \mathbf{S}^1}\big)^{\wedge}_p \land   (\Sigma Y_{h \mathbf{S}^1})^{\wedge}_p \ar@<5ex>[u]^{\wr}  \ar@<-5ex>[u]^{\wr}  \ar[r]  & \big(\Sigma (X\land Y)_{h \mathbf{S}^1}\big)^{\wedge}_p \ar[u]^{\wr}
}
\end{gather*}
\end{proposition}
For the proof, we need some preparations on the products in equivariant stable homotopy theory.

We adapt Definition \ref{def:topological-product} to the $G$-equivariant spectra in the sense of \cite{LMS86}. Let $G$ still be a compact Lie group, and $U$ be a complete $G$-universe. Let $D\in G \mathcal{S}U$. Let $H$ be a closed normal subgroup of $G$, and $J=G/H$. Recall from \cite[Definition I.3.7]{LMS86} that the fixed point $J$-spectrum $D^H\in J\mathcal{S}U^H$ is defined as
\begin{gather*}
(D^H)(V)=(DV)^H\ \mbox{for finite-dimensional $J$-subspaces}\ V\subset U^H.
\end{gather*}
For the orbit spectrum, we first define the prespectrum $\ell D/H\in J\mathcal{P}U^H$ by
\begin{gather*}
(\ell D/H)_{\mathcal{P}}(V)=(DV)/H\ \mbox{for finite-dimensional $J$-subspaces}\ V\subset U^H.
\end{gather*}
Then the spectrum $D/H\in J\mathcal{S}U^H$ is defined to be $D/H=L(\ell D/H)$. 

Let $G$ be a compact Lie group. Let $U$ be a complete $G$-universe. 
We fix a $G$-linear isometry $f\in \mathcal{L}(U\oplus U,U)$. For $D,E\in G \mathcal{S}U$ we  define $D\land E=f_*(D\land E)$ and $\mathrm{F}(D,E)=\mathrm{F}(D,f^*E)$. Let $i:U^G\rightarrow U$ be the inclusion of $G$-universes. 
\begin{lemma}\label{lem:functionSpectrum-induce-mapOnFixPointsOrbits}
Let $D,E\in G \mathcal{S}U$.
\begin{enumerate}[(i)]
  \item There exists a natural map  $\mathrm{F}(D,E)^G\rightarrow \mathrm{F}(D^G,E^G)$ of spectra indexed on $U^G$.
  \item There exists a natural map $\mathrm{F}(D,E)^G\rightarrow \mathrm{F}(D/G,E/G)$  of spectra indexed on $U^G$.
\end{enumerate}
\end{lemma}
\begin{proof}
(i) For an indexing subspace $V\subset U^G$, by the definition of function spectra (\cite[Definition II.3.3]{LMS86}), we have $\mathrm{F}(D,E)^G(V)= G\mathcal{P}U(D,f^*E[V])$, and $\mathrm{F}(D^G,E^G)(V)= \mathcal{P}U^G(D^G,f^{G*}E^G[V])$. For any indexing subspace $W\subset U^G$, any $G$-map $D(W)\rightarrow E(W\oplus V)$ induces a map $D^G(W)=D(W)^G\rightarrow E(W\oplus V)^G=(f^{G*}E^G)(W\oplus V)$. Moreover, for  any indexing subspaces $W\subset W'\subset U^G$, if we have a commutative diagram of $G$-maps
\begin{gather*}
  \xymatrix{
    D(W) \ar[r] \ar[d] & \Omega^{W'-W}D(W') \ar[d] \\
    E(W\oplus V) \ar[r]  & \Omega^{W'-W}E(W'\oplus V),
  }
\end{gather*}
then the induced diagram
\begin{gather*}
  \xymatrix{
    D^G(W) \ar[r] \ar[d] & \Omega^{W'-W}D^G(W') \ar[d] \\
    (f^{G*}E^G)(W\oplus V) \ar[r]  & \Omega^{W'-W}(f^{G*}E^G)(W'\oplus V)
  }
\end{gather*}
is also commutative.  Thus we have a natural map $\mathrm{F}(D,E)^G(V)\rightarrow \mathrm{F}(D^G,E^G)(V)$, and by checking compatibility with structural maps in the same way, we get the map $\mathrm{F}(D,E)^G\rightarrow \mathrm{F}(D^G,E^G)$.

(ii) For an indexing subspace $V\subset U^G$,  we have $\mathrm{F}(D/G,E/G)(V)= \mathcal{P}U^G(D/G,f^{G*}(E/G)[V])$. A $G$-map $D\rightarrow (f^*E)[V]$ induces a map $D/G\rightarrow (f^{G*}E)[V]/G$. It remains to define $(f^{G*}E)[V]/G\rightarrow f^{G*}(E/G)[V]$. Recall that $(f^{G*}E)[V]/G=L\Big(\ell \big((f^{G*}E)[V]\big)/G\Big)$,  by adjunction it suffices to define 
\begin{gather*}
\ell \big((f^{G*}E)[V]\big)/G\rightarrow  \ell\big(f^{G*}(E/G)[V]\big).
\end{gather*}
But for an indexing subspace $W\subset U^G$, 
\begin{gather}\label{eq:functionSpectrum-induce-mapOnOrbits-1}
\big((f^{G*}E)[V]\big)/G(W)=\big((f^{G*}E)[V]\big)(W)/G=(f^{G*}E)(W\oplus V)/G=E\big(f(W\oplus V)\big)/G,
\end{gather}
and
\begin{gather}\label{eq:functionSpectrum-induce-mapOnOrbits-2}
\ell\big(f^{G*}(E/G)[V]\big)(W)=f^{G*}(E/G)(W\oplus V)=(E/G)\big(f(W\oplus V)\big).
\end{gather}
Since $E/G=L\big((\ell E)/G\big)$, there is a natural map from \eqref{eq:functionSpectrum-induce-mapOnOrbits-1} to \eqref{eq:functionSpectrum-induce-mapOnOrbits-2}, and by checking compatibility with structural maps as in (i), we get the map $\mathrm{F}(D,E)^G\rightarrow \mathrm{F}(D/G,E/G)$.
\end{proof}
\begin{construction}\label{cons-fixedPoint-orbit-G-spectrum-product}
Let $D,E\in G \mathcal{S}U$.
\begin{enumerate}[(i)]
  \item The map $D\rightarrow \mathrm{F}(E,D\land E)$ induces a map $D^G\rightarrow \mathrm{F}(E,D\land E)^G$. 
  We  define $D^G\land E^G\rightarrow (D\land E)^G$ by applying adjunction to the composition
 \begin{gather*}
  \xymatrix{
    D^G \ar[r] \ar[dr] & \mathrm{F}(E,D\land E)^G \ar[d] \\
    & \mathrm{F}\big(E^G,(D\land E)^G\big),
  }
 \end{gather*}
  where the vertical map is the map defined in Lemma \ref{lem:functionSpectrum-induce-mapOnFixPointsOrbits}(i).
 \item We define $D^G\land (E/G)\rightarrow (D\land E)/G$ by applying adjunction to the composition
 \begin{gather*}
  \xymatrix{
    D^G \ar[r] \ar[dr] & \mathrm{F}(E,D\land E)^G \ar[d] \\
    & \mathrm{F}\big(E/G,(D\land E)/G\big),
  }
 \end{gather*}
 where the vertical map is the map defined in Lemma \ref{lem:functionSpectrum-induce-mapOnFixPointsOrbits}(ii).
  \item We define  $D^{hG}\land (EG_+\land E)^G\rightarrow  (EG_+\land D\land E)^G$ to be the composition
\begin{align*}
D^{hG}\land (EG_+\land E)^G&=\mathrm{F}(EG_+,D)^G\land (EG_+\land E)^G\\
 &\rightarrow \big(\mathrm{F}(EG_+,D)\land EG_+\land E\big)^G \rightarrow (EG_+\land D\land E)^G
\end{align*}
where the first arrow is the map defined in (i) and the second arrow is induced by the evaluation map $\mathrm{F}(EG_+,D)\land EG_+\rightarrow D$.
  \item We define $D^G\land E_{hG}\rightarrow (D\land E)_{hG}$ to be the composition
  \begin{gather*}
    D^G\land E_{hG}=D^G\land (EG_+\land E)/G\rightarrow (EG_+\land D\land E)/G=(D\land E)_{hG},
  \end{gather*}
  where the  arrow is the map defined in (ii).
\end{enumerate}
\end{construction}

\begin{definition}
Let $G$ be a compact Lie group and $U$ a $G$-universe. Let $D,E\in G \mathcal{S}U$. For any indexing subspaces $V\subset W$ in $U$, by definition of $G$-spectra we have the structural $G$-homeomorphism $DV\xrightarrow{\cong}\Omega^{W-V}DW$, and $EV\xrightarrow{\cong}\Omega^{W-V}EW$, and thus the induced $G$-map $\mathrm{F}(DW,EW)\rightarrow \mathrm{F}(DV,EV)$ between the function spaces. We define the \emph{function space}
\begin{gather*}
\mathsf{F}(D,E):=\varprojlim_{V\subset U}\mathrm{F}(DV,EV)
\end{gather*}
where the limit is taken over the indexing subspaces in $U$.
\end{definition}

\begin{lemma}\label{lem:function-G-space}
\begin{enumerate}[(i)]
  \item For any $X\in G \mathcal{T}$, and $D,E\in G \mathcal{S}U$, there is an adjunction isomorphism
  \begin{gather*}
  G \mathcal{T}(X,\mathsf{F}(D,E))\cong G \mathcal{S}U(X\land D,E)
  \end{gather*}
  which is natural in $X$, $D$, and $E$.
  \item There are natural maps of spaces $\mathsf{F}(D,E)^G\rightarrow \mathsf{F}(D^G,E^G)$ and $\mathsf{F}(D,E)^G\rightarrow \mathsf{F}(D/G,E/G)$.
\end{enumerate}
\end{lemma}
\begin{proof}
We leave the proof to the reader.
\end{proof}

\begin{construction}\label{cons-fixedPoint-orbit-G-space-spectrum-product}
Let $X\in G \mathcal{T}$ and $E\in G \mathcal{S}U$.
\begin{enumerate}[(i)]
  \item The map $X\rightarrow \mathsf{F}(E,X\land E)$ induces a map $X^G\rightarrow \mathsf{F}(E,X\land E)^G$. 
  We  define $X^G\land E^G\rightarrow (X\land E)^G$ by applying adjunction in Lemma \ref{lem:function-G-space}(i) to the composition
 \begin{gather*}
  \xymatrix{
    X^G \ar[r] \ar[dr] & \mathsf{F}(E,X\land E)^G \ar[d] \\
    & \mathsf{F}\big(E^G,(X\land E)^G\big),
  }
 \end{gather*}
  where the vertical map is the map defined in Lemma \ref{lem:function-G-space}(ii).
 \item We define $X^G\land (E/G)\rightarrow (X\land E)/G$ by applying adjunction to the composition
 \begin{gather*}
  \xymatrix{
    X^G \ar[r] \ar[dr] & \mathsf{F}(E,X\land E)^G \ar[d] \\
    & \mathsf{F}\big(E/G,(X\land E)/G\big),
  }
 \end{gather*}
 where the vertical map is the map defined in Lemma \ref{lem:function-G-space}(ii).
  \item We define  $X^{hG}\land (EG_+\land E)^G\rightarrow  (EG_+\land X\land E)^G$ to be the composition
\begin{align*}
X^{hG}\land (EG_+\land E)^G&=\mathrm{F}(EG_+,X)^G\land (EG_+\land E)^G\\
 &\rightarrow \big(\mathrm{F}(EG_+,X)\land EG_+\land E\big)^G \rightarrow (EG_+\land D\land E)^G
\end{align*}
where the first arrow is the map defined in (i) and the second arrow is induced by the evaluation map $\mathrm{F}(EG_+,X)\land EG_+\rightarrow X$.
  \item We define $X^G\land E_{hG}\rightarrow (X\land E)_{hG}$ to be the composition
  \begin{gather*}
    X^G\land E_{hG}=X^G\land (EG_+\land E)/G\rightarrow (EG_+\land X\land E)/G=(X\land E)_{hG}
  \end{gather*}
  where the  arrow is the map defined in (ii).
\end{enumerate}
\end{construction}

Now we recall the Adams isomorphism (\cite{Ada84}, \cite{LMS86}). Let $N$ be a  normal closed subgroup of $G$, let $J=G/N$, and let $A$ be the tangent space of $N$ at $e\in G$, which is regarded as a representation of $G$. Let $i:U^N\hookrightarrow U$ be the inclusion of $G$-universes. For any $N$-free $G$-spectra $D$ indexed on $U^N$,   in \cite[\S II.7]{LMS86} a map of $J$-spectra
\begin{gather*}
\tilde{\tau}: D/N\rightarrow (\Sigma^{-A}i_*D)^N
\end{gather*}
is constructed and is shown to be an equivalence of $J$-spectra. We call $\tilde{\tau}$ an \emph{Adams map}.  From the construction of $\tilde{\tau}$ in loc. cit., it is  natural in $D$ in $hGS\mathcal{U}^N$
i.e for a map $D\rightarrow E$ in $G \mathcal{S}U^N$,  the induced diagram
\begin{gather*}
  \xymatrix{
    D/N \ar[r] \ar[d]_{\tilde{\tau}}   & E/N \ar[d]_{\tilde{\tau}} \\
  \big( \Sigma^{-A}i_*D\big)^N\ar[r] & \big( \Sigma^{-A}i_*E\big)^N
  }
\end{gather*}
commutes in $hGS\mathcal{U}^N$.
Besides this, we need the following form of naturality  of $\tilde{\tau}$. 
\begin{lemma}\label{lem:adams-iso-homotopical-functoriality}
For $X\in G \mathcal{T}$ and $E\in G \mathcal{S}U^N$, the  diagram
\begin{equation}\label{eq:adams-iso-homotopical-functoriality}
    \begin{gathered}
\xymatrix{
  X^G\land (E/N) \ar[r] \ar[d]_{1\land \tilde{\tau}} & \big(X\land E\big)/N \ar[d]_{\tilde{\tau}}   \\
  X^G\land (\Sigma^{-A}i_*E)^N \ar[r] & (X\land \Sigma^{-A}i_*E)^N 
  }
\end{gathered}
\end{equation}
commutes in $hGS\mathcal{U}^N$, where   we have used the identification
 $i_*(X\land E)=X\land i_*E$ (see \cite[Prop.~II.1.4]{LMS86}).
\end{lemma}
To show this lemma, we have to recall several key points in the construction of $\tilde{\tau}$ from \cite[\S II.7]{LMS86}. WLOG (\cite[Theorem I.5.12]{LMS86}), let $D$ be a $G$-CW-spectrum indexed on $U^N$. Let $\Gamma=G\ltimes N$ be the semidirect product with multiplication $(h,m) (g,n) = (hg,g^{-1}mgn)$  for $h, g\in G$ and $m,n\in N$. We identify $G$ with a subgroup of $\Gamma$ by $g\mapsto (g,1)$. There are group homomorphisms
\begin{gather*}
\theta:\Gamma\rightarrow G,\ \theta (g,n) =gn,\\
\epsilon:\Gamma\rightarrow G,\ \epsilon (g,n) =g,
\end{gather*}
and we let $\Pi:=e\ltimes N=\ker(\epsilon)$. Let $\Gamma$ act on $N$ by $(g,n)m = gnmg^{-1}$.
Via $\epsilon$, we regard $U$ as an (incomplete) $\Gamma$-universe, and $A$ as a representation of  $\Gamma$. Take a complete $\Gamma$-universe $U'$ and let $j:U\hookrightarrow U'$ be an inclusion of $\Gamma$-universes such that $U=(U')^{\Pi}$. Then the construction of $\tilde{\tau}$ consists of the following steps:
\begin{enumerate}
   \item[$1^{\circ}$] By the Thom-type construction \cite[Cons. II.5.1 and Page 99 of \S II.7]{LMS86}, one finds a map $t : S \rightarrow \Sigma^{-A}\Sigma^{\infty}(N_+)$ of $\Gamma$-spectra indexed on $U'$, which induces a map
   \begin{gather*}
   1\land t: j_*i_*\theta^* D\rightarrow j_*\Sigma^{-A}(i_*\theta^*D\land N_+).
   \end{gather*}
   \item[$2^{\circ}$] Suppose now that $D$ is $N$-free. Then $i_*\theta^*D$ is a $\Pi$-free $\Gamma$-spectrum. By 
   \cite[Theorem II.2.8(i)]{LMS86}, $j_*$ induces an isomorphism 
   \begin{gather*}
   [i_*\theta^* D,\Sigma^{-A}(i_*\theta^*D\land N_+)]_{\Gamma}\xrightarrow{\cong}   [j_*i_*\theta^* D, j_*\Sigma^{-A}(i_*\theta^*D\land N_+)]_{\Gamma}
   \end{gather*}
   of $\Gamma$-equivariant homotopy classes. Hence there exists a map $\hat{\tau}:i_*\theta^* D\rightarrow \Sigma^{-A}(i_*\theta^*D\land N_+)$ which induces $1\land t$ by $j_*$.
   \item[$3^{\circ}$] By \cite[Lemma II.7.4]{LMS86}, there are natural isomorphisms
\begin{gather*}
(i_*\theta^*D \land N_+) /\Pi \cong i_*D\ \mbox{and}\ (i_*\theta^* D)/\Pi\cong i_*(D/N)
\end{gather*}
of $G$-spectra indexed on $U$. Applying these to $\tilde{\tau}$ yields 
\begin{gather*}
\tau:i_*(D/N)\rightarrow  \Sigma^{-A}i_*\theta^*D.
\end{gather*}
   \item[$4^{\circ}$] Finally, by adjunction, we obtain a map $\tilde{\tau}:D/N\rightarrow (\Sigma^{-A}i_*\theta^*D)^N$.
 \end{enumerate} 
 We can choose a map $t$ common for all $D$. 
Then the only obstacle to lift the natural Adams map in $hG \mathcal{S}U^N$ to a functorial one in $G \mathcal{S}U^N$ lies in the step $2^{\circ}$; even in the case $G=N$ and $G$ is commutative, there is no functorial lifting. 
\begin{remark}\label{rem:adams-map-functoriality}
For finite groups $G$, replacing $D$ with $EN_+\land D$, Reich and Varisco (\cite{ReV16}) constructed a functorial Adams map  in the spectrum level, in terms of orthogonal $G$-spectra.
\end{remark}

We bypass this obstacle in an ad hoc way by the following lemma which slightly enhances \cite[Theorem II.2.8(i)]{LMS86}.
\begin{lemma}\label{lem:LMS-II.2.8-enhance}
Let $\Gamma$ be a compact Lie group, and $\Pi$ be a normal subgroup. Let $U'$ be a complete $\Gamma$-universe and $U=U^{\prime\Pi}$, and $j:U \hookrightarrow U'$ be the inclusion.
Let $E_1\in \Gamma \mathcal{S}U$ be a $\Pi$-free $\Gamma$-CW spectrum, and $E_2\in \Gamma \mathcal{S}U$ be a $\Pi$-free sub-$\Gamma$-CW spectrum of $E_1$. Let   $F\in \Gamma \mathcal{S}U$ be any $\Gamma$-spectrum. Then in the following diagram
\begin{equation}\label{eq:LMS-II.2.8-enhance}
\begin{gathered}
\xymatrix{
  [E_1,F]_{\Gamma} \ar[d]_{|_{E_2}} \ar[r]^<<<<{j_*} & [j_* E_1,j_*F]_{\Gamma} \ar[d]^{|_{j_*E_2}} \\
    [E_2,F]_{\Gamma} \ar[r]^<<<<{j_*} & [j_* E_2,j_*F]_{\Gamma} 
}
\end{gathered}
\end{equation}
the horizontal maps are isomorphisms. Moreover, for any $\Gamma$-maps $\alpha:j_* E_1\rightarrow j_*F$ and $\beta:E_2\rightarrow F$ such that $j_*[\beta]=[\alpha|_{E_2}]$, there exists a $\Gamma$-map $\gamma:E_1\rightarrow F$ such that $j_* \gamma=\alpha$ and $\gamma|_{E_2}=\beta$.
\end{lemma}
\begin{proof}
The first statement is \cite[Theorem II.2.8(i)]{LMS86}. For the second, we use the argument in the proof of \cite[Lemma II.2.4(ii)]{LMS86}. Let $\mathcal{F}(\Pi)$ be the set of closed subgroups of $\Gamma$ such that $H\cap \Pi=\{e\}$.
Then by \cite[Cor. II.1.8 and Lem. II.2.4(i)]{LMS86}, the adjunction map $\eta:F \rightarrow j^*j_*F$ is an $\mathcal{F}(\Pi)$-equivalence. The diagram \eqref{eq:LMS-II.2.8-enhance} factors as
\begin{gather*}
\xymatrix{
  [E_1,F]_{\Gamma} \ar[d]_{|_{E_2}} \ar[r]^<<<<{\eta_*} & [E_1,j^*j_*F]_{\Gamma} \ar@{}[r]|-{\cong} \ar[d]^{|_{E_2}}  & [j_* E_1,j_*F]_{\Gamma} \ar[d]^{|_{j_*E_2}} \\
    [E_2,F]_{\Gamma}  \ar[r]^<<<<{\eta_*} & [E_2,j^*j_*F]_{\Gamma} \ar@{}[r]|-{\cong} & [j_* E_2,j_*F]_{\Gamma} 
}
\end{gather*}
Then since $E_2$ is obtained from $E_1$ by attaching $\Pi$-free cells $S^n_H$ where $H\subset \mathcal{F}(\Pi)$, the second statement follows from the $\mathcal{F}(\Pi)$-Whitehead theorem \cite[Theorem II.2.2]{LMS86}
\end{proof}

\begin{proof}[Proof of Lemma \ref{lem:adams-iso-homotopical-functoriality}]
Do the above construction of $\tilde{\tau}$ for $E$ and for $X\land E$. In step $2^{\circ}$, first run it for $E$, so that we choose a $\hat{\tau}$ for $D=E$, then apply Lemma \ref{lem:LMS-II.2.8-enhance} to $F=\Sigma^{-A}(i_*\theta^*(X\land E)\land N_+)$, $E_1=i_*\theta^* (X\land E)$,  $E_2=X^G\land i_*\theta^* E$, $\alpha=$ the map $1\land t$ for $D=X\land E$, and $\beta=\iota\circ (1_{X^G}\land \hat{\tau})$, where $\iota$ is the inclusion $ X^G\land \Sigma^{-A} ( i_*\theta^* E\land N_+)\hookrightarrow F$. Thus the diagram
\begin{gather*}
 \xymatrix{
    X^G \land  i_*\theta^* E \ar@{^{(}->}[r] \ar[d]_{1_{X^G}\land \hat{\tau}} & i_*\theta^* (X\land E) \ar[d]^{\gamma} \\
    X^G\land \Sigma^{-A} ( i_*\theta^* E\land N_+) \ar[r]^{\iota} & \Sigma^{-A} i_*\theta^* (X\land E) \land N_+
 }
 \end{gather*} 
is homotopy commutative. Hence running the remaining steps yields that the diagram \eqref{eq:adams-iso-homotopical-functoriality} commutes in $hGS\mathcal{U}^N$.
\end{proof}

\begin{corollary}\label{cor:transfer-products-homotopyOrbit-fixetPoints}
Let $G$ be a compact Lie group, and $A$ be the adjoint representation of $G$.
Let $X\in G \mathcal{T}$ and $D\in G \mathcal{S}U^G$, and let $\tilde{\tau}$ be the Adams map.
Then the diagram 
\begin{gather*}
\xymatrix{
  X^{hG}\land D_{hG} \ar[r] \ar[d]_{\wr}^{1\land \tilde{\tau}} \ar@{}[dr]|{\circlearrowleft} & (X\land D)_{hG} \ar[d]_{\wr}^{\tilde{\tau}} \\
  X^{hG}\land (\Sigma^{-A}EG_+\land i_*D)^G \ar[r] & (\Sigma^{-A}EG_+\land X\land i_*D)^G
  }
\end{gather*}
commutes in $hGS \mathcal{U}^{G}$.
\end{corollary}
\begin{proof}
By Construction \ref{cons-fixedPoint-orbit-G-space-spectrum-product}(iii) and (iv), this diagram factors as 
\begin{gather*}
\xymatrix{
  X^{hG}\land D_{hG} \ar[r] \ar[d]_{1\land \tilde{\tau}} & \big(\mathrm{F}(EG_+,X)\land EG_+\land D\big)/G \ar[r] \ar[d]_{\tilde{\tau}}   & (EG_+\land X\land D)/G \ar[d]_{\tilde{\tau}} \\
  X^{hG}\land (\Sigma^{-A}EG_+\land i_*D)^G \ar[r] & \big(\Sigma^{-A}\mathrm{F}(EG_+,X)\land EG_+\land i_*D\big)^G \ar[r] & (\Sigma^{-A}EG_+\land X\land i_*D)^G.
  }
\end{gather*}
The left-hand square commutes by Lemma \ref{lem:adams-iso-homotopical-functoriality}, and the right-hand square commutes by the discussion preceeding Lemma \ref{lem:adams-iso-homotopical-functoriality}.
\end{proof}

Now we are ready to show Proposition \ref{prop:commutativity-action-hCpn-hS1}.
\begin{proof}[Proof of Proposition \ref{prop:commutativity-action-hCpn-hS1}]
The map in \eqref{eq:transfer-iso-homotopyOrbits} is defined via the Adams isomorphism and the obvious map from $\mathbf{S}^1$-fixed points to $C_{p^n}$-fixed points. Hence the conclusion follows from Corollary \ref{cor:transfer-products-homotopyOrbit-fixetPoints} and the commutativity of exterior products of (homotopy-)fixed points with change of the groups.
\end{proof}

\begin{remark}\label{rem:commutativity-action-S1-spectra}
We expect that Proposition \ref{prop:commutativity-action-hCpn-hS1} holds for $\mathbf{S}^1$-spectra as well. But for the proof, one needs some effort to generalize Lemma \ref{lem:adams-iso-homotopical-functoriality}.
\end{remark}


\subsection{A recap of filtered \texorpdfstring{$\mathrm{THH}$}{THH} and filtered \texorpdfstring{$\mathrm{TC}$}{TC}} 
\label{sub:a_recap_of_filtered_thh_and_filtered_tc}
We recall several notions of filtered objects from \cite{Bru01}. Let $\mathcal{S}_*$ be the category of pointed simplicial sets. Let $\Gamma$ be the category of pointed finite sets. A $\Gamma$-filtered space is a pointed functor  $\Gamma \rightarrow \mathcal{S}_*^{\mathbb{Z}}$. The smash product of two $\Gamma$-filtered spaces $X$ and $Y$ is given by
\begin{gather*}
(X\land Y)(n^+)=\colim_{n_1^+\land n_2^+ \rightarrow n^+}X(n_1^+)\land Y(n_2^+).
\end{gather*}
Let $A$ be a  simplicial ring, and $I$ an ideal of $A$. Let $A(s)=A$ for $s\geq 0$ and $A(s)=I^{-s}$ for $s\leq 0$. Let $\widetilde{A}$ be  the $\Gamma$-filtered space with
$\widetilde{A}(s)(n^+)=\widetilde{\mathbb{Z}}(n^+)\otimes_{\mathbb{Z}}A(s)= \mathbb{Z}\{n+\}/\mathbb{Z}\{0\}\otimes_{\mathbb{Z}}A(s)$. 
Let $\mathcal{I}$ be the Bökstedt subcategory of $\Gamma$ which consists of all objects but only injective maps.  
 The $\Gamma$-filtered space $\operatorname{THH}(A)$ is defined by
\begin{gather*}
\operatorname{THH}(A,n^+)(s)_q=\holim_{\mathbf{x}\in \mathcal{I}^{q+1}}\mathrm{Map}_*\left(
 \bigwedge_{i=0}^q S^{x_i},n^+\land \Big(\bigwedge_{i=1}^q \widetilde{A}(S^{x_i})\Big)(s)
\right)
\end{gather*}
for $n^+\in \operatorname{ob}\Gamma$ and $s\in \mathbb{Z}$.
In particular, $\operatorname{THH}(A)(0)$ is the usual $\Gamma$-space $\operatorname{THH}(A)$ (e.g. \cite[\S4.2]{DGM13}). The \emph{filtered space} $\operatorname{THH}(A)$ is defined  so that $\operatorname{THH}(A)(s)=\operatorname{THH}(A,S^0)(s)$, which is an \emph{epicyclic filtered space} in the sense of \cite[\S 4.2]{Bru01}. For $s\leq 0$, there are maps
\begin{gather*}
\operatorname{THH}(A)(s) \xrightarrow{i} \operatorname{THH}(A)(\lfloor \frac{s}{p}\rfloor),\quad
\big|\operatorname{THH}(A)(s)\big|^{C_{p^n}} \xrightarrow{R} \Big|\operatorname{THH}(A)(\lfloor \frac{s}{p}\rfloor)\Big|^{C_{p^{n-1}}},\\
\operatorname{THH}(A)(s)^{C_{p^n}} \xrightarrow{F} \operatorname{THH}(A)(s)^{C_{p^{n-1}}}
\end{gather*}
where $i$ is the obvious inclusion, $R$ is induced by the edgewise subdivision, and $F$ is the obvious inclusion of $C_{p^n}$-fixed points. Then (\cite[\S 5]{Bru01}) the filtered TR, TF, and TC are defined by homotopy equalizers:
\begin{align}
\mathrm{TR}(A)(s)&:=\mathrm{hoeq}\bigg(\prod_{n\geq 0}\big|\operatorname{THH}(A)(s)\big|^{C_{p^n}}
\overset{R}{\underset{i}{\rightrightarrows}}\prod_{n\geq d0}\Big|\operatorname{THH}(A)(\lfloor \frac{s}{p}\rfloor)\Big|^{C_{p^{n}}}\bigg),\label{eq:def-filteredTR} \\
\mathrm{TF}(A)(s)&:=\mathrm{hoeq}\bigg(\prod_{n\geq 0}\big|\operatorname{THH}(A)(s)\big|^{C_{p^n}}
\overset{F}{\underset{\mathrm{id}}{\rightrightarrows}}\prod_{n\geq d0}\Big|\operatorname{THH}(A)(s)\Big|^{C_{p^{n}}}\bigg),\nn
\\
\mathrm{TC}(A)(s)&:=\mathrm{hoeq}\Big(\operatorname{TR}(A)(s)
\overset{F}{\underset{\mathrm{id}}{\rightrightarrows}}\operatorname{TR}(A)(s)\Big).\label{eq:def-filteredTC}
\end{align}

\subsection{Comparison of  multiplications of \texorpdfstring{$\mathrm{TF}$}{TF} on relative \texorpdfstring{$\mathrm{HC}$}{HC}} 
\label{sub:the_multiplication_of_TF_on_relative_HC}
Let $A$ be a commutative simplicial ring, and $I$ be an ideal of $A$ with $I^m=0$. 
For brevity of notations, we let  $X=|\mathrm{THH}(A)|$,
and any map  induced by the multiplication $A\otimes A\rightarrow A$ will be denoted by $\mu$.

\begin{construction}\label{cons:multiplication-TF-on-relHC-hS1}
There is a natural map $\operatorname{TF}(A)=\holim_{F}X(0)^{C_{p^n}}\rightarrow \holim X(0)^{hC_{p^n}}$.
Recall the equivalence (\cite[Lemma 6.3.1.1]{DGM13})
\begin{gather*}
 \holim X(0)^{hC_{p^n}}\simeq X(0)^{h \mathbf{S}^1}.
\end{gather*}
We define the multiplication
\begin{gather}\label{eq:product_TF_on_relative_HC}
\operatorname{TF}(A;p) \land \bigg(\left(\frac{X(-1)}{X(-p)}\right)_{h \mathbf{S}^1}\bigg)_p^{\land} 
 \xrightarrow{\mu}   \left(\left(\frac{X(-1)}{X(-p)}\right)_{h \mathbf{S}^1}\right)_p^{\land}
\end{gather}
to be the composition
\begin{gather*}
\operatorname{TF}(A;p) \land \left(\left(\frac{X(-1)}{X(-p)}\right)_{h \mathbf{S}^1}\right)_p^{\land}
\rightarrow \holim_{F}X(0)^{hC_{p^n}} \land \left(\left(\frac{X(-1)}{X(-p)}\right)_{h \mathbf{S}^1}\right)_p^{\land}\\
\simeq  \left(X(0)^{h \mathbf{S}^1}\right)_p^{\land} \land  \left(\left(\frac{X(-1)}{X(-p)}\right)_{h \mathbf{S}^1}\right)_p^{\land}
\simeq \left(X(0)^{h \mathbf{S}^1} \land \left(\frac{X(-1)}{X(-p)}\right)_{h \mathbf{S}^1}\right)_p^{\land}\\
 \xrightarrow{\mu}   \left(\left(\frac{X(-1)}{X(-p)}\right)_{h \mathbf{S}^1}\right)_p^{\land}.
\end{gather*}
\end{construction}

\begin{construction}\label{cons:multiplication-TF-on-relHC-hCpn}
Taking homotopy limits of the commutative diagrams
\begin{equation}\label{eq:product-X(0)CpnFixPoints-on-X(-1)-X(-p)CpnOrbits}
  \begin{gathered}
\xymatrix{
 X(0)^{C_{p^n}}\land \left(\frac{X(-1)}{X(-p)}\right)_{hC_{p^n}} \ar[r]^<<<<<<<{\mu} \ar[d]_{F\land \mathrm{trf}} &  \left(\frac{X(-1)}{X(-p)}\right)_{hC_{p^n}}  \ar[d]^{\mathrm{trf}}\\
 X(0)^{C_{p^{n-1}}}\land \left(\frac{X(-1)}{X(-p)}\right)_{hC_{p^{n-1}}} \ar[r]^<<<<{\mu} &  \left(\frac{X(-1)}{X(-p)}\right)_{hC_{p^{n-1}}}
 }
\end{gathered}
\end{equation}
yields a map
\begin{gather*}
\holim_{F}X(0)^{C_{p^n}} \land \holim_{\mathrm{trf}}\left(\frac{X(-1)}{X(-p)}\right)_{hC_{p^n}} \overset{\mu}{\longrightarrow}  \holim_{\mathrm{trf}}\left(\frac{X(-1)}{X(-p)}\right)_{hC_{p^n}}.
\end{gather*}
Combining with the equivalence \eqref{eq:transfer-iso-homotopyOrbits}, we get a map
\begin{gather}\label{eq:TF-on-THH(-1)/THH(-p)/hS1}
\operatorname{TF}(A;p) \land \left(\left(\frac{X(-1)}{X(-p)}\right)_{h \mathbf{S}^1}\right)_p^{\land} \overset{\mu}{\longrightarrow}   \left(\left(\frac{X(-1)}{X(-p)}\right)_{h \mathbf{S}^1}\right)_p^{\land}\quad.
\end{gather}
\end{construction}
\begin{proposition}\label{prop:action-homotopyPointOnHomotopyFixedPoint}
  The multiplication \eqref{eq:TF-on-THH(-1)/THH(-p)/hS1} is naturally equivalent to \eqref{eq:product_TF_on_relative_HC}.
\end{proposition}
\begin{proof}
This follows by  applying Proposition \ref{prop:commutativity-action-hCpn-hS1} to $X=\operatorname{THH}(A)(0)$ and $Y=\frac{\operatorname{THH}(A)(-1)}{\operatorname{THH}(A)(-p)}$.
\end{proof}


\subsection{Comparision with the multiplication of \texorpdfstring{$\mathrm{TC}$}{TC} on  \texorpdfstring{$\left(\frac{\operatorname{THH}(-1)}{\operatorname{THH}(-p)}\right)_{hC_{p^n}}$}{(THH(-1)/THH(-p))hCpn}} 
\label{sub:the_multiplication_of_TF_on_relative_THH(-1)/THH(-p))hCpn}
\begin{lemma}[Brun]\label{lem:rel-TC-equals-rel-homotopyOrbits-Cpn}
There is a natural equivalence
\begin{gather*}
\frac{\mathrm{TC}(A)(-1)}{\mathrm{TC}(A)(-p)}\simeq \holim_{n,\mathrm{trf}}\left(\frac{X(-1)}{X(-p)}\right)_{hC_{p^n}}.
\end{gather*}
where $\mathrm{trf}:\left(\frac{X(-1)}{X(-p)}\right)_{hC_{p^n}} \rightarrow \left(\frac{X(-1)}{X(-p)}\right)_{hC_{p^{n-1}}}$ is the transfer map induced by the Adams isomorphism.
\end{lemma}
This is an intermediate statement  shown in the proof of \cite[Lemma 5.3]{Bru01}. For Construction \ref{cons:multiplication-TC(0)-on-TC(-1)/TC(-p)}, we have to recall Brun's arguments.
\begin{proof}
As we recalled in \S\ref{sub:the_multiplication_of_TF_on_relative_HC}, for $s\leq 0$, there are  $\mathbf{S}^1$-maps $X(s)\xrightarrow{i}X\big(\lfloor \frac{s}{p}\rfloor\big)$, which induce maps $X(s)^{C_{p^n}}\xrightarrow{i}X(\lfloor \frac{s}{p}\rfloor)^{C_{p^n}}$, and for $s\in \mathbb{Z}$, there are maps $X(s)^{C_{p^n}}\xrightarrow{R}X\big(\lfloor \frac{s}{p}\rfloor\big)^{C_{p^{n-1}}}$. 
Then $i$ and $R$ induce maps between the quotients:
\begin{gather*}
\frac{X(-1)}{X(-p)}\xrightarrow{i}\frac{X(-1)}{X(-1)}=\ast,\
\left(\frac{X(-1)}{X(-p)}\right)^{C_{p^n}}\xrightarrow{R} \left(\frac{X(-1)}{X(-1)}\right)^{C_{p^{n-1}}}=\ast.
\end{gather*}
By the definition \eqref{eq:def-filteredTR}, we get equivalences
\begin{gather*}
\frac{\mathrm{TR}(A)(-1)}{\mathrm{TR}(A)(-p)}\simeq\prod_{n\geq 0}\left(\frac{X(-1)}{X(-p)}\right)^{C_{p^n}}\simeq\prod_{n\geq 0}\frac{X(-1)^{C_{p^n}}}{X(-p)^{C_{p^n}}}
\end{gather*}
where we have used that finite colimits commute with limits in the stable homotopy category.
The filtered fundamental cofibration sequence (\cite[Page 220]{Bru01})
\begin{gather*}
X(s)_{hC_{p^n}}\xrightarrow{N} X^{C_{p^n}}(s)\xrightarrow{R} X^{C_{p^{n-1}}}(\lfloor \frac{s}{p}\rfloor),\ \mbox{for}\ s\in \mathbb{Z}
\end{gather*}
where $N$ is the \emph{norm map}, implies the horizontal equivalences in the commutative diagram
\begin{equation}\label{eq:equivalence-trf-Frob}
  \begin{gathered}
\xymatrix{
  \left(\frac{X(-1)}{X(-p)}\right)_{hC_{p^n}} \ar[r]^{\sim}_{N} \ar[d]_{\mathrm{trf}} & \left(\frac{X(-1)}{X(-p)}\right)^{C_{p^n}} \ar[d]^{F} \\
  \left(\frac{X(-1)}{X(-p)}\right)_{hC_{p^{n-1}}} \ar[r]^{\sim}_{N} & \left(\frac{X(-1)}{X(-p)}\right)^{C_{p^{n-1}}}.
}
\end{gathered}
\end{equation}
Hence by definition \eqref{eq:def-filteredTC}, we obtain
\begin{equation}\label{eq:equiv-relTC-relHC}
  \begin{gathered}
\frac{\mathrm{TC}(A)(-1)}{\mathrm{TC}(A)(-p)}= \operatorname{hoeq}\left(\frac{\mathrm{TR}(A)(-1)}{\mathrm{TR}(A)(-p)}
\overset{F}{\underset{i}{\rightrightarrows}}\frac{\mathrm{TR}(A)(-1)}{\mathrm{TR}(A)(-p)}\right)\\
\simeq \holim_{n,F}\left(\frac{X(-1)}{X(-p)}\right)^{C_{p^n}}\overset{N^{-1}}{\cong} \holim_{n,\mathrm{trf}}\left(\frac{X(-1)}{X(-p)}\right)_{hC_{p^n}}.
\end{gathered}
\end{equation}
\end{proof}

\begin{construction}\label{cons:multiplication-TC(0)-on-TC(-1)/TC(-p)}
We have products $X(0)\land X(s)\rightarrow X(s)$ for $s\leq 0$. The   diagrams
\begin{gather*}
\xymatrix{
  X(0)^{C_{p^n}}\land  X(-1)^{C_{p^n}} \ar[r]^<<<<<<<{\mu} \ar[d]_{R\land R} &  X(-1)^{C_{p^n}} \ar[d]^{R}\\
   X(0)^{C_{p^{n-1}}}\land  X(-1)^{C_{p^{n-1}}} \ar[r]^<<<<{\mu} &  X(-1)^{C_{p^{n-1}}}
}\ \mbox{and}\ 
\xymatrix{
  X(0)^{C_{p^n}}\land  X(-p)^{C_{p^n}} \ar[r]^<<<<<<<{\mu} \ar[d]_{R\land R} &  X(-p)^{C_{p^n}} \ar[d]^{R}\\
   X(0)^{C_{p^{n-1}}}\land  X(-1)^{C_{p^{n-1}}} \ar[r]^<<<<{\mu} &  X(-1)^{C_{p^{n-1}}}
}
\end{gather*}
are commutative, thus induce a map
\begin{gather*}
\mathrm{TR}(A)(0)\tsmash \prod_{n\geq 0}\left(\frac{X(-1)}{X(-p)}\right)^{C_{p^n}}\overset{\mu}{\longrightarrow} \prod_{n\geq 0}\left(\frac{X(-1)}{X(-p)}\right)^{C_{p^n}}.
\end{gather*}
Moreover, the   diagrams
\begin{equation}\label{eq:product-X(0)CpnFixPoints-on-X(-1)-X(-p)CpnFixPoints}
\begin{gathered}
\xymatrix{
  X(0)^{C_{p^n}}\land  X(-1)^{C_{p^n}} \ar[r]^<<<<<<<{\mu} \ar[d]_{F\land F} &  X(-1)^{C_{p^n}} \ar[d]^{F}\\
   X(0)^{C_{p^{n-1}}}\land  X(-1)^{C_{p^{n-1}}} \ar[r]^<<<<{\mu} &  X(-1)^{C_{p^{n-1}}}
}
\mbox{and}
\xymatrix{
  X(0)^{C_{p^n}}\land  X(-p)^{C_{p^n}} \ar[r]^<<<<<<<{\mu} \ar[d]_{F\land F} &  X(-p)^{C_{p^n}} \ar[d]^{F}\\
   X(0)^{C_{p^{n-1}}}\land  X(-p)^{C_{p^{n-1}}} \ar[r]^<<<<{\mu} &  X(-p)^{C_{p^{n-1}}}
}
\end{gathered}
\end{equation}
are commutative. By  \eqref{eq:equiv-relTC-relHC} and the commutative diagram \eqref{eq:equivalence-trf-Frob},  we get a product 
\begin{gather}\label{eq:TC-rel-product}
\mathrm{TC}(A)(0)\land \frac{\mathrm{TC}(A)(-1)}{\mathrm{TC}(A)(-p)}\overset{\mu}{\longrightarrow} \frac{\mathrm{TC}(A)(-1)}{\mathrm{TC}(A)(-p)}.
\end{gather}
Then by Lemma \ref{lem:rel-TC-equals-rel-homotopyOrbits-Cpn}, we get a product
\begin{gather}\label{eq:TC-on-THH(-1)/THH(-p)/hS1}
\mathrm{TC}(A)(0)\land \holim_{n,\mathrm{trf}}\left(\frac{X(-1)}{X(-p)}\right)_{hC_{p^n}}\overset{\mu}{\longrightarrow} \holim_{n,\mathrm{trf}}\left(\frac{X(-1)}{X(-p)}\right)_{hC_{p^n}}.
\end{gather}
\end{construction}

\begin{lemma}\label{lem:commutativity-mu-1_land_Norm}
The  diagram 
\begin{equation}\label{eq:commutativity-mu-1_land_Norm}
\begin{gathered}
\xymatrix{
  X(0)^{C_{p^n}}\land \left(\frac{X(-1)}{X(-p)}\right)_{hC_{p^n}} \ar[r]^{\mathrm{id}\land N} \ar[d]_{\mu} &  X(0)^{C_{p^n}}\land \left(\frac{X(-1)}{X(-p)}\right)^{C_{p^n}} \ar[d]^{\mu} \\
   \left(\frac{X(-1)}{X(-p)}\right)_{hC_{p^n}}  \ar[r]^{N}&  \left(\frac{X(-1)}{X(-p)}\right)^{C_{p^n}}
}
\end{gathered}
\end{equation}
commutes.
\end{lemma}
\begin{proof}
For a finite group $G$ and a $G$-spectrum $Y$, the norm map $Y_{hG}\rightarrow Y^{G}$ is defined to be the composite $EG_+\land_G Y\xrightarrow{\tilde{\tau}} (EG_+\land Y)^G\xrightarrow{\pi} (\mathbf{S}^0\land Y)^G=Y^G$, where $\pi$ is induced by the  projection $EG_+\rightarrow \mathbf{S}^0$. Thus the diagram \eqref{eq:commutativity-mu-1_land_Norm} factors as
\begin{gather*}
\xymatrix{
  X(0)^{C_{p^n}}\land \left(\frac{X(-1)}{X(-p)}\right)_{hC_{p^n}} \ar[r]^<<<<{\mathrm{id}\land \tilde{\tau}} \ar[d]_{\mu} & X(0)^{C_{p^n}}\land \left(EC_{p^n+}\land \frac{X(-1)}{X(-p)}\right)^{C_{p^n}} \ar[d]^{\mu} \ar[r]^<<<<{\pi} &  X(0)^{C_{p^n}}\land \left(\frac{X(-1)}{X(-p)}\right)^{C_{p^n}} \ar[d]^{\mu} \\
   \left(\frac{X(-1)}{X(-p)}\right)_{hC_{p^n}}  \ar[r]^<<<<<<<<<<<{\tilde{\tau}}&  \left(EC_{p^n+}\land\frac{X(-1)}{X(-p)}\right)^{C_{p^n}} \ar[r]^<<<<<<<<<<<<{\pi} &  \left(\frac{X(-1)}{X(-p)}\right)^{C_{p^n}}.
}
\end{gather*}
The left-hand square commutes by Lemma \ref{lem:adams-iso-homotopical-functoriality} and the naturality of the map $\tilde{\tau}$, while the right-hand square commutes by the naturality of  $\pi$. Hence \eqref{eq:commutativity-mu-1_land_Norm} commutes.
\end{proof}

\begin{corollary}\label{cor:comparison-product-by-TF-and-by-TC}
The product \eqref{eq:TF-on-THH(-1)/THH(-p)/hS1} is compatible with \eqref{eq:TC-on-THH(-1)/THH(-p)/hS1}. Namely, the diagram
\begin{gather*}
\xymatrix{
  \mathrm{TC}(A)(0)\land \holim\limits_{n,\mathrm{trf}}\left(\frac{X(-1)}{X(-p)}\right)_{hC_{p^n}} \ar[d] \ar[r]^<<<<{\mu} &
\holim\limits_{n,\mathrm{trf}}\left(\frac{X(-1)}{X(-p)}\right)_{hC_{p^n}} \ar@{=}[d] \\
 \mathrm{TF}(A)(0)\land \holim\limits_{n,\mathrm{trf}}\left(\frac{X(-1)}{X(-p)}\right)_{hC_{p^n}} \ar[r]^<<<<{\mu} &
\holim\limits_{n,\mathrm{trf}}\left(\frac{X(-1)}{X(-p)}\right)_{hC_{p^n}} 
}
\end{gather*}
commutes.
\end{corollary}
\begin{proof}
By Lemma \ref{lem:commutativity-mu-1_land_Norm}, the products \eqref{eq:product-X(0)CpnFixPoints-on-X(-1)-X(-p)CpnFixPoints} induce the same product  as \eqref{eq:product-X(0)CpnFixPoints-on-X(-1)-X(-p)CpnOrbits} on the quotients. Then the conclusion follows from the commutativity of the maps $R$ and $F$.
\end{proof}


\subsection{Completion of the proof of Proposition \ref{prop:brun-iso-product}} 
\label{sub:completion_of_the_proof_of_proposition_prop:brun-iso-product}

\begin{proof}[Proof of Proposition \ref{prop:brun-iso-product}]
By the discussion in \S\ref{sub:compatibility_of_the_brun_map_with_the_products}, we have to show the commutativity of \eqref{eq:brun-iso-product-TC-to-HC}. By the first item in the discussion following \eqref{eq:zigzag-graph-Brun-map}, the left-hand vertical arrow in \eqref{eq:brun-iso-product-TC-to-HC} is the map on homotopy groups of the map \eqref{eq:TC-rel-product}. 
By Corollary \ref{cor:comparison-product-by-TF-and-by-TC}, the latter map factors through the product \eqref{eq:TF-on-THH(-1)/THH(-p)/hS1}, which by Proposition \ref{prop:action-homotopyPointOnHomotopyFixedPoint} is equivalent to \eqref{eq:product_TF_on_relative_HC}.

The linearization $\mathrm{THH}(A)\rightarrow \mathrm{HH}(A)$ (\cite[Page 156]{DGM13}) is a cyclic map and thus induces a commutative diagram
\begin{gather*}
\xymatrix{
  X(0)^{h \mathbf{S}^1} \land  \left(\frac{X(-1)}{X(-p)}\right)_{h \mathbf{S}^1} \ar[r] \ar[d] &    \left(\frac{X(-1)}{X(-p)}\right)_{h \mathbf{S}^1} \ar[d]\\
  \mathrm{HH}(A)(0)^{h \mathbf{S}^1} \land  \left(\frac{\mathrm{HH}(A)(-1)}{\mathrm{HH}(A)(-p)}\right)_{h \mathbf{S}^1} \ar[r] &    \left(\frac{\mathrm{HH}(A)(-1)}{\mathrm{HH}(A)(-p)}\right)_{h \mathbf{S}^1}
}
\end{gather*}
where  the bottom horizontal is the map (\ref{eq-HC-action}). By Construction \ref{cons:multiplication-TF-on-relHC-hS1}, the product \eqref{eq:product_TF_on_relative_HC} factors through the top horizontal arrow. Hence the commutativity of the diagram \eqref{eq:brun-iso-product-TC-to-HC}  follows.
\end{proof}

\section{Reduction to derived cyclic homology over 
\texorpdfstring{$W(\Bbbk)$}{Wn(k)}} 
\label{sec:reduction_to_cyclic_homology_over_Wn(k)}

For a smooth algebra over a general commutative ring $R$, it seems implausible to compute $\operatorname{HH}(A/\mathbb{Z})$ or $\operatorname{HC}(A/\mathbb{Z})$, or the derived versions,  directly from the definition. The goal of this section is to show Corollary \ref{cor:reduction-HH-HC-overZ-overW(k)}. 
We need to use the following facts about cotangent complexes and their relations to HH and HC. Recall that for an algebra $A$ over $R$ and a simplicial resolution $P_{\bullet} \rightarrow A$ of $A$ by a degreewise  flat  $R$-algebra, the cotangent complex  $\mathbb{L}_{A/R}$ is defined, up to natural weak equivalences, to be the simplicial $A$-module $\Omega^1_{P_{\bullet}/R}\otimes_{P_{\bullet}}A$.  Similarly, for  $i\geq 0$, we define $\mathbb{L}^i_{A/R}:=\Omega^i_{P_{\bullet}/R}\otimes_{P_{\bullet}}A$. We denote their associated chain complexes also by $\mathbb{L}_{A/R}$ and $\mathbb{L}^i_{A/R}$.

\begin{proposition}[{\cite[Prop.~2.27]{Mor19}}]
Let $R$ be a commutative ring and $A$ a commutative $R$-algebra. If $B$ is a commutative $R$-algebra such that $\mathrm{Tor}_i^{R}(A,B)=0$ for $i>0$, then there are natural quasi-isomorphisms
  \begin{align}\label{eq-baseChange-cotangentComplex}
  \mathbb{L}_{A/R}\otimes_R^{\mathbf{L}}B\xrightarrow{\sim} \mathbb{L}_{(A\otimes_R B)/B},\tag{Base change}
  \end{align}
  \begin{align}\label{eq-Kunneth}
  \mathbb{L}_{A\otimes_R B/R}\cong (\mathbb{L}_{A/R}\otimes_R^{\mathbf{L}}B)\oplus (\mathbb{L}_{B/R}\otimes_R^{\mathbf{L}}A).\tag{Künneth}
  \end{align} 
  Moreover, if $B$ is a commutative $A$-algebra, then the resulting sequence
  \begin{align}\label{eq-transitivity-seq}
  \mathbb{L}_{A/R}\otimes_A^{\mathbf{L}}B\rightarrow \mathbb{L}_{B/R}\rightarrow \mathbb{L}_{B/A} \tag{Transitivity
  }
  \end{align}
  is a fiber sequence.  
\end{proposition}
\begin{proposition}[{\cite[Proposition IV.4.1]{NiS18}, \cite[Prop.~2.28]{Mor19}}]\label{prop:HH-HC-filtration}
Let $R$ be a commutative ring and $A$ be a commutative $R$-algebra. 
\begin{enumerate}[(i)]
  \item The derived Hochschild complex $\mathrm{HH}(A/R)$, viewed as an object of $D(A)$, admits a natural, complete, descending $\mathbb{N}$-indexed filtration whose $i$-th graded piece is equivalent to $\mathbb{L}_{A/R}^i[i]$, for $i\geq 0$.
  \item The derived cyclic complex $\mathrm{HC}(A/R)$ admits a natural, complete, descending $\mathbb{N}$-indexed filtration whose $i$-th graded piece is $\bigoplus_{n\geq 0}\mathbb{L}_{A/R}^i[i+2n]$, for $i\geq 0$. 
\end{enumerate}
\end{proposition}

 For any commutative ring $R$, let $R[T^{-\infty}]:=\varinjlim_m R[T^{p^{-m}}]$, with transition maps the inclusions $T^{p^{-m}} \mapsto T^{p^{-m}}$. For $x\in R$, denote by $\underline{x}$ the Teichmüller representative of $x$ in $W(R)$ and $W_n(R)$.

\begin{lemma}\label{lem-WnKInf}
Let $R$ be a commutative ring satisfying $pR=0$. Then the maps of $W_n(R)$-algebras
\begin{gather*}
W_n(R)[T^{p^{-m}}]\xrightarrow{g_m}  W_n(R[T^{p^{-m}}]),\ T^{p^{-m}}\mapsto \underline{T^{p^{-m}}}
\end{gather*}
induce an isomorphism on colimits
\[
W_n(R)[T^{p^{-\infty}}]\cong W_n(R[T^{p^{-\infty}}]).
\]
\end{lemma}
\begin{proof}
By the definition of Witt vectors, the map $W_n(R)[T] \rightarrow W_n(R[T])$ is surjective; one can also see this from a description of the generators of $W_n(R[T])$ as an abelian group in \cite[Cor. 2.4]{LaZ04}.
Then in the commutative diagram
\[
\xymatrix{
  W_n(R)[T] \ar[r] \ar[d]_{g_0} & W_n(R)[T^{\frac{1}{p}}] \ar[r] \ar[d]_{g_1} & \dots \ar[r] & W_n(R)[T^{\frac{1}{p^{n-1}}}]  \ar[d]_{g_{n-1}}  \\
  W_n(R[T])  \ar[r] &  W_n\big(R[T^{\frac{1}{p}}]\big) \ar[r] & \dots \ar[r] & W_n\big(R[T^{\frac{1}{p^{n-1}}}]\big)
}
\]
the vertical arrows are injective. Thus the induced map on (filtered) colimits is also injective. Since $pR=0$, we have
\[
p^{n-1}\underline{T^{\frac{1}{p^{n-1}}}}=V^{n-1}\underline{T}.
\]
It follows that the image of $W_n(R[T])$ in $W_n(R[T^{\frac{1}{p^{n-1}}}])$ lies in the image of $W_n(R)[T^{\frac{1}{p^{n-1}}}]$. Hence the induced map on colimits is surjective.
\end{proof}

\begin{proposition}\label{prop:cotangentComp-Wnk}
Let $\Bbbk$ be a perfect field of characteristic $p$. Then $\mathbb{L}_{W_n(\Bbbk)/W_n(\mathbb{F}_p)}\simeq 0$.
\end{proposition}
\begin{proof}
We use the argument of \cite[Lemma 5.5]{HeM97}. Let $\{T_j|j\in J\}$ be a transcendental basis of $\Bbbk$ over $\mathbb{F}_p$. Let $l=\varinjlim_{r\in \mathbb{N}} \mathbb{F}_p(T_j^{p^{-r}}| j\in J)$, which is a perfect field. Then we have $\Bbbk=\varinjlim_{\alpha}\Bbbk_{\alpha}$, where the colimit runs over finite subextensions $l\subset \Bbbk_{\alpha}\subset \Bbbk$ and is filtered. 
Since $l$ is perfect, $l \rightarrow \Bbbk_{\alpha}$ are étale, and then by  \cite[Chap.~0 Prop.~1.5.8]{Ill79},
$W_n(\Bbbk_{\alpha})$ are étale over $W_n(l)$. Thus $\mathbb{L}_{W_n(\Bbbk_{\alpha})/W_n(l)}\simeq 0$, so by \eqref{eq-transitivity-seq} we have
\begin{gather*}
H_i\big(\mathbb{L}_{W_n(\Bbbk)/W_n(\mathbb{F}_p)}\big)=\varinjlim_{\alpha}H_i\big(\mathbb{L}_{W_n(\Bbbk_{\alpha})/W_n(\mathbb{F}_p)}\big)
=\varinjlim_{\alpha}W_n(\Bbbk_{\alpha})\otimes_{W_n(l)}H_i\big(\mathbb{L}_{W_n(l)/W_n(\mathbb{F}_p)}\big)\\
=W_n(\Bbbk)\otimes_{W_n(l)}H_i\big(\mathbb{L}_{W_n(l)/W_n(\mathbb{F}_p)}\big).
\end{gather*}
For any finite subset $J$ of $I$, let $l_J=\mathbb{F}_p(T_j^{p^{-\infty}}| j\in J)$, and $L_J=\mathbb{F}_p[T_j^{p^{-\infty}}| j\in J]$. Then
\begin{gather*}
H_i\big(\mathbb{L}_{W_n(l)/W_n(\mathbb{F}_p)}\big)=\varinjlim_{\mathrm{finite\ subsets}\ J\subset I}H_i\big(\mathbb{L}_{W_n(l_J)/W_n(\mathbb{F}_p)}\big).
\end{gather*}
Since $l_J$ is a localization of $L_J$, and in particular is essentially étale over $L_J$, 
we have
\begin{gather*}
 H_i\big(\mathbb{L}_{W_n(l_J)/W_n(\mathbb{F}_p)}\big)
=W_n(l_J)\otimes_{W_n(L_J)}H_i\big(\mathbb{L}_{W_n(L_J)/W_n(\mathbb{F}_p)}\big).
\end{gather*} 
Then applying Lemma \ref{lem-WnKInf} to $L_J$ yields
\begin{gather*}
H_i\big(\mathbb{L}_{W_n(L_J)/W_n(\mathbb{F}_p)}\big)\cong H_i\big(\mathbb{L}_{W_n(\mathbb{F}_p)[T_j^{p^{-\infty}}| j\in J]/W_n(\mathbb{F}_p)}\big)\\
=\varinjlim_{k_j \rightarrow +\infty,j\in J} H_i\big(\mathbb{L}_{W_n(\mathbb{F}_p)[T_j^{p^{-k_j}}| j\in J]/W_n(\mathbb{F}_p)}\big).
\end{gather*}
For any $\{k_j\in \mathbb{N}\}_{j\in J}$, $W_n(\mathbb{F}_p)[T_j^{p^{-k_j}}| j\in J]$ is a polynomial algebra over $W_n(\mathbb{F}_p)$. Thus 
\begin{gather*}
H_i\big(\mathbb{L}_{W_n(\mathbb{F}_p)[T_j^{p^{-k_j}}| j\in J]/W_n(\mathbb{F}_p)}\big)\cong 
\begin{cases}
\Omega^1_{W_n(\mathbb{F}_p)[T_j^{p^{-k_j}}| j\in J]/W_n(\mathbb{F}_p)} &  \mbox{if}\ i=0 \\
0 & \mbox{otherwise}
\end{cases}.
\end{gather*}
But in $W_n(\mathbb{F}_p)[T_j^{p^{-(k_j+n)}}| j\in J]$, $\mathrm{d}T_j^{p^{-k_j}}=0$. Hence $H_i\big(\mathbb{L}_{W_n(L_J)/W_n(\mathbb{F}_p)}\big)=0$ for all $i$, and the conclusion follows.
\end{proof}

\begin{proposition}\label{prop:cotangentComp-Wk}
Let $\Bbbk$ be a perfect field of characteristic $p$. Let $A_n$ be a commutative smooth  $W_n(\Bbbk)$-algebra. Then there is a natural equivalence of cotangent complexes
\begin{gather*}
\mathbb{L}_{A_n/\mathbb{Z}}\simeq \mathbb{L}_{A_n/W(\Bbbk)}.
\end{gather*}
\end{proposition}
\begin{proof}
Let $W=W(\Bbbk)$.
By Elkik's theorem (\cite[Théorème 6]{Elk73}, see also \cite[Theorem 1.3.1]{Ara01}), there exists a smooth $W(\Bbbk)$-algebra $A$ lifting $A_n=A/p^n A$. Then there exists a smooth polynomial algebra $B=W[X_1,\dots,X_n]$ over $W$, and an étale $W$-homomorphism $B\rightarrow A$, such that $W\rightarrow A$ factors as $W\rightarrow B\rightarrow A$. Let  $B_n=B/p^n B$. Then (\ref{eq-transitivity-seq}) yields a fiber sequence
\begin{gather}\label{eq-fiberSeq-An/Z}
\mathbb{L}_{W/\mathbb{Z}}\otimes^{\mathbf{L}}_W A_n\rightarrow \mathbb{L}_{A_n/\mathbb{Z}}\rightarrow \mathbb{L}_{A_n/W}.
\end{gather}
We have
\begin{gather*}
\mathbb{L}_{W/\mathbb{Z}}\otimes^{\mathbf{L}}_W A_n
\simeq\mathbb{L}_{W/\mathbb{Z}}\otimes^{\mathbf{L}}_W B_n\otimes_{B_n}A_n,
\end{gather*}
and
\begin{gather*}
\mathbb{L}_{W/\mathbb{Z}}\otimes^{\mathbf{L}}_W B_n
\simeq
\mathbb{L}_{W/\mathbb{Z}}\otimes^{\mathbf{L}}_{\mathbb{Z}}W_n(\mathbb{F}_p)\otimes_{W_n(\mathbb{F}_p)}W_n(\mathbb{F}_p)[X_1,\dots,X_n].
\end{gather*}
By (\ref{eq-baseChange-cotangentComplex}) and Proposition \ref{prop:cotangentComp-Wnk}, we have
\begin{gather*}
\mathbb{L}_{W/\mathbb{Z}}\otimes^{\mathbf{L}}_{\mathbb{Z}}W_n(\mathbb{F}_p)\simeq \mathbb{L}_{W_n(\Bbbk)/W_n(\mathbb{F}_p)}\simeq 0.
\end{gather*}
Hence $\mathbb{L}_{W/\mathbb{Z}}\otimes^{\mathbf{L}}_W A_n\simeq 0$ and thus (\ref{eq-fiberSeq-An/Z}) implies $\mathbb{L}_{A_n/\mathbb{Z}}\simeq \mathbb{L}_{A_n/W}$.
\end{proof}

\begin{corollary}\label{cor:reduction-HH-HC-overZ-overW(k)}
Let $\Bbbk$ be a perfect field of characteristic $p$. Let $A$ be a commutative smooth  $W_n(\Bbbk)$-algebra. Then for all $i\geq 0$, there are natural isomorphisms 
\begin{align*}
\widetilde{\mathrm{HH}}_i(A)\cong \widetilde{\mathrm{HH}}_i\big(A/W(\Bbbk)\big),\quad
\widetilde{\mathrm{HC}}_i(A)\cong \widetilde{\mathrm{HC}}_i\big(A/W(\Bbbk)\big).
\end{align*}
\end{corollary}
\begin{proof}
By Proposition \ref{prop:cotangentComp-Wk} we have $\mathbb{L}_{A_n/\mathbb{Z}}\simeq \mathbb{L}_{A_n/W(\Bbbk)}$, and thus $\mathbb{L}^i_{A_n/\mathbb{Z}}\simeq \mathbb{L}^i_{A_n/W(\Bbbk)}$ for $i\in \mathbb{N}$. Namely, the natural map 
$\widetilde{\operatorname{HC}}(A/\mathbb{Z}) \rightarrow \widetilde{\operatorname{HC}}(A/W(\Bbbk))$ induces equivalences on all graded pieces of the filtrations in Proposition \ref{prop:HH-HC-filtration}(ii).  Then from the completeness of this filtration,  the conclusion follows.
\end{proof}


\section{Derived Hochschild homology of smooth algebras}
\label{sec:derived_hochschild_homology_of_smooth_algebras}
In this section we review  the HKR isomorphism for a smooth cdga (\cite[Prop.~2.5]{BuV88}, \cite[Prop.~5.4.6]{Lod98}), as a preparation for the computation of cyclic homology in the next section.
We fix a commutative ring $k$ as our base ring, and all the Hochschild homology and cyclic homology are those over $k$, so we omit $k$ in the notations such as $\operatorname{HH}(A/k)$. 

\subsection{A cdga resolution} 
\label{sub:a_typical_cdga}
We present the cdga which will be studied from \S\ref{sub:formality-cyclic-complex} to \S\ref{sec:mod_p_rel_cyclic_homology_and_the_products}.
Let $k$ be a commutative ring, $\ell\in k$. Let $A$ be a commutative flat $k$-algebra, and let $\overline{A}=A/\ell A$.
Let $\mathscr{A}$ be the cdga with $\mathscr{A}_0=A$, $\mathscr{A}_1=A\{\varepsilon\}$ (namely, a free $A$-module of rank 1 with a  generator $\varepsilon$), and $\mathscr{A}_i=0$ for $i\geq 2$, and
\begin{gather*}
\updelta:\mathscr{A}_1\rightarrow \mathscr{A}_0,\ r \varepsilon\mapsto \ell r\ \mbox{for}\ r\in A.
\end{gather*}
By Lemma \ref{lem:equivalence-Shukla-derivedHH-HC}  we have canonical equivalences
\begin{gather*}
\widetilde{\mathrm{HH}}(\overline{A})\simeq \mathrm{HH}(\mathscr{A}),\ 
\widetilde{\mathrm{HC}}(\overline{A})\simeq \mathrm{HC}(\mathscr{A}),\
\widetilde{\mathrm{HC}}^{-}(\overline{A})\simeq \mathrm{HC}^{-}(\mathscr{A}).
\end{gather*}


\subsection{Formality of the cyclic complex  of a smooth algebra}
\label{sub:formality-cyclic-complex}
Recall that (\cite[\S5.4.3]{Lod98}) for a cdga over $k$, the Kähler differential module $\Omega^{\bullet}_{\mathscr{A}/k}$ is defined by the relations
\begin{gather*}
\mathrm{d}(ab)=a \mathrm{d}b+(\mathrm{d}a)b=a \mathrm{d}b+(-1)^{|a||b|}b \mathrm{d}a,\label{eq-differentialModule-cdga-Leibniz}\\
\mathrm{d}a \mathrm{d}b=-(-1)^{|a||b|}\mathrm{d}b \mathrm{d}a.\label{eq-differentialModule-cdga-commutativity}
\end{gather*}

Now keep the setup of \S\ref{sub:a_typical_cdga} and assume further that  $A$ is a smooth algebra over $k$ of finite type; in fact, most of this subsection holds for all smooth cdga $\mathscr{A}$ over $k$, but we restrict ourselves to this case.  Then an element of $\Omega^{\bullet}_{\mathscr{A}/k}$ can be uniquely written as a sum of elements of the form $\alpha \varepsilon^j (\mathrm{d}\varepsilon)^{m}$ with $\alpha\in \Omega^{i}_{A/k}$, $i,m\in \mathbb{Z}$ satisfying $i,m\geq 0$, and $j=0$ or $1$. We define the \emph{differential degree} $\|-\|$ of $\alpha \varepsilon^j (\mathrm{d}\varepsilon)^{m}$ to be its degree as a Kähler differential, namely,
\begin{gather*}
\| \alpha \varepsilon^j (\mathrm{d}\varepsilon)^{m}\|=i+m,
\end{gather*}
and define its the \emph{total degree} to be
\begin{gather*}
\deg(\alpha \varepsilon^j \big(\mathrm{d}\varepsilon)^{m}\big)=|\alpha|+\|\alpha\|=i+j+2m.
\end{gather*}

For any integer $n\geq 0$, let $\mathbf{C}_n(A/\ell)$ be the submodule of $\Omega^{\bullet}_{\mathscr{A}/k}$ of the elements of total degree $n$. We have operators
\begin{align*}
\updelta:\mathbf{C}_n(A/\ell)\rightarrow \mathbf{C}_{n-1}(A/\ell),\ \mbox{and}\ \mathrm{d}:\mathbf{C}_n(A/\ell)\rightarrow \mathbf{C}_{n+1}(A/\ell),
\end{align*}
where
\begin{align}\label{eq:HHbf-partial-delta}
\updelta\big(\varepsilon \alpha (\mathrm{d}\varepsilon)^{m-1}\big)= \ell\alpha (\mathrm{d}\varepsilon)^{m-1},\quad
\updelta\big(\beta (\mathrm{d}\varepsilon)^{m}\big)= 0,
\end{align}
and
\begin{align*}
\mathrm{d}\big(\varepsilon \alpha (\mathrm{d}\varepsilon)^{m-1}\big)
&= (\mathrm{d}\alpha) \varepsilon (\mathrm{d}\varepsilon)^{m-1}
+(-1)^{\|\alpha\|}\alpha (\mathrm{d}\varepsilon)^{m},\\
\mathrm{d}\big(\beta (\mathrm{d}\varepsilon)^{m}\big)&=(\mathrm{d}\beta) (\mathrm{d}\varepsilon)^{m}.
\end{align*}

\begin{definition}\label{def:formal-HH}
Define $\mathbf{HH}_n(A/\ell)$ to be the $n$-th homology of the chain complex $\big(\mathbf{C}_{\bullet}(A/\ell),\updelta\big)$. 
Define a homomorphism 
$\pi:C_n(\mathscr{A},\updelta) \rightarrow \mathbf{C}_n(A/\ell)$ by
\begin{gather*}
\pi(a_0,\dots,a_n)=a_0 \mathrm{d}a_1\cdots \mathrm{d}a_n
\end{gather*}
for $a_i\in \mathscr{A}$, and let
$\varphi:C_n(\mathscr{A},\updelta) \rightarrow \mathbf{C}_n(A/\ell)[\frac{1}{n!}]$ be
\begin{gather*}
\varphi(a_0,\dots,a_n)=\frac{1}{n!}\pi(a_0,\dots,a_n).
\end{gather*}
\end{definition}

\begin{lemma}\label{lem:phi-commutesWith-partial-and-delta-product}
\begin{enumerate}[(i)]
  \item $\updelta\circ \mathrm{d}=\mathrm{d}\circ \updelta$.
  \item  $\updelta\circ \varphi=\varphi\circ \updelta$ and $\mathrm{d}\circ\varphi=\varphi\circ B$.
  \item $\varphi$ intertwines the shuffle product and the product of Kähler differentials.
\end{enumerate}
\end{lemma}
\begin{proof}
(i) and (ii) are standard, see \cite[Page 185-186]{Lod98}. (iii) can be verified directly from \eqref{eq:shuffleProduct-cdga}.
\end{proof}

\begin{theorem}\label{thm:formality}
For integers $n\geq 0$,
\begin{enumerate}[(i)]
  \item $\pi$ induces an isomorphism $\mathrm{HH}_n(\mathscr{A},0)[\frac{1}{n!}]\xrightarrow{\cong} \mathbf{C}_n(A/\ell)[\frac{1}{n!}]$, and
  \item $\varphi$ induces an isomorphism
\begin{gather*}
\widetilde{\mathrm{HH}}_n(A/\ell)[\frac{1}{n!}] \xrightarrow{\cong} \mathbf{HH}_n(A/\ell)[\frac{1}{n!}].
\end{gather*}
\end{enumerate}
\end{theorem}
\begin{proof}
 By the base change of the Hochschild homology, it suffices to show the case that $A$ is a polynomial algebra $k[V]$, where $V$ is a finite free $k$-module. Then $(\mathscr{A},0)=\Lambda V$. In this case, there is an antisymmetric map which induces an isomorphism $\epsilon:\mathbf{C}_n(A/\ell) \xrightarrow{\cong}\mathrm{HH}_n(\mathscr{A},0)$, and $\pi\circ \epsilon=n!$ on the $n$-th homology (\cite[Prop.~2.5]{BuV88}, and  also \cite[Prop.~5.4.4 and 5.4.6]{Lod98}) and then (i) follows.

(ii) follows from (i) and Lemma \ref{lem:phi-commutesWith-partial-and-delta-product}(i).
\end{proof}

\subsection{Towards the computation of (relative) cyclic homology} 
\label{sub:towards_the_computation_of_(relative)_cyclic_homology}
Keep the setup and notations in \S\ref{sub:formality-cyclic-complex}.
We make some preparations for the computations in \S\ref{sec:relative_derived_cyclic_homology_of_smooth_algebras_over_texorpdfstring_w_n_bbbk_}. As we recalled in \S\ref{sub:the_hochschild_and_cyclic_homology_of_cdga_s}, and by Theorem \ref{thm:formality}, the suitably localized derived cyclic homology of $A/\ell$ can be computed by alternating the operator $\updelta$; more precisely, by the following chain complex $\big(\mathbf{CC}_{\bullet}(A/\ell),\mathbf{D}'\big)$. However, it will turn out to be convenient\footnote{The convenience is in the construction of the map $\Psi$ in Lemma \ref{lem:map-Psi-complexes-twist(r)-component} and the proof of 
Theorem \ref{thm:rel-HC-homotopyEquiv}.} to introduce 
the operator
\begin{align*}
B:\mathbf{C}_n(A/\ell)\rightarrow \mathbf{C}_{n+1}(A/\ell)
\end{align*}
given by
\begin{align*}
B\big(\varepsilon \alpha (\mathrm{d}\varepsilon)^{m-1}\big)&= -\varepsilon\mathrm{d}\alpha (\mathrm{d}\varepsilon)^{m-1}
+\alpha (\mathrm{d}\varepsilon)^{m},\\
B\big(\alpha (\mathrm{d}\varepsilon)^{m}\big)&= \mathrm{d}\big(\alpha (\mathrm{d}\varepsilon)^{m}\big)=(\mathrm{d}\alpha) (\mathrm{d}\varepsilon)^{m},
\end{align*}
for $\alpha\in \Omega^{\bullet}_{A/k}$. 
One checks immediately that $B^2=0$ and
\begin{align*}
\updelta\circ B+ B\circ \updelta=0.
\end{align*}

\begin{definition}
We define a complex
\begin{align*}
\mathbf{CC}_n(A/\ell):=\mathbf{C}_n(A/\ell)\oplus \mathbf{C}_{n-2}(A/\ell)\oplus\dots\oplus\begin{cases}
\mathbf{C}_0(A/\ell),& \mbox{if}\ 2|n\\
\mathbf{C}_1(A/\ell),& \mbox{if}\ 2\nmid n
\end{cases}
\end{align*}
and   operators
\begin{gather*}
  D':=(-1)^{\|-\|}\updelta+\mathrm{d}:\mathbf{CC}_n(A/\ell)\rightarrow \mathbf{CC}_{n+1}(A/\ell),\\
  D:=\updelta+B:\mathbf{CC}_n(A/\ell)\rightarrow \mathbf{CC}_{n+1}(A/\ell),
\end{gather*}
where $(-1)^{\|-\|}$ is the operator $\alpha\mapsto (-1)^{\|\alpha\|}\alpha$ for $\alpha\in \Omega^{\bullet}_{\mathscr{A}/k}$ of pure differential degrees. Then we define  compositions
\begin{gather*}
\mathbf{D}':\mathbf{CC}_n(A/\ell)\xrightarrow{D'} \mathbf{CC}_{n+1}(A/\ell)\xrightarrow{\mathrm{proj}}\mathbf{CC}_{n-1}(A/\ell),\\
\mathbf{D}:\mathbf{CC}_n(A/\ell)\xrightarrow{D} \mathbf{CC}_{n+1}(A/\ell)\xrightarrow{\mathrm{proj}}\mathbf{CC}_{n-1}(A/\ell),
\end{gather*}
where $\mathrm{proj}$ is the obvious projection. Then $(\mathbf{D}')^2=\mathbf{D}^2=0$, and we denote the associated chain complex by $\big(\mathbf{CC}_{\bullet}(A/\ell),\mathbf{D}'\big)$ and $\big(\mathbf{CC}_{\bullet}(A/\ell),\mathbf{D}\big)$, respectively.
\end{definition}
The following lemma is straightforward.
\begin{lemma}\label{lem:iso-(pmDelta,d)-to-(delta,B)}
Let $g:\mathbb{N}\cup\{0\} \rightarrow \mathbb{Z}$ be a function such that $g(m)=g(m-1)+m-1$ for all $m$. 
Then the $k$-linear automorphism $\varrho$ of $\Omega^{\bullet}_{\mathscr{A}/k}$ defined by
\begin{align*}
\varrho\big(\alpha \varepsilon (\mathrm{d}\varepsilon)^{m-1}\big)&= (-1)^{\|\alpha\|+g(m)}\alpha \varepsilon (\mathrm{d}\varepsilon)^{m-1}\\
\mbox{and}\quad 
\varrho\big(\alpha (\mathrm{d}\varepsilon)^{m}\big)&= (-1)^{g(m)}\alpha (\mathrm{d}\varepsilon)^{m}
\end{align*}
satisfies
\begin{gather*}
B\circ \varrho=\varrho\circ \mathrm{d},\quad \updelta\circ \varrho=\varrho\circ (-1)^{\|-\|}\updelta.
\end{gather*}
In particular, $\varrho$ induces a quasi-isomorphism
\begin{gather*}
\big(\mathbf{CC}_{\bullet}(A/\ell),\mathbf{D}'\big)\xrightarrow{\sim} \big(\mathbf{CC}_{\bullet}(A/\ell),\mathbf{D}\big).
\end{gather*}
\end{lemma}

\begin{definition}
We define
$\mathbf{HC}_n(A/\ell)$ to be the $n$-th homology of the complex $\big(\mathbf{CC}_{\bullet}(A/\ell),\mathbf{D}\big)
$.
\end{definition}

\begin{lemma}\label{lem:connesExaSeq-HCbf}
There is a short exact sequence of chain complexes
\begin{gather*}
0 \rightarrow \big(\mathbf{C}(A/\ell),\updelta\big) \rightarrow \big(\mathbf{CC}(A/\ell),\mathbf{D}\big) \rightarrow \big(\mathbf{CC}(A/\ell)[2],\mathbf{D}\big) \rightarrow 0,
\end{gather*}
where the second arrow is given by the obvious inclusions $\mathbf{C}_n(A/\ell) \rightarrow \mathbf{CC}_n(A/\ell)$, and the third arrow the projections $\mathbf{CC}_n(A/\ell) \twoheadrightarrow \mathbf{CC}_{n-2}(A/\ell)$.
\end{lemma}
\begin{proof}
The  assertion amounts to showing the commutativity of the diagram
\begin{gather*}
\xymatrix{
  0 \ar[r] & \mathbf{C}_n \ar[d]_{\updelta} \ar[r] & \mathbf{CC}_n \ar[r] \ar[d]_{\mathbf{D}} & \mathbf{CC}_{n-2} \ar[r] \ar[d]^{\mathbf{D}} & 0 \\
    0 \ar[r] & \mathbf{C}_{n-1} \ar[r] & \mathbf{CC}_{n-1} \ar[r]  & \mathbf{CC}_{n-3} \ar[r]  & 0.}
\end{gather*}
Taking care of the projection, it is straightforward to check the commutativity.
\end{proof}

\begin{corollary}\label{cor:relHC-relBoldHC-quasiIsom}
The composition $\varrho\circ \varphi$ induces quasi-isomorphisms
\begin{gather*}
\tau_{\leq n} \widetilde{CC}(A/\ell)[\frac{1}{n!}]  \xrightarrow{\sim} \tau_{\leq n} \widetilde{\mathbf{CC}}(A/\ell)[\frac{1}{n!}]
\end{gather*}
for all integers $n\geq 0$,
\end{corollary}
\begin{proof}
 The quasi-isomorphisms follow from Theorem \ref{thm:formality}, Lemma \ref{lem:connesExaSeq-HCbf}, and induction on $n$.
\end{proof}

\section{Relative derived cyclic homology of smooth algebras over \texorpdfstring{$W_n(\Bbbk)$}{Wn(k)}}
\label{sec:relative_derived_cyclic_homology_of_smooth_algebras_over_texorpdfstring_w_n_bbbk_}
Let $\Bbbk$ be a perfect field of characteristic $p$. 
In this section, we fix a  smooth algebra $A$ of finite type over $W=W(\Bbbk)$, and compute the relative derived cyclic homology $\widetilde{\mathrm{HC}}_i(A_L,A_M)$ in the range $i\leq p-2$. We keep the notations of \S \ref{sub:formality-cyclic-complex}. Denote $\Omega^i_{A/W}$ by $\Omega^i$ for brevity. For $n\in \mathbb{N}$, we denote $A/p^n A$ by $A_n$. 
Suppose   $L,M\in \mathbb{Z}$ and $L>M\geq 1$.   We impose transition maps 
\begin{align*}
\mathbf{C}_n(A_L)\rightarrow \mathbf{C}_n(A_M)
\end{align*}
given by
\begin{equation}\label{eq:transitionMaps}
\begin{gathered}
\varepsilon\alpha   (\mathrm{d}\varepsilon)^{m-1}\mapsto \overline{\varepsilon \alpha  (\mathrm{d}\varepsilon)^{m-1}}:=p^{(L-M)m}\varepsilon  \alpha (\mathrm{d}\varepsilon)^{m-1},\\
 \beta (\mathrm{d}\varepsilon)^{m}\mapsto \overline{ \beta (\mathrm{d}\varepsilon)^{m}}:=p^{(L-M)m} \beta (\mathrm{d}\varepsilon)^{m},
\end{gathered}
\end{equation}
for $\alpha\in \Omega^{n-2m+1}$, $\beta\in \Omega^{n-2m}$. 

Recall our notations on (co)chain complexes in \S \ref{sub:notations}(iv).
Let $\mathbf{C}_{\bullet}(A_L,A_M)$ be  the total complex of the double complex
\begin{align*}
\xymatrix{
  \dots \ar[r]^<<<<<{\updelta} & \mathbf{C}_{n+1}(A_L) \ar[r]^{\updelta} \ar[d] & \mathbf{C}_{n}(A_L) \ar[r]^{\updelta} \ar[d] & \mathbf{C}_{n-1}(A_L) \ar[r]^>>>>>{\updelta} \ar[d] & \dots & (\mathrm{deg}\ 0) \\
  \dots \ar[r]^<<<<<{-\updelta} & \mathbf{C}_{n+1}(A_M) \ar[r]^{-\updelta} & \mathbf{C}_{n}(A_M) \ar[r]^{-\updelta} & \mathbf{C}_{n-1}(A_M) \ar[r]^>>>>>{-\updelta} & \dots & (\mathrm{deg}\ -1) \\
}
\end{align*}
and define
\begin{align*}
\mathbf{HH}_n(A_L,A_M):=\mbox{the $n$-th homology of the complex}\ \mathbf{C}_{\bullet}(A_L,A_M). 
\end{align*}
Let $\mathbf{CC}_{\bullet}(A_L,A_M)$ be  the total complex of the double complex
\begin{align*}
\xymatrix{
  \dots \ar[r]^<<<<<{\mathbf{D}} & \mathbf{CC}_{n+1}(A_L) \ar[r]^{\mathbf{D}} \ar[d] & \mathbf{CC}_{n}(A_L) \ar[r]^{\mathbf{D}} \ar[d] & \mathbf{CC}_{n-1}(A_L) \ar[r]^>>>>>{\mathbf{D}} \ar[d] & \dots & (\mathrm{deg}\ 0) \\
  \dots \ar[r]^<<<<<{-\mathbf{D}} & \mathbf{CC}_{n+1}(A_M) \ar[r]^{-\mathbf{D}} & \mathbf{CC}_{n}(A_M) \ar[r]^{-\mathbf{D}} & \mathbf{CC}_{n-1}(A_M) \ar[r]^>>>>>{-\mathbf{D}} & \dots & (\mathrm{deg}\ -1) \\
}
\end{align*}
and define
\begin{align*}
\mathbf{HC}_n(A_L,A_M):=\mbox{the $n$-th homology of the complex}\ \mathbf{CC}_{\bullet}(A_L,A_M). 
\end{align*}
Taking the cone of the exact sequence in Lemma \ref{lem:connesExaSeq-HCbf}, we obtain  a short exact sequence of chain complexes
\begin{gather}\label{eq:connesExaSeq-HCbf-HHbf-rel}
0 \rightarrow \mathbf{C}(A_L,A_M) \rightarrow \mathbf{CC}(A_L,A_M) \rightarrow \mathbf{CC}(A_L,A_M)[2] \rightarrow 0.
\end{gather}
By Corollary \ref{cor:relHC-relBoldHC-quasiIsom}, $\varrho\circ \varphi$ induces quasi-isomorphisms
\begin{gather}\label{eq:relHC-relBoldHC-quasiIsom} 
\tau_{\leq n} \widetilde{CC}(A_L,A_M)[\frac{1}{(n+1)!}]  \xrightarrow{\sim} \tau_{\leq n} \widetilde{\mathbf{CC}}(A_L,A_M)[\frac{1}{(n+1)!}] .
\end{gather}

\begin{proposition}\label{prop:relHC-relBoldHC-bounded-pExp-Isom}
Let $n\geq 0$ and $L>M\geq 1$ be integers. 
\begin{enumerate}[(i)]
  \item $\widetilde{\operatorname{HH}}_n(A_L,A_M)$ and $\widetilde{\operatorname{HC}}_n(A_L,A_M)$ are $p$-primary torsion groups of finite exponents.
  \item $\mathbf{HH}_n(A_L,A_M)$ and $\mathbf{HC}_n(A_L,A_M)$ are $p$-primary torsion groups of finite exponents.
  \item For $0\leq n\leq p-2$, $\varrho\circ \varphi$ induces isomorphisms
\begin{gather*}
\widetilde{\operatorname{HC}}_n(A_L,A_M)\xrightarrow{\cong} \mathbf{HC}_n(A_L,A_M).
\end{gather*}
\end{enumerate}
\end{proposition}
\begin{proof} 
(i) We have $\widetilde{\operatorname{HH}}_0(\mathbb{Z}/p^L)=\mathbb{Z}/p^L$. Since $\widetilde{\operatorname{HH}}_i(A_L)$ are $\widetilde{\operatorname{HH}}_0(\mathbb{Z}/p^L)$-modules, they are $p$-primary groups with finite $p$-exponents. The claim for $\widetilde{\operatorname{HC}}_i(A_L,A_M)$ then follows from Connes' exact sequence \eqref{eq:ISB-sequence} and the absolute-to-relative exact sequence. 

(ii) For $\mathbf{HH}_n(A_L,A_M)$, this follows from \eqref{eq:HHbf-partial-delta}. Then by the exact sequence induced by \eqref{eq:connesExaSeq-HCbf-HHbf-rel}  and induction on $n$, the assertion on $\mathbf{HC}_n(A_L,A_M)$ follows. 
The isomorphism (iii) then follows from  \eqref{eq:relHC-relBoldHC-quasiIsom}.
\end{proof}

\begin{definition}\label{def:relative-infinitesimal-complexes}
For an integer $L\geq 1$, we define $p^{r,L}\Omega^{\bullet}_{A}$ to be the complex
\begin{gather*}
p^{rL}A\rightarrow p^{(r-1)L}\Omega^1_{A/W}\rightarrow
\dots\rightarrow p^L\Omega_{A/W}^{r-1}
\rightarrow \Omega^r_{A/W}\rightarrow \Omega^{r+1}_{A/W}\rightarrow \dots
\end{gather*}
For integers $L>M\geq 1$, we define $p^{r,M}_{r,L}\Omega^{\bullet}_{A}$ to be the complex
\begin{gather*}
p^{rM} A_{rL}\rightarrow p^{(r-1)M}\Omega^1_{A_{(r-1)L}/W_{(r-1)L}}\rightarrow
\dots\rightarrow p^M\Omega_{A_{L}/W_{L}}^{r-1}\rightarrow 0\rightarrow \dots
\end{gather*}
\end{definition}
In the following of this section, we fix integers $L>M\geq 1$. There is an obvious map of complexes $p^{r,L}\Omega^{\bullet}_{A}\rightarrow p^{r,M}\Omega^{\bullet}_{A}$ and quasi-isomorphisms
\begin{gather}\label{eq:quasi-iso-prML}
\mathrm{Cone}\big(p^{r,L}\Omega^{\bullet}_{A}\rightarrow p^{r,M}\Omega^{\bullet}_{A}\big)\xrightarrow{\sim} 
\mathrm{Cone}\big(p^{r,L}\Omega^{<r}_{A}\rightarrow p^{r,M}\Omega^{<r}_{A}\big)\xrightarrow{\sim} p^{r,M}_{r,L}\Omega^{\bullet}_{A}\quad.
\end{gather}

\begin{lemma}\label{lem:map-Psi-complexes-twist(r)-component}
For $r\geq 0$, there exists a morphism of complexes
\begin{gather*}
\Psi: \mathrm{Cone}\big(p^{r+1,L}\Omega^{<r}_{A}\rightarrow p^{r+1,M}\Omega^{<r}_{A}\big)[2r] \rightarrow \mathbf{CC}^{\bullet}(A_L,A_M)
\end{gather*}
defined as follows. For $n\leq r$, the degree $-n$ component of $\mathrm{Cone}\big(p^{r+1,L}\Omega^{<r}_{A}\rightarrow p^{r+1,M}\Omega^{<r}_{A}\big)[2r]$ is zero.
Let $n\geq r+1$ be an integer. For an element $(\theta_{2r-n+1},\lambda_{2r-n})$ in the (cohomological)  degree $-n$ component, there exists unique $\eta_{2r-n+1} \in  \Omega^{2r-n+1}$ and $\kappa_{2r-n} \in  \Omega^{2r-n}$ such that 
\begin{gather*}
\theta_{2r-n+1}= p^{(n-r)L}\eta_{2r-n+1},\ \lambda_{2r-n}= p^{(n-r+1)M}\kappa_{2r-n}.
\end{gather*}
Moreover, if $n=r+1$, then $\theta_{2r-n+1}=0$. 
Then we define
\begin{align*}
\Psi(\theta_{2r-n+1},\lambda_{2r-n})={}&
\bigg(\sum_{m=1}^{n-r} (-1)^{m-1}p^{(n-r-m)L}\eta_{2r-n+1}\varepsilon(\mathrm{d}\varepsilon)^{m-1}
+\lambda_{2r-n},\\
&\quad\sum_{m=1}^{n-r+1}(-1)^{m-1}p^{(n-r-m+1)M}\kappa_{2r-n}\varepsilon(\mathrm{d}\varepsilon)^{m-1}\bigg).
\end{align*}
\end{lemma}
\begin{proof}
We need to show $\Psi\circ \mathrm{d}=\mathbf{D}\circ \Psi$. We have
\begin{align*}
&\Psi\circ\mathrm{d}(\theta_{2r-n+1},\lambda_{2r-n})\\
={}&\Psi(- \mathrm{d}\theta_{2r-n+1},\theta_{2r-n+1}+ \mathrm{d}\lambda_{2r-n})\\
={}&\Psi\big(p^{(n-r-1)L}(-p^L \mathrm{d}\eta_{2r-n+1}),p^{(n-r)L}\eta_{2r-n+1}+ p^{(n-r+1)M}\mathrm{d}\kappa_{2r-n}\big)\\ 
={}&\Psi\big(p^{(n-r-1)L}(-p^L \mathrm{d}\eta_{2r-n+1}),p^{(n-r)M}(p^{(n-r)(L-M)}\eta_{2r-n+1}+ p^{M}\mathrm{d}\kappa_{2r-n})\big)\\ 
={}&\bigg(\sum_{m=1}^{n-r-1} (-1)^{m-1}p^{(n-r-1-m)L}(-p^L \mathrm{d}\eta_{2r-n+1})\varepsilon(\mathrm{d}\varepsilon)^{m-1}\\
&+p^{(n-r)L}\eta_{2r-n+1}+ p^{(n-r+1)M}\mathrm{d}\kappa_{2r-n},\\
&\sum_{m=1}^{n-r}(-1)^{m-1}p^{(n-r-m)M}(p^{(n-r)(L-M)}\eta_{2r-n+1}+ p^{M}\mathrm{d}\kappa_{2r-n})\varepsilon(\mathrm{d}\varepsilon)^{m-1}\bigg)\\
={}&\bigg(\sum_{m=1}^{n-r-1} (-1)^{m}p^{(n-r-m)L}\mathrm{d}\eta_{2r-n+1}\varepsilon(\mathrm{d}\varepsilon)^{m-1}
+p^{(n-r)L}\eta_{2r-n+1}+ \mathrm{d}\lambda_{2r-n},\\
&\quad\sum_{m=1}^{n-r}(-1)^{m-1}p^{(n-r)L-mM}\eta_{2r-n+1}\varepsilon(\mathrm{d}\varepsilon)^{m-1}
+\sum_{m=1}^{n-r}(-1)^{m-1}p^{(n-r-m+1)M}\mathrm{d}\kappa_{2r-n}\varepsilon(\mathrm{d}\varepsilon)^{m-1}\bigg).
\end{align*}
On the other hand, we have
\begin{align*}
& \mathbf{D}\circ \Psi (\theta_{2r-n+1},\lambda_{2r-n}) \\
={}&\bigg(\mathbf{D}\sum_{m=1}^{n-r} (-1)^{m-1}p^{(n-r-m)L}\eta_{2r-n+1}\varepsilon(\mathrm{d}\varepsilon)^{m-1}
+\mathbf{D}\lambda_{2r-n},\\
&\quad\quad \sum_{m=1}^{n-r} (-1)^{m-1}p^{(n-r-m)L}p^{m(L-M)}\eta_{2r-n+1}\varepsilon(\mathrm{d}\varepsilon)^{m-1}
+\lambda_{2r-n}\\
&\hspace{3cm}- \mathbf{D}\sum_{m=1}^{n-r+1}(-1)^{m-1}p^{(n-r-m+1)M}\kappa_{2r-n}\varepsilon(\mathrm{d}\varepsilon)^{m-1}\bigg).
\end{align*}
Recall that $\mathbf{D}=\mathrm{proj}\circ D$. 
Let us first compute 
\begin{eqnarray}\label{eq:map-Psi-complexes-D}
&\bigg(D\sum_{m=1}^{n-r} (-1)^{m-1}p^{(n-r-m)L}\eta_{2r-n+1}\varepsilon(\mathrm{d}\varepsilon)^{m-1}
+D\lambda_{2r-n},\\
&\quad\quad \sum_{m=1}^{n-r} (-1)^{m-1}p^{(n-r-m)L}p^{m(L-M)}\eta_{2r-n+1}\varepsilon(\mathrm{d}\varepsilon)^{m-1}
+\lambda_{2r-n}\nn\\
&\hspace{3cm}- D\sum_{m=1}^{n-r+1}(-1)^{m-1}p^{(n-r-m+1)M}\kappa_{2r-n}\varepsilon(\mathrm{d}\varepsilon)^{m-1}\bigg)\nn\\
=&\bigg(\sum_{m=1}^{n-r}  (-1)^{m-1}p^{(n-r-m)L}
\big(p^L\eta_{2r-n+1}(\mathrm{d}\varepsilon)^{m-1}- \mathrm{d}\eta_{2r-n+1}\varepsilon(\mathrm{d}\varepsilon)^{m-1}+\eta_{2r-n+1}(\mathrm{d}\varepsilon)^{m}\big)
+\mathrm{d}\lambda_{2r-n},\nn\\
 &\sum_{m=1}^{n-r}(-1)^{m-1}p^{(n-r)L-mM}\eta_{2r-n+1}\varepsilon(\mathrm{d}\varepsilon)^{m-1}
+\lambda_{2r-n}\nn\\
&-\sum_{m=1}^{n-r+1}(-1)^{m-1}p^{(n-r-m+1)M}
\big(p^M\kappa_{2r-n}(\mathrm{d}\varepsilon)^{m-1}- \mathrm{d}\kappa_{2r-n}\varepsilon(\mathrm{d}\varepsilon)^{m-1}+\kappa_{2r-n}(\mathrm{d}\varepsilon)^{m}\big)\bigg).\nn
\end{eqnarray}
We have cancellations
\begin{align*}
&\sum_{m=1}^{n-r}  (-1)^{m-1}p^{(n-r-m)L}
\big(p^L\eta_{2r-n+1}(\mathrm{d}\varepsilon)^{m-1}+\eta_{2r-n+1}(\mathrm{d}\varepsilon)^{m}\big)\\
={}&p^{(n-r)L}\eta_{2r-n+1}+(-1)^{n-r+1}\eta_{2r-n+1}(\mathrm{d}\varepsilon)^{n-r},
\end{align*}
and
\begin{align*}
&\sum_{m=1}^{n-r+1}(-1)^{m-1}p^{(n-r-m+1)M}
\big(p^M\kappa_{2r-n}(\mathrm{d}\varepsilon)^{m-1}+\kappa_{2r-n}(\mathrm{d}\varepsilon)^{m}\big)\\
={}&p^{(n-r+1)M}\kappa_{2r-n}+(-1)^{n-r}\kappa_{2r-n}(\mathrm{d}\varepsilon)^{n-r+1}.
\end{align*}
Thus the RHS of \eqref{eq:map-Psi-complexes-D} is equal to
\begin{align*}
&\bigg(\sum_{m=1}^{n-r} (-1)^{m}p^{(n-r-m)L}
\mathrm{d}\eta_{2r-n+1}\varepsilon(\mathrm{d}\varepsilon)^{m-1}
+p^{(n-r)L}\eta_{2r-n+1}\\
&\quad\quad +(-1)^{n-r+1}\eta_{2r-n+1}(\mathrm{d}\varepsilon)^{n-r}
+\mathrm{d}\lambda_{2r-n},\\
&\quad \sum_{m=1}^{n-r}  (-1)^{m-1}p^{(n-r)L-mM}\eta_{2r-n+1}\varepsilon(\mathrm{d}\varepsilon)^{m-1}
+\lambda_{2r-n}\\
&\quad\quad -\sum_{m=1}^{n-r+1}(-1)^{m}p^{(n-r-m+1)M}
\mathrm{d}\kappa_{2r-n}\varepsilon(\mathrm{d}\varepsilon)^{m-1}\\
&\quad\quad\quad -p^{(n-r+1)M}\kappa_{2r-n}-(-1)^{n-r}\kappa_{2r-n}(\mathrm{d}\varepsilon)^{n-r+1}\bigg).
\end{align*}
But we have
\begin{gather*}
\deg \mathrm{d}\eta_{2r-n+1}\varepsilon(\mathrm{d}\varepsilon)^{(n-r)-1}=\deg \eta_{2r-n+1}(\mathrm{d}\varepsilon)^{n-r}=n+1,\\
\deg \mathrm{d}\kappa_{2r-n}\varepsilon(\mathrm{d}\varepsilon)^{(n-r+1)-1}=\deg \kappa_{2r-n}(\mathrm{d}\varepsilon)^{n-r+1}=n+2,
\end{gather*}
thus these terms are eliminated by the projection. Moreover, $\lambda_{2r-n}=p^{(n-r+1)M}\kappa_{2r-n}$.
Therefore 
\begin{align*}
& \mathbf{D}\circ \Psi (\theta_{2r-n+1},\lambda_{2r-n}) \\
={}&\bigg(\sum_{m=1}^{n-r-1} (-1)^{m}p^{(n-r-m)L}
\mathrm{d}\eta_{2r-n+1}\varepsilon(\mathrm{d}\varepsilon)^{m-1}
+p^{(n-r)L}\eta_{2r-n+1}
+\mathrm{d}\lambda_{2r-n},\\
&\quad \sum_{m=1}^{n-r}  (-1)^{m-1}p^{(n-r)L-mM}\eta_{2r-n+1}\varepsilon(\mathrm{d}\varepsilon)^{m-1}
-\sum_{m=1}^{n-r}(-1)^{m}p^{(n-r-m+1)M}
\mathrm{d}\kappa_{2r-n}\varepsilon(\mathrm{d}\varepsilon)^{m-1}
\bigg).
\end{align*}
Hence $\Psi\circ \mathrm{d}=\mathbf{D}\circ \Psi$.
\end{proof}

\begin{lemma}\label{lem:rel-HC-nonreduced-mod-p}
Assume $L> M\geq  1$.   Then 
\begin{gather*}
H_n\left( \bigoplus_{r\geq 0}p^{r+1,M}_{r+1,L}\Omega^{\bullet}_A[2r];\mathbb{Z}/p\right)\cong \bigoplus_{s=0}^n \Omega^s_{A_1/\Bbbk}.
\end{gather*}
Moreover, 
$H_n\big(p^{r,M}_{r,L}\Omega^{\bullet}_A;\mathbb{Z}/p^l\big)$ is a $A_l$-module for $1\leq l\leq M$.
\end{lemma}
\begin{proof}
The quasi-isomorphism \eqref{eq:quasi-iso-prML} gives the following flat resolution of $p^{r,M}_{r,L}\Omega^{\bullet}_A$:
\begin{flalign}\label{eq-infinitesimalMOtivicComplex-flatResolution}
\xymatrix{
A\ar[r]^{p^L \mathrm{d}} \ar[d]_{p^{r(L-M)}} &  \Omega^1_{A/W}\ar[r]^{p^L \mathrm{d}} \ar[d]_{p^{(r-1)(L-M)}}  &  \Omega^2_{A/W}\ar[r]^{p^L \mathrm{d}} \ar[d]_{p^{(r-2)(L-M)}} & 
\dots \ar[r]^<<<{p^L \mathrm{d}} &  \Omega_{A/W}^{r-1} \ar[d]^{p^{L-M}} \\
 A\ar[r]^{p^M \mathrm{d}} \ar[d]_{p^{rM}}
 &  \Omega^1_{A/W}\ar[r]^{p^M \mathrm{d}} \ar[d]_{p^{(r-1)M}} &  \Omega^2_{A/W}\ar[r]^{p^M \mathrm{d}} \ar[d]_{p^{(r-2)M}} & 
\dots \ar[r]^<<<{p^M\mathrm{d}} &  \Omega_{A/W}^{r-1} \ar[d]^{p^M} \\
  p^{rM} A_{rL}\ar[r]^<<<{\mathrm{d}} &  p^{(r-1)M}\Omega^1_{A_{(r-1)L}/W_{(r-1)L}}\ar[r]^{\mathrm{d}}  &  p^{(r-2)M}\Omega^2_{A_{(r-2)L}/W_{(r-2)L}}\ar[r]^<<<{\mathrm{d}} & 
\dots \ar[r]^<<<{\mathrm{d}} &  p^M\Omega_{A_{L}/W_{L}}^{r-1}
}
\end{flalign}
Thus for $L>M\geq 1$ there are canonical isomorphisms
\begin{gather}\label{eq:rel-HC-nonreduced-modP-components}
H^i\big(p^{r,M}_{r,L}\Omega^{\bullet}_A\otimes_{\mathbb{Z}}^{\mathbf{L}}\mathbb{Z}/p\big)\cong\begin{cases}
\quad\quad\quad\quad  A_1,& i=-1,\\
\Omega^{i}_{A_1/\Bbbk}\oplus \Omega^{i+1}_{A_1/\Bbbk},& 0\leq i\leq r-2,\\
\Omega^{r-1}_{A_1/\Bbbk},& i=r-1.
\end{cases}
\end{gather}
Hence
\begin{align*}
& H_n\left( \bigoplus_{r\geq 0}p^{r+1,M}_{r+1,L}\Omega^{\bullet}_A[2r];\mathbb{Z}/p\right)=\bigoplus_{r\geq 0}
H^{-n+2r}\big(p^{r+1,M}_{r+1,L}\Omega^{\bullet}_A\otimes_{\mathbb{Z}}^{\mathbf{L}}\mathbb{Z}/p\big)\\
& \cong \bigoplus_{r\leq n-1}(\Omega^{-n+2r}_{A_1/\Bbbk}\oplus \Omega^{-n+2r+1}_{A_1/\Bbbk})\oplus \Omega^{n}_{A_1/\Bbbk}
=\bigoplus_{s=0}^n \Omega^s_{A_1/\Bbbk},
\end{align*}
where we have set  $\Omega^{i}_{A_1/\Bbbk}=0$ if $i<0$.
\end{proof}
\begin{remark}\label{rem:p(r)Complex-modP}
It seems to be an essential feature of $p^{r,M}_{r,L}\Omega_{A}^{\bullet}$ that its mod $p$ cohomology are $A_1$-modules. In contrast, 
the complex $p^{r} A_n \xrightarrow{\mathrm{d}}  p^{r-1}\Omega^1_{A_{n}/W_{n}}
 \xrightarrow{\mathrm{d}}  p^{r-2}\Omega^2_{A_{n}/W_{n}} \xrightarrow{\mathrm{d}}  \dots $  after  sheafification and taking limits $n \rightarrow \infty$ also becomes the pro-complex $p(r)\Omega_{X_{\centerdot}}^{\bullet}$ in \cite[\S 2]{BEK14}, while its  mod $p$ cohomology groups  are in general  not $A_1$-modules. 
\end{remark}

\begin{theorem}\label{thm:rel-HC-homotopyEquiv}
The morphism of complexes
\begin{gather*}
\Psi:
\bigoplus_{r=0}^{+\infty}
\mathrm{Cone}\big(p^{r+1,L}\Omega^{<r}_{A}\rightarrow p^{r+1,M}\Omega^{<r}_{A}\big)[2r] \rightarrow \mathbf{CC}^{\bullet}(A_L,A_M),
\end{gather*}  
which in each direct summand is defined in Lemma \ref{lem:map-Psi-complexes-twist(r)-component}, is a quasi-isomorphism.
\end{theorem}
\begin{proof}
By Proposition \ref{prop:relHC-relBoldHC-bounded-pExp-Isom}(ii), the homology groups of the RHS are $p$-primary groups of finite exponents.
The cohomology groups of the LHS are also $p$-primary groups of finite exponents because of the quasi-isomorphism \eqref{eq:quasi-iso-prML} and the components of $p^{r,M}_{r,L}\Omega^{\bullet}$ are already $p$-primary groups of finite exponents. Thus it suffices to show that $\Psi\otimes^{\mathbf{L}}\mathbb{Z}/p$ is a quasi-isomorphism.
We first compute $\mathbf{HC}_n(A_L,A_M;\mathbb{Z}/p)$ in several steps. In this proof, we denote $\Omega_{A/W}^{i}$ by $\Omega^i$ for brevity.

\noindent\textbf{Step 1:} 
By the definition, $\mathbf{HC}_n(A_L,A_M)$ is equal to the following  quotient group:
the group  
\begin{gather}\label{eq:relativeHC-numerator}
\begin{array}{c}
\left\{\omega=\left.\begin{array}{c}\sum_{m=1}^{
\lfloor \frac{n+1}{2}\rfloor}\sum_{n-s-2m+1\in 2\mathbb{Z}_{\geq 0}} \alpha_{m;s}\varepsilon(\mathrm{d}\varepsilon)^{m-1}\\
+\sum_{m=0}^{\lfloor \frac{n}{2}\rfloor}\sum_{n-s-2m\in 2\mathbb{Z}_{\geq 0}}\beta_{m;s}(\mathrm{d}\varepsilon)^{m}
\end{array}\right| \alpha_{m;s},\beta_{m;s}\in\Omega^s\right\}\\
\oplus\left\{\upsilon=\left.\begin{array}{c}\sum_{m=1}^{
\lfloor \frac{n+2}{2}\rfloor}\sum_{n-s-2m+2\in 2\mathbb{Z}_{\geq 0}}\gamma_{m;s}\varepsilon(\mathrm{d}\varepsilon)^{m-1}\\
+\sum_{m=0}^{\lfloor \frac{n+1}{2}\rfloor}\sum_{n-s-2m+1\in 2\mathbb{Z}_{\geq 0}} \delta_{m;s}(\mathrm{d}\varepsilon)^m
\end{array}\right|  \gamma_{m;s},\delta_{m;s}\in\Omega^s\right\}
\end{array}
\end{gather}
consisting of  the elements $(\omega,\upsilon)$ satisfying 
\begin{gather*}
\mathbf{D}(\omega)=0,\ \mbox{and}\
\mathbf{D}(\upsilon)=\overline{\omega},
\end{gather*}
modulo the subgroup generated by 
\begin{gather}
\Big(\mathbf{D}\big(\alpha_{m;s} \varepsilon(\mathrm{d}\varepsilon)^{m-1}\big),\overline{\alpha_{m;s}\varepsilon(\mathrm{d}\varepsilon)^{m-1}}\Big)\ \mbox{for}\ 1\leq m\leq \lfloor \frac{n}{2}\rfloor,n-s-2m\in 2\mathbb{Z}_{\geq 0},\alpha_{m;s}\in \Omega^s;\label{rel-nonreduced-1}\tag{rel-1}\\
\Big(\mathbf{D}\big(\beta_{m;s}(\mathrm{d}\varepsilon)^{m}\big),\overline{\beta_{m;s}(\mathrm{d}\varepsilon)^{m}}\Big)\ \mbox{for}\ 0\leq m\leq \lfloor \frac{n-1}{2}\rfloor, n-s-2m-1\in 2\mathbb{Z}_{\geq 0},\beta_{m;s}\in \Omega^s;\label{rel-nonreduced-2}\tag{rel-2}\\
\Big(\updelta\big(\alpha_{m;s} \varepsilon(\mathrm{d}\varepsilon)^{m-1}\big),\overline{\alpha_{m;s}\varepsilon(\mathrm{d}\varepsilon)^{m-1}}\Big)\ \mbox{for}\ 1\leq m\leq \lfloor \frac{n+2}{2}\rfloor, s+2m=n+2,\alpha_{m;s}\in \Omega^s;\label{rel-nonreduced-3}\tag{rel-3}\\
\Big(\updelta\big(\beta_{m;s}(\mathrm{d}\varepsilon)^{m}\big),\overline{\beta_{m;s}(\mathrm{d}\varepsilon)^{m}}\Big)\ \mbox{for}\ 0\leq m\leq \lfloor \frac{n+1}{2}\rfloor,s+2m=n+1,\beta_{m;s}\in \Omega^s;\label{rel-nonreduced-4}\tag{rel-4}\\
\Big(0,\mathbf{D}\big(\gamma_{m;s}\varepsilon(\mathrm{d}\varepsilon)^{m-1}\big)\Big)\ \mbox{for}\ 1\leq m\leq \lfloor \frac{n+1}{2}\rfloor, n+1-s-2m\in 2\mathbb{Z}_{\geq 0},\gamma_{m;s}\in \Omega^s;\label{rel-nonreduced-5}\tag{rel-5}\\
\Big(0,\mathbf{D}\big(\delta_{m;s}(\mathrm{d}\varepsilon)^m\big)\Big)\ \mbox{for}\ 0\leq m\leq \lfloor \frac{n}{2}\rfloor, n-s-2m\in 2\mathbb{Z}_{\geq 0},\delta_{m;s}\in \Omega^s;\label{rel-nonreduced-6}\tag{rel-6}\\
\Big(0,\updelta\big(\gamma_{m;s}\varepsilon(\mathrm{d}\varepsilon)^{m-1}\big)\Big)\ \mbox{for}\ 1\leq m\leq \lfloor \frac{n+3}{2}\rfloor,s+2m=n+3,\gamma_{m;s}\in \Omega^s;\label{rel-nonreduced-7}\tag{rel-7}\\
\Big(0,\updelta\big(\delta_{m;s}(\mathrm{d}\varepsilon)^m\big)\Big)\ \mbox{for}\ 0\leq m\leq \lfloor \frac{n+2}{2}\rfloor,s+2m=n+2,\delta_{m;s}\in \Omega^s.\label{rel-nonreduced-8}\tag{rel-8}
\end{gather}
Since  $\Big(0,\updelta\big(\delta_{m;s}(\mathrm{d}\varepsilon)^m\big)\Big)=0$, \eqref{rel-nonreduced-8} is trivial.
Mod $p$, \eqref{rel-nonreduced-3} and \eqref{rel-nonreduced-7} become trivial, and the remaining five relations become
\begin{gather}
\left(-\varepsilon(\mathrm{d}\alpha_{m;s})(\mathrm{d}\varepsilon)^{m-1}+\alpha_{m;s}(\mathrm{d}\varepsilon)^{m},0\right)\label{rel-nonreduced-1-mod-$p$}\tag{rel-1-mod-$p$},\
 \mbox{for}\ 1\leq m\leq \lfloor \frac{n}{2}\rfloor.\nn\\
(\mathrm{d}\beta_{0;s},\beta_{0;s}),\ 
\mbox{and}\ 
\left(\mathrm{d}\beta_{m;s}(\mathrm{d}\varepsilon)^{m},0\right)\ \mbox{for}\ 1\leq m\leq \lfloor \frac{n-1}{2}\rfloor.\label{rel-nonreduced-2-mod-$p$}\tag{rel-2-mod-$p$}\\
\big(0,\beta_{0;n+1}\big).\label{rel-nonreduced-4-mod-$p$}\tag{rel-4-mod-$p$}\nn\\
\big(0,-\varepsilon(\mathrm{d}\gamma_{m;s})(\mathrm{d}\varepsilon)^{m-1}+\gamma_{m;s}(\mathrm{d}\varepsilon)^{m}\big),\label{rel-nonreduced-5-mod-$p$}\tag{rel-5-mod-$p$}\ 
\mbox{for}\ 1\leq m\leq \lfloor \frac{n+1}{2}\rfloor.\nn\\
\big(0, (\mathrm{d}\delta_{m;s})(\mathrm{d}\varepsilon)^m\big),\label{rel-nonreduced-6-mod-$p$}\tag{rel-6-mod-$p$}\
\mbox{for}\ 0\leq m\leq \lfloor \frac{n}{2}\rfloor.\nn
\end{gather}

\noindent\textbf{Step 2:}
In $\mathbf{HC}_n(A_L,A_M; \mathbb{Z}/p)$, by \eqref{rel-nonreduced-1-mod-$p$} we obtain, for $1\leq m\leq \lfloor \frac{n}{2}\rfloor$,
\begin{align*}
\big(\alpha(\mathrm{d}\varepsilon)^m,0\big)=\left(\varepsilon\mathrm{d}\alpha(\mathrm{d}\varepsilon)^{m-1},0\right).
\end{align*}
Then \eqref{rel-nonreduced-2-mod-$p$} reduces to $(\mathrm{proj}\circ\mathrm{d}\alpha,\alpha)=0$ for $\alpha\in \Omega^{s}$ with $n-s-1\in 2 \mathbb{Z}_{\geq 0}$.
By \eqref{rel-nonreduced-5-mod-$p$} we obtain, for $1\leq m\leq \lfloor \frac{n+1}{2}\rfloor$,
\begin{align*}
\big(0,\alpha(\mathrm{d}\varepsilon)^m\big)&=\left(0,\varepsilon\mathrm{d}\alpha(\mathrm{d}\varepsilon)^{m-1}\right).
\end{align*}
Then (\ref{rel-nonreduced-6-mod-$p$}) reduces to $(0,\mathrm{d}\alpha)=0$. Thus together with \eqref{rel-nonreduced-4-mod-$p$}, the effect of the mod $p$ relations is the elimination of the terms $\beta_{m;s}(\mathrm{d}\varepsilon)^{m}$ for $m\geq 1$ in the first summand, and $\delta_{m;s}(\mathrm{d}\varepsilon)^{m}$ for $m\geq 0$ in the second summand in \eqref{eq:relativeHC-numerator}.  It follows that
\begin{flalign*}
&\mathbf{HC}_n(A_L,A_M;\mathbb{Z}/p)\\
\cong&\left\{\left.\begin{array}{c}
\left\{\omega=\left.\begin{array}{c}\sum_{m=1}^{
\lfloor \frac{n+1}{2}\rfloor}\sum_{n-s-2m+1\in 2\mathbb{Z}_{\geq 0}} \varepsilon\alpha_{m;s}(\mathrm{d}\varepsilon)^{m-1}\\
+\sum_{n-s\in 2\mathbb{Z}_{\geq 0}}\beta_{0;s}
\end{array}\right| \alpha_{m;s},\beta_{0;s}\in\Omega^s/p\Omega^s\right\}\\
\oplus\left\{\upsilon=\left.\begin{array}{c}\sum_{m=1}^{
\lfloor \frac{n+2}{2}\rfloor}\sum_{n-s-2m+2\in 2\mathbb{Z}_{\geq 0}}\varepsilon\gamma_{m;s}(\mathrm{d}\varepsilon)^{m-1}
\end{array}\right|  \gamma_{m;s}\in\Omega^s/p\Omega^s\right\}
\end{array}\right| \begin{array}{c}
\mathbf{D}(\omega)\equiv0 \\
\mathbf{D}(\upsilon)\equiv\overline{\omega}\\ \mod p\end{array} \right\}.\nn
\end{flalign*}

\noindent\textbf{Step 3:} Now we compute $\mathbf{D}(\omega)$ and $\mathbf{D}(\upsilon)$.
We have
\begin{align*}
&\mathbf{D}(\upsilon)=\mathrm{pr}_{\leq n}\circ D\big(\sum_{m=1}^{\lfloor \frac{n+2}{2}\rfloor}\sum_{s\in n-2m+2- 2\mathbb{Z}_{\geq 0}}\varepsilon\gamma_{m;s}(\mathrm{d}\varepsilon)^{m-1}\big)\\
\equiv& -\sum_{m=1}^{\lfloor \frac{n}{2}\rfloor}\sum_{s\in n-2m- 2\mathbb{Z}_{\geq 0}}\varepsilon(\mathrm{d}\gamma_{m;s})(\mathrm{d}\varepsilon)^{m-1}
+\sum_{m=1}^{\lfloor \frac{n}{2}\rfloor}\sum_{s\in n-2m- 2\mathbb{Z}_{\geq 0}}\gamma_{m;s}(\mathrm{d}\varepsilon)^{m}\ \mod p
\end{align*}
and
\begin{align*}
&\mathbf{D}(\omega)=\mathrm{pr}_{\leq n-1}\circ D\big(\sum_{m=1}^{\lfloor \frac{n+1}{2}\rfloor}\sum_{s\in n-2m+1- 2\mathbb{Z}_{\geq 0}}\varepsilon\alpha_{m;s}(\mathrm{d}\varepsilon)^{m-1}+\sum_{s\in n- 2\mathbb{Z}_{\geq 0}}\beta_{0;s}\big)\\
\equiv{}& -\sum_{m=1}^{\lfloor \frac{n-1}{2}\rfloor}\sum_{s\in n-2m-1- 2\mathbb{Z}_{\geq 0}}\varepsilon(\mathrm{d}\alpha_{m;s})(\mathrm{d}\varepsilon)^{m-1}\\
&+\sum_{m=1}^{\lfloor \frac{n-1}{2}\rfloor}\sum_{s\in n-2m-1- 2\mathbb{Z}_{\geq 0}}\alpha_{m;s}(\mathrm{d}\varepsilon)^{m}
+\sum_{s\in n-2- 2\mathbb{Z}_{\geq 0}}\mathrm{d}\beta_{0;s}\ \mod p\quad.
\end{align*}
Thus $\mathbf{D}(\upsilon)\equiv\overline{\omega}\mod p$ is equivalent to 
\begin{align*}
\begin{cases}
\gamma_{m;s}\equiv0\mod p,& \mbox{for}\ 1\leq m\leq \lfloor \frac{n}{2}\rfloor,\ s\in n-2m-2\mathbb{Z}_{\geq 0};\\
\beta_{0;s}\equiv0\mod p,& \mbox{for}\ s\in n- 2\mathbb{Z}_{\geq 0}.
\end{cases}
\end{align*}
Similarly, $\mathbf{D}(\omega)\equiv0 \mod p$ is equivalent to
\begin{align*}
\begin{cases}
\alpha_{m;s}\equiv0\mod p,& \mbox{for}\ 1\leq m\leq \lfloor \frac{n-1}{2}\rfloor,\ s\in n-2m-1- 2\mathbb{Z}_{\geq 0};\\
\mathrm{d}\beta_{0;s-1}\equiv0\mod p,& \mbox{for}\ s\in n-1- 2\mathbb{Z}_{\geq 0},
\end{cases}
\end{align*}
Hence
\begin{flalign}\label{eq:rel-HC-nonreduced-modP}
&\mathbf{HC}_n(A_L,A_M;\mathbb{Z}/p)\\
\cong&\left\{
\bigg(\sum_{m=0}^{\lfloor \frac{n-1}{2}\rfloor} \varepsilon\alpha_{n-1-2m}(\mathrm{d}\varepsilon)^{m}
,\sum_{m=0}^{\lfloor \frac{n}{2}\rfloor}\varepsilon\gamma_{n-2m}(\mathrm{d}\varepsilon)^{m}\bigg)
\bigg| \alpha_{s},\gamma_s\in\Omega^s/p\Omega^s \right\}\nn\\
={}&\bigoplus_{i\geq 0}\Omega^{n-1-2i}_{A_1/\Bbbk}\oplus \bigoplus_{i\geq 0}\Omega^{n-2i}_{A_1/\Bbbk}\quad.\nn
\end{flalign}
This completes the computation. From \eqref{eq-infinitesimalMOtivicComplex-flatResolution} and \eqref{eq:rel-HC-nonreduced-modP}, it is clear that $\Psi/p$ induces isomorphisms on homology groups. So the theorem is proved.
\end{proof}

\begin{corollary}\label{cor:rel-HC-nonreduced-overZ}
For any integer $n\geq 0$ and integers $L>M\geq 1$, there is a canonical isomorphism
\begin{align}\label{eq:rel-HC-nonreduced-overZ}
\mathbf{HC}_n(A_L,A_M)
\cong & \bigoplus_{s\in n-2 \mathbb{Z}_{\geq 0}}\frac{\{\omega\in p^{(\frac{n-s}{2}+1)M} \Omega^{s}_{A/W}| \mathrm{d}\omega\in p^{\frac{n-s}{2}L} \Omega^{s+1}_{A/W}\}}{p^{(\frac{n-s}{2}+1)L} \Omega^{s}_{A/W}+\mathrm{d}\big(p^{(\frac{n-s}{2}+2)M} \Omega^{s-1}_{A/W}\big)}.
\end{align}
\end{corollary}
The underwaved terms in the following proposition will be used explicitly in Propositions \ref{prop:K1-multiplication-on-rel-HC} and \ref{prop:BottElement-multiplication-on-rel-HC}.

\begin{proposition}\label{prop:decomposition-HHbf-HCbf}
Let $n\geq 0$ be an integer. 
\begin{enumerate}[(i)]
  \item There is a canonical isomorphism
\begin{gather}\label{eq:HHbf-rel-mod-p-decomp}
\mathbf{HH}_n(A_L,A_M;\mathbb{Z}/p)\cong \bigoplus_{i=0}^{n-1}\Omega_{A_1/\Bbbk}^i \oplus \bigoplus_{i=0}^{n}\Omega_{A_1/\Bbbk}^i,
\end{gather}
 which is characterized by the following choices of representatives:
    \begin{align*}
    &\uwave{\big(\varepsilon \alpha (\mathrm{d}\varepsilon)^{m-1},0\big) \mapsto (\alpha \mod p,0)} &  \mbox{for}\ m\geq 1\ \mbox{and}\ \alpha\in \Omega_{A/W}^{n+1-2m},\\
    &\big(\alpha (\mathrm{d}\varepsilon)^{m},0\big) \mapsto (\alpha \mod p,0) & \mbox{for}\ m\geq 1\ \mbox{and}\ \alpha\in \Omega_{A/W}^{n-2m},\\
    &\uwave{\big(0,\varepsilon \alpha (\mathrm{d}\varepsilon)^{m-1}\big) \mapsto (0,\alpha \mod p)} &  \mbox{for}\ m\geq 1\ \mbox{and}\ \alpha\in \Omega_{A/W}^{n+2-2m},\\
    &\big(0,\alpha (\mathrm{d}\varepsilon)^{m}\big) \mapsto (0,\alpha \mod p) & \mbox{for}\ m\geq 1\ \mbox{and}\ \alpha\in \Omega_{A/W}^{n+1-2m}.
    \end{align*}
  \item There is a canonical isomorphism
\begin{gather}\label{eq:HCbf-rel-mod-p-decomp}
\mathbf{HC}_n(A_L,A_M;\mathbb{Z}/p)\cong \bigoplus_{i=0}^{n}\Omega_{A_1/\Bbbk}^i
\end{gather}
which is characterized by the following choices of representatives 
\begin{align*}
&\big(\varepsilon\alpha(\mathrm{d}\varepsilon)^{m},0\big)\mapsto \alpha \mod p & \mbox{for}\ m\geq 0\ \mbox{and}\ \alpha\in \Omega_{A/W}^{n-1-2m},\\
&\big(0,\varepsilon \alpha(\mathrm{d}\varepsilon)^{m}\big)\mapsto \alpha \mod p & \mbox{for}\ m\geq 0\ \mbox{and}\ \alpha\in \Omega_{A/W}^{n-2m}.
\end{align*}
  \item The long exact sequence of mod $p$ homology induced by \eqref{eq:connesExaSeq-HCbf-HHbf-rel} splits into short exact sequences
\begin{gather*}
0 \rightarrow \mathbf{HC}_{n-1}(A_L,A_M;\mathbb{Z}/p) \xrightarrow{\boldsymbol{\delta}} \mathbf{HH}_n(A_L,A_M;\mathbb{Z}/p) \xrightarrow{\mathbf{I}}
\mathbf{HC}_n(A_L,A_M;\mathbb{Z}/p) \rightarrow 0.
\end{gather*}
We denote the summand $\Omega_{A_1/\Bbbk}^i$ in the first big direct sum in RHS of \eqref{eq:HHbf-rel-mod-p-decomp} by 
  $(\Omega_{A_1/\Bbbk}^i)^{(1)}$ for $0\leq i\leq n-1$, and that in the second big direct sum by $(\Omega_{A_1/\Bbbk}^i)^{(2)}$ for $0\leq i\leq n$.
  Then via the identification \eqref{eq:HHbf-rel-mod-p-decomp} and \eqref{eq:HCbf-rel-mod-p-decomp}, the maps $\boldsymbol{\delta}$ and $\mathbf{I}$ have the following formulas:
  \begin{align*}
  &\boldsymbol{\delta}|_{\Omega_{A_1/\Bbbk}^{n-2i}}= 1-\mathrm{d}:\Omega_{A_1/\Bbbk}^{n-2-2i} \rightarrow (\Omega_{A_1/\Bbbk}^{n-2-2i})^{(1)}\oplus (\Omega_{A_1/\Bbbk}^{n-1-2i})^{(1)}\quad & \mbox{for}\ i\geq 0,\\
  &\boldsymbol{\delta}|_{\Omega_{A_1/\Bbbk}^{n-1-2i}}= 1-\mathrm{d}:\Omega_{A_1/\Bbbk}^{n-1-2i} \rightarrow (\Omega_{A_1/\Bbbk}^{n-1-2i})^{(2)}\oplus (\Omega_{A_1/\Bbbk}^{n-2i})^{(2)}\quad & \mbox{for}\ i\geq 0,\\
  &\uwave{\mathbf{I}|_{(\Omega_{A_1/\Bbbk}^{n-1-2i})^{(1)}}= 1:(\Omega_{A_1/\Bbbk}^{n-1-2i})^{(1)} \rightarrow \Omega_{A_1/\Bbbk}^{n-1-2i}}\quad & \mbox{for}\ i\geq 0,\\
  &\mathbf{I}|_{(\Omega_{A_1/\Bbbk}^{n-2-2i})^{(1)}}= \mathrm{d}:(\Omega_{A_1/\Bbbk}^{n-2-2i})^{(1)} \rightarrow \Omega_{A_1/\Bbbk}^{n-1-2i}\quad & \mbox{for}\ i\geq 0,\\
  &\uwave{\mathbf{I}|_{(\Omega_{A_1/\Bbbk}^{n-2i})^{(2)}}= 1:(\Omega_{A_1/\Bbbk}^{n-2i})^{(2)} \rightarrow \Omega_{A_1/\Bbbk}^{n-2i}}\quad & \mbox{for}\ i\geq 0,\\
  &\mathbf{I}|_{(\Omega_{A_1/\Bbbk}^{n-1-2i})^{(2)}}= \mathrm{d}:(\Omega_{A_1/\Bbbk}^{n-1-2i})^{(2)} \rightarrow \Omega_{A_1/\Bbbk}^{n-2i}\quad & \mbox{for}\ i\geq 0.
  \end{align*}
\end{enumerate}
\end{proposition}
\begin{proof}
(i) By definition, $\mathbf{C}_n(A/\ell)$ is a free $k$-module with a basis consisting of the elements
\begin{gather*}
\begin{cases}
\varepsilon \alpha (\mathrm{d}\varepsilon)^{m-1},\
\mbox{with}\ m\geq 1,\ 0\leq i\leq n,\ i+2m-1=n,\ \mbox{and}\ \alpha\in \Omega^i_{A/k}, \\
\beta (\mathrm{d}\varepsilon)^{m},\ \mbox{with}\ m\geq 0,\ 0\leq i\leq n,\ i+2m=n,\ \beta\in \Omega^i_{A/k}.
\end{cases}
\end{gather*}
 By \eqref{eq:transitionMaps}, the boundary operator of $\mathbf{C}_{\bullet}(A_L,A_M)\otimes \mathbb{Z}/p$ annihilates its generators except that it sends $\big(\beta (\mathrm{d}\varepsilon)^0,0\big)\mapsto (0,\beta (\mathrm{d}\varepsilon)^0)$ for $\beta \in \Omega^n \mod p$ and $n\geq 0$. Thus the assertion of part (i)  follows.

(ii) This is a reformulation of \eqref{eq:rel-HC-nonreduced-modP}.

(iii) This is obtained by a diagram chase on the exact sequence \eqref{eq:connesExaSeq-HCbf-HHbf-rel}, and we leave the details as an exercise.
\end{proof}

Then by Proposition \ref{prop:relHC-relBoldHC-bounded-pExp-Isom}(iii), we obtain
\begin{corollary}\label{cor:decomposition-HH-HC}
For any integer $n$ satisfying $0\leq n\leq p-2$, replacing $\mathbf{HH}$ with $\widetilde{\mathrm{HH}}$, and $\mathbf{HC}$ with $\widetilde{\mathrm{HC}}$, Corollary \ref{cor:rel-HC-nonreduced-overZ} and Proposition \ref{prop:decomposition-HHbf-HCbf} hold.
\end{corollary}

\section{Mod \texorpdfstring{$p$}{p} relative cyclic homology and the products} 
\label{sec:mod_p_rel_cyclic_homology_and_the_products}
In this section, for $i=1,2$ and $0\leq n\leq n+i\leq p-2$, by using the results in \S\ref{sec:derived_hochschild_homology_of_smooth_algebras} and \S\ref{sec:relative_derived_cyclic_homology_of_smooth_algebras_over_texorpdfstring_w_n_bbbk_},  we study the mod $p$ product 
\begin{gather*}
 K_i(A_L; \mathbb{Z}/p)\times  \widetilde{\operatorname{HC}}_n(A_L,A_M; \mathbb{Z}/p) \rightarrow \widetilde{\operatorname{HC}}_{n+i}(A_L,A_M; \mathbb{Z}/p)
 \end{gather*}  
 induced by the Chern character $\mathrm{ch}^{-}:K_i(A_L) \rightarrow \operatorname{HC}^{-}_i(A_L)$ and the topological product
 \begin{gather*}
 \operatorname{HC}^{-}_i(A_L; \mathbb{Z}/p)\times \widetilde{\operatorname{HC}}_n(A_L,A_M; \mathbb{Z}/p) \rightarrow \widetilde{\operatorname{HC}}_{n+i}(A_L,A_M; \mathbb{Z}/p).
\end{gather*} 
In \S\ref{sub:ring_complexes_and_mod_p_products} we set up representatives of mod $p$ cohomology and their products, which will be used in \S\ref{sub:the_k1_action_on_mod_p_relative_HC}, \ref{sub:the_bott_element_in_hh}, and also \S\ref{sub:the_mod_p_product_structure}.
\subsection{Ring complexes and mod \texorpdfstring{$p$}{p} products} 
\label{sub:ring_complexes_and_mod_p_products}

We recall the notion of ring complexes (\cite[Chap.~I \S 2]{Kat87}).
\begin{definition}\label{def:Kato's-product-structure}
Let $T$ be a topos. Let $A=(\cdots\rightarrow A^{i-1}\rightarrow A^i\rightarrow A^{i+1}\rightarrow\cdots)$ be a complex of abelian groups in $T$.  We say that $A$ is a \emph{ring complex} if there exists a global section $1$ of $A^0$, a homomorphism of complexes $A\otimes_{\mathbb{Z}} A\rightarrow A$, $x\otimes y\mapsto xy$, satisfying
\begin{gather*}
1x=x=x1,\quad x(yz)=(xy)z\quad.
\end{gather*}
In particular, we have
\begin{gather*}
\mathrm{d}(xy)=\mathrm{d}(x)y+(-1)^q x \mathrm{d}(y),\ \mbox{if}\ x\in A^q,\ y\in A^{q'}.
\end{gather*}
A \emph{ring map} of ring complexes is a morphism of complexes that preserves $1$ and the product.
\end{definition}

For a ring complex $S$, the elements in $\mathcal{H}^i(S\otimes^{\mathbf{L}}\mathbb{Z}/p)$ can be represented in two ways, by resolving $S$ or $\mathbb{Z}/p$. We record such representatives in the following two lemmas, and give a formula for the products of mod $p$ cohomology classes.

\begin{lemma}\label{lem:mod-p-product}
Let $T$ be a topos, and $S$ be a ring complex in $T$. Let $\mathbb{Z}/p \mathbb{Z}$ be the constant abelian group  in $T$ with the obvious product. Then $S\otimes^{\mathbf{L}}\mathbb{Z}/p \mathbb{Z}$ has an induced product structure, and the product on cohomology is given by the following description. There are splittable short exact sequences
\begin{gather*}
0\rightarrow H^i(S)/p\rightarrow H^i(S\otimes^{\mathbf{L}}\mathbb{Z}/p \mathbb{Z})\rightarrow \leftidx{_p}H^{i+1}(S)\rightarrow 0.
\end{gather*}
Moreover, $H^i(S\otimes^{\mathbf{L}}\mathbb{Z}/p \mathbb{Z})$ is computed by the complex
\begin{gather*}
\cdots \rightarrow S^{i-1}\oplus S^i\rightarrow S^i\oplus S^{i+1}\rightarrow S^{i+1}\oplus S^{i+2}\rightarrow \cdots
\end{gather*}
with
\begin{gather}\label{eq:mod-p-differential}
\mathrm{d}(a_i\oplus a_{i+1})=(\mathrm{d}a_i+p a_{i+1})\oplus(- \mathrm{d}a_{i+1}),
\end{gather}
and the product on this complex is given by
\begin{gather}\label{eq:mod-p-product}
(a_i\oplus a_{i+1})\cdot (b_j\oplus b_{j+1})=(a_ib_j)\oplus (a_{i+1}b_j+(-1)^ia_ib_{j+1}).
\end{gather}
Moreover, via the Eilenberg-Mac Lane functor $H$, this differential and product coincide with that induced by identifying $H(S\otimes^{\mathbf{L}}\mathbb{Z}/p)$ with $H(\mathbb{Z}/p)\land H(S)$ and using the product 
\begin{equation*}
\begin{gathered}
\big(H(\mathbb{Z}/p)\land H(S)\big)\times \big(H(\mathbb{Z}/p)\land H(S)\big)
\cong H(\mathbb{Z}/p)\land H(\mathbb{Z}/p)\land H(S)\land H(S)\\
\xrightarrow{\mathrm{aw}} H(\mathbb{Z}/p\otimes^{\mathbf{L}}\mathbb{Z}/p)\land H(S \otimes^{\mathbf{L}} S) \rightarrow H(\mathbb{Z}/p)\land H(S)
\end{gathered}
\end{equation*}
where the first arrow $\mathrm{aw}$ is induced by the Alexander-Whitney map (see e.g. \cite[Definition 29.7]{May67}), and $H(\mathbb{Z}/p\otimes^{\mathbf{L}}\mathbb{Z}/p) \rightarrow H(\mathbb{Z}/p)$ is given by the projection
\begin{gather*}
\mathbb{Z}/p\otimes^{\mathbf{L}}\mathbb{Z}/p\simeq \mathbb{Z}/p\oplus \mathbb{Z}/p[-1] \rightarrow \mathbb{Z}/p\quad.
\end{gather*}
\end{lemma}
\begin{proof}
We need only to show that the differential \eqref{eq:mod-p-differential} and the product \eqref{eq:mod-p-product} coincide with the induced ones. We can compute the induced differential and product by a cdga model of $H(\mathbb{Z}/p)\land H(S)$, namely, $\mathscr{S}=S\oplus \varepsilon S$ with $\varepsilon^2=0$ and $\mathrm{d}\varepsilon=p\in$ the first summand $S$. Then
\begin{gather*}
\mathrm{d}(a_i\oplus a_{i+1})=\mathrm{d}(a_i+\varepsilon a_{i+1})= \mathrm{d}a_{i}+pa_{i+1}-\varepsilon \mathrm{d}a_{i+1}
=(\mathrm{d}a_{i}+pa_{i+1})\oplus (- \mathrm{d}a_{i+1}),
\end{gather*}
and
\begin{align*}
(a_i\oplus a_{i+1})\cdot (b_j\oplus b_{j+1})&=(a_i+\varepsilon a_{i+1})\cdot (b_j+\varepsilon b_{j+1})\\
&=a_ib_j+\varepsilon a_{i+1}b_j+\varepsilon(-1)^ia_ib_{j+1}=(a_ib_j)\oplus (a_{i+1}b_j+(-1)^ia_ib_{j+1}).
\end{align*}
\end{proof}

\begin{remark}\label{rem:choice-mod-p-product}
One can also identify  $H(S\otimes^{\mathbf{L}}\mathbb{Z}/p)$ with $ H(S)\land H(\mathbb{Z}/p)$, 
and then replace the formula \eqref{eq:mod-p-product} with $(a_i\oplus a_{i+1})\cdot (b_j\oplus b_{j+1})=(a_ib_j)\oplus ((-1)^j a_{i+1}b_j+a_ib_{j+1})$, and the formula \eqref{eq:mod-p-differential} with $\mathrm{d}(a_i,a_{i+1})=(\mathrm{d}a_{i}+(-1)^{i+1}pa_{i+1},\mathrm{d}a_{i+1})$. A different identification will affect some intermediate mod $p$ computations, but not the results with integral coefficients, e.g. the main Theorem \ref{thm:relative-comparison-local}.
\end{remark}

The following two lemmas are straightforward.
\begin{lemma}\label{lem:mod-p-representative-flatResolution}
Suppose $S$ to be a bounded above complex of abelian groups in a topos $T$. Let $Q\xrightarrow{\varphi} S$ be a flat resolution. We have isomorphisms
\begin{gather*}
S\otimes^{\mathbf{L}}\mathbb{Z}/p \mathbb{Z}\cong Q\otimes \mathbb{Z}/p \mathbb{Z} \cong Q\otimes (\underset{\mathrm{deg}=-1}{\mathbb{Z}}\xrightarrow{\times p}
\underset{\mathrm{deg}=0}{\mathbb{Z}})\cong S\otimes (\underset{\mathrm{deg}=-1}{\mathbb{Z}}\xrightarrow{\times p}\underset{\mathrm{deg}=0}{\mathbb{Z}})
\end{gather*}
in the derived category. Let $\alpha$ be a cocycle in $(Q\otimes \mathbb{Z}/p \mathbb{Z})^i$. That is, $\alpha\in Q^i$ and there exists some $\beta\in Q^{i+1}$ (and thus unique by flatness) such that  $\mathrm{d}\alpha+p\beta=0$; by flatness we have $\mathrm{d}\beta=0$. Then $(\alpha,\beta)\in Q^i\oplus Q^{i+1}$ represents the class of $\alpha$ in $\big(Q\otimes (\mathbb{Z}\xrightarrow{\times p}\mathbb{Z})\big)^i$ via the above isomorphisms. Moreover, $(\varphi\alpha,\varphi\beta)\in S^i\oplus S^{i+1}$ represents the class of $\alpha$ in $\big(S\otimes (\mathbb{Z}\xrightarrow{\times p}\mathbb{Z})\big)^i$ in terms of the representatives in Lemma \ref{lem:mod-p-product}.
\end{lemma}

\begin{lemma}
For a cdga $(A,\updelta)$ over a commutative ring $k$, the Hochschild complex  $C(A,\updelta)$ with the shuffle product \eqref{eq:shuffleProduct-cdga} is a ring complex (on a point topos).
\end{lemma}

\subsection{The action of \texorpdfstring{$K_1$}{K1} on the mod \texorpdfstring{$p$}{p}  relative HC} 
\label{sub:the_k1_action_on_mod_p_relative_HC}
\begin{proposition}\label{prop:K1-multiplication-on-rel-HC}
Let $n$ be an integer such that $0\leq n\leq p-3$. 
Let $x\in A^{\times}$. Then the multiplication of $\mathrm{ch}^{-}(x)\in \mathrm{HC}_1^{-}(A)$ on $\widetilde{\mathrm{HC}}_n(A_L,A_M;\mathbb{Z}/p)$ is given by
\begin{gather*}
\xymatrix@R=1pc@C=4pc{
  \widetilde{\mathrm{HC}}_{n}(A_L,A_M;\mathbb{Z}/p) \ar[r]^{\mathrm{ch}^-(x)} & \widetilde{\mathrm{HC}}_{n+1}(A_L,A_M;\mathbb{Z}/p) \\
  \Omega_{A_1/\Bbbk}^{n-2i-1} \ar[r]^{-(\mathrm{d}\log x)\land} \ar@{}[u]|-{\rotatebox{90}{$\subset$}} &  
  \Omega_{A_1/\Bbbk}^{n-2i} \ar@{}[u]|-{\rotatebox{90}{$\subset$}}
}
\end{gather*}
and
\begin{gather*}
\xymatrix@R=1pc@C=4pc{
  \widetilde{\mathrm{HC}}_{n}(A_L,A_M;\mathbb{Z}/p) \ar[r]^{\mathrm{ch}^-(x)} & \widetilde{\mathrm{HC}}_{n+1}(A_L,A_M;\mathbb{Z}/p) \\
  \Omega_{A_1/\Bbbk}^{n-2i} \ar[r]^{(\mathrm{d}\log x)\land} \ar@{}[u]|-{\rotatebox{90}{$\subset$}} &  
  \Omega_{A_1/\Bbbk}^{n-2i+1} \ar@{}[u]|-{\rotatebox{90}{$\subset$}}
}
\end{gather*}
where the vertical inclusions are induced by the decomposition \eqref{eq:HCbf-rel-mod-p-decomp} and Corollary \ref{cor:decomposition-HH-HC}.
\end{proposition}
\begin{proof}
By Proposition \ref{prop:decomposition-HHbf-HCbf}(iii) (and Corollary \ref{cor:decomposition-HH-HC}), especially the underwaved terms, the map $\mathbf{I}$ is surjective on mod $p$ homology. Thus by Proposition \ref{prop:algebraic-and-topological-products-weak-comparison} 
and Variant \ref{var:generalizations_to_dgas_simplicial_rings_and_homology_with_coefficients}, the topological product on $\widetilde{\mathrm{HC}}_n(A_L,A_M;\mathbb{Z}/p)$ coincides with the algebraic product. 

By \cite[Prop.~8.4.9]{Lod98}, $h\circ\operatorname{ch}^{-}(x)=(x^{-1},x)$. 
Then by lemma  \ref{lem:phi-commutesWith-partial-and-delta-product}, we have
\begin{gather*}
\varphi\big((x^{-1},x)\times (a_0,\dots,a_n)\big)
=\mathrm{d}\log x\cdot \varphi(a_0,\dots,a_n).
\end{gather*}
where the product in the parenthesis on the LHS is the shuffle product (namely the algebraic product). Hence the conclusion follows from Lemma \ref{lem:mod-p-representative-flatResolution}, the formula \eqref{eq:mod-p-product}, the underwaved maps in Proposition \ref{prop:decomposition-HHbf-HCbf}(i), and the map $\varrho$ in Lemma \ref{lem:iso-(pmDelta,d)-to-(delta,B)}.
\end{proof}

\subsection{The Bott element in Hochschild homology} 
\label{sub:the_bott_element_in_hh}

\begin{definition}[{\cite[2.7.2]{Wei91}}]\label{def:bottElement}
Let $m\geq 2$ be an integer.
Let $R$ be a ring.  Suppose $x\in R$ and $x^m=1$. Let $\mu_m$ be the group of $m$-roots of 1 in $\mathbb{C}$, and $\zeta_m=\exp(\frac{2\pi\sqrt{-1}}{m})$.
The group homomorphism $\mu_m\rightarrow R^{\times}$, $\zeta_m\mapsto x$, induces a map $B\mu_m\rightarrow B GL(R)$, and thus the composition of the maps
\begin{gather*}
\mu_m\cong \leftidx_m\pi_1(B\mu_m)\cong \pi_2(B\mu_m;\mathbb{Z}/m)\rightarrow \pi_2(BGL(R);\mathbb{Z}/m)\rightarrow \pi_2(BGL(R)^+;\mathbb{Z}/m)=K_2(R;\mathbb{Z}/m)\quad.
\end{gather*}
We denote  the image of $\zeta_m$ under this map by $\beta_{x}$, and call it the Bott element of $K_2(R;\mathbb{Z}/m)$  (associated with $x$).
\end{definition}
For any ring $R$ and any $m\in \mathbb{Z}$, we have the Dennis trace map $\mathrm{Dtr}:K_i(R;\mathbb{Z}/m)\rightarrow \widetilde{\mathrm{HH}}_i(R;\mathbb{Z}/m)$ (see e.g. \cite{McC96}). In this subsection we compute the Dennis trace of the Bott element $\beta_x$, and use the result to compute the (filtered) multiplication by $\beta_x$ on the relative derived cyclic homology with coefficients $\mathbb{Z}/p$. We will  use  the following facts:
\begin{enumerate}[(i)]
   \item $\mathrm{Dtr}$ is functorial in $R$.
   \item  (\cite[\S 8.4.2]{Lod98})
   The composition
   \begin{gather*}
   K_i(R;\mathbb{Z}/m)\xrightarrow{\mathrm{Dtr}} \widetilde{\mathrm{HH}}_i(R;\mathbb{Z}/m)\rightarrow \mathrm{HH}_i(R;\mathbb{Z}/m)
   \end{gather*}
   is the Dennis trace map valued in the underived Hochschild homology and we denote it still by $\mathrm{Dtr}$. Recall the trace map $\mathrm{tr}:H_i(GL(R);\mathbb{Z}/m) \rightarrow \mathrm{HH}_i(R;\mathbb{Z}/m)$ induced by
   \begin{gather}\label{eq:DennisTrace}
   \xymatrix{
     \mathbb{Z}[GL_r(R)^i]\ar@{^{(}->}[r] & \mathbb{Z}[GL_r(R)^{i+1}] \ar[r] & M_r(R)^{i+1} \ar[r]^<<<<<{\mathrm{tr}} & R^{i+1}
     }
   \end{gather}
   where the first arrow is given by
   \begin{gather*}
        (g_1,\dots,g_i) \mapsto  \big((g_1\cdots g_i)^{-1},g_1,\dots,g_i\big)\quad.
   \end{gather*}
   Then the underived $\mathrm{Dtr}$ is equal to the composition
   \begin{gather*}
   K_i(R;\mathbb{Z}/m)\xrightarrow{\mathrm{Hurewicz\ map}} H_i(GL(R);\mathbb{Z}/m) \xrightarrow{\mathrm{tr}}  \mathrm{HH}_i(R;\mathbb{Z}/m).
   \end{gather*}
   \end{enumerate} 

 \begin{proposition}\label{prop:dennisTrace-bottEle-HH}
 Let $R$ and $x\in R$ be as in Definition \ref{def:bottElement}. Then in the normalized Hochschild complex which computes $\mathrm{HH}_2(R;\mathbb{Z}/m)$, we have
 \begin{gather}\label{eq:dennisTrace-bottEle-HH}
 \mathrm{Dtr}(\beta_x)=\sum_{i=1}^{m-1}(x^{-i-1},x^i,x)\quad.
 \end{gather}
 \end{proposition}
 \begin{proof}
Recall (see e.g.  \cite[Chap.~1]{Nei10}) that for a topological space $T$, $\pi_2(T; \mathbb{Z}/m \mathbb{Z})$ is the set of homotopy classes of pointed maps $[P^2(\mathbb{Z}/m \mathbb{Z}),T]_*$. 
Here the  Moore space $P^2(\mathbb{Z}/m \mathbb{Z})=\mathbf{S}^1\cup_m e^2$, where $e^2$ is a topological $2$-simplex, and the subscript $m$ means that the attaching map $\partial e^2=\mathbf{S}^1\rightarrow \mathbf{S}^1$ is an $m$-fold covering map which preserves the orientation. Moreover, the boundary map $\pi_2(T; \mathbb{Z}/m \mathbb{Z})\rightarrow \leftidx_m\pi_1(T)$ is given by $[f]\mapsto [f|_{\mathbf{S}^1}]$, and the Hurewicz map $\pi_2(T; \mathbb{Z}/m \mathbb{Z}) \rightarrow H_2(T; \mathbb{Z}/m \mathbb{Z})$ is given by $[f]\mapsto [f|_{e^2}]$.

Let $\zeta=\zeta_m$ for brevity.
Regard $B\mu_m$ as the geometrical realization of the nerve of $\mu_m$. The map $f:e^2\rightarrow B\mu_m$ represented by the subdivision
 
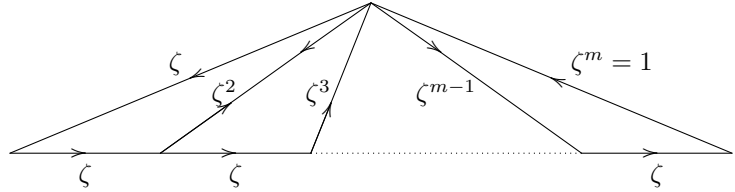
\begin{gather*}
\xy
(0,0); (-24,-10) **\dir{-}?>*\dir{>};(-48,-20), **\dir{-};(-38,-20) **\dir{-}?>*\dir{>}; (-28,-20) **\dir{-}; 
(-18,-20) **\dir{-}?>*\dir{>}; (-8,-20), **@{-}; (28,-20), **@{.}; (38,-20) **\dir{-}?>*\dir{>}; 
(48,-20),**@{-};(24,-10) **\dir{-}?>*\dir{>}; (0,0), **@{-};
(0,0);(-9.3,-6.6) **\dir{-}?>*\dir{>};  (-28,-20), **@{-};(-18.6,-13.3) **\dir{-}?>*\dir{>};
(0,0); (-8,-20), **@{-};(-5.3,-13.3) **\dir{-}?>*\dir{>};
(0,0); (9.3,-6.6) **\dir{-}?>*\dir{>}; (28,-20), **@{-}; 
(-26,-8)*{\zeta}; (-38,-23)*{\zeta};(-18,-23)*{\zeta}; (38,-23)*{\zeta};(32,-8)*{\zeta^m=1};
(-19.5,-12)*{\zeta^2};(-7,-12)*{\zeta^3};(10,-12)*{\zeta^{m-1}};
\endxy
 \end{gather*}
 maps to $\zeta\in \pi_1(B\mu_m)$ by the boundary map  $\pi_2(B\mu_m; \mathbb{Z}/m \mathbb{Z})\rightarrow \leftidx_m\pi_1(B\mu_m)$. Thus the composition 
 \begin{gather*}
 e^2\xrightarrow{f}B\mu_m \xrightarrow{\zeta\mapsto x} BGL_1(R) \rightarrow BGL(R) \rightarrow BGL(R)^+
 \end{gather*}
 represents the Bott element $\beta_x$. This composition factors through $BGL_1(R)=BR^{\times}$, and its homology class in $H_2(R^\times; \mathbb{Z}/m \mathbb{Z})$ is 
 \begin{gather*}
 [x|x]+[x^2|x]+\dots+[x^{m-1}|x]
 \end{gather*}
 in terms of the bar resolution definition of the group homology. Applying the map \eqref{eq:DennisTrace} to this element yields \eqref{eq:dennisTrace-bottEle-HH}.
 \end{proof}
 Now we specialize to the case of the \emph{non-flat} ring $\mathbb{Z}/p^n \mathbb{Z}$ and $m=p$, and compute the Dennis trace of the Bott element in its \emph{derived} Hochschild homology.
\begin{proposition}\label{prop:dennisTrace-bottEle-derivedHH}
Let $p$ be a prime number, $n\in \mathbb{N}$ and $n\geq 2$.  
Let $x=1+p^{n-1}\in \mathbb{Z}$. Assume $(p,n)\neq (2,2)$,  so that  $\bar{x}=x\mod p^n$ is a primitive $p$-th root of unity in $\mathbb{Z}/p^n \mathbb{Z}$, and $\frac{x^p-1}{p^n}\in \mathbb{Z}\cap \mathbb{Z}_{(p)}^{\times}$. Let $\widetilde{\mathbb{Z}/p^n \mathbb{Z}}$ be the cdga $\mathbb{Z}\oplus \mathbb{Z}\varepsilon$  with $\deg(\varepsilon)=1$ and $\updelta \varepsilon=p^n$. Then via the identification $\widetilde{\mathrm{HH}}_2(\mathbb{Z}/p^n \mathbb{Z}; \mathbb{Z}/p \mathbb{Z})\cong \mathrm{HH}_2(\widetilde{\mathbb{Z}/p^n \mathbb{Z}}; \mathbb{Z}/p \mathbb{Z})$, we have
\begin{gather}\label{eq:dennisTrace-bottEle-derivedHH}
 \mathrm{Dtr}(\beta_{\bar{x}})=(1,\varepsilon) 
\end{gather}
in the normalized Hochschild complex which computes $\mathrm{HH}_2(\widetilde{\mathbb{Z}/p^n \mathbb{Z}}; \mathbb{Z}/p \mathbb{Z})$.
\end{proposition}
\begin{proof}
Let $u:=\frac{x^p-1}{p^n}$.
Let $R=\mathbb{Z}[X]/(X^p-1)$, and let $\widetilde{R}$ be the cdga $\mathbb{Z}[X]\oplus \mathbb{Z}[X]\varepsilon$ with $\updelta \varepsilon= X^p-1$. Then $\widetilde{R}$ is a cdga resolution of $R$, and since $R$ is flat, we have $\mathrm{HH}_i(R;\mathbb{Z}/p)\cong \mathrm{HH}_i(\widetilde{R};\mathbb{Z}/p)$.
There is a map  $g:\widetilde{R} \rightarrow \widetilde{\mathbb{Z}/p^n \mathbb{Z}}$ of cdga's given by the commutative diagram
 \begin{gather*}
 \xymatrix@C=8pc{
   \mathbb{Z}[X]\varepsilon \ar[r]^{X^k\mapsto (1+p^{n-1})^k u} \ar[d]_{\times (X^p-1)} & \mathbb{Z} \varepsilon \ar[d]^{\times p^n} \\
   \mathbb{Z}[X] \ar[r]^{X\mapsto 1+p^{n-1}} & \mathbb{Z},
 }
 \end{gather*}
 where the lower horizontal map is a ring map, which makes $\mathbb{Z}$ a $\mathbb{Z}[X]$-algebra, and the upper horizontal map is then a map of $\mathbb{Z}[X]$-modules induced by $\varepsilon\mapsto u$.

 By Proposition \ref{prop:dennisTrace-bottEle-HH}, $ \mathrm{Dtr}(\beta_X)=\sum_{i=1}^{m-1}(X^{-i-1},X^i,X)\in \mathrm{HH}_2(R;\mathbb{Z}/p)$. We have to find its corresponding element in $\mathrm{HH}_2(\widetilde{R};\mathbb{Z}/p)$. In the Hochschild complex of $\widetilde{R}$ we have
 \begin{gather*}
b\sum_{i=1}^{p-1} (X^{p-i-1},X^i,X)=p(X^{p-1},X)-(1,X^p)=p(X^{p-1},X)-(1,1)-(1,\updelta \varepsilon),
 \end{gather*}
 and
 \begin{gather*}
b(1, \varepsilon)=\varepsilon-\varepsilon=0,\ \updelta(1,\varepsilon)=(1,\updelta\varepsilon).
 \end{gather*}
Hence in the normalized Hochschild complex computing $\mathrm{HH}_2(\widetilde{R};\mathbb{Z}/p)$, we have
\begin{gather*}
\mathrm{Dtr}(\beta_X)=(1, \varepsilon)+\sum_{i=1}^{p-1} (X^{p-i-1},X^i,X).
\end{gather*}
By the functoriality of $\mathrm{Dtr}$, we apply the map $g:\widetilde{R} \rightarrow \widetilde{\mathbb{Z}/p^n \mathbb{Z}}$ to $(1, \varepsilon)+\sum_{i=1}^{p-1} (X^{p-i-1},X^i,X)$.  Passing to the normalized Hochschild complex and using $u\equiv 1 \mod p$,  we obtain \eqref{eq:dennisTrace-bottEle-derivedHH}.
\end{proof}


\begin{proposition}\label{prop:BottElement-multiplication-on-rel-HC}
Let $p$ be an odd prime number. Let $L>M\geq 1$ be integers. Let $n$ be an integer such that $0\leq n\leq p-4$. 
 Let $x=1+p^{L-1}$. 
Then the multiplication of the Bott element $\beta_x$ on $\widetilde{\mathrm{HC}}_n(A_L,A_M;\mathbb{Z}/p)$ via Brun's isomorphism preserves the decomposition \eqref{eq:HCbf-rel-mod-p-decomp} and  is given by 
\begin{equation}\label{eq:BottElement-multiplication-on-rel-HC-1}
\begin{gathered}
\xymatrix@R=1pc@C=4pc{
  \widetilde{\mathrm{HC}}(A_L,A_M;\mathbb{Z}/p) \ar[r]^{\beta_x} & \widetilde{\mathrm{HC}}_{n+2}(A_L,A_M;\mathbb{Z}/p) \\
  \Omega_{A_1/\Bbbk}^{n-1-2i} \ar[r]^{(-1)^{n-i}} \ar@{}[u]|-{\rotatebox{90}{$\subset$}} &  
  \Omega_{A_1/\Bbbk}^{n-1-2i} \ar@{}[u]|-{\rotatebox{90}{$\subset$}}
}
\end{gathered}
\end{equation}
and 
\begin{equation}\label{eq:BottElement-multiplication-on-rel-HC-2}
\begin{gathered}
\xymatrix@R=1pc@C=4pc{
  \widetilde{\mathrm{HC}}_{n}(A_L,A_M;\mathbb{Z}/p) \ar[r]^{\beta_x} & \widetilde{\mathrm{HC}}_{n+2}(A_L,A_M;\mathbb{Z}/p) \\
  \Omega_{A_1/\Bbbk}^{n-2i} \ar[r]^{0} \ar@{}[u]|-{\rotatebox{90}{$\subset$}} &  
  \Omega_{A_1/\Bbbk}^{n-2i} \ar@{}[u]|-{\rotatebox{90}{$\subset$}}
}
\end{gathered}
\end{equation}
for $i\geq 0$, where the vertical inclusions are induced by the decomposition \eqref{eq:HCbf-rel-mod-p-decomp}.
\end{proposition}
\begin{proof}
By the argument of first paragraph in the proof of Proposition \ref{prop:K1-multiplication-on-rel-HC}, we only need to compute the shuffle multiplication by $\operatorname{Dtr}\beta_{\zeta}$. 
By \eqref{eq:dennisTrace-bottEle-derivedHH}, $\varphi(\operatorname{Dtr}\beta_{\zeta})=\mathrm{d}\varepsilon$. Then by Lemma \ref{lem:phi-commutesWith-partial-and-delta-product}(iii), the left multiplication by $\operatorname{Dtr}\beta_{\zeta}$ on $\big(\mathbf{CC}_*(A_L;\mathbb{Z}/p),\mathbf{D}'\big)$ is given by 
\begin{gather*}
\mathrm{d}\varepsilon\cdot
\alpha\varepsilon (\mathrm{d}\varepsilon)^{m-1}=(-1)^{\|\alpha\|+1}  \alpha \varepsilon(\mathrm{d}\varepsilon)^{m}.
\end{gather*}
Transferring this multiplication to $\big(\mathbf{CC}_*(A_L;\mathbb{Z}/p),\mathbf{D}\big)$ is given by the commutative diagram
\begin{gather*}
\xymatrix{
  \alpha \varepsilon(\mathrm{d}\varepsilon)^{m-1} \ar@{|->}[d]_{\beta_{\zeta}} \ar@{|->}[r]^<<<<<<<{\varrho} & 
  (-1)^{\|\alpha\|+g(m)} \alpha \varepsilon(\mathrm{d}\varepsilon)^{m-1} \ar@{|-->}[d] \\
  (-1)^{\|\alpha\|+1}  \alpha \varepsilon(\mathrm{d}\varepsilon)^{m} \ar[r]^<<<<<{\varrho} & (-1)^{1+g(m+1)} \alpha \varepsilon(\mathrm{d}\varepsilon)^{m}
}
\end{gather*}
and thus the induced right vertical arrow is given by $\alpha \mapsto (-1)^{1+g(m+1)-\|\alpha\|-g(m)}\alpha=(-1)^{\|\alpha\|+m}\alpha$. Hence the lower horizontal arrow in the diagram \eqref{eq:BottElement-multiplication-on-rel-HC-1} follows. The 0 map in the  diagram  \eqref{eq:BottElement-multiplication-on-rel-HC-2}  follows from $\varphi(\operatorname{Dtr}\beta_{\zeta})=\mathrm{d}\varepsilon$ and \eqref{eq:transitionMaps}.
\end{proof}

\section{Infinitesimal motivic complexes} 
\label{sec:infinitesimal_motivic_complexes}
Let $\Bbbk$ be a perfect field with $\operatorname{char}(\Bbbk)=p$. Let $W=W(\Bbbk)$, and $W_n=W_n(\Bbbk)$.
Let $\mathrm{Sm}_{W_{\centerdot}}$ be the category in \S\ref{sub:notations}(vi).
Let $X_{\centerdot}\in \mathrm{Sm}_{W_{\centerdot}}$. For a smooth scheme $Y_m$ over $W_m$, we let $\Omega_{Y_m}^{i}$ stand for $\Omega_{Y_m/W_m}^{i}$.
 In this section, we define the infinitesimal motivic complexes $\mathbb{Z}_{X_n}(r)$ and their product structure.

\subsection{Pro-complexes} 
\label{sub:pro_complexes}
The complex $\mathbb{Z}_{X_n}(r)$  can be represented by a complex of Nisnevich sheaves on $X_1$. However, we still need   constructions and results on the pro-complexes. Let us recall several necessary notions and facts about pro-complexes  (see e.g. \cite{ArM69}, \cite{Isa04}, and \cite[Appendices]{BEK14}).
\begin{itemize}
  \item Let $(\mathbb{N},\geq )$ be the category with objects $\mathbb{N}$ and morphisms $n_1 \rightarrow n_2$ if $n_1 \geq n_2$. For a category  $\mathrm{C}$, let $\mathrm{C}_{\mathrm{pro}}$ be the category with objects the functors $Y_{\centerdot}:(\mathbb{N},\geq)\rightarrow\mathrm{C}$, and with morphisms
\begin{equation*}
  \mathrm{Mor}_{\mathrm{C}_{\mathrm{pro}}}(Y_{\centerdot},Y'_{\centerdot})=\varprojlim_{n}\varinjlim_m \mathrm{Mor}_{\mathrm{C}}(Y_m,Y'_n).
\end{equation*}
A \emph{level representation} of a morphism $f\in  \mathrm{Mor}_{\mathrm{C}_{\mathrm{pro}}}(Y_{\centerdot},Y'_{\centerdot})$ consists of two strictly increasing functions $\alpha,\beta:\mathbb{N} \rightarrow \mathbb{N}$ and a commutative diagram
\begin{gather*}
\xymatrix{
  \dots \ar[r] & Y_{\alpha(n)} \ar[r] \ar[d] & Y_{\alpha(n-1)} \ar[r] \ar[d] & \dots \ar[r] & Y_{\alpha(1)} \ar[d] \\
  \dots \ar[r] & Y'_{\beta(n)} \ar[r]  & Y'_{\beta(n-1)} \ar[r] & \dots \ar[r] & Y'_{\beta(1)} 
}
\end{gather*}
which represents $f$.
  \item For a ringed topos $(T,\mathcal{O})$, let $\mathrm{C}(T)$ (reps. $\mathrm{C}^{-}(T)$, resp, $\mathrm{C}^+(T)$) denote the category of complexes (resp. right bounded, resp. left bounded) of sheaves of $\mathcal{O}$-modules on $T$. Let $\mathrm{C}_{\mathrm{pro}}(T):=\mathrm{C}(T)_{\mathrm{pro}}$, and similarly for $\mathrm{C}^{-}_{\mathrm{pro}}(T)$ and $\mathrm{C}^{+}_{\mathrm{pro}}(T)$.

  Let the weak equivalences in $\mathrm{C}(T)$ be the quasi-isomorphisms, and $\mathrm{D}(T)$ be the associated homotopy category. Let the weak equivalences in $\mathrm{C}_{\mathrm{pro}}(T)$ be the morphisms which have a level representation which is a quasi-isomorphism, and $\mathrm{D}_{\mathrm{pro}}(T)$ be the associated homotopy category.
  Similarly we have $\mathrm{D}^{-}(T)$, $\mathrm{D}^{+}(T)$, $\mathrm{D}_{\mathrm{pro}}^{-}(T)$, and $\mathrm{D}_{\mathrm{pro}}^{+}(T)$.
  \item There is an obvious inclusion functor $\iota:\mathrm{C}(T)\rightarrow \mathrm{C}_{\mathrm{pro}}(T)$ which is fully faithful, and an induced functor $\iota:\mathrm{D}(T)\rightarrow \mathrm{D}_{\mathrm{pro}}(T)$. There is a functor $\mathbf{R} \varprojlim: \mathrm{D}_{\mathrm{pro}}(X_1) \rightarrow \mathrm{D}(X_1)$, such that there is a canonical equivalence $\mathrm{id}\simeq \mathbf{R}\varprojlim \circ \iota$. Thus $\iota:\mathrm{D}(T)\rightarrow \mathrm{D}_{\mathrm{pro}}(T)$ is also fully faithful. We regard an object of $\mathrm{C}(T)$ (resp. $\mathrm{D}(T)$) as an object of $\mathrm{C}_{\mathrm{pro}}(T)$ (resp. $\mathrm{D}_{\mathrm{pro}}(T)$) via these inclusion functors.
  \item  For a morphism of ringed topoi $f:(T,\mathcal{O}_T) \rightarrow (S,\mathcal{O}_S)$, we define the derived pullback $\mathbf{L}f^*:\mathrm{D}^{-}_{\mathrm{pro}}(S) \rightarrow \mathrm{D}^{-}_{\mathrm{pro}}(T)$ index-wise. More precisely, for $E_{\centerdot}\in \mathrm{C}^{-}_{\mathrm{pro}}(S)$, taking $E'_{\centerdot}\in \mathrm{C}^{-}_{\mathrm{pro}}(S)$ with a commutative diagram in $\mathrm{C}(S)$
  \begin{gather*}
\xymatrix{
  \dots \ar[r] & E'_n \ar[r] \ar[d] & E'_{n-1} \ar[r] \ar[d] & \dots \ar[r] & E'_1 \ar[d] \\
  \dots \ar[r] & E_n \ar[r]  & E_{n-1} \ar[r] & \dots \ar[r] & E_1
}
\end{gather*}
such that each $E'_n \rightarrow E_n$ is a flat resolution, we define $\mathbf{L}f^*E_{\centerdot}=f^{-1}E'_{\centerdot}\otimes_{f^{-1}\mathcal{O}_S}\mathcal{O}_T$. In the same way, we define the derived tensor product $\otimes^{\mathbf{L}}:\mathrm{D}^{-}_{\mathrm{pro}}(S)\times \mathrm{D}^{-}_{\mathrm{pro}}(S) \rightarrow \mathrm{D}^{-}_{\mathrm{pro}}(S)$.
  \item  For a morphism of ringed topoi $f:(T,\mathcal{O}_T) \rightarrow (S,\mathcal{O}_S)$, where $T$ has \emph{a conservative set of points} \cite[Exp.~XVII Définition 4.2.2]{AGV71}, we define the derived pushforward $\mathbf{R}f_*:\mathrm{D}^{+}_{\mathrm{pro}}(T) \rightarrow \mathrm{D}^{+}_{\mathrm{pro}}(S)$  index-wise by using the index-wise Godement resolution. Then for $E_{\centerdot}\in \mathrm{D}^{-}_{\mathrm{pro}}(S)$ and $F_{\centerdot}\in \mathrm{D}^{+}_{\mathrm{pro}}(T)$ there is an adjunction isomorphism
  \begin{gather}\label{eq:pro-adjunction}
  \mathrm{Hom}_{\mathrm{D}_{\mathrm{pro}}(T)}(\mathbf{L}f^* E_{\centerdot},F_{\centerdot})\cong \mathrm{Hom}_{\mathrm{D}_{\mathrm{pro}}(S)}(E_{\centerdot}, \mathbf{R}f_*F_{\centerdot})\quad.
  \end{gather}
  This can be shown by using the level representations and the usual adjunction isomorphism. 
  \item Let $\tau\in\{\mathrm{Zar}, \mathrm{Nis},\et\}$. For a scheme $X$, let $X_{\tau}$  be the corresponding topos. Then let
  \begin{gather*}
  \mathrm{C}(X)_{\tau}=\mathrm{C}(X_{\tau}),\ \mathrm{D}(X)_{\tau}=\mathrm{D}(X_{\tau}), \\
    \mathrm{C}_{\mathrm{pro}}(X)_{\tau}=\mathrm{C}_{\mathrm{pro}}(X_{\tau}),\ 
    \mathrm{D}_{\mathrm{pro}}(X)_{{\tau}}=\mathrm{D}_{\mathrm{pro}}(X_{{\tau}}),
  \end{gather*}
  and similarly the category of right or left bounded complexes. As in \cite{BEK14}, Nisnevich topology is our default topology, and we often omit the subscript Nis in these notations.
\end{itemize}
Let $r$ be an integer such that $0\leq r<p$. We recall several pro-complexes in \cite[\S2, \S7]{BEK14}. 
\begin{definition}\label{def:BEK-proComplexes}
\begin{enumerate}[(i)]
  \item  $p(r)\Omega_{X_{\centerdot}}^{\bullet}\in \mathrm{C}_{\mathrm{pro}}(X_1)_{\et/\mathrm{Nis}}$ is the sub-pro-complex
\begin{gather*}
  p^r \mathcal{O}_{X_{\centerdot}}\rightarrow p^{r-1}\Omega_{X_{\centerdot}}^1\rightarrow\dots\rightarrow
  p\Omega_{X_{\centerdot}}^{r-1}\rightarrow\Omega_{X_{\centerdot}}^{r}\rightarrow \Omega_{X_{\centerdot}}^{r+1}\rightarrow\dots
\end{gather*}
 of the de Rham pro-complex $\Omega_{X_{\centerdot}}^{\bullet}$.
 \item   $q(r)W_{\centerdot}\Omega_{X_{1}}^{\bullet}\in \mathrm{C}_{\mathrm{pro}}(X_1)_{\et/\mathrm{Nis}}$ is the sub-pro-complex
\begin{gather*}
  p^{r-1}VW_{\centerdot} \mathcal{O}_{X_{1}}\rightarrow p^{r-2}VW_{\centerdot}\Omega_{X_{1}}^{1}\rightarrow\dots\rightarrow
  pVW_{\centerdot}\Omega_{X_{1}}^{r-2}\rightarrow VW_{\centerdot}\Omega_{X_{1}}^{r-1}\rightarrow W_{\centerdot}\Omega_{X_{1}}^{r}\rightarrow 
  W_{\centerdot}\Omega_{X_{1}}^{r+1}\rightarrow \dots
\end{gather*}
of the de Rham-Witt pro-complex $W_{\centerdot}\Omega_{X_{1}}^{\bullet}$.
\item Assume that there is a closed  embedding $X_{\centerdot}\hookrightarrow Z_{\centerdot}$, where $Z_{\centerdot}$ is an object of $\mathrm{Sm}_{W_{\centerdot}}$ associated with a smooth scheme $Z$ over $W$, and there is a morphism  $f:Z \rightarrow Z$ lifting the Frobenius $F:W \rightarrow W$ and such that $f\otimes_{W}\Bbbk$ is the absolute Frobenius on $Z_1$. Let $D_n=D_{X_n}(Z_n)$, and let $J_n$ be the ideal of $X_n\subset D_n$ and $I_n=(J_n,p)$. Then  $J(r)\Omega_{D_{\centerdot}}^{\bullet}\in \mathrm{C}_{\mathrm{pro}}(X_1)_{\et/\mathrm{Nis}}$ is defined to be the pro-complex 
\begin{gather*}
{J}_{\centerdot}^r\rightarrow {J}_{\centerdot}^{(r-1)}\otimes_{\mathcal{O}_{Z_{\centerdot}}}\Omega_{Z_{\centerdot}/W_{\centerdot}}^1\rightarrow \cdots\rightarrow 
{J}_{\centerdot}\otimes_{\mathcal{O}_{Z_{\centerdot}}}\Omega_{Z_{\centerdot}/W_{\centerdot}}^{r-1}\rightarrow \mathcal{O}_{D_{\centerdot}}\otimes_{\mathcal{O}_{Z_{\centerdot}}}\Omega_{Z_{\centerdot}/W_{\centerdot}}^r\rightarrow \cdots
\end{gather*}
and $I(r)\Omega_{D_{\centerdot}}^{\bullet}\in \mathrm{C}_{\mathrm{pro}}(X_1)_{\et/\mathrm{Nis}}$ is defined to be the pro-complex 
\begin{gather*}
I_{\centerdot}^r\rightarrow I_{\centerdot}^{(r-1)}\otimes_{\mathcal{O}_{Z_{\centerdot}}}\Omega_{Z_{\centerdot}/W_{\centerdot}}^1\rightarrow \cdots\rightarrow 
I_{\centerdot}\otimes_{\mathcal{O}_{Z_{\centerdot}}}\Omega_{Z_{\centerdot}/W_{\centerdot}}^{r-1}\rightarrow \mathcal{O}_{D_{\centerdot}}\otimes_{\mathcal{O}_{Z_{\centerdot}}}\Omega_{Z_{\centerdot}/W_{\centerdot}}^r\rightarrow \cdots
\end{gather*}
\item The motivic pro-complex $\mathbb{Z}_{X_{\centerdot}}(r)$ is an object of $\mathrm{D}_{\mathrm{pro}}(X_1)$, and there is an exact triangle
\begin{gather}\label{eq:pro-motivicFundamentalTriangle}
p(r)\Omega^{<r}_{X_{\centerdot}}[-1] \rightarrow \mathbb{Z}_{X_{\centerdot}}(r) \rightarrow \mathbb{Z}_{X_1}(r) \rightarrow p(r)\Omega^{<r}_{X_{\centerdot}}
\end{gather}
where $\mathbb{Z}_{X_1}(r)$ is the motivic complex $\mathbb{Z}(r)$ localized on the small Nisnevich site of $X_1$.
\end{enumerate}
\end{definition}

\subsection{The infinitesimal motivic complex} 
\label{sub:infinitesimal_motivic_complexes}

Brun's theorem \ref{thm:Brun} and our Theorem \ref{thm:rel-HC-homotopyEquiv} suggest that the following complex  $p(r)\Omega^{\bullet}_{X_{L}}$, which is a sheafification of a  complex defined in Definition \ref{def:relative-infinitesimal-complexes}, is a correct infinitesimal analog of the truncated pro-complex $p(r)\Omega_{X_{\centerdot}}^{<r}$.
\begin{definition}\label{def:proComplexes}
Let $r$ be an integer such that $0\leq r<p$.
For an integer $L\geq 1$, we define $p^{r,L}\Omega^{\bullet}_{X_{\centerdot}}\in\mathrm{C}_{\mathrm{pro}}(X_1)_{\et/\mathrm{Nis}}$ to be the pro-complex
\begin{gather*}
p^{rL}\mathcal{O}_{X_{\centerdot}}\rightarrow p^{(r-1)L}\Omega^1_{X_{\centerdot}/W_{\centerdot}}\rightarrow
\dots\rightarrow p^L\Omega_{X_{\centerdot}/W_{\centerdot}}^{r-1}
\rightarrow \Omega^r_{X_{\centerdot}/W_{\centerdot}}\rightarrow \Omega^{r+1}_{X_{\centerdot}/W_{\centerdot}}\rightarrow \dots
\end{gather*}
For integers $L>M\geq 1$, we define $p^{r,M}_{r,L}\Omega^{\bullet}_{X_{\centerdot}}\in \mathrm{C}(X_1)_{\et/\mathrm{Nis}}$ to be the complex
\begin{gather*}
p^{rM} \mathcal{O}_{X_{rL}}\rightarrow p^{(r-1)M}\Omega^1_{X_{(r-1)L}/W_{(r-1)L}}\rightarrow
\dots\rightarrow p^M\Omega_{X_{L}/W_{L}}^{r-1}\rightarrow 0\rightarrow \dots
\end{gather*}
In particular,  $p^{r,1}_{r,L}\Omega^{\bullet}_{X_{\centerdot}}$ is denoted also by $p(r)\Omega^{\bullet}_{X_{L}}$.
\end{definition}

\begin{definition}\label{def:infinitesimalMotComp}
Let $r$ be an integer such that $0\leq r<p$.
We define the \emph{infinitesimal motivic complex} $\mathbb{Z}_{X_{n}}(r)\in \mathrm{C}(X_1)$ as follows.
 First we assume the situation in item (iii) in Definition \ref{def:proComplexes}; this is achievable e.g. when $X_{\centerdot}$ is a quasi-projective object of $\mathrm{Sm}_{W_{\centerdot}}$; see \S\ref{sub:notations}(vi). By \cite[Theorem 7.2]{BeO78}, the natural map $I(r)\Omega_{D_{\centerdot}}^{\bullet} \rightarrow  p(r)\Omega_{X_{\centerdot}}^{\bullet}$ is a quasi-isomorphism.
 By \cite[Proposition 2.8]{BEK14}, the Frobenius on $Z_{\centerdot}$ induces a quasi-isomorphism $\Phi(F):I(r)\Omega_{D_{\centerdot}}^{\bullet} \rightarrow q(r)W_{\centerdot} \Omega_{X_1}^{\bullet}$. Denote the composition
\begin{gather*}
\mathbb{Z}_{X_{1}}(r) \xrightarrow{\operatorname{dlog}} W_{\centerdot} \Omega^r_{X_1,\log}[-r] \rightarrow W_{\centerdot} \Omega_{X_1}^{\geq r} \rightarrow q(r)W_{\centerdot} \Omega_{X_1}^{\bullet}
\end{gather*}
still by $\operatorname{dlog}$.   Let
\begin{gather*}
 \mathbb{Z}_{X_{1}}(r)^{\sim}:=\mathrm{Cone}\left(\mathbb{Z}_{X_{1}}(r)\oplus I(r)\Omega_{D_{\centerdot}}^{\bullet} \xrightarrow{\big(\operatorname{dlog},-\Phi(F)\big)} q(r)W_{\centerdot} \Omega_{X_1}^{\bullet}\right)[-1]\quad \in \mathrm{C}_{\mathrm{pro}}(X_1)\ .
\end{gather*}
Then the canonical map $\big(\mathbb{Z}_{X_{1}}(r)\big)^{\sim} \rightarrow \mathbb{Z}_{X_{1}}(r)$ is a quasi-isomorphism. We define
\begin{gather}\label{eq:infinitesimalMotComp}
\mathbb{Z}_{X_{n}}(r)^{\circ}_{Z_{\centerdot},f}=\operatorname{Cone}\left(\mathbb{Z}_{X_{1}}(r)^{\sim} \rightarrow p(r)\Omega_{X_{n}}^{\bullet} \right)[-1]\quad \in \mathrm{C}_{\mathrm{pro}}(X_1)\ ,
\end{gather}
where the arrow is the composition
\begin{gather*}
\mathbb{Z}_{X_{1}}(r)^{\sim} \rightarrow
 I(r)\Omega_{D_{\centerdot}}^{\bullet} \xrightarrow{\backsimeq}  p(r)\Omega_{X_{\centerdot}}^{\bullet} \rightarrow p(r)\Omega_{X_{n}}^{\bullet}\quad.
\end{gather*}

In the general case, we use the argument of \cite[Remark 1.8]{Kat87}. Namely, let  $X'_{\centerdot}$ be an affine open covering of $X_{\centerdot}$ and $X'_{\centerdot}\hookrightarrow Z_{\centerdot}$  a $W_{\centerdot}$-immersion  such that $Z_{\centerdot}$ is smooth and equipped with a Frobenius lifting $f$. Then let $X^{(i)}_{\centerdot}=X'_{\centerdot}\times_{X_{\centerdot}}\cdots \times_{X_{\centerdot}} X'_{\centerdot}$ ($i$-times), and $Z_{\centerdot}^{(i)}=Z_{\centerdot}\times_{W_{\centerdot}}\times \cdots \times_{W_{\centerdot}}Z_{\centerdot}$, and let $\pi_i:X_{\centerdot}^{(i)}\rightarrow X_{\centerdot}$ the canonical projection.
Let $\mathbb{Z}_{X_n}(r)^{\circ}_{X'_{\centerdot},Z_{\centerdot},f}$ be the complex associated with the double complex
\begin{gather*}
\underset{\mbox{(degree 0)}}{\pi_{1*}\big(\mathbb{Z}_{X_n^{(1)}}(r)^{\circ}\big)} \underset{\mbox{(degree 1)}}{\rightarrow {\pi_{2*}\big(\mathbb{Z}_{X_n^{(2)}}(r)^{\circ}\big)}} \rightarrow \dots
\end{gather*}
Then we define
\begin{gather*}
\mathbb{Z}_{X_n}(r)^{\circ}=\varinjlim_{(X'_{\centerdot},Z_{\centerdot},f)}\mathbb{Z}_{X_n}(r)^{\circ}_{X'_{\centerdot},Z_{\centerdot},f}\quad \in \mathrm{C}_{\mathrm{pro}}(X_1)\ .
\end{gather*}
This inductive system  is filtered. So when $X$ admits a global immersion $\mathbb{Z}_{X_n}(r)^{\circ}$ is isomorphic to \eqref{eq:infinitesimalMotComp} in $\mathrm{D}_{\mathrm{pro}}(X_1)$ by the Mayer-Vietoris property of the motivic complex $\mathbb{Z}(r)$. 
Finally, we define
\begin{gather*}
\mathbb{Z}_{X_{n}}(r):=\mathbf{R}\varprojlim \mathbb{Z}_{X_{n}}(r)^{\circ} \in \mathrm{D}(X_1)\ .
\end{gather*}
\end{definition}

Since $\mathbb{Z}_{X_1}(r)^{\sim}$ is quasiisomorphic to $\mathbb{Z}_{X_1}(r)$, $\mathbb{Z}_{X_{n}}(r)$ is isomorphic to $\mathbb{Z}_{X_{n}}(r)^{\circ}$ in $\mathrm{D}_{\mathrm{pro}}(X_1)$. Thus in $\mathrm{D}(X_1)$ there is an exact triangle
\begin{gather}\label{eq:motivicFundamentalTriangle}
p(r)\Omega^{\bullet}_{X_n}[-1] \rightarrow \mathbb{Z}_{X_n}(r) \rightarrow \mathbb{Z}_{X_1}(r) \rightarrow p(r)\Omega^{\bullet}_{X_n}\quad .
\end{gather}

\begin{proposition}\label{prop:functoriality-infinitesimalMotivComp}
$\mathbb{Z}_{X_{n}}(r)^{\circ}$ is functorial in $X_{\centerdot}$ in the sense that any morphism $a: X_{\centerdot} \rightarrow Y_{\centerdot}$ in $\mathrm{Sm}_{W_{\centerdot}}$ induces a morphism 
$a^{-1}\mathbb{Z}_{Y_n}(r)^{\circ} \rightarrow \mathbb{Z}_{X_n}(r)^{\circ}$ in $\mathrm{C}_{\mathrm{pro}}(X_1)$, and for $X_{\centerdot} \xrightarrow{a} Y_{\centerdot} \xrightarrow{b}Z_{\centerdot}$ the diagram
\begin{equation}\label{eq:functoriality-infinitesimalMotivComp}
\begin{gathered}
\xymatrix{
 (b\circ a)^{-1} \mathbb{Z}_{Z_n}(r)^{\circ} \ar@{=}[r] \ar[d] & a^{-1}b^{-1} \mathbb{Z}_{Z_n}(r)^{\circ} \ar[d]\\
 \mathbb{Z}_{X_n}(r)^{\circ} & a^{-1}\mathbb{Z}_{Y_n}(r)^{\circ} \ar[l]
}
\end{gathered}
\end{equation}
in $\mathrm{C}_{\mathrm{pro}}(X_1)$ commutes. Moreover, $\mathbb{Z}_{X_{n}}(r)$ is  functorial in $X_{\centerdot}$ in the same sense.
\end{proposition}
\begin{proof}
For $X_{\centerdot}\in \mathrm{Sm}_{W_{\centerdot}}$, let $\mathcal{A}_{X_{\centerdot}}$ be the category whose  objects the triples $(X'_{\centerdot},Z_{\centerdot},f)$ as above, and morphisms are the refinements, namely, the commutative diagrams
\begin{gather*}
\xymatrix{
  X_{\centerdot} \ar@{=}[d] & X'_{\centerdot} \ar[l] \ar@{^{(}->}[r] \ar[d] & Z_{\centerdot} \ar[d] \\
  X_{\centerdot} & X''_{\centerdot} \ar[l] \ar@{^{(}->}[r] & Z'_{\centerdot} 
}
\end{gather*}
satisfying that $Z_{\centerdot} \rightarrow Z'_{\centerdot}$ commutes with the Frobenius liftings. 
For a morphism $a:X_{\centerdot} \rightarrow Y_{\centerdot}$  in $\mathrm{Sm}_{W_{\centerdot}}$,  $\mathcal{A}_{a}$ be the category with objects the tuples $\big((X'_{\centerdot},U_{\centerdot},f_1),(Y'_{\centerdot},V_{\centerdot},f_2),\square\big)$, where $(X'_{\centerdot},U_{\centerdot},f_1)\in \mathcal{A}_{X_{\centerdot}}$, $(Y'_{\centerdot},V_{\centerdot},f_2)\in \mathcal{A}_{Y_{\centerdot}}$, and $\square$ is a commutative diagram
\begin{equation}\label{eq:morphism-frobenius-liftings}
  \begin{gathered}
\xymatrix{
  X_{\centerdot} \ar[d]_{a} & X'_{\centerdot} \ar[l] \ar@{^{(}->}[r] \ar[d] & U_{\centerdot} \ar[d] \\
  Y_{\centerdot} & Y'_{\centerdot} \ar[l] \ar@{^{(}->}[r] & V_{\centerdot} 
}
\end{gathered}
\end{equation}
satisfying that $U_{\centerdot} \rightarrow V_{\centerdot}$ commutes with the Frobenius liftings $f_1$ and $f_2$. The morphisms of $\mathcal{A}_a$ is the refinements in the obvious sense. There are forgetful functors 
$\pi_1:\mathcal{A}_a \rightarrow \mathcal{A}_{X_{\centerdot}}$ and $\pi_2:\mathcal{A}_a \rightarrow \mathcal{A}_{Y_{\centerdot}}$. Both $\pi_1$ and $\pi_2$ are cofinal. In fact, for any object $(X'_{\centerdot},U_{\centerdot},f_1)\in \mathcal{A}_{X_{\centerdot}}$ and $(Y'_{\centerdot},V_{\centerdot},f_2)\in \mathcal{A}_{Y_{\centerdot}}$, by shrinking $(X'_{\centerdot},U_{\centerdot},f_1)\in \mathcal{A}_{X_{\centerdot}}$ and then replacing $U_{\centerdot}$ with $U_{\centerdot}\times_{W_{\centerdot}}V_{\centerdot}$, we can find a commutative diagram as \eqref{eq:morphism-frobenius-liftings}; the remaining cofinality condition is shown similarly.

Every square \eqref{eq:morphism-frobenius-liftings} induces a morphism $\square^*:a^{-1}\mathbb{Z}_{Y_n}(r)^{\circ}_{Y'_{\centerdot},V_{\centerdot},f_2} \rightarrow \mathbb{Z}_{X_n}(r)^{\circ}_{X'_{\centerdot},U_{\centerdot},f_1}$.
 Then we define the morphism $a^{-1}\mathbb{Z}_{Y_n}(r)^{\circ} \rightarrow \mathbb{Z}_{X_n}(r)^{\circ}$ induced by $a$ to be the composition
\begin{gather*}
\xymatrix{
\varinjlim\limits_{\beta\in \mathcal{A}_{Y_{\centerdot}}}a^{-1}\mathbb{Z}_{Y_n}(r)^{\circ}_{\beta} &
\varinjlim\limits_{(\alpha,\beta,\square)\in \mathcal{A}_{a}}a^{-1}\mathbb{Z}_{Y_n}(r)^{\circ}_{\beta}  \ar[l]_{\cong} \ar[r]^<<<<{\square^*} & \varinjlim\limits_{(\alpha,\beta,\square)\in \mathcal{A}_{a}}\mathbb{Z}_{X_n}(r)^{\circ}_{\alpha} \ar[r]^{\cong} &
\varinjlim\limits_{\alpha\in \mathcal{A}_{X_{\centerdot}}}\mathbb{Z}_{X_n}(r)^{\circ}_{\alpha}\quad.
}
\end{gather*}
For $X_{\centerdot} \xrightarrow{a} Y_{\centerdot} \xrightarrow{b}Z_{\centerdot}$, introducing a category $\mathcal{A}_{a,b}$ with objects the commutative diagram
 \begin{gather*}
\xymatrix{
  X_{\centerdot} \ar[d]_{a} & X'_{\centerdot} \ar[l] \ar@{^{(}->}[r] \ar[d] & U_{\centerdot} \ar[d] \\
  Y_{\centerdot} \ar[d] & Y'_{\centerdot} \ar[l] \ar[d] \ar@{^{(}->}[r] & V_{\centerdot} \ar[d]  \\
  Z_{\centerdot} & Z'_{\centerdot} \ar[l] \ar@{^{(}->}[r] & T_{\centerdot}  
}
\end{gather*}
with $(X'_{\centerdot},U_{\centerdot},f_1)\in \mathcal{A}_{X_{\centerdot}}$, $(Y'_{\centerdot},V_{\centerdot},f_2)\in \mathcal{A}_{Y_{\centerdot}}$, and $(Z'_{\centerdot},V_{\centerdot},f_3)\in \mathcal{A}_{Z_{\centerdot}}$, by similar arguments one can show the commutativity of \eqref{eq:functoriality-infinitesimalMotivComp}. This completes the proof of the functoriality of $\mathbb{Z}_{X_{n}}(r)^{\circ}$, and that of $\mathbb{Z}_{X_{n}}(r)$ follows.
\end{proof}

\begin{remark}\label{rem:define-infinitesimalMotivicComp-Xnr}
Let $X_{nr}$ be a smooth scheme over $W_{nr}$.
By \cite[Théorème 6]{Elk73}, and also \cite[Théorème 1.3.1]{Ara01}, any affine open of $X_{nr}$  lifts to a smooth affine scheme over $W$. Then by the above argument, one might define $\mathbb{Z}_{X_{n}}(r)^{\circ}$ and $\mathbb{Z}_{X_{n}}(r)$ which are functorial in $X_{nr}$.
\end{remark}


\subsection{Syntomic pro-complexes \texorpdfstring{$\mathfrak{S}_{X_{\centerdot}}(r,n)$}{SX(r,n)}} 
\label{sub:infinitesimal_syntomic_complexes}
To study the products in $\bigoplus_r \mathbb{Z}_{X_n}(r)$ and do explicit computations in the subsequent sections, we need to introduce an analog of the syntomic pro-complex in \cite[\S 4]{BEK14}.
\begin{definition}
Let $r$ be an integer such that $0\leq r<p$.
Assume the situation of Definition \ref{def:BEK-proComplexes}(iii). Let $\mathcal{J}_n=(J_n,p^n)$, $m\geq n$, and let $J(r,n)\Omega_{D_{m}}^{\bullet}\in \mathrm{C}(X_1)_{\et/\mathrm{Nis}}$ be the complex 
\begin{gather*}
\mathcal{J}_n^r\rightarrow  \mathcal{J}_n^{r-1}\otimes_{\mathcal{O}_{Z_m}}\Omega_{Z_m/W_m}^1\rightarrow \cdots\rightarrow 
\mathcal{J}_n\otimes_{\mathcal{O}_{Z_m}}\Omega_{Z_m/W_m}^{r-1}\rightarrow \mathcal{O}_{D_m}\otimes_{\mathcal{O}_{Z_m}}\Omega_{Z_m/W_m}^r\rightarrow \cdots
\end{gather*}
By \cite[Theorem 7.2]{BeO78},   in $\mathrm{D}(X_1)$ $J(r,n)\Omega_{D_{m}}^{\bullet}$ is independent of the choice of $Z_m$.
By \cite[Chap.~I  \S1]{Kat87} (see also \cite[\S4]{BEK14}), there is a morphism\footnote{A local explicit expression of $f_r$ is given in \eqref{eq:formula-fr}.} $f_r:I(r)\Omega_{D_{\centerdot}}^{\bullet} \rightarrow \Omega_{D_{\centerdot}}^{\bullet}$.
Denote the composition of maps of pro-complexes
\begin{gather*}
J(r,n)\Omega_{D_{\centerdot}}^{\bullet}\rightarrow 
I(r)\Omega_{D_{\centerdot}}^{\bullet}\xrightarrow{1-f_r}\Omega_{D_{\centerdot}}^{\bullet}
\end{gather*}
still by $1-f_r$, and define 
\begin{gather*}
\mathfrak{S}_{X_{\centerdot}}(r,n)_{\et}=\mathrm{Cone}\big( J(r,n)\Omega_{D_{\centerdot}}^{\bullet}\xrightarrow{1-f_r}\Omega_{D_{\centerdot}}^{\bullet}\big)[-1]\ \in\mathrm{C}_{\mathrm{pro}}(X_1)_{\et}.
\end{gather*}
Let $\epsilon: X_{1,\et}\rightarrow X_{1,\mathrm{Nis}}$ be the morphism of sites. Define
\begin{gather}\label{eq:infinitesimal-syntomicComplex-Nis}
\mathfrak{S}_{X_{\centerdot}}(r,n)=\tau^{\leq r}\circ \mathbf{R} \epsilon_{*}\mathfrak{S}_{X_{\centerdot}}(r,n)_{\et}\ 
\in\mathrm{D}_{\mathrm{pro}}(X_1).
\end{gather}
\end{definition}

\begin{lemma}\label{lem:quasiIso-J(r,n)->I(r)-to-p(r)}
There is a canonical quasi-isomorphism of pro-complexes
\begin{gather}\label{eq:quasiIso-J(r,n)->I(r)-to-p(r)}
\mathrm{Cone}\big(J(r,n)\Omega_{D_{\centerdot}}^{\bullet}\rightarrow I(r)\Omega_{D_{\centerdot}}^{\bullet}\big)\xrightarrow{\sim} p(r)\Omega_{X_{n}}^{\bullet}\quad.
\end{gather}
\end{lemma}
\begin{proof}
Let  $I(r)\Omega_{D_{m}}^{\bullet}$
be the complex 
\begin{gather*}
I_{m}^r\rightarrow I_{m}^{(r-1)}\otimes_{\mathcal{O}_{Z_{m}}}\Omega_{Z_{m}}^1\rightarrow \cdots\rightarrow 
I_{m}\otimes_{\mathcal{O}_{Z_{m}}}\Omega_{Z_{m}}^{r-1}\rightarrow \mathcal{O}_{D_{m}}\otimes_{\mathcal{O}_{Z_{m}}}\Omega_{Z_{m}}^r\rightarrow \cdots
\end{gather*}
By \cite[Theorem 7.2]{BeO78},   in $\mathrm{D}(X_1)$, both $J(r,n)\Omega_{D_{m}}^{\bullet}$ and $I(r)\Omega_{D_{m}}^{\bullet}$ are independent of the choice of $Z_m$. 
It suffices to show that when $m\geq nr$, there is a canonical quasi-isomorphism
\begin{gather*}
\mathrm{Cone}\big(J(r,n)\Omega_{X_{m}}^{\bullet}\rightarrow I(r)\Omega_{X_{m}}^{\bullet}\big)\xrightarrow{\sim} p(r)\Omega_{X_{n}}^{\bullet}\quad.
\end{gather*}
Since this has nothing to do with the Frobenius liftings, we can take $Z_{m}=X_m$. Then $J(r,n)\Omega_{D_{m}}^{\bullet}=p^{r,n}\Omega_{X_m}^{\bullet}$ and $I(r)\Omega_{D_{m}}^{\bullet}=p^{r,1}\Omega_{X_m}^{\bullet}$.
When $m\geq nr$, we have the short exact sequence
\begin{gather*}
0 \rightarrow p^{(r-i)n}\Omega_{Z_{m}}^i \rightarrow p^{r-i}\Omega_{Z_{m}}^i 
\rightarrow p^{r-i}\Omega_{Z_{(r-i)n}}^i \rightarrow 0.
\end{gather*}
Hence the conclusion follows.
\end{proof}

\begin{proposition}
There is an exact triangle
\begin{gather}\label{eq:fundamentalTriangle-et}
p(r)\Omega^{\bullet}_{X_{n}}[-1]\rightarrow \mathfrak{S}_{X_{\centerdot}}(r,n)_{\et}\xrightarrow{\Phi^{\mathcal{J}}} W_{\centerdot}\Omega^{r}_{X_1,\log,\et}[-r]\xrightarrow{[1]}\cdots
\end{gather}
in $\mathrm{D}_{\mathrm{pro}}(X_1)_{\et}$, and an exact triangle
\begin{gather}\label{eq:fundamentalTriangle-Nis}
p(r)\Omega^{\bullet}_{X_{n}}[-1]\rightarrow \mathfrak{S}_{X_{\centerdot}}(r,n)\xrightarrow{\Phi^{\mathcal{J}}} W_{\centerdot}\Omega^{r}_{X_1,\log,\mathrm{Nis}}[-r]\xrightarrow{[1]}\cdots
\end{gather}
in $\mathrm{D}_{\mathrm{pro}}(X_1)$. In particular, the support of $\mathfrak{S}_{X_{\centerdot}}(r,n)$ lies in degrees $[1,r]$ for $r\geq 1$.
\end{proposition}
\begin{proof}
Consider the commutative diagram
\begin{gather*}
\xymatrix{
  J(r,n)\Omega_{X_{\centerdot}}^{\bullet} \ar[r]^<<<<{1-f_r} \ar[d] & \Omega_{D_{\centerdot}}^{\bullet} \ar@{=}[d] & \\
  I(r)\Omega_{D_{\centerdot}}^{\bullet} \ar[r]^{1-f_r} & \Omega_{D_{\centerdot}}^{\bullet} &.
}
\end{gather*}
By Lemma \ref{lem:quasiIso-J(r,n)->I(r)-to-p(r)} we have $\mathrm{Cone}\big(J(r,n)\Omega_{D_{\centerdot}}^{\bullet}\rightarrow I(r)\Omega_{D_{\centerdot}}^{\bullet}\big)\cong p(r)\Omega^{\bullet}_{X_{n}}$
in $\mathrm{D}_{\mathrm{pro}}(X_1)_{\et/\mathrm{Nis}}$. 
By \cite[Prop.~5.3]{BEK14}, $\mathrm{Cone}\big( I(r)\Omega_{D_{\centerdot}}^{\bullet} \xrightarrow{1-f_r}  \Omega_{D_{\centerdot}}^{\bullet}\big)[-1]$  in the étale topology is isomorphic to $\Omega^{r}_{X_1,\log}[-r]$ in $\mathrm{D}_{\mathrm{pro}}(X_1)_{\et}$. So the assertion in $\mathrm{D}_{\mathrm{pro}}(X_1)_{\et}$ follows. Applying $\tau^{\leq r}\circ \mathbf{R} \epsilon_{*}$ to \eqref{eq:fundamentalTriangle-et} and using
\begin{gather*}
\epsilon_{*}W_{\centerdot}\Omega^{r}_{X_1,\log,\et}\cong W_{\centerdot}\Omega^{r}_{X_1,\log,\mathrm{Nis}}
\end{gather*}
 (\cite[Remark 1 in Page 224]{Kat82}), we obtain \eqref{eq:fundamentalTriangle-Nis}.
\end{proof}
Recall from \cite[\S4]{BEK14}
\begin{align*}
\mathfrak{S}_{X_{\centerdot}}(r)_{\et}&:=\mathrm{Cone}\big(J(r)\Omega_{D_{\centerdot}}^{\bullet}\xrightarrow{1-f_r}\Omega_{D_{\centerdot}}^{\bullet}\big)[-1]\ \in\mathrm{C}_{\mathrm{pro}}(X_1)_{\et},\\
\mathfrak{S}_{X_{\centerdot}}(r)&:=\tau^{\leq r}\circ \mathbf{R} \epsilon_{*}\mathfrak{S}_{X_{\centerdot}}(r)_{\et}\ 
\in\mathrm{D}_{\mathrm{pro}}(X_1),
\end{align*}
and an exact triangle (\cite[Theorem 5.4]{BEK14})
\begin{gather}\label{eq:fundamentalTriangle-S(r)-Nis}
p(r)\Omega^{<r}_{X_{\centerdot}}[-1]\rightarrow \mathfrak{S}_{X_{\centerdot}}(r)\xrightarrow{\Phi^{J}} W_{\centerdot}\Omega^{r}_{X_1,\log,\mathrm{Nis}}[-r]\xrightarrow{[1]}\cdots
\end{gather}
There is a map from \eqref{eq:fundamentalTriangle-S(r)-Nis} to \eqref{eq:fundamentalTriangle-Nis}, from which we get an epimorphism
\begin{gather*}
\mathcal{H}^r(\mathfrak{S}_{X_{\centerdot}}(r)) \twoheadrightarrow\mathcal{H}^r(\mathfrak{S}_{X_{\centerdot}}(r,n)).
\end{gather*}

\subsection{Products in infinitesimal complexes} 
\label{sub:the_product_structure}
We recall the following construction from \cite[Chap.~I \S 2]{Kat87}.
\begin{definition}\label{def:Kato-product-ringComplexes}
Let $A$ and $B$ be ring complexes (Definition \ref{def:Kato's-product-structure}) in a topos $T$, and $g,h:A\rightarrow B$ be two ring maps of ring complexes. Let $S=\mathrm{Cone}(A\xrightarrow{h-g}B)[-1]$. Thus $S^q=A^q\oplus B^{q-1}$, and $\mathrm{d}:S^q\rightarrow S^{q+1}$ is given by 
\begin{gather}\label{eq:Kato-differential}
\mathrm{d}(x,y)=\big(\mathrm{d}x,g(x)-h(x)- \mathrm{d}y\big).
\end{gather}
We define the product $S\otimes_{\mathbb{Z}} S\rightarrow S$ by 
\begin{gather}\label{eq:Kato-product}
(x,y)(x',y')=\big(xx',(-1)^q g(x)y'+yh(x')\big),
\end{gather}
for $(x,y)\in S^q$, $(x',y')\in S^{q'}$.
 This defines a ring complex structure on $S$ with $1=(1_A,0)\in S^0$, and the obvious homomorphism $S \rightarrow A$ is a ring map of ring complexes.
\end{definition}
\begin{remark}\label{rem:Kato-BEK-sign}
Kato called the complex $S$ (ignoring the product) the \emph{mapping fiber of} $g-h$. But as we understand the mapping fiber as the shifted cone $\mathrm{Cone}(\dots)[-1]$,  $S$ is the mapping fiber of $h-g$. Thus the syntomic complex defined in \cite{Kat87} coincides with that in \cite{BEK14}. Thes consistent choices of sign conventions turn out to be crucial in the computations in \S\ref{sub:the_mod_p_product_structure}.
\end{remark}
The following lemma slightly enhances \cite[Chap.~I, Lemma 2.3]{Kat87}.
\begin{lemma}\label{lem:ringComplex-product-gradedCommutative}
Let $A$, $B$, and $S$ be as in Definition \ref{def:Kato-product-ringComplexes}.
Assume the following:
\begin{enumerate}[(i)]
  \item The product on $B$ is graded commutative. Namely, for $y\in B^q$ and $y'\in B^{q'}$, we have $yy'=(-1)^{qq'}y'y$.
  \item For   $x\in A^q$ and $x'\in A^{q'}$, there exists $z\in A^{q+q'-1}$ such that  have $xx-(-1)^{qq'}x'x=\mathrm{d} z$ and $(g-h)z=0$.
\end{enumerate}
 Then the induced product on $\mathcal{H}^*(S)$ is  graded commutative.
\end{lemma}
\begin{proof}
Suppose $x\in A^q$, $y\in B^{q+1}$, $x'\in A^{q'}$, $y'\in B^{q'+1}$, such that  $(x,y)$ represents an element of $\mathcal{H}^q(S)$ and  $(x',y')$ represents an element of $\mathcal{H}^{q'}(S)$.
Since $d(x,y)=0$ and $d(x',y')=0$, we have 
\begin{gather*}
dx=0,\ dx'=0,\ g(x)-h(x)=dy,\ g(x')-h(y')=dy'.
\end{gather*}
Moreover, let $z\in A^{q+q'-1}$ be as in the assumption (ii).
Thus
\begin{align*}
&(-1)^{qq'}(x',y')(x,y)-(x,y)(x',y')\\
={}&(-1)^{qq'}\big(x'x,(-1)^{q'}g(x')y+y'h(x)\big)
-\big(xx',(-1)^q g(x)y'+yh(x')\big)\\
={}&(-1)^{qq'}\big(\mathrm{d}z,(-1)^{q'}(g(x')-h(x'))y+y'(h(x)-g(x))\big)\\
={}& (-1)^{qq'}\big(\mathrm{d}z,(-1)^{q'}(\mathrm{d}y')y-y'(\mathrm{d}y)\big)\\
={}&(-1)^{qq'}\mathrm{d}(z,0)+(-1)^{(q+1)q'+1}\mathrm{d}(0,y'y).
\end{align*}
\end{proof}

\begin{lemma}\label{lem:ringComplexes-truncations}
Let $A$ be a ring complex on a topos $T$. Assume that $A$ has a decomposition $A=\bigoplus_{r\in \mathbb{Z}}A(r)$ into complexes in $T$, which respects the product in $A$ in the sense that 
\begin{gather*}
A^i(r)\cdot A^j(s)\subset A^{i+j}(r+s)
\end{gather*}
for $r,s\in \mathbb{Z}$. Then $\bigoplus_{r\in \mathbb{Z}}\tau^{\leq r}A(r)$ has a natural structure of ring complex. Moreover, there is a ring map $\bigoplus_{r\in \mathbb{Z}}\tau^{\leq r}A(r) \rightarrow \bigoplus_{r\in \mathbb{Z}}\mathcal{H}^r(A)[-r]$.
\end{lemma}
\begin{proof}
It suffices to observe that $\big(\tau^{\leq r}A(r)\big)^r\cdot \big(\tau^{\leq s}A(s)\big)^s\subset \big(\tau^{\leq r+s}A(r+s)\big)^{r+s}$.
\end{proof}

\begin{definition}
Both $\bigoplus_{0\leq r<p}J(r,n)\Omega_{D_{m}}^{\bullet}$ and $\Omega_{D_{\centerdot}}^{\bullet}$ are ring  complexes; the ring structure on the former is given by
\begin{gather*}
(J_n,p^n)^{r-i}\otimes_{\mathcal{O}_{Z_m}}\Omega^i_{Z_m}\ \otimes\ (J_n,p^n)^{s-j}\otimes_{\mathcal{O}_{Z_m}}\Omega^j_{Z_m} \rightarrow (J_n,p^n)^{r+s-i-j}\otimes_{\mathcal{O}_{Z_m}}\Omega^{i+j}_{Z_m}.
\end{gather*}
We define  products in $\bigoplus_{0\leq r<p}\mathfrak{S}_{X_{\centerdot}}(r,n)_{\et}$ by the construction in Definition \ref{def:Kato's-product-structure} by taking $g=\bigoplus_{0\leq r<p}1$ and $h=\bigoplus_{0\leq r<p}f_r$. This induces products in $\bigoplus_{0\leq r<p}\mathfrak{S}_{X_{\centerdot}}(r,n)$ by adjunction and the canonical truncation. In fact,
by adjunction \eqref{eq:pro-adjunction}, there are morphisms $\mathbf{L}\epsilon^*\mathbf{R}\epsilon_*\mathfrak{S}_{X_{\centerdot}}(r,n)_{\et} \rightarrow \mathfrak{S}_{X_{\centerdot}}(r,n)_{\et}$ (note that $\mathbf{R}\epsilon_*\mathfrak{S}_{X_{\centerdot}}(r,n)_{\et}$ is right bounded as we will observe in Construction \ref{cons:infinitesimalSyntomicComplex-Nis-rep}) and thus
\begin{align*}
\mathbf{L}\epsilon^*\big(\mathbf{R}\epsilon_*\mathfrak{S}_{X_{\centerdot}}(r,n)_{\et} \otimes^{\mathbf{L}}\mathbf{R}\epsilon_*\mathfrak{S}_{X_{\centerdot}}(s,n)_{\et}\big)&\cong 
\mathbf{L}\epsilon^*\mathbf{R}\epsilon_*\mathfrak{S}_{X_{\centerdot}}(r,n)_{\et} \otimes^{\mathbf{L}}\mathbf{L}\epsilon^*\mathbf{R}\epsilon_*\mathfrak{S}_{X_{\centerdot}}(s,n)_{\et}\\
&\rightarrow\mathfrak{S}_{X_{\centerdot}}(r,n)_{\et} \otimes^{\mathbf{L}}\mathfrak{S}_{X_{\centerdot}}(s,n)_{\et} \rightarrow\mathfrak{S}_{X_{\centerdot}}(r+s,n)_{\et},
\end{align*}
and by adjunction again we get 
$\mathbf{R}\epsilon_*\mathfrak{S}_{X_{\centerdot}}(r,n)_{\et} \otimes^{\mathbf{L}}\mathbf{R}\epsilon_*\mathfrak{S}_{X_{\centerdot}}(s,n)_{\et} \rightarrow \mathbf{R}\epsilon_*\mathfrak{S}_{X_{\centerdot}}(r+s,n)_{\et}$. Then the canonical truncations induce
\begin{gather}\label{eq:infinitesimalSyntomicComplex-Nis}
\tau^{\leq r}\mathbf{R}\epsilon_*\mathfrak{S}_{X_{\centerdot}}(r,n)_{\et} \otimes^{\mathbf{L}} \tau^{\leq s}\mathbf{R}\epsilon_*\mathfrak{S}_{X_{\centerdot}}(s,n)_{\et} \rightarrow \tau^{r+s} \mathbf{R}\epsilon_*\mathfrak{S}_{X_{\centerdot}}(r+s,n)_{\et},
\end{gather}
and thus $\mathfrak{S}_{X_{\centerdot}}(r,n)\otimes^{\mathbf{L}}\mathfrak{S}_{X_{\centerdot}}(s,n) \rightarrow
\mathfrak{S}_{X_{\centerdot}}(r+s,n)$. 
\end{definition}

One can then define the product $\mathbb{Z}_{X_n}(r)\otimes^{\mathbf{L}}\mathbb{Z}_{X_n}(s) \rightarrow \mathbb{Z}_{X_n}(r+s)$ for $0\leq r\leq s+r<p$ by the argument of \cite[Proposition 7.2(v)]{BEK14}.
We would like to give a   more explicit  construction. For this, we need to define the product \eqref{eq:infinitesimalSyntomicComplex-Nis} as a ring map of ring complexes.
\begin{construction}\label{cons:infinitesimalSyntomicComplex-Nis-rep}
Let $D_{\centerdot}$ and $Z_{\centerdot}$ be as in Definition \ref{def:BEK-proComplexes}(iii). The schemes $D_n$ and $X_n$ have the same underlying Zariski (resp. étale, resp. Nisnevich) topology. Then since $\mathcal{J}_n^{r-i}\otimes_{\mathcal{O}_{Z_m}}\Omega_{Z_m/W_m}^i$ and  $I^{r-i}\otimes_{\mathcal{O}_{Z_m}}\Omega_{Z_m/W_m}^i$ are coherent sheaves on $D_m$, we have
\begin{gather*}
\mathbf{R} \epsilon_{*}\mathfrak{S}_{X_{\centerdot}}(r,n)_{\et}\cong \mathrm{Cone}\big(\epsilon_* J(r,n)\Omega_{D_{\centerdot}}^{\bullet}\xrightarrow{1-\epsilon_* f_r}\epsilon_*\Omega_{D_{\centerdot}}^{\bullet}\big)[-1].
\end{gather*}
The complexes $\bigoplus_{0\leq r<p}\epsilon_* J(r,n)\Omega_{D_{\centerdot}}^{\bullet}$ and $\bigoplus_{0\leq r<p}\epsilon_*\Omega_{D_{\centerdot}}^{\bullet}$ are ring complexes and $1$ and $\epsilon_* f_r$ are ring maps of ring complexes. Hence by Lemma \ref{lem:ringComplexes-truncations}, we have a ring complex
\begin{gather}\label{eq:infinitesimalSyntomicComplex-Nis-rep}
\bigoplus_{0\leq r<p}\tau^{\leq r}\Big(
\mathrm{Cone}\big(\epsilon_* J(r,n)\Omega_{D_{\centerdot}}^{\bullet}\xrightarrow{1-\epsilon_* f_r}\epsilon_*\Omega_{D_{\centerdot}}^{\bullet}\big)[-1]\Big)
\end{gather}
which in $\mathrm{D}_{\mathrm{pro}}(X_1)$ is independent of the choice of $D_{\centerdot}$ and $Z_{\centerdot}$, and the product in it represents the product defined above in $\mathrm{D}_{\mathrm{pro}}(X_1)$.
\end{construction}

\begin{proposition}\label{prop:infinitesimalMotComp-product}
\begin{enumerate}[(i)]
  \item The cone
\begin{gather}\label{eq:infinitesimalMotivicComplex-derivedDef}
\mathrm{cone}\big(\mathfrak{S}_{X_{\centerdot}}(r,n)\oplus \mathbb{Z}_{X_1}(r)\xrightarrow{\Phi^{\mathcal{J}}\oplus(- \mathrm{d}\log)}W_{\centerdot}\Omega^{r}_{X_1,\log}[-r]\big)[-1]
\end{gather}
is well  defined up to unique isomorphisms in $\mathrm{D}_{\mathrm{pro}}(X_1)$ and it is naturally isomorphic to $\mathbb{Z}_{X_n}(r)$ in $\mathrm{D}_{\mathrm{pro}}(X_1)$.
\item Assume the existence of $Z_{\centerdot}$ as in Definition \ref{def:BEK-proComplexes}(iii). Regard  $\bigoplus_{0\leq r<p}\mathfrak{S}_{X_{\centerdot}}(r,n)$ as the complex \eqref{eq:infinitesimalSyntomicComplex-Nis-rep}. Then the map
\begin{gather*}
  \Phi: \bigoplus_{0\leq r<p}\mathfrak{S}_{X_{\centerdot}}(r,n) \rightarrow \bigoplus_{0\leq r<p}W_{\centerdot}\Omega^{r}_{X_1,\log}[-r]
  \end{gather*}
  is a ring map of ring complexes.
  \item There is a ring complex
  \begin{gather*}
  \bigoplus_{r\geq 0}\mathbb{Z}_{X_1}(r)
  \end{gather*}
  and a ring map of ring complexes
  \begin{gather*}
  \bigoplus_{r\geq 0}\mathbb{Z}_{X_1}(r) \xrightarrow{\mathrm{d}\log} \bigoplus_{r\geq 0}W_{\centerdot}\Omega^{r}_{X_1,\log}[-r]
  \end{gather*}
  induced by the composition 
  \begin{gather*}
  \mathcal{H}^r\big(\mathbb{Z}_{X_1}(r)\big) \cong  \mathcal{K}_{X_1,1}^{M} \xrightarrow{\mathrm{d}\log}W_{\centerdot}\Omega^{r}_{X_1,\log}\quad.
  \end{gather*}
  \item There is a morphism
  \begin{gather*}
  \bigoplus_{0\leq r<p}\mathbb{Z}_{X_n}(r) \ \otimes^{\mathbf{L}}\ \bigoplus_{0\leq r<p}\mathbb{Z}_{X_n}(r) \longrightarrow \bigoplus_{0\leq r<p}\mathbb{Z}_{X_n}(r)
  \end{gather*}
  in $\mathrm{D}(X_1)$ which is unital and  associative in the obvious sense, and the projections to $\bigoplus_{0\leq r<p}\mathfrak{S}_{X_{\centerdot}}(r,n)$ and to $\bigoplus_{0\leq r<p}W_{\centerdot}\Omega^{r}_{X_1,\log}[-r]$ induced by \eqref{eq:infinitesimalMotivicComplex-derivedDef} preserve the products.
   \item The induced product in the cohomology sheaves of $\bigoplus_{0\leq r<p}\mathbb{Z}_{X_n}(r)$ is graded commutative.
  \item There is a commutative diagram of exact triangles in $\mathrm{D}_{\mathrm{pro}}(X_1)$
\begin{equation}\label{eq:motivicFundamentalTriangle-to-syntomicFundamentalTriangle}
\begin{gathered}
  \xymatrix{
  p(r)\Omega_{X_{n}}^{\bullet}[-1] \ar[r] \ar@{=}[d] & \mathbb{Z}_{X_{n}}(r) \ar[r]^{\alpha_1} \ar[d]^{\alpha_2} & \mathbb{Z}_{X_{1}}(r) \ar[r] \ar[d]^{\mathrm{d}\log} & p(r)\Omega_{X_{n}}^{\bullet}  \ar@{=}[d]  \\
  p(r)\Omega_{X_{n}}^{\bullet}[-1] \ar[r]  & \mathfrak{S}_{X_{\centerdot}}(r,n) \ar[r]^<<<<<{\Phi^{\mathcal{J}}}  & 
  W_{\centerdot}\Omega_{X_{1},\log}^r[-r] \ar[r]^>>>>>>{\alpha} & p(r)\Omega_{X_{n}}^{\bullet}
  }
\end{gathered}
\end{equation}
where the bottom is the fundamental triangle \eqref{eq:fundamentalTriangle-Nis}, and both $\alpha_1$ and $\alpha_2$ preserve products.
\end{enumerate}
\end{proposition}
\begin{proof}
(i) The first assertion follows from  \cite[Lemma A.2]{BEK14}. The natural isomorphism then follows from the exact triangles \eqref{eq:infinitesimalMotComp} and \eqref{eq:fundamentalTriangle-Nis}.

(ii) By Lemma \ref{lem:ringComplexes-truncations}, there is a ring map of ring complexes
\begin{gather*}
\bigoplus_{0\leq r<p}\mathfrak{S}_{X_{\centerdot}}(r,n) \rightarrow \bigoplus_{0\leq r<p}\mathcal{H}^r\big(\mathfrak{S}_{X_{\centerdot}}(r,n)\big)[-r]\quad.
\end{gather*}
The map $\bigoplus_{0\leq r<p}\mathcal{H}^r\big(\mathfrak{S}_{X_{\centerdot}}(r,n)\big)[-r] \rightarrow \bigoplus_{0\leq r<p}W_{\centerdot}\Omega^{r}_{X_1,\log}[-r]$ induced by $\Phi$ preserves products: this follows from the definition of $\Phi$, and finally follows from that the map $\Phi(f): \mathcal{O}_{D_n} \rightarrow W_n(\mathcal{O}_{X_1})$ induced by the Frobenius lifting $f$ on $Z_{\centerdot}$ is a ring map (\cite[Chap 0, (1.3.16)]{Ill79}, see also \cite[(2.6)]{BEK14}).

(iii) We use the product in motivic complexes constructed as in \cite[3.10-3.11]{MVW06}. It is associative because the Eilenberg-MacLane map $\nabla$ (=inverse of the Alexander-Whitney map) is associative (\cite[Proposition 29.9]{May67})\footnote{Rather than merely homotopy associative as stated in \cite[3.11]{MVW06}.}. Then by Lemma \ref{lem:ringComplexes-truncations}, $\mathrm{d}\log$ is a ring map because $\mathbb{Z}_{X_1}(r)$ is supported in degree $\leq r$.

Apply  Definition \ref{def:Kato's-product-structure} to
\begin{equation}\label{eq:infinitesimalMotComp-product-A&B}
\begin{aligned}
A&=\bigoplus_{0\leq r<p}\mathfrak{S}_{X_{\centerdot}}(r,n)\oplus \bigoplus_{0\leq r<p}\mathbb{Z}_{X_1}(r),\\
B&=\bigoplus_{0\leq r<p}W_{\centerdot}\Omega^{r}_{X_1,\log}[-r],
\end{aligned}
\end{equation}
and $g=(\Phi,0)$, and $h=(0,\mathrm{d}\log)$. Then (iv) follows from (ii) and (iii). The product map is defined in $\mathrm{D}_{\mathrm{pro}}(X_1)$ by the construction, but it is then  defined in $\mathrm{D}(X_1)$ because the inclusion $\iota:\mathrm{D}_{\mathrm{pro}}(X_1) \rightarrow \mathrm{D}_{\mathrm{pro}}(X_1)$ is fully faithful.

(v) follows\footnote{I hope  this ad hoc proof can be replaced with a more natural one by exploring a certain operadic nature of the infinitesimal motivic complex, just as the case  of the motivic complex (see \cite[Page 122]{MVW06}).} from applying Lemma \ref{lem:ringComplex-product-gradedCommutative} to the above $A$, $B$, $g$, and $h$. Since $(\Phi,0)$ and $(0,\mathrm{d}\log)$ respect the decompositions in \eqref{eq:infinitesimalMotComp-product-A&B}, and each $W_{\centerdot}\Omega^{r}_{X_1,\log}[-r]$ is concentrated in one single degree, the condition  Lemma \ref{lem:ringComplex-product-gradedCommutative}(ii) holds. 
 
(vi) follows from (i) and (iv).
\end{proof}

\begin{proposition}\label{prop:properties-Zn(r)}
\begin{enumerate}[(i)]
  \item The complex $\mathbb{Z}_{X_n}(r)$ is supported in degree $\leq r$. 
  \item In $\mathrm{D}_{\mathrm{pro}}(X_1)$, the pro-complex  $\dots \rightarrow \mathbb{Z}_{X_n}(r) \rightarrow \mathbb{Z}_{X_{n-1}}(r) \rightarrow \dots \rightarrow \mathbb{Z}_{X_1}(r)$ is isomorphic to  the motivic pro-complex $\mathbb{Z}_{X_{\centerdot}}(r)$defined in \cite[\S 7]{BEK14}.
\end{enumerate}
Moreover, we have the following canonical isomorphisms in $\mathrm{D}(X_1)$:
\begin{enumerate}
  \item[(iii)] $\mathbb{Z}_{X_n}(0)\cong\mathbb{Z}$.
  \item[(iv)] $\mathbb{Z}_{X_n}(1)\cong \mathbb{G}_{m,X_n}[-1]$.
  \item[(v)] $\mathcal{H}^r\big(\mathbb{Z}_{X_n}(r)\big)\cong \mathcal{K}_{X_n,r}^{M}$.
  \item[(vi)] $\mathbb{Z}_{X_{n}}(r)\otimes^{\mathbf{L}}\mathbb{Z}/p^i \cong \mathfrak{S}_{X_{\centerdot}}(r,n)\otimes^{\mathbf{L}}\mathbb{Z}/p^i $ for $i\geq 1$.
\end{enumerate}
\end{proposition}
\begin{proof}
The statement (i) follows from \eqref{eq:motivicFundamentalTriangle} and that $\mathbb{Z}_{X_1}(r)$ is supported in degree $\leq r$ by its definition. 

Comparing the definition of $\mathbb{Z}_{X_{\centerdot}}(r)$ (\cite[Definition 7.1]{BEK14}) with \eqref{eq:infinitesimalMotivicComplex-derivedDef}, there is a map of exact triangles from \eqref{eq:pro-motivicFundamentalTriangle} to \eqref{eq:motivicFundamentalTriangle}.
But the pro-complex  $\dots \rightarrow p(r)\Omega^{\bullet}_{X_n} \rightarrow p(r)\Omega^{\bullet}_{X_{n-1}} \rightarrow \cdots $ is isomorphic to  the pro-complex $p(r)\Omega^{<r}_{X_{\centerdot}}$. Then (ii) follows.

The statement (iii) follows from \eqref{eq:infinitesimalMotComp} and that $\mathbb{Z}_{X_1}(0)=\mathbb{Z}$. (vi) follows the diagram \eqref{eq:motivicFundamentalTriangle-to-syntomicFundamentalTriangle} and the isomorphism (of Zariski, thus Nisnevich, sheaves)  $\mathbb{Z}_{X_1}(r)\otimes^{\mathbf{L}}\mathbb{Z}/p^i\cong W_{\centerdot}\Omega_{X_{1},\log}^r[-r]\otimes^{\mathbf{L}}\mathbb{Z}/p^i$ (\cite[Theorem 8.3]{GeL00}). For (v), we recall the following commutative diagram from \cite[Page 695]{BEK14}
\[
\xymatrix{
  0 \ar[r] & p \Omega_{X_{\centerdot}}^{r-1}/p^2\Omega_{X_{\centerdot}}^{r-2} \ar[r] \ar @{=}[d] & \mathcal{H}^r(\mathfrak{S}_{X_{\centerdot}}(r)) \ar[r]^{\Phi^J} & W_{\centerdot}\Omega_{X_1,\log}^r \ar[r] & 0 \\
  0 \ar[r] & p \Omega_{X_{\centerdot}}^{r-1}/p^2\Omega_{X_{\centerdot}}^{r-2} \ar[r] \ar@{=}[d] & \mathcal{H}^r({\mathbb{Z}_{X_{\centerdot}}(r)}) \ar[u] \ar[r] & \mathcal{H}^r(\mathbb{Z}_{X_1}(r)) \ar[u]_{\mathrm{d}\log} \ar[r] & 0\\
  0 \ar[r] & p \Omega_{X_{\centerdot}}^{r-1}/p^2\Omega_{X_{\centerdot}}^{r-2} \ar[r]^<<<<<{\mathrm{Exp}} & \mathcal{K}^M_{X_{\centerdot},r} \ar[r] \ar[u]^{\gamma} & \mathcal{K}^M_{X_1,r} \ar[u]_{\cong} \ar[r] & 0
}
\]
where the map $\gamma$ is induced by the composition $ \mathcal{K}^M_{X_{\centerdot},r} \rightarrow  \mathcal{K}^M_{X_{1},r} \xrightarrow{\cong } \mathcal{H}^r(\mathbb{Z}_{X_1}(r)) $ (\cite{ElM02}, \cite{Ker09}, \cite{Ker10}) and Kato's symbol (\cite[Chap.~I \S3]{Kat87}) $\mathcal{K}^M_{X_{\centerdot},r} \rightarrow   \mathcal{H}^r(\mathfrak{S}_{X_{\centerdot}}(r))$.
There is a commutative diagram
\begin{gather*}
\xymatrix{
  0 \ar[r] & p \Omega_{X_{\centerdot}}^{r-1}/p^2\Omega_{X_{\centerdot}}^{r-2} \ar[d] \ar[r]^<<<{\mathrm{Exp}} & \mathcal{K}^M_{X_{\centerdot},r} \ar[r] \ar[d]  & \mathcal{K}^M_{X_1,r}  \ar[r] \ar@{=}[d] & 0 \\
  0 \ar[r] & p \Omega_{X_{n}}^{r-1}/p^2\Omega_{X_{n}}^{r-2} \ar[r]^<<<{\mathrm{Exp}} & \mathcal{K}^M_{X_{n},r} \ar[r] & \mathcal{K}^M_{X_1,r}  \ar[r] & 0
}
\end{gather*}
where both maps $\mathrm{Exp}$ are given by (\cite[(12.1)]{BEK14})
 \[
 p x d \log y_1\wedge \cdots \wedge d \log y_{r-1} \mapsto \{\exp(px),y_1,...,y_{r-1}\},
 \]
 and the bottom row is also exact by \cite[Theorem 12.3]{BEK14}. We set
 \begin{align*}
 U_1\mathcal{K}^M_{X_{n},r}&:=\mathrm{Ker}\big( \mathcal{K}^M_{X_{n},r} \rightarrow \mathcal{K}^M_{X_{1},r} \big),\\
  U_1\mathcal{K}^M_{X_{\centerdot},r}&:=\mathrm{Ker}\big( \mathcal{K}^M_{X_{\centerdot},r} \rightarrow \mathcal{K}^M_{X_{1},r} \big),\\
\mathcal{N}&:=\mathrm{Ker}\big(\mathcal{K}^M_{X_{\centerdot},r} \rightarrow \mathcal{K}^M_{X_n,r}\big).
 \end{align*}
The map $U_1\mathcal{K}^M_{X_{\centerdot},r} \rightarrow \mathcal{K}^M_{X_{\centerdot},r}$ identifies $\mathrm{Ker}\big(U_1\mathcal{K}^M_{X_{\centerdot},r} \rightarrow U_1\mathcal{K}^M_{X_n,r}\big)$ with $\mathcal{N}$.  The commutative diagram
 \begin{equation}\label{eq:gammaPrime}
\begin{gathered}
\xymatrix{
  0 \ar[r] & p \Omega_{X_{n}}^{r-1}/p^2\Omega_{X_{n}}^{r-2} \ar[r] & \mathcal{H}^r({\mathbb{Z}_{X_{n}}(r)}) \ar[r] & \mathcal{H}^r(\mathbb{Z}_{X_1}(r))  \ar[r] & 0\\
  0 \ar[r] & p \Omega_{X_{\centerdot}}^{r-1}/p^2\Omega_{X_{\centerdot}}^{r-2} \ar[r]^<<<<<{\mathrm{Exp}} \ar[u] & \mathcal{K}^M_{X_{\centerdot},r} \ar[r] \ar[u]^{\gamma'} & \mathcal{K}^M_{X_1,r} \ar[u]_{\cong} \ar[r] & 0
}
\end{gathered}
\end{equation}
where $\gamma'$ is the composition $\mathcal{K}^M_{X_{\centerdot},r} \xrightarrow{\gamma} \mathcal{H}^r({\mathbb{Z}_{X_{\centerdot}}(r)})  \rightarrow \mathcal{H}^r({\mathbb{Z}_{X_{n}}(r)})$, induces a diagonal map $\theta$ in 
\begin{gather*}
\xymatrix{
   p \Omega_{X_{\centerdot}}^{r-1}/p^2\Omega_{X_{\centerdot}}^{r-2} \ar[d] \ar[r]^<<<{\mathrm{Exp}} & U_1\mathcal{K}^M_{X_{\centerdot},r}  \ar[d] \ar[dl]_{\theta} \\
 p \Omega_{X_{n}}^{r-1}/p^2\Omega_{X_{n}}^{r-2} \ar[r]^<<<{\mathrm{Exp}} & U_1\mathcal{K}^M_{X_{n},r} 
}
\end{gather*}
such that the left-upper triangle commutes. Since the outer square commutes and the upper $\mathrm{Exp}$ is an isomorphism, the right-lower triangle also commutes. It follows that $\theta$, and thus $\gamma'$, send $\mathcal{N}$ to 0.  Then the diagram \eqref{eq:gammaPrime} factors through
\begin{gather*}
\xymatrix{
  0 \ar[r] & p \Omega_{X_{n}}^{r-1}/p^2\Omega_{X_{n}}^{r-2} \ar[r] & \mathcal{H}^r({\mathbb{Z}_{X_{n}}(r)}) \ar[r] & \mathcal{H}^r(\mathbb{Z}_{X_1}(r))  \ar[r] & 0\\
  0 \ar[r] & p \Omega_{X_{n}}^{r-1}/p^2\Omega_{X_{n}}^{r-2} \ar[r]^<<<<<{\mathrm{Exp}} \ar@{=}[u] & \mathcal{K}^M_{X_{n},r} \ar[r] \ar[u] & \mathcal{K}^M_{X_1,r} \ar[u]_{\cong} \ar[r] & 0
}
\end{gather*}
and hence we obtain the isomorphism $\mathcal{K}^M_{X_{n},r} \xrightarrow{\cong} \mathcal{H}^r(\mathbb{Z}_{X_1}(r))$.
 
Finally we show (iv).
 Since $\mathbb{Z}_{X_1}(1)\cong \mathcal{K}_{X_1,1}^M[-1]\cong \mathbb{G}_{m,X_1}[-1]$, by the exact triangle \eqref{eq:motivicFundamentalTriangle}
 one sees that 
the cohomology sheaves of $ \mathbb{Z}_{X_{n}}(1)$ are concentrated at degree 1. Since $ \mathbb{Z}_{X_{n}}(1)$ is supported in degree $\leq 1$, we have $ \mathbb{Z}_{X_{n}}(1)\cong \mathcal{H}^1(\mathbb{Z}_{X_{n}}(1))\cong \mathcal{K}_{X_n,1}^M[-1]\cong \mathbb{G}_{m,X_n}[-1]$ as a consequence of  (iii).
\end{proof}




\section{Infinitesimal Chern classes and Chern characters} 
\label{sec:infinitesimal_chern_classes_and_chern_characters}
Recall the definition of the category $\mathrm{Sm}_{W_{\centerdot}}$ in \S\ref{sub:notations}(vi).
We call a collection of morphisms $\{U_{i \centerdot}\rightarrow X_{\centerdot}\}_{i\in I}$ in $\mathrm{Sm}_{W_{\centerdot}}$ \emph{a Nisnevich} (resp. \emph{Zariski}, resp. \emph{étale}) \emph{covering}  if  
$\{U_{i, n}\rightarrow X_{n}\}_{i\in I}$ is a Nisnevich covering for each $n\in \mathbb{N}$. 
Let $(\mathrm{Sm}_{W_{\centerdot}})_{\mathrm{Nis}}$ (resp. $(\mathrm{Sm}_{W_{\centerdot}})_{\mathrm{Zar}}$, resp. $(\mathrm{Sm}_{W_{\centerdot}})_{\et}$) be the associated site. 

For every $n\in \mathbb{N}$, let $\mathrm{red}_n: \mathrm{Sm}_{W_{\centerdot}} \rightarrow \mathrm{Sm}_{W_{n}}$ be the functor sending $X_{\centerdot}$ to $X_n$. Then $\mathrm{red}_n$ is continuous and cocontinuous with respect to the Nisnevich (resp. Zariski, resp. étale) topology. Thus there is an associated morphism of topoi 
$\mathrm{red}_n: (\mathrm{Sm}_{W_{\centerdot}})_{\mathrm{Nis}}^{\sim} \rightarrow (\mathrm{Sm}_{W_{n}})_{\mathrm{Nis}}^{\sim}$ (resp. $(\mathrm{Sm}_{W_{\centerdot}})_{\mathrm{Zar}}^{\sim} \rightarrow (\mathrm{Sm}_{W_{n}})_{\mathrm{Zar}}^{\sim}$, resp. $(\mathrm{Sm}_{W_{\centerdot}})_{\et}^{\sim} \rightarrow (\mathrm{Sm}_{W_{n}})_{\et}^{\sim}$). As abuse of notations, for a sheaf $\mathcal{F}$ on $(\mathrm{Sm}_{W_{n}})_{\mathrm{Nis}}$, we denote $\mathrm{red}_n^{-1}\mathcal{F}$ still by $\mathcal{F}$; we have $\mathcal{F}(X_{\centerdot})= \mathcal{F}(X_n)$.

By Proposition \ref{prop:functoriality-infinitesimalMotivComp} 
there exists a complex of sheaves $\mathbb{Z}_{?n}(r)$ on  $(\mathrm{Sm}_{W_{\centerdot}})_{\mathrm{Nis}}$ such that the restriction of $\mathbb{Z}_{?n}(r)$ to $X_{\centerdot}$ is $\mathbb{Z}_{X_n}(r)^{\circ}$.

Recall the cohomology theory of simplicial (pre)sheaves from \cite[Chapter 8]{Jar15}:
\begin{definition}
Let $\mathcal{C}$ be a site. A map between two simplicial presheaves on $\mathcal{C}$ is a \emph{local weak equivalence} if it induces isomorphisms on the sheaves associated with their presheaves of homotopy groups.  Let $s \mathbf{Pre}(\mathcal{C})$ be the category of simplicial presheaves on $\mathcal{C}$, and the homotopy category 
$\mathrm{Ho}(s \mathbf{Pre}(\mathcal{C}))$ is the localization by the local weak equivalences.  For two simplicial presheaves $E$ and $F$ on $\mathcal{C}$, we denote 
\begin{gather*}
[E,F]:=\mathrm{Hom}_{\mathrm{Ho}(s \mathbf{Pre}(\mathcal{C}))}(E,F).
\end{gather*}
For   a simplicial presheaf $E$ on $\mathcal{C}$, and a complex $F$ of abelian sheaves  on $\mathcal{C}$, the $n$-th cohomology group of $E$ with coefficients in $F$ is defined to be
\begin{gather*}
H^n(E,F):=[E,\mathrm{K}(F,n)].
\end{gather*}
When $E$ is the presheaf represented by an object $X$ of $\mathcal{C}$, by \cite[Theorem 8.26]{Jar15},  $H^n(E,F)$ coincides with the usual hypercohomology group $\mathbb{H}^n(X,F)$.
\end{definition}

\begin{lemma}\label{lem:simplicialQuasiCoherentSheafCohomology}
Let $n\in \mathbb{N}$, and $X=\{X_i\}_{i\in \mathbb{N}\cup\{0\}}$ be a simplicial smooth affine scheme over $W_n$. Let $F$ be a left bounded complex of abelian sheaves on $(\mathrm{Sm}_{W_{\centerdot}})_{\mathrm{Nis}}$, such that for each $i\in \mathbb{Z}$, $F^i=\mathrm{red}_m^{-1}G^i$ for a   quasi-coherent sheaf $G^i$ on $(\mathrm{Sm}_{W_m})_{\mathrm{Nis}}$ for a certain $m$. Identifying $X$ with the simplicial sheaf $\mathrm{red}_n^{-1}X$ on $(\mathrm{Sm}_{W_{\centerdot}})_{\mathrm{Nis}}$, there is a canonical isomorphism
\begin{gather*}
H^q(X,F)\cong H^q\big(\operatorname{Tot} \Gamma(X,G)\big),
\end{gather*}
where $\operatorname{Hom}(X,G)$ is the double complex with $\Gamma(X,G)^{i,j}= \Gamma(X_{i},G^j)$,  and the differential in the first direction is induced by the differential of the Moore complex associated with the simplicial abelian sheaf $\mathbb{Z}X$.
\end{lemma}
\begin{proof}
By a long exact sequence  and induction, we may assume that $F$ is  concentrated in degree 0. Then let $F=\mathrm{red}_m^{-1}G$ for a   quasi-coherent sheaf $G$ on $(\mathrm{Sm}_{W_m})_{\mathrm{Nis}}$ for some $m$.
By \cite[Lemma 8.24]{Jar15}, it suffices to show $H^0_{(\mathrm{Sm}_{W_{\centerdot}})_{\mathrm{Nis}}}(X_i,F)=\Gamma(X_i,G)$ and $H^q_{(\mathrm{Sm}_{W_{\centerdot}})_{\mathrm{Nis}}}(X_i,F)=0$ for $q>0$. Since $X_i$ is affine and smooth over $W_n$, there exists a smooth scheme $Y$ over $W$ such that $X_i\cong Y_n:= Y\times_W W_n$. Let $Y_{\centerdot}$ be the associated object in $\mathrm{Sm}_{W_{\centerdot}}$. The cohomology $H^q_{(\mathrm{Sm}_{W_{\centerdot}})_{\mathrm{Nis}}}(X_i,F)$ can be computed in the localization $(\mathrm{Sm}_{W_{\centerdot}})_{\mathrm{Nis}}/Y_{\centerdot}$, and then can be computed by the small Nisnevich site over $Y_{\centerdot}$. But this latter site is equivalent to the small Nisnevich site over $Y_j$ for any $j\in \mathbb{N}$. Hence the wanted assertion follows from that the Nisnevich cohomology of a quasi-coherent sheaf is isomorphic to its Zariski cohomology, and that $Y_j$ is affine for all $j$.
\end{proof}

Our definition of infinitesimal Chern classes and the multiplicativity of the infinitesimal Chern  character is based on Pushin's motivic Chern classes \cite{Pus04} and the following fundamental theorem of Totaro \cite[Theorem 9.2]{Tot18}.
\begin{theorem}[Totaro]\label{thm:Totaro}
For any integer $n\geq 1$, there is a canonical isomorphism of graded rings 
\begin{gather*}
\bigoplus_{i,j\geq 0} H^{i}(\mathrm{BGL}_{n,\mathbb{Z}},\Omega^j) \cong \mathbb{Z}[c_1,\dots,c_n]
\end{gather*}
where $\deg c_i=(i,i)$ and $\Omega^j$ is the sheaf $U \mapsto \Omega^j_{U/\mathbb{Z}}(U)$.
\end{theorem}
Totaro's theorem is stated for the Zariski and the étale sites.
Since $\Omega^i$ is quasi-coherent on each $\mathrm{GL}_{n,\mathbb{Z}}$, this result holds equally on the Nisnevich  site. For any ring $R$, we understand $\mathrm{BGL}_R$ as the simplicial ind-scheme $\dots \rightarrow \mathrm{BGL}_{n,R} \rightarrow \mathrm{BGL}_{n+1,R} \rightarrow \dots$.
Then we have
\begin{gather*}
\bigoplus_{i,j\geq 0} H^{i}(\mathrm{BGL}_{\mathbb{Z}},\Omega^j) \cong \mathbb{Z}[c_1,c_2,\dots].
\end{gather*}
In the rest of this section, the  cohomology of (simplicial) sheaves is the cohomology on the site  $(\mathrm{Sm}_{W_{\centerdot}})_{\mathrm{Nis}}$. We introduce the following sheaves on this site.
\begin{itemize}
  \item $\Omega^i_{?/W_n}(X_{\centerdot}):=\Omega^i_{X_n/W_n}$.
  \item $p^{r,M}_{r,L}\Omega^{\bullet}_{?\centerdot}(X_{\centerdot}):=p^{r,M}_{r,L}\Omega^{\bullet}_{X_{\centerdot}}$, 
  $p(r)\Omega^{\bullet}_{?n}(X_{\centerdot}):=p(r)\Omega^{\bullet}_{X_{n}}$.
\end{itemize}

\begin{proposition}\label{prop:truncatedMotCoh-BGL}
For $n\geq 1$, there is a canonical isomorphism of graded groups
\begin{gather*}
\bigoplus_{r<p}H^{2r}\big(\mathrm{BGL}_{W_1},\mathbb{Z}_{?n}(r)\big)\cong \mathbb{Z}[c_1,c_2,\dots]_{<2p},
\end{gather*}
where $\deg c_i=2i$, and the subscript “$<2p$" means taking only the sums of monomials of degree less than $2p$. 
\end{proposition}
\begin{proof}
For $n=1$ this is \cite[Proposition 2]{Pus04} without restricting the degrees.
By Totaro's Theorem \ref{thm:Totaro}
and the universal coefficient theorem, $H^i(\mathrm{BGL}_{R},\Omega^j_{?/R})=0$ for $i\neq j$ and any ring $R$. So using Lemma \ref{lem:simplicialQuasiCoherentSheafCohomology} we get
\begin{gather*}
H^i(\mathrm{BGL}_{W_1},\Omega^j_{?/W_n})=H^i(\mathrm{BGL}_{W_n},\Omega^j_{?/W_n})=0\quad \mbox{for}\ i\neq j.
\end{gather*}
Then  by the Hodge-to-de Rham spectral sequence, we get
\begin{gather*}
H^{2r-1}(\mathrm{BGL}_{W_1},p(r)\Omega^{\bullet}_{?n})=0,\\
H^{2r}(\mathrm{BGL}_{W_1},p(r)\Omega^{\bullet}_{?n})=0.
\end{gather*}
So the conclusion follows from \eqref{eq:motivicFundamentalTriangle} and the $n=1$ case. 
\end{proof}

\begin{proposition}\label{prop:relativeDeRhamCoh-BGLtimesBGL}
For $m,n,r\in \mathbb{N}$, and integers $L>M\geq 1$, we have
\begin{subequations}
\begin{gather}
H^{2r-1}(\mathrm{BGL}_{n,W_m}\times \mathrm{BGL}_{n,W_m}, p^{r,M}_{r,L}\Omega^{\bullet}_{?{\centerdot}} )=0,\label{eq:relativeDeRhamCoh-BGLtimesBGL-times}\\
H^{2r-1}(\mathrm{BGL}_{n,W_m}\land \mathrm{BGL}_{n,W_m}, p^{r,M}_{r,L}\Omega^{\bullet}_{?{\centerdot}} )=0.\label{eq:relativeDeRhamCoh-BGLtimesBGL-land}
\end{gather}
\end{subequations}
\end{proposition}
\begin{proof}
There is a cofiber sequence of simplicial sheaves
\begin{gather*}
\mathrm{BGL}_{n,W_m}\lor \mathrm{BGL}_{n,W_m} \rightarrow \mathrm{BGL}_{n,W_m}\times  \mathrm{BGL}_{n,W_m} \rightarrow \mathrm{BGL}_{n,W_m}\land \mathrm{BGL}_{n,W_m}.
\end{gather*}
Thus (\cite[Cor. 6.4.2(b)]{Hov99}) there is an exact sequence
\begin{multline*}
\dots \rightarrow H^{2r-2}(\mathrm{BGL}_{n,W_m}\times \mathrm{BGL}_{n,W_m},p^{r,M}_{r,L}\Omega^{\bullet}_{?{\centerdot}}) \rightarrow
 H^{2r-2}(\mathrm{BGL}_{n,W_m}\lor \mathrm{BGL}_{n,W_m},p^{r,M}_{r,L}\Omega^{\bullet}_{?{\centerdot}})\\
\rightarrow H^{2r-1}(\mathrm{BGL}_{n,W_m}\land \mathrm{BGL}_{n,W_m},p^{r,M}_{r,L}\Omega^{\bullet}_{?{\centerdot}})
\rightarrow H^{2r-1}(\mathrm{BGL}_{n,W_m}\times \mathrm{BGL}_{n,W_m},p^{r,M}_{r,L}\Omega^{\bullet}_{?{\centerdot}}) 
\rightarrow \dots
\end{multline*}
The proposition follows then from the following two claims:
\begin{enumerate}[(i)]
  \item $H^{2r-2}(\mathrm{BGL}_{n,W_m}\times \mathrm{BGL}_{n,W_m},p^{r,M}_{r,L}\Omega^{\bullet}_{?{\centerdot}})$
  maps onto  $H^{2r-2}(\mathrm{BGL}_{n,W_m}\lor \mathrm{BGL}_{n,W_m},p^{r,M}_{r,L}\Omega^{\bullet}_{?{\centerdot}})$.
  \item $H^{2r-1}(\mathrm{BGL}_{n,W_m}\times \mathrm{BGL}_{n,W_m},p^{r,M}_{r,L}\Omega^{\bullet}_{?{\centerdot}})=0$.
\end{enumerate}
We first show (ii). By Lemma \ref{lem:simplicialQuasiCoherentSheafCohomology} and the Hodge-to-de Rham spectral sequence, it suffices to show $H^{2r-1-i}(\mathrm{BGL}_{n,W_m}\times \mathrm{BGL}_{n,W_m},\Omega^{i}_{?/W_m})=0$. By the Eilenberg-Zilber theorem, 
\begin{gather*}
H^{2r-1-i}(\mathrm{BGL}_{n,W_m}\times \mathrm{BGL}_{n,W_m},\Omega^{i}_{?/W_m})\\
=\bigoplus_{j+k=i}H^{2r-i-1}\big(\Gamma(\mathrm{BGL}_{n,W_m},\Omega^{j}_{?/W_m})\otimes_{W_m}\Gamma(\mathrm{BGL}_{n,W_m},\Omega^{k}_{?/W_m})\big),
\end{gather*}
which vanishes by Theorem \ref{thm:Totaro}.

Then we show (i). By Theorem \ref{thm:Totaro}, Lemma \ref{lem:simplicialQuasiCoherentSheafCohomology}, and the Hodge-to-de Rham spectral sequence, we have
\begin{gather*}
H^{2r-2}(\mathrm{BGL}_{n,W_m}\times \mathrm{BGL}_{n,W_m},p^{r,M}_{r,L}\Omega^{\bullet}_{?{\centerdot}})
=H^{r-1}(\mathrm{BGL}_{n,W_m}\times \mathrm{BGL}_{n,W_m},\Omega^{r-1}_{?/W_m}).
\end{gather*}
By the Eilenberg-Zilber theorem, 
\begin{gather}\label{eq:relativeDeRhamCoh-BGLtimesBGL-decomp-1}
H^{r-1}(\mathrm{BGL}_{n,W_m}\times \mathrm{BGL}_{n,W_m},\Omega^{r-1}_{?/W_m})\\
=\bigoplus_{j+k=r-1}
H^{r-1}\big(\Gamma(\mathrm{BGL}_{n,W_m},\Omega^{j}_{?/W_m})\otimes_{W_m}\Gamma(\mathrm{BGL}_{n,W_m},\Omega^{k}_{?/W_m})\big).\nn
\end{gather}
On the other hand, applying \cite[Cor. 6.4.2(b)]{Hov99} to the pushout square
\begin{gather*}
\xymatrix{
  \ast \ar[r] \ar[d] &  \mathrm{BGL}_{n,W_m} \ar[d] \\
  \mathrm{BGL}_{n,W_m} \ar[r] & \mathrm{BGL}_{n,W_m}\lor \mathrm{BGL}_{n,W_m}
}
\end{gather*}
yields a homomorphism between two long exact sequences
\begin{gather*}
\scalebox{0.75}{
\xymatrix@C=0.5pc{
\dots \ar[r] & H^i(\mathrm{BGL}_{n,W_m}/\ast,p^{r,M}_{r,L}\Omega^{\bullet}_{?{\centerdot}}) \ar[r] \ar@{=}[d] &
 H^i(\mathrm{BGL}_{n,W_m}\lor \mathrm{BGL}_{n,W_m},p^{r,M}_{r,L}\Omega^{\bullet}_{?{\centerdot}}) \ar[r] \ar[d] & H^i(\mathrm{BGL}_{n,W_m}, p^{r,M}_{r,L}\Omega^{\bullet}_{?{\centerdot}}) \ar[r] \ar[d] & H^{i+1}(\mathrm{BGL}_{n,W_m}/\ast,p^{r,M}_{r,L}\Omega^{\bullet}_{?{\centerdot}}) \ar[r] \ar@{=}[d] & \dots \\
\dots \ar[r] & H^i(\mathrm{BGL}_{n,W_m}/\ast,p^{r,M}_{r,L}\Omega^{\bullet}_{?{\centerdot}}) \ar[r]  & H^i(\mathrm{BGL}_{n,W_m},p^{r,M}_{r,L}\Omega^{\bullet}_{?{\centerdot}}) \ar[r]  & H^i(\ast, p^{r,M}_{r,L}\Omega^{\bullet}_{?{\centerdot}}) \ar[r]  & H^{i+1}(\mathrm{BGL}_{n,W_m}/\ast,p^{r,M}_{r,L}\Omega^{\bullet}_{?{\centerdot}}) \ar[r]  & \dots 
 } }
\end{gather*}
which by a diagram chasing implies  an exact sequence
\begin{multline*}
\dots \rightarrow H^{2r-3}(\ast,p^{r,M}_{r,L}\Omega^{\bullet}_{?{\centerdot}}) \rightarrow
 H^{2r-2}(\mathrm{BGL}_{n,W_m}\lor \mathrm{BGL}_{n,W_m},p^{r,M}_{r,L}\Omega^{\bullet}_{?{\centerdot}})\\
\rightarrow H^{2r-2}(\mathrm{BGL}_{n,W_m},p^{r,M}_{r,L}\Omega^{\bullet}_{?{\centerdot}})^{\oplus 2}
\rightarrow H^{2r-2}(\ast,p^{r,M}_{r,L}\Omega^{\bullet}_{?{\centerdot}}) 
\rightarrow \dots
\end{multline*}
If $r>1$, then 
\begin{gather}\label{eq:relativeDeRhamCoh-BGLtimesBGL-decomp-2}
H^{2r-2}(\mathrm{BGL}_{n,W_m}\lor \mathrm{BGL}_{n,W_m},p^{r,M}_{r,L}\Omega^{\bullet}_{?{\centerdot}})=H^{2r-2}(\mathrm{BGL}_{n,W_m},p^{r,M}_{r,L}\Omega^{\bullet}_{?{\centerdot}})^{\oplus 2}
\end{gather}
 and the summands with $(j,k)=(r-1,0)$ and $(0,r-1)$ in \eqref{eq:relativeDeRhamCoh-BGLtimesBGL-decomp-1} are sent to the two summands of \eqref{eq:relativeDeRhamCoh-BGLtimesBGL-decomp-2}. If $r=1$ then 
 $H^{0}(\mathrm{BGL}_{n,W_m}\lor \mathrm{BGL}_{n,W_m},p^{r,M}_{r,L}\Omega^{\bullet}_{?{\centerdot}})=H^{0}(\mathrm{BGL}_{n,W_m},p^{r,M}_{r,L}\Omega^{\bullet}_{?{\centerdot}})$ and \eqref{eq:relativeDeRhamCoh-BGLtimesBGL-decomp-1} maps isomorphically onto it. Hence the proof is complete.
\end{proof}

By \cite[Theorem 5.9]{Jar15}, on a site $\mathcal{C}$, the category $s \mathbf{Shv}(\mathcal{C})$ has an \emph{injective model structure} where the weak equivalences are the local weak equivalences, the cofibrations are the monomorphisms of sheaves, and the fibrations, called \emph{injective fibrations}, are the maps which have the right lifting property with respect to all trivial cofibrations. Moreover, this model category is proper.

\begin{proposition}\label{prop:infinitesimalChernClass}
Let $\mathbb{Z}[c_1,\dots]$ be the graded ring of formal variables $c_i$ with $\deg(c_i)=i$.
Let $P(\mathbf{c})\in \mathbb{Z}[c_1,\dots]$ be a polynomial of degree $r<p$.
\begin{enumerate}[(i)]
  \item For $i\geq 0$, there exists a canonical Chern class map
  \begin{gather}\label{eq:infinitesimalChernClass-absolute}
  P(\mathbf{c}): K_i(X_n)\rightarrow H^{2r-i}\big(X_1,\mathbb{Z}_{X_n}(r)\big)
  \end{gather}
  for  $X_{\centerdot}\in \mathrm{Sm}_{W_{\centerdot}}$, which is compatible for $n\geq 1$; 
  in particular, these Chern class maps are functorial in $X_{\centerdot}\in \mathrm{Sm}_{W_{\centerdot}}$. 
  Here the compatibility means that for $n\geq n'$, the diagram
  \begin{gather*}
  \xymatrix{
     K_i(X_n)\ar[r]^<<<<{P(\mathbf{c})} \ar[d] & H^{2r-i}\big(X_1,\mathbb{Z}_{X_n}(r)\big)  \ar[d] \\
     K_i(X_{n'})\ar[r]^<<<<{P(\mathbf{c})}  & H^{2r-i}\big(X_1,\mathbb{Z}_{X_{n'}}(r)\big)  
  }
  \end{gather*}
  commutes.
  \item For  $i\geq 0$, there exists
  a canonical relative Chern class map
  \begin{gather}\label{eq:infinitesimalChernClass-rel}
  P(\mathbf{c}): K_i(X_n,X_m)\rightarrow H^{2r-i}\big(X_1,p^{r,m}_{r,n}\Omega_{X_{\centerdot}}^{\bullet}\big)
  \end{gather}
  for $X_{\centerdot}\in \mathrm{Sm}_{W_{\centerdot}}$,   which is compatible with $n\geq m\geq 1$; 
  in particular, these Chern class maps are functorial in $X_{\centerdot}\in \mathrm{Sm}_{W_{\centerdot}}$. 
  Here the compatibility means that for $n,m,n',m'\in \mathbb{N}$ satisfying
  \begin{gather*}
  \xymatrix@C=0.5pc@R=0.7pc{
    n \ar@{}[r]|-{\geq} \ar@{}[d]|-{\rotatebox{-90}{$\geq$}} & m \ar@{}[d]|-{\rotatebox{-90}{$\geq$}} \\
    n' \ar@{}[r]|-{\geq} & m' 
  }
  \end{gather*}
  the diagram
  \begin{gather*}
  \xymatrix{
    K_i(X_n,X_m)\ar[r]^<<<<{P(\mathbf{c})} \ar[d] & H^{2r-i}\big(X_1,p^{r,m}_{r,n}\Omega_{X_{\centerdot}}^{\bullet}\big) \ar[d] \\
    K_i(X_{n'},X_{m'})\ar[r]^<<<<{P(\mathbf{c})}  & H^{2r-i}\big(X_1,p^{r,m'}_{r,n'}\Omega_{X_{\centerdot}}^{\bullet}\big) 
  }
  \end{gather*}
  commutes.
\end{enumerate}
\end{proposition}\label{prop:infinitesimalChernCharacter}

\begin{remark}\label{rem:functoriality-vs-canonicality}
We would like to remind the reader of the difference between functoriality and canonicality. The former has a general mathematical definition. The latter is in general somewhat aesthetic, and in Proposition \ref{prop:infinitesimalChernClass} it means the independence on certain universal choices of objects and maps relating to the classifying spaces, more precisely, the statements $(1^{\circ})$ and $(2^{\circ})$ in the following proof.
\end{remark}

\begin{proof}
By Proposition \ref{prop:truncatedMotCoh-BGL}, the canonical maps
\begin{gather}\label{eq:infinitesimalChernCharacter-0}
\bigoplus_{r<p}H^{2r}\big(\mathrm{BGL}_{W_n},\mathbb{Z}_{?n}(r)\big) \rightarrow \bigoplus_{r<p}H^{2r}\big(\mathrm{BGL}_{W_{n-1}},\mathbb{Z}_{?n-1}(r)\big)
\end{gather}
are isomorphisms and send $P(\mathbf{c})$ to $P(\mathbf{c})$. In other words, in the homotopy category $\mathrm{Ho}\big(s\mathbf{Shv}(\mathrm{Sm}_{W_{\centerdot}})_{\mathrm{Nis}}\big)$, there is a commutative diagram
\begin{gather}\label{diag:infinitesimalChernClass-1}
\xymatrix{
  \dots \ar[r] & \mathrm{BGL}_{W_n} \ar[r] \ar[d]_{P(\mathbf{c})} & \mathrm{BGL}_{W_{n-1}} \ar[r] \ar[d]_{P(\mathbf{c})} & \dots \ar[r] & \mathrm{BGL}_{W_1} \ar[d]_{P(\mathbf{c})} \\
  \dots \ar[r] & \mathrm{K}\big(\mathbb{Z}_{?n}(r),2r\big) \ar[r] & \mathrm{K}\big(\mathbb{Z}_{?n-1}(r),2r\big) \ar[r] & \dots \ar[r] & \mathrm{K}\big(\mathbb{Z}_{?1}(r),2r\big)
}
\end{gather}
where the horizontal maps are maps in $s\mathbf{Shv}(\mathrm{Sm}_{W_{\centerdot}})_{\mathrm{Nis}}$. 
In the injective model category $s\mathbf{Shv}(\mathrm{Sm}_{W_{\centerdot}})_{\mathrm{Nis}}$, by the factorization property of model categories, there is a commutative diagram
\begin{gather}\label{diag:infinitesimalChernClass-2}
\xymatrix{
  \dots \ar[r] & \mathrm{K}\big(\mathbb{Z}_{?n}(r),2r\big) \ar[r] \ar[d] & \mathrm{K}\big(\mathbb{Z}_{?{n-1}}(r),2r\big) \ar[r] \ar[d] & \dots \ar[r] & \mathrm{K}\big(\mathbb{Z}_{?1}(r),2r\big) \ar[d] \\
  \dots \ar[r] & \Gamma_{n}(r) \ar[r] & \Gamma_{n-1}(r) \ar[r] & \dots \ar[r] & \Gamma_{1}(r)
  }
\end{gather}
such that the vertical maps are local weak equivalences, $\Gamma_{1}(r)$ is fibrant, and the maps $\Gamma_{n}(r) \rightarrow \Gamma_{n-1}(r)$ are fibrations for all $n\geq 2$.
Since all simplicial sheaves are cofibrant in the injective model category, there is a commutative diagram
\begin{gather}\label{diag:infinitesimalChernClass-3}
\xymatrix{
  \dots \ar[r] & \mathrm{BGL}_{W_n} \ar[r] \ar[d]_{P(\mathbf{c})} & \mathrm{BGL}_{W_{n-1}} \ar[r] \ar[d]_{P(\mathbf{c})} & \dots \ar[r] & \mathrm{BGL}_{W_1} \ar[d]_{P(\mathbf{c})} \\
  \dots \ar[r] & \Gamma_{n}(r) \ar[r] & \Gamma_{n-1}(r) \ar[r] & \dots \ar[r] & \Gamma_{1}(r)
}
\end{gather}
in $s\mathbf{Shv}(\mathrm{Sm}_{W_{\centerdot}})_{\mathrm{Nis}}$, which thus lifts the  diagram  \eqref{diag:infinitesimalChernClass-1} in the homotopy category because of the isomorphisms \eqref{eq:infinitesimalChernCharacter-0}.

Now we apply Gillet's construction of Chern classes (\cite[\S 2]{Gil81}). For any $X_{\centerdot}\in \mathrm{Sm}_{W_{\centerdot}}$ and  natural numbers $k,m,n,l$,  we have canonical maps
\begin{gather*}
\operatorname{Hom}_{(\mathrm{Sm}_{W_{\centerdot}})_{\mathrm{Nis}}}\big(B_k \mathrm{GL}_{m,W_n},\Gamma_n(r)_l\big)
\rightarrow \operatorname{Hom}_{(\mathrm{Sm}_{W_{\centerdot}})_{\mathrm{Nis}}}\big(\mathcal{H}om(X_{\centerdot},B_k \mathrm{GL}_{m,W_n}),\mathcal{H}om(X_{\centerdot},\Gamma_n(r)_l)\big)\\
\cong \operatorname{Hom}_{(X_{\centerdot})_{\mathrm{NIS}}}\big(B_k \mathrm{GL}_m(\mathcal{O}_{X_n}), \Gamma_n(r)_l|_{(X_{\centerdot})_{\mathrm{NIS}}}\big) \rightarrow
\operatorname{Hom}_{(X_{\centerdot})_{\mathrm{Nis}}}\big(B_k \mathrm{GL}_m(\mathcal{O}_{X_n}), \Gamma_n(r)_l|_{(X_{\centerdot})_{\mathrm{Nis}}}\big),
\end{gather*}
where $(X_{\centerdot})_{\mathrm{NIS}}$ is the big Nisnevich site of $X_{\centerdot}$ and $(X_{\centerdot})_{\mathrm{Nis}}$ is the small one. Since $(X_{\centerdot})_{\mathrm{Nis}}$ is equivalent to $(X_{n})_{\mathrm{Nis}}$, $\Gamma_n(r)_l|_{(X_{\centerdot})_{\mathrm{Nis}}}$ corresponds to a Nisnevich sheaf on $(X_{n})_{\mathrm{Nis}}$, denoted by $\Gamma_n(r)_l|_{(X_{n})_{\mathrm{Nis}}}$. Moreover, \eqref{diag:infinitesimalChernClass-2} induces a local weak equivalence  $\mathrm{K}\big(\mathbb{Z}_{X_n}(r),2r\big) \simeq \Gamma_n(r)|_{(X_{n})_{\mathrm{Nis}}}$. Then the element $P(\mathbf{c})\in \operatorname{Hom}_{(\mathrm{Sm}_{W_{\centerdot}})_{\mathrm{Nis}}}\big(\mathrm{BGL}_{m,W_n},\Gamma_n(r)\big)$ induces an element in  $\operatorname{Hom}_{(X_{\centerdot})_{\mathrm{Nis}}}\big(\mathrm{BGL}_m(\mathcal{O}_{X_n}), \Gamma_n(r)|_{(X_{\centerdot})_{\mathrm{Nis}}}\big)$, hence an element $P(\mathbf{c})_{X_n}$ in 
\begin{gather*}
[\mathbb{Z}_{\infty}\mathrm{BGL}_m(\mathcal{O}_{X_n}), \mathbb{Z}_{\infty}\Gamma_n(r)|_{(X_{\centerdot})_{\mathrm{Nis}}}]\cong [\mathbb{Z}_{\infty}\mathrm{BGL}_m(\mathcal{O}_{X_n}), \Gamma_n(r)|_{(X_{\centerdot})_{\mathrm{Nis}}}]\\
\cong \big[\mathbb{Z}_{\infty}\mathrm{BGL}_m(\mathcal{O}_{X_n}), \mathrm{K}\big(\mathbb{Z}_{X_n}(r),2r\big)\big],
\end{gather*}
where $[\ ,\ ]$ are taken in $\mathrm{Ho}(s \mathbf{Shv}(X_{\centerdot})_{\mathrm{Nis}})$. There is an equivalence of Nisnevich sheaves $\mathcal{K}_{X_n}\simeq \mathbb{Z}\times \mathbb{Z}_{\infty}\mathrm{BGL}(\mathcal{O}_{X_n})$ (see \cite[Proposition 2.15]{Gil81} and \cite[Page 702]{BEK14}). By the Nisnevich descent of $K$-theory (\cite[\S 10]{ThT90}), $P(\mathbf{c})_{X_n}$ induces the wanted Chern class maps \eqref{eq:infinitesimalChernClass-absolute}. 

The compatibility of the Chern class maps follows from the commutativity of the diagram \eqref{diag:infinitesimalChernClass-3}. The Chern class maps are canonical in the following sense:
\begin{enumerate}
  \item[$(1^{\circ})$] they are independent of the choices of the vertical arrows $P(\mathbf{c})$ in \eqref{diag:infinitesimalChernClass-3}, and
  \item[$(2^{\circ})$] they are independent of the choices of the $\Gamma_n(r)$'s in  
  \eqref{diag:infinitesimalChernClass-2}.
\end{enumerate}
For $(1^{\circ})$, we note that any two choices of a vertical arrow $P(\mathbf{c})$ in \eqref{diag:infinitesimalChernClass-3} are simplicially homotopic to each other in the injective model category $s\mathbf{Shv}(\mathrm{Sm}_{W_{\centerdot}})_{\mathrm{Nis}}$:
\begin{gather}\label{diag:infinitesimalChernClass-2cell}
\xymatrix{
\mathrm{BGL}_{W_n}\rtwocell^{P(\mathbf{c})}_{P(\mathbf{c})'} & \Gamma_n(r)\quad .
}
\end{gather}
Then the induced maps are homotopic by induced homotopies. For $(2^{\circ})$, we observe that, by the left properness of the injective model structure of simplicial sheaves, two choices of the factorization
\begin{gather*}
\mathrm{K}\big(\mathbb{Z}_{?{n-1}}(r),2r\big) \rightarrow \Gamma_n(r) \rightarrow \Gamma_{n-1}(r)\\
(\text{resp.}\
\mathrm{K}\big(\mathbb{Z}_{?{n-1}}(r),2r\big) \rightarrow \Gamma_n(r)' \rightarrow \Gamma_{n-1}(r)\quad )
\end{gather*}
where the first arrow is a trivial cofibration and the second a fibration, 
factor through a common one:
\begin{gather*}
\xymatrix{
  \mathrm{K}\big(\mathbb{Z}_{?{n-1}}(r),2r\big) \ar[r] \ar[d] & \Gamma_n(r) \ar[d] \ar@/^1pc/[ddr] & \\
  \Gamma_n(r)' \ar[r] \ar@/_1pc/[rrd] & \Gamma_n(r)'' \ar@{-->}[dr] & \\
  &&  \Gamma_{n-1}(r),
}
\end{gather*}
where $\Gamma_n(r)''$ is a pushout of the square. Applying the above argument to $\Gamma_n(r)''$ we get the homotopical independence of the relative Chern classes on the choice of $\Gamma_n(r)$.
 This completes the proof of (i). Now we show (ii).
The sub-diagram 
\begin{gather*}
\xymatrix{
   \mathrm{BGL}_{W_n} \ar[r] \ar[d]_{P(\mathbf{c})} & \mathrm{BGL}_{W_{m}} \ar[d]^{P(\mathbf{c})} \\
  \Gamma_n(r) \ar[r] & \Gamma_m(r)
}
\end{gather*}
of \eqref{diag:infinitesimalChernClass-3} induces a commutative diagram
\begin{gather*}
\xymatrix{
  & \mathrm{BGL}(\mathcal{O}_{X_n}) \ar[r] \ar[d]_{P(\mathbf{c})} & \mathrm{BGL}(\mathcal{O}_{X_m}) \ar[d]^{P(\mathbf{c})} & \\
  & \Gamma_n(r)|_{(X_{\centerdot})_{\mathrm{Nis}}} \ar[r] & \Gamma_m(r)|_{(X_{\centerdot})_{\mathrm{Nis}}} &. 
}
\end{gather*}
Applying $\mathbb{Z}_{\infty}$ to this diagram yields a map on homotopy fibers, and thus the relative Chern class maps \eqref{eq:infinitesimalChernClass-rel}. The compatibility\footnote{The functoriality of the relative Chern class maps also follows directly from the commutativity of \eqref{diag:infinitesimalChernClass-3}, while the canonicality does not. This is an illustration of Remark \ref{rem:functoriality-vs-canonicality}.}
 follows from the commutativity of \eqref{diag:infinitesimalChernClass-3}. 
We are left to show the canonicality of the relative Chern class maps in the sense of $(1^{\circ})$ and $(2^{\circ})$ as above.

For the sake of clarity, from now on in this proof we denote a map $P(\mathbf{c}):\mathrm{BGL}_{W_n} \rightarrow \Gamma_n(r)$ by $P(\mathbf{c})_n$.
For $(1^{\circ})$, suppose given two choices of the map $P(\mathbf{c})_n:\mathrm{BGL}_{W_n} \rightarrow \Gamma_n(r)$, and two choices of the map $P(\mathbf{c})_m:\mathrm{BGL}_{W_{m}} \rightarrow \Gamma_{m}(r)$, such that both diagrams
\begin{gather*}
\xymatrix{
   \mathrm{BGL}_{W_n} \ar[r] \ar[d]_{P(\mathbf{c})_n} & \mathrm{BGL}_{W_{m}} \ar[d]^{P(\mathbf{c})_m} \\
  \Gamma_n(r) \ar[r] & \Gamma_m(r)
}\quad\mbox{and}\quad
\xymatrix{
   \mathrm{BGL}_{W_n} \ar[r] \ar[d]_{P(\mathbf{c})'_n} & \mathrm{BGL}_{W_{m}} \ar[d]^{P(\mathbf{c})'_m} \\
  \Gamma_n(r) \ar[r] & \Gamma_m(r)
}
\end{gather*}
are commutative. Let $\text{“}P(\mathbf{c})_m\Rightarrow P(\mathbf{c})'_m "$ be  a homotopy  in \eqref{diag:infinitesimalChernClass-2cell} (replacing $n$ with $m$). Since $\Gamma_n(r)^{\Delta^1} \rightarrow \Gamma_m(r)^{\Delta^1}$ is a fibration, there exists a map $f$ making the diagram
\begin{gather*}
\xymatrix{
   \mathrm{BGL}_{W_n} \ar[r] \ar@{-->}[d]_{f} & \mathrm{BGL}_{W_{m}} \ar[d]^{\text{“}P(\mathbf{c})_m\Rightarrow P(\mathbf{c})'_m "} \\
  \Gamma_n(r)^{\Delta^1} \ar[r] & \Gamma_m(r)^{\Delta^1}
}
\end{gather*}
commutative. Both $P(\mathbf{c})_n$ and $f|_{ \mathrm{BGL}_{W_n}\times\{0\}}$ lies in the fiber of 
\begin{gather*}
\mathrm{Hom}\big(\mathrm{BGL}_{W_n},\Gamma_n(r)\big) \rightarrow \mathrm{Hom}\big(\mathrm{BGL}_{W_n},\Gamma_m(r)\big)
\end{gather*}
over the composition map $\mathrm{BGL}_{W_n} \rightarrow \mathrm{BGL}_{W_m} \xrightarrow{P(\mathbf{c})_m}\Gamma_m(r)$.
Since 
\begin{gather*}
\big[\mathrm{BGL}_{W_n},\operatorname{hofib}\big(\Gamma_n(r) \rightarrow \Gamma_m(r)\big)\big]=
H^{2r-1}(\mathrm{BGL}_{W_n}, p^{r,m}_{r,n}\Omega^{\bullet}_{?{\centerdot}} )=0
\end{gather*}
(the latter equality follows from Theorem \ref{thm:Totaro}, see the proof of Proposition \ref{prop:truncatedMotCoh-BGL}),
there exists a map $\mathrm{BGL}_{W_n}\times \Delta^1 \xrightarrow{g} \Gamma_n(r)$ connecting $P(\mathbf{c})_n$ and 
$f|_{ \mathrm{BGL}_{W_n}\times\{0\}}$ such that the composition
\begin{gather*}
\mathrm{BGL}_{W_n}\times \Delta^1 \xrightarrow{g} \Gamma_n(r) \rightarrow \Gamma_m(r)
\end{gather*}
is the constant map $\mathrm{BGL}_{W_n} \rightarrow \mathrm{BGL}_{W_m} \xrightarrow{P(\mathbf{c})_m}\Gamma_m(r)$. 
The same argument gives a map $\mathrm{BGL}_{W_n}\times \Delta^1 \xrightarrow{h} \Gamma_n(r)$ connecting $P(\mathbf{c})'_n$ and $f|_{ \mathrm{BGL}_{W_n}\times\{1\}}$ such that the composition
\begin{gather*}
\mathrm{BGL}_{W_n}\times \Delta^1 \xrightarrow{h} \Gamma_n(r) \rightarrow \Gamma_m(r)
\end{gather*}
is the constant map $\mathrm{BGL}_{W_n} \rightarrow \mathrm{BGL}_{W_m} \xrightarrow{P(\mathbf{c})'_m}\Gamma_m(r)$. Then since $\Gamma_n(r)$ is fibrant, there exists a homotopy $\text{“}P(\mathbf{c})_n\Rightarrow P(\mathbf{c})'_n "$ making the diagram
\begin{gather*}
\xymatrix{
   \mathrm{BGL}_{W_n} \ar[r] \ar[d]_{\text{“}P(\mathbf{c})_n\Rightarrow P(\mathbf{c})'_n "} & \mathrm{BGL}_{W_{m}} \ar[d]^{\text{“}P(\mathbf{c})_m\Rightarrow P(\mathbf{c})'_m "} \\
  \Gamma_n(r)^{\Delta^1} \ar[r] & \Gamma_m(r)^{\Delta^1}
}
\end{gather*}
commutative. It follows that  the  map between the homotopy fibers 
\begin{gather*}
\xymatrix{
 \operatorname{hofib}(\mathbb{Z}_{\infty}\mathrm{BGL}_{W_n} \rightarrow \mathbb{Z}_{\infty}\mathrm{BGL}_{W_{m}}) \ar[r] \ar@{-->}[d] &  \mathbb{Z}_{\infty}\mathrm{BGL}_{W_n} \ar[r] \ar[d]_{\mathbb{Z}_{\infty}P(\mathbf{c})_n}& \mathbb{Z}_{\infty}\mathrm{BGL}_{W_{m}} \ar[d]_{\mathbb{Z}_{\infty}P(\mathbf{c})_m} \\
\operatorname{hofib}(\mathbb{Z}_{\infty}\Gamma_n(r) \rightarrow \mathbb{Z}_{\infty}\Gamma_m(r)) \ar[r] &  \mathbb{Z}_{\infty}\Gamma_n(r) \ar[r] & \mathbb{Z}_{\infty}\Gamma_m(r)
}\quad
\end{gather*}
induced by $P(\mathbf{c})$ is homotopic to that induced by $P(\mathbf{c})'$. That is, the homotopy class of the relative Chern class map does not depend on the choice of $P(\mathbf{c})$'s. The proof of the independence of the choice of $\Gamma_n(r)$'s follows by the same argument as that for the absolute Chern classes.
\end{proof}

\begin{proposition}\label{prop:definition-infinitesimal-ChernCharacter}
\begin{enumerate}[(i)]
  \item For $0\leq r<p$ and $X_{\centerdot}\in \mathrm{Sm}_{W_{\centerdot}}$, there is a canonical Chern character map
  \begin{gather*}
  \mathrm{ch}_r: K_i(X_n)\rightarrow H^{2r-i}\big(X_1,\mathbb{Z}_{X_n}(r)\big)[\frac{1}{r!}]\quad.
  \end{gather*}
  Moreover, the Chern character map
  \begin{gather*}
  \mathrm{ch}=\sum_{r<p}\mathrm{ch}_r:\bigoplus_{i\geq 0}K_i(X_n)\rightarrow \bigoplus_{i\geq 0}\bigoplus_{0\leq r<p}  H^{2r-i}\big(X_1,\mathbb{Z}_{X_n}(r)\big)[\frac{1}{r!}]
    \end{gather*}
 is multiplicative.
  \item For $1\leq r\leq p-1$ and $X_{\centerdot}\in \mathrm{Sm}_{W_{\centerdot}}$, there exists  
  a canonical relative Chern character map
  \begin{gather*}
  \mathrm{ch}_r: K_i(X_n,X_m)\rightarrow H^{2r-i}\big(X_1,p^{r,m}_{r,n}\Omega_{X_{\centerdot}}^{\bullet}\big)\quad.
  \end{gather*}
  which is compatible for $n>m\geq 1$.
    Moreover, the Chern character map
\begin{gather*}
  \mathrm{ch}=\sum_{r<p}\mathrm{ch}_r:\bigoplus_{i\geq 0}K_i(X_n,X_m)\rightarrow \bigoplus_{i\geq 0}\bigoplus_{0\leq r<p}  H^{2r-i}\big(X_1,p^{r,m}_{r,n}\Omega_{X_{\centerdot}}^{\bullet}\big)
    \end{gather*}
  is multiplicative, namely, the diagram
  \begin{gather*}
  \xymatrix{
    K_i(X_n)\times K_j(X_n,X_m) \ar[r]^<<<<{\mathrm{ch}} \ar[d] & \bigoplus_{0\leq s<p}  H^{2s-i}\big(X_1,\mathbb{Z}_{X_n}(s)\big)[\frac{1}{s!}]\times   \bigoplus_{0\leq r<p}  H^{2r-j}\big(X_1,p^{r,m}_{r,n}\Omega_{X_{\centerdot}}^{\bullet}\big) \ar[d] \\
    K_{i+j}(X_n,X_m) \ar[r]^<<<<<<<<<<<<<{\mathrm{ch}} & \bigoplus_{0\leq r<p}  H^{2r-i-j}\big(X_1,p^{r,m}_{r,n}\Omega_{X_{\centerdot}}^{\bullet}\big)
  }
  \end{gather*}
  commutes for $i,j\geq 0$.
\end{enumerate}
\end{proposition}
\begin{proof}
(i) For any abelian category $\mathbf{A}$ and $A,B,C\in \operatorname{Ch}_+(\mathbf{A})$, and a morphism $\operatorname{Tot}(A\times B) \rightarrow C$, the composed  morphism of complexes 
\begin{gather*}
\mathrm{N}\big(\mathrm{K}(A)\times \mathrm{K}(B)\big)
= \mathrm{diag}(A\times B) \xrightarrow{\sim} \operatorname{Tot}(A\times B) \rightarrow C
\end{gather*}
where $\simeq$ is the Alexander-Whitney map, induces
a map of simplicial abelian groups $\mathrm{K}(A)\times \mathrm{K}(B) \rightarrow \mathrm{K}(C)$.
Applying this to the product  $\mathbb{Z}_{?n}(i)\otimes^{\mathbf{L}}_{\mathbb{Z}} \mathbb{Z}_{?n}(j) \rightarrow \mathbb{Z}_{?n}(i+j)$, and using that $\mathbb{Z}_{?n}(r)$ is supported in degree $[0,r]$ (Proposition \ref{prop:properties-Zn(r)}), one sees that there are product maps $\Gamma_n(i) \times \Gamma_n(j) \rightarrow \Gamma_n(i+j)$ which are homotopy additive in each summand, and are homotopy associative, and these product maps are homotopy commutative. Namely, the following diagrams are homotopy commutative:
\begin{gather*}
\xymatrix{
  \big(\Gamma_n(i)\times \Gamma_n(i)\big)\times \Gamma_n(j) \ar[r]^<<<<{(+,\mathrm{id})} \ar[d] & \Gamma_n(i)\times \Gamma_n(j) \ar[d] \\
  \Gamma_n(i+j)\times \Gamma_n(i+j) \ar[r]^<<<<<<{+} & \Gamma_n(i+j)
}\quad
\xymatrix{
  \Gamma_n(i)\times \big(\Gamma_n(j)\times \Gamma_n(j)\big) \ar[r]^<<<<{(\mathrm{id},+)} \ar[d] & \Gamma_n(i)\times \Gamma_n(j) \ar[d] \\
  \Gamma_n(i+j)\times \Gamma_n(i+j) \ar[r]^<<<<<<{+} & \Gamma_n(i+j)
}
\end{gather*}
\begin{gather*}
\xymatrix{
  \Gamma_n(i)\times \Gamma_n(j)\times \Gamma_n(k) \ar[r] \ar[d] & \Gamma_n(i)\times \Gamma_n(j+k) \ar[d] \\
  \Gamma_n(i+j)\times \Gamma_n(k) \ar[r] & \Gamma_n(i+j+k)
}
\end{gather*}
\begin{gather*}
\xymatrix{
  \Gamma_n(i)\times \Gamma_n(j) \ar[dr] \ar[rr]^{\tau} && \Gamma_n(j)\times \Gamma_n(i) \ar[dl] \\
 & \Gamma_n(i+j)&
}
\end{gather*}
Then we define operations 
\begin{gather*}
 \big(\mathbb{Z}\times \prod_{1\leq i<p}\Gamma_n(i)\big)\times \big(\mathbb{Z}\times \prod_{1\leq i<p}\Gamma_n(i)\big) \xrightarrow{\cdot} \mathbb{Z}\times \prod_{1\leq i<p}\Gamma_n(i),\\
 \big(\mathbb{Z}\times \prod_{1\leq i<p}\Gamma_n(i)\big)\times \big(\mathbb{Z}\times \prod_{1\leq i<p}\Gamma_n(i)\big) \xrightarrow{\star} \mathbb{Z}\times \prod_{1\leq i<p}\Gamma_n(i)
\end{gather*}
by the formulas in \cite[Definition 2.27]{Gil81}\footnote{This is originally defined by Grothendieck using different notations in \cite[Page 28]{SGA6}.}. We have to show that both diagrams
\begin{gather*}
\xymatrix{
   \mathrm{BGL}_{W_n}\times \mathrm{BGL}_{W_n}  \ar[r]^<<<<<<<<<<<<<<<{\oplus} \ar[d]_{\widetilde{C}\times \widetilde{C}} & \mathrm{BGL}_{W_n} \ar[d]^{\widetilde{C}} \\
   \big(\mathbb{Z}\times \prod_{1\leq i<p}\Gamma_n(i)\big)\times \big(\mathbb{Z}\times \prod_{1\leq i<p}\Gamma_n(i)\big) \ar[r]^<<<<{\cdot} & \mathbb{Z}\times \prod_{1\leq i<p}\Gamma_n(i)
}
\end{gather*}
and
\begin{gather*}
\xymatrix{
   \mathrm{BGL}_{W_n}\land \mathrm{BGL}_{W_n}  \ar[r]^{\mu} \ar[d]_{\widetilde{C}\land \widetilde{C}} & \mathrm{BGL}_{W_n} \ar[d]^{\widetilde{C}} \\
   \big(\mathbb{Z}\times \prod_{1\leq i<p}\Gamma_n(i)\big)\land \big(\mathbb{Z}\times \prod_{1\leq i<p}\Gamma_n(i)\big) \ar[r]^<<<<{\star} & \mathbb{Z}\times \prod_{1\leq i<p}\Gamma_n(i)
}
\end{gather*}
homotopically commute for all $n\in \mathbb{N}$, where $\oplus$ is induced by the map $\mathrm{GL}_{W_n}\times \mathrm{GL}_{W_n} \xrightarrow{\oplus} \mathrm{GL}_{W_n}$ which are induced by choosing a bijection $\mathbb{N}\sqcup \mathbb{N} \rightarrow \mathbb{N}$, and $\mu$ is induced by the map
$\mathrm{GL}_{W_n}\times \mathrm{GL}_{W_n} \xrightarrow{\otimes} \mathrm{GL}_{W_n}$
which is induced  by choosing a bijection $\mathbb{N}\times \mathbb{N} \rightarrow \mathbb{N}$ (see \cite[discussions around Lemma 2.32]{Gil81}).
In the diagram
\begin{gather*}
\scalebox{0.7}{
\xymatrix@C=0.6pc{
  & \mathrm{BGL}_{W_1}\times \mathrm{BGL}_{W_1} \ar[rr]^{\oplus} \ar'[d][dd] && \mathrm{BGL}_{W_1} \ar[dd] \\
   \mathrm{BGL}_{W_n}\times \mathrm{BGL}_{W_n} \ar[ru] \ar[rr]^{\oplus} \ar[dd] && \mathrm{BGL}_{W_n} \ar[dd] \ar[ru]  &  \\
  & \big(\mathbb{Z}\times \prod_{1\leq i<p}\Gamma_1(i)\big)\times\big(\mathbb{Z}\times \prod_{1\leq i<p}\Gamma_1(i)\big)  \ar'[r][rr]^<<<<<<<<<{\cdot} && \mathbb{Z}\times \prod_{1\leq i<p}\Gamma_1(i) \\
   \big(\mathbb{Z}\times \prod_{1\leq i<p}\Gamma_n(i)\big)\times \big(\mathbb{Z}\times \prod_{1\leq i<p}\Gamma_n(i)\big) \ar[ru]  \ar[rr]^<<<<<<<<<<<<<<<<<<<<<<<<<<<<<<<<<{\cdot} && \mathbb{Z}\times \prod_{1\leq i<p}\Gamma_n(i) \ar[ru] & 
}}
\end{gather*}
both the left and the right squares commute by the construction of the vertical arrows in the proof of Proposition \ref{prop:infinitesimalChernClass}, the top square commutes, and the bottom square homotopically commutes by the above discussion. The back square homotopically commutes by the Whitney sum formula of the construction of the Chern classes valued in Chow groups \cite{Pus04}.  A diagram chasing then shows that the composition of the front square with the arrow $\mathbb{Z}\times \prod_{1\leq i<p}\Gamma_n(i) \rightarrow \mathbb{Z}\times \prod_{1\leq i<p}\Gamma_1(i)$ homotopically commutes. Then by \eqref{eq:relativeDeRhamCoh-BGLtimesBGL-times}, the front square homotopically commutes. Applying similar arguments to the diagram
\begin{gather*}
\scalebox{0.6}{
\xymatrix@C=0.5pc{
  & \mathrm{BGL}_{W_1}\land \mathrm{BGL}_{W_1} \ar[rr]^{\mu} \ar'[d][dd] && \mathrm{BGL}_{W_1} \ar[dd] \\
   \mathrm{BGL}_{W_n}\land \mathrm{BGL}_{W_n} \ar[ru] \ar[rr]^{\mu} \ar[dd] && \mathrm{BGL}_{W_n} \ar[dd] \ar[ru]  &  \\
  & \big(\mathbb{Z}\times \prod_{1\leq i<p}\Gamma_1(i)\big)\land \big(\mathbb{Z}\times \prod_{1\leq i<p}\Gamma_1(i)\big)  \ar'[r][rr]^<<<<<<<<<{\star} && \mathbb{Z}\times \prod_{1\leq i<p}\Gamma_1(i) \\
   \big(\mathbb{Z}\times \prod_{1\leq i<p}\Gamma_n(i)\big)\land \big(\mathbb{Z}\times \prod_{1\leq i<p}\Gamma_n(i)\big) \ar[ru]  \ar[rr]^<<<<<<<<<<<<<<<<<<<<<<<<<<<<<<<<<{\star} && \mathbb{Z}\times \prod_{1\leq i<p}\Gamma_n(i) \ar[ru] & 
}}
\end{gather*}
and using \eqref{eq:relativeDeRhamCoh-BGLtimesBGL-land} show that the front square homotopically commutes. This completes the proof of (i). For  (ii) we need to show that the above homotopy commutativity can be chosen compatibly, namely, such that  the diagrams
\begin{gather*}
\xymatrix@C=2pc{
  \mathrm{BGL}_{W_{n+1}}\times \mathrm{BGL}_{W_{n+1}}  \ar[d]_{\text{“}\widetilde{C}\circ \oplus \Rightarrow \cdot \circ (\widetilde{C}\times \widetilde{C})"}  \ar[r]  & \mathrm{BGL}_{W_n}\times \mathrm{BGL}_{W_n}   \ar[d]^{\text{“}\widetilde{C}\circ \oplus \Rightarrow \cdot \circ (\widetilde{C}\times \widetilde{C})"}  \\
  \big(\mathbb{Z}\times \prod_{1\leq i<p}\Gamma_{n+1}(i)\big)^{\Delta^1} \ar[r] & \big(\mathbb{Z}\times \prod_{1\leq i<p}\Gamma_n(i)\big)^{\Delta^1}
}
\end{gather*}
and 
\begin{gather*}
\xymatrix@C=2pc{
  \mathrm{BGL}_{W_{n+1}}\land \mathrm{BGL}_{W_{n+1}}  \ar[d]_{\text{“}\widetilde{C}\circ \mu \Rightarrow \star \circ (\widetilde{C}\times \widetilde{C})"} 
   \ar[r]  & \mathrm{BGL}_{W_n}\land \mathrm{BGL}_{W_n}  \ar[d]^{\text{“}\widetilde{C}\circ \mu \Rightarrow \star \circ (\widetilde{C}\times \widetilde{C})"} \\
  \big(\mathbb{Z}\times \prod_{1\leq i<p}\Gamma_{n+1}(i)\big)^{\Delta^1} \ar[r] & \big(\mathbb{Z}\times \prod_{1\leq i<p}\Gamma_n(i)\big)^{\Delta^1}
}
\end{gather*}
commute.  This is obtained by induction on $n$ and by  using \eqref{eq:relativeDeRhamCoh-BGLtimesBGL-times} and \eqref{eq:relativeDeRhamCoh-BGLtimesBGL-land} respectively as in the last part of the proof of Proposition \ref{prop:infinitesimalChernClass}.
\end{proof}

\begin{proposition}\label{prop:c1-K1}
By the identification $\mathbb{Z}_{X_n}(1)\simeq \mathbb{G}_{m,X_n}[-1]$ in Proposition \ref{prop:properties-Zn(r)}(iv), the sheafified $c_1: \mathcal{K}_{X_n,1} \rightarrow \mathbb{Z}_{X_n}(1)[1]$ can be taken to be the identification $\mathcal{K}_1\cong \mathbb{G}_m$.
\end{proposition}
\begin{proof}
By our convention in \S\ref{sub:notations}(iv), the complex $ \mathbb{Z}_{X_n}(1)[2]\simeq \mathbb{G}_{m,X_n}[1]$ regarded as a chain complex is equivalent to the one concentrated in degree 1. By the construction of the usual Chern class, the $c_1$ is induced by the map of simplicial schemes $\mathrm{BGL}_{\Bbbk} \rightarrow \mathrm{K}(\mathbb{G}_{m,\Bbbk}[1])$ which by Dold-Kan corresponds to the map $\mathrm{GL}_{\Bbbk} \xrightarrow{\mathrm{det}}\mathbb{G}_{m,\Bbbk}$. By the construction of our infinitesimal Chern class in the proof of Proposition \ref{prop:infinitesimalChernClass}, the natural lifting $\mathrm{GL}_{W_n(\Bbbk)} \xrightarrow{\mathrm{det}}\mathbb{G}_{m,W_n(\Bbbk)}$  induces the infinitesimal $c_1$, and the proof is complete.
\end{proof}


\section{The infinitesmial Chern character isomorphism} 
\label{sec:an_infinitesmial_chern_character_isomorphism}
With the preparations in the preceding sections, we are now ready to state and show  the following main theorem of this paper.
\begin{theorem}\label{thm:relative-comparison-local}
Let $\Bbbk$ be a perfect field of characteristic $p$.  Let $X_{\centerdot}\in \operatorname{ob}\mathrm{Sm}_{W_{\centerdot}(\Bbbk)}$.
For $0\leq i\leq p-4$ (resp. $0\leq i\leq p-3$) and $n> m\geq 1$, the sheafified relative Chern character map
\begin{gather}\label{eq:relative-comparison-local}
\mathrm{ch}=\sum_{r=1}^i \mathrm{ch}_r:\mathcal{K}_{X_n,X_m,i}\rightarrow  \bigoplus_{r=1}^i \mathcal{H}^{2r-i-1}\big(p^{r,m}_{r,n}\Omega^{\bullet}_{X_{\centerdot}}\big)
\end{gather}
is an isomorphism (resp. an epimorphism) of Nisnevich sheaves on $X_1$.
\end{theorem}
By Proposition \ref{prop:relHC-relBoldHC-bounded-pExp-Isom}(iii) and Theorem \ref{thm:rel-HC-homotopyEquiv}, for $i\leq p-1$ there is a canonical isomorphism
\begin{gather*}
\bigoplus_{r=1}^i \mathcal{H}^{2r-i-1}\big(p^{r,m}_{r,n}\Omega^{\bullet}_{X_{\centerdot}}\big) \xrightarrow{\Psi}
\mathcal{HC}_{X_n,X_m,i-1}\quad.
\end{gather*}
We identify the two sides via this isomorphism and denote the composition $\Psi\circ \mathrm{ch}$ still by $\mathrm{ch}$. Then by the 5-lemma and the following commutative diagram, Theorem \ref{thm:relative-comparison-local} follows from the cases $n-m=1$.
\begin{gather*}
  \xymatrix@C=0.6pc{
  \mathcal{K}_{X_{m+1},X_m,i+1} \ar[r] \ar[d]^{\mathrm{ch}} &  
  \mathcal{K}_{X_n,X_{m+1},i} \ar[r] \ar[d]^{\mathrm{ch}} & 
  \mathcal{K}_{X_n,X_m,i} \ar[r] \ar[d]^{\mathrm{ch}} & 
  \mathcal{K}_{X_{m+1},X_m,i} \ar[d]^{\mathrm{ch}} \ar[r]  & \mathcal{K}_{X_{n},X_{m+1},i-1} \ar[d]^{\mathrm{ch}} \\
  \mathcal{HC}_{X_{m+1},X_m,i} \ar[r] & \mathcal{HC}_{X_n,X_{m+1},i-1} \ar[r] & \mathcal{HC}_{X_n,X_m,i-1} \ar[r] & \mathcal{HC}_{X_{m+1},X_m,i-1} \ar[r] & \mathcal{HC}_{X_n,X_{m+1},i-2}
  }
  \end{gather*}

There is a simple fact in homological algebra:
\begin{lemma}
Let $C \rightarrow D$ be a map of chain complexes of abelian groups which are bounded from below (namely, $C_i=D_i=0$ for $i< N$ for some $N \in \mathbb{Z}$). Let $n\in \mathbb{Z}$. Assume that the induced homomorphisms $H_i(C; \mathbb{Z}/p \mathbb{Z}) \rightarrow H_i(D; \mathbb{Z}/p \mathbb{Z})$ are
\begin{gather}\label{eq:condition-n-mod-p-homology}
\begin{cases}
\mbox{isomorphisms} & \mbox{for}\ i\leq n \\
\mbox{surjective} & \mbox{for}\ i\leq n+1
\end{cases}.
\end{gather}
Assume moreover that for  $i\leq n+1$, $H_i(C)$ and  $H_i(D)$ are $p^M$-torsion for some $M\in \mathbb{N}$. Then the homomorphisms $H_i(C) \rightarrow H_i(D)$ also satisfy \eqref{eq:condition-n-mod-p-homology}.
\end{lemma}
Since both the relative $K$-sheaves and the relative $\mathrm{HC}$-sheaves in consideration are $p$-primary of bounded exponents (for relative $K$ see \cite[Theorem A]{GeH11} and \cite[Theorem 2.25]{LaT19}; for relative HC this is Proposition \ref{prop:relHC-relBoldHC-bounded-pExp-Isom}(i)), the cases $n-m=1$ of Theorem \ref{thm:relative-comparison-local} follow from its slightly stronger mod $p$ version:
\begin{proposition}\label{prop:relative-comparison-m=n-1-modP}
Let $\Bbbk$ be a perfect field of characteristic $p$.  Let $X_{\centerdot}\in \operatorname{ob}\mathrm{Sm}_{W_{\centerdot}(\Bbbk)}$.
Then for $0\leq i\leq p-3$ and $n\geq 2$, the sheafified relative mod $p$ Chern character map
\begin{gather}\label{eq:relative-comparison-m=n-1-modP}
\mathrm{ch}=\sum_{r=1}^i \mathrm{ch}_r:(\mathcal{K}/p)_{X_{n},X_{n-1},i}\rightarrow 
\bigoplus_{r=1}^i \mathcal{H}^{2r-i-1}\big(p^{r,n-1}_{r,n}\Omega^{\bullet}_{X_{\centerdot}}\otimes^{\mathbf{L}}\mathbb{Z}/p\big)
\end{gather}
is an isomorphism.
\end{proposition}
 The proof of Proposition \ref{prop:relative-comparison-m=n-1-modP} occupies \S \ref{sub:local_frobenius_liftings}-\S \ref{sub:completion_of_the_proof_of_relative-comparison}.

\begin{remark}\label{rem:iso-surj-interval}
Ideally, we should pursue to show that \eqref{eq:relative-comparison-m=n-1-modP} is surjective if $i=p-2$, and thus  the isomorphism and the surjection ranges in Theorem \ref{thm:relative-comparison-local} would be improved by 1. However, our proof  does not achieve this.
\end{remark}

\begin{corollary}\label{cor:smAlg-relKTheory}
Let $\Bbbk$ be a perfect field of characteristic $p$. Let $n\in \mathbb{N}$.
\begin{enumerate}[(i)]
  \item Let $X_{\centerdot}$ be an object of $\operatorname{Sm}_{W_{\centerdot}(\Bbbk)}$ of relative dimension $d$. Then for $0\leq i\leq p-5-d$ and $n>m\geq 1$, the infinitesimal Chern character map
  \begin{gather*}
  \mathrm{ch}=\sum_{r=1}^{p-1} \mathrm{ch}_r:K_i(X_n,X_m)\rightarrow \bigoplus_{r=1}^{p-1} \mathbb{H}^{2r-i-1}\big(X_1,p^{r,m}_{r,n}\Omega^{\bullet}_{X_{\centerdot}}\big)
  \end{gather*}
  is an isomorphism.
  \item Let $A_n$ be a smooth $W_{n}(\Bbbk)$-algebra of relative dimension $d$.  Take a smooth $W(\Bbbk)$-algebra $A$ which lifts $A_n$. 
  Then for $i\leq p-5-d$, the Chern character induces an isomorphism
  \begin{align*}
&\mathrm{ch}:K_{i}\big(A_n,A_1\big)
\cong \bigoplus_{r=1}^{\lfloor \frac{i+1}{2}\rfloor}\frac{\{\omega\in p^{r} \Omega^{i+1-2r}_{A/W}| \mathrm{d}\omega\in p^{(r-1)n} \Omega^{i+2-2r}_{A/W}\}}{p^{rn} \Omega^{i+1-2r}_{A/W}+\mathrm{d}(p^{r+1} \Omega^{i-2r}_{A/W})}.
\end{align*}
\end{enumerate}
\end{corollary}
\begin{proof}
By Thomason-Trobaugh \cite[Theorem 10.3]{ThT90}, there is a  local-to-global spectral sequence $E_2^{i,j}=H^i_{\mathrm{Nis}}(X_1,\mathcal{K}_{X_n,X_m,j})\Longrightarrow K_{j-i}(X_n,X_m)$ with differentials $d_s:E_s^{i,j} \rightarrow E_s^{i+s,j+s-1}$. On the other hand, there is  a spectral sequence $E_2^{i,j}=H^i\big(X_1,\mathcal{H}^{j}(p^{r,m}_{r,n}\Omega^{\bullet}_{X_{\centerdot}})\big)\Longrightarrow \mathbb{H}^{i+j}\big(X_1,p^{r,m}_{r,n}\Omega^{\bullet}_{X_{\centerdot}}\big)$  with differentials $d_s:E_s^{i,j} \rightarrow E_s^{i+s,j-s+1}$. Let $\leftidx{^\prime}E_2^{i,j}=\bigoplus_{r}H^i\big(X_1,\mathcal{H}^{2r-j-1}(p^{r,m}_{r,n}\Omega^{\bullet}_{X_{\centerdot}})\big)$. Since both spectral sequences are of the Brown-Gersten type (\cite{BrG73}),
the sheafification of the relative Chern character defined in Proposition \ref{prop:definition-infinitesimal-ChernCharacter} induces a natural map $\mathrm{ch}:\{E_s^{i,j}\} \rightarrow \{\leftidx{^\prime}E_{s}^{i,j}\}$ between spectral sequences which converges to $\mathrm{ch}:K_{j-i}(X_n,X_m) \rightarrow \bigoplus_{r=1}^{p-1} \mathbb{H}^{2r-i+j-1}\big(X_1,p^{r,m}_{r,n}\Omega^{\bullet}_{X_{\centerdot}}\big)$.

Recall that $\mathrm{cd}_{\mathrm{Nis}}(X_1)\leq d$ (\cite[Theorem 1.32]{Nis89}).
Thus by  induction on $s$, and noting that the RHS of \eqref{eq:relative-comparison-local} is equal to the direct sum over all $1\leq r\leq p-1$, Theorem \ref{thm:relative-comparison-local} implies $E_s^{i,j} \xrightarrow{\cong} \leftidx{^\prime}E_{s}^{i,j}$ for $j-i\leq p-5-d$ (see also \cite[Lemma B.10]{BEK14}). Then (i) follows.
For (ii), by Elkik's lifting theorem, such a lifting $A$ exists. Thus (ii) follows from (i).
\end{proof}

\subsection{Local Frobenius liftings} 
\label{sub:local_frobenius_liftings}
The following result is extracted from the proof of \cite[Proposition 3.2]{LaZ04}. Recall that, since $\mathrm{char}(\Bbbk)=p$, for $n\geq 1$ there is a Frobenius ring map $F:W_n(\Bbbk) \rightarrow W_n(\Bbbk)$, $(a_0,\dots,a_{n-1})\mapsto (a_0^p,\dots,a_{n-1}^p)$.
\begin{proposition}\label{prop:zar-local-Frob-lift}
Let $X_{\centerdot}\in \mathrm{Sm}_{W_{\centerdot}}$. Then 
\emph{Zariski locally} on $X_{\centerdot}$ there exist Frobenius liftings. More precisely, 
for any point $x\in X_1$, there exist affine open subschemes $U_n\subset X_n$ containing $x$ for $n\geq 1$,
and morphisms $f_n:U_n \rightarrow U_{n}$, satisfying 
\begin{enumerate}[(i)]
  \item $U_{n-1}=U_n\times_{\operatorname{Spec} W_n}\operatorname{Spec}W_{n-1}$ as subschems of $X_{n-1}$,
  \item $f_1$ is the absolute Frobenius on $U_1$, and
  \item $f_n$ lifts $F:W_n \rightarrow W_{n}$ and the morphisms $f_n$ are compatible, namely, the diagrams
  \begin{gather*}
  \xymatrix{
    U_n \ar[r]^{f_n} \ar[d] & U_{n} \ar[d] \\
   \operatorname{Spec} W_n \ar[r]^{F} & \operatorname{Spec} W_n
  }\quad
  \xymatrix{
   {}
    \ar@{}[d]|-{\mbox{and}}\\
{}
  }\quad
  \xymatrix{
    U_n \ar[r] \ar[d]_{f_n} & U_{n+1} \ar[d]^{f_{n+1}} \\
    U_n \ar[r] & U_{n+1}
  }
  \end{gather*}
  are commutative.
\end{enumerate}
\end{proposition}
\begin{proof}
Let $U_1$ be an arbitrary affine open subset of $X_1$ containing $x$.  Shrink $U_1=\operatorname{Spec}B_1$ such that we have an étale homomorphism of $\rho_1:\Bbbk$-algebras $A_1=\Bbbk[T_1,\dots,T_d] \rightarrow B_1$. 
For $n\geq 2$, let $U_n$ be the open subscheme of $X_n$ whose underlying space coincides with $U_1$. By \cite[Chap I Cor. 5.1.10]{EGA I}, $U_n$'s are also affine. Let $B_n$ be a smooth $W_n$-algebra for  $n\geq 2$  such that  $U_n=\operatorname{Spec}B_n$ and $B_{n-1}=B_n\otimes_{W_n} W_{n-1}$. 
Let $A=W[T_1,\dots,T_d]$ and  $A_n=A\otimes_W W_n$. Then there is a commutative diagram
\begin{gather*}
\xymatrix{
\dots \ar[r] & A_n \ar[r] \ar[d]_{\rho_n} & A_{n-1} \ar[r] \ar[d]_{\rho_{n-1}} & \dots \ar[r] \ar[d] & A_2 \ar[r] \ar[d]_{\rho_2} & A_1 \ar[d]^{\rho_1} \\
\dots \ar[r] & B_n \ar[r]  & B_{n-1} \ar[r]  & \dots \ar[r]  & B_2 \ar[r]  & B_1
}
\end{gather*}
such that  every $\rho_n$ is an étale homomorphism of $W_n$-algebras. In fact, we can construct $\rho_n$ inductively: suppose we have defined  $\rho_{n-1}$ , then we choose $b_1^{(n)},\dots,b_d^{(n)}\in B_n$ such that $b_i^{(n)}$ maps to $\rho_{n-1}(T_i)\in B_{n-1}$, and define $\rho_n$ by $\rho_n(T_i)=b_i^{(n)}$; by the Jacobian criterion,
$\rho_n$ is étale because so is $\rho_1$.
 
Now we repeat the arguments of \cite[Page 284-285]{LaZ04}.
Let  $\phi_n:A_n \rightarrow A_n$ be the ring map which lifts  $F: W_n \rightarrow W_n$ such that $\phi_n(T_i)=T_i^p$. In particular, $\phi_1$ is the absolute Frobenius of $A_1$.
Let $B_n^*=B_n\otimes_{\rho_n,A_n,\phi_n}A_n$, and we regard $B_n^*$ as an $A_n$-algebra by $a\mapsto 1\otimes a$. Since $\phi_n$ lifts $\phi_1$, there is an isomorphism
\begin{gather*}
B_n^*\otimes_{A_n}A_1=B_n\otimes_{A_n,\phi_n}A_n \otimes_{A_n}A_1\cong B_n\otimes_{A_n}A_1 \otimes_{A_1,\mathrm{Frob}}A_1= B_1\otimes_{A_1, \mathrm{Frob}}A_1\ .
\end{gather*}
Since $B_1$ is étale over $A_1$, the relative Frobenius
\begin{gather*}
B_1\otimes_{A_1, \mathrm{Frob}}A_1 \xrightarrow{\cong} B_1,\quad
b\otimes a \mapsto b^p a
\end{gather*}
is an isomorphism (\cite[\S 1.2 Proposition 2]{Hou66}).
Thus $B_n^*$ is an étale $A_n$-algebra which lifts  the étale $A_1$-algebra $B_1$ with respect to the reduction map $A_n \rightarrow A_1$. By \cite[18.1.2]{EGA IV-4}, there is a unique isomorphism of $A_n$-algebras $g_n:B_n^*\cong B_n$, and in the diagram
\begin{gather*}
\xymatrix{
  B_n \ar[r]^{b\mapsto b\otimes 1} \ar[d] & B_n^* \ar[r]^{g_n} \ar[d] & B_n \ar[d] \\
  B_{n-1} \ar[r]^{b\mapsto b\otimes 1}  & B_{n-1}^* \ar[r]^{g_{n-1}} &  B_{n-1}
}
\end{gather*}
the square on the right commutes.
Let $\psi_n$ be the composition 
$B_n \rightarrow B_n^* \xrightarrow{g_n}B_n$, and let $f_n$ be the associated endomorphism of $U_n=\operatorname{Spec}B_n$. Then the conditions (ii) and (iii) are satisfied.
\end{proof}

\subsection{The mod \texorpdfstring{$p$}{p} product structure} 
\label{sub:the_mod_p_product_structure}
For $L>M\geq 1$, the sheafification of \eqref{eq:rel-HC-nonreduced-modP-components} yields a canonical decomposition
\begin{gather}\label{eq:rel-HC-nonreduced-modP-components-sheafified}
\mathcal{H}^i\big(p^{r,M}_{r,L}\Omega^{\bullet}_{X_{\centerdot}}[-1]\otimes_{\mathbb{Z}}^{\mathbf{L}}\mathbb{Z}/p\big)\cong\begin{cases}
\quad\quad\quad\ \mathcal{O}_{X_1},& i=0,\\
\Omega^{i-1}_{X_1}\oplus \Omega^{i}_{X_1},& 1\leq i\leq r-1,\\
\Omega^{r-1}_{X_1},& i=r.
\end{cases}
\end{gather}
In this subsection we study the multiplication
\begin{gather}\label{eq:c1-product-relativeMotivicComplex}
\mathcal{K}_j(X_L; \mathbb{Z}/p)\times  \mathcal{H}^i\big(p^{r,M}_{r,L}\Omega^{\bullet}_{X_{\centerdot}}[-1]\otimes_{\mathbb{Z}}^{\mathbf{L}}\mathbb{Z}/p\big) 
\rightarrow \mathcal{H}^{i+2-j}\big(p^{r+1,M}_{r+1,L}\Omega^{\bullet}_{X_{\centerdot}}[-1]\otimes_{\mathbb{Z}}^{\mathbf{L}}\mathbb{Z}/p\big)\\
(\alpha,\gamma) \mapsto c_1(\alpha)\cdot \gamma\nn
\end{gather}
for $j=1,2$,  
where the multiplication is induced by the product in $\bigoplus_{0\leq r<p}\mathbb{Z}_{X_L}(r)$ and the fiber sequence
\begin{gather*}
0 \rightarrow p^{r,M}_{r,L}\Omega^{\bullet}_{X_{\centerdot}}[-1] \rightarrow \mathbb{Z}_{X_L}(r) \rightarrow \mathbb{Z}_{X_M}(r) \rightarrow 0.
\end{gather*}
The problem is local, so by Proposition \ref{prop:zar-local-Frob-lift} we assume the existence of a Frobenius lifting on $X_{\centerdot}$. Then we take $D_{\centerdot}=X_{\centerdot}$, and thus
\begin{gather*}
J(r,n)\Omega_{D_{\centerdot}}^{\bullet}=J(r,n)\Omega_{X_{\centerdot}}^{\bullet}=p^{r,n}\Omega_{X_{\centerdot}}^{\bullet}\ ,\\
I(r)\Omega_{D_{\centerdot}}^{\bullet}=I(r)\Omega_{X_{\centerdot}}^{\bullet}=p^{r,1}\Omega_{X_{\centerdot}}^{\bullet}\ ,
\end{gather*}
and $f_r: I(r)\Omega_{X_{\centerdot}}^{\bullet} \rightarrow \Omega_{X_{\centerdot}}^{\bullet}$ can be described as follows. Let $x_i\in \mathcal{O}_{X_{\centerdot}}$ for $0\leq i\leq k$. Since $f \mod p$ is the absolute Frobenius on $X_1$, there exists unique $y_i\in \mathcal{O}_{X_{\centerdot}}$ such that $f_{\centerdot}^*(x_i)=x_i^p+p y_i$. Then
\begin{gather}\label{eq:formula-fr}
f_r(p^{r-k}x_0 \mathrm{d}x_1\cdots \mathrm{d}x_k)=(f_{\centerdot}^*x_0)(x_1^{p-1}\mathrm{d}x_1+ \mathrm{d}y_1)\cdots 
(x_k^{p-1}\mathrm{d}x_k+ \mathrm{d}y_k).
\end{gather}
In particular, since for any $a\in \mathcal{O}_{X_{\centerdot}}^{\times}$, there exists $b\in \mathcal{O}_{X_{\centerdot}}$ such that $\log (f(a)a^{-p})=pb$, we have
\begin{gather}\label{eq:formula-f1}
(f_1-1)\mathrm{d}\log a= \mathrm{d}b.
\end{gather}

In the following computations we use the representatives of mod $p$ classes and their products in Lemma \ref{lem:mod-p-product}.
The mod $p$ product in Lemma \ref{lem:mod-p-product} coincides with the one induced by the mod $p$ product via the Eilenberg-Mac Lane functor. Consequently, the mod $p$ Chern character map  preserves mod $p$ products.

By the isomorphism \eqref{eq:quasiIso-J(r,n)->I(r)-to-p(r)} we have
\begin{gather}\label{eq:quasiIso-J(r,n)->I(r)-to-p(r)-cohomology}
H^i\big(p(r)\Omega^{\bullet}_{X_n}[-1]\otimes^{\mathbf{L}}\mathbb{Z}/p \mathbb{Z}\big)=
H^i\Big(\mathrm{Cone}\big(J(r,n)\Omega_{X_{\centerdot}}^{\bullet}\rightarrow I(r)\Omega_{X_{\centerdot}}^{\bullet}\big)[-1]\otimes^{\mathbf{L}}\mathbb{Z}/p \mathbb{Z}\Big).
\end{gather}
Recall that
\begin{gather*}
\mathrm{Cone}\big(J(r,n)\Omega_{X_{\centerdot}}^{\bullet}\rightarrow I(r)\Omega_{X_{\centerdot}}^{\bullet}\big)[-1]^{i}
=J(r,n)\Omega_{X_{\centerdot}}^{i}\oplus I(r)\Omega_{X_{\centerdot}}^{i-1},
\end{gather*}
and noting the shift $[-1]$, the differential is given by
\begin{gather}\label{eq:quasiIso-J(r,n)->I(r)-to-p(r)-cone-differential}
\mathrm{d}(\alpha_{i},\alpha_{i-1})=(\mathrm{d}\alpha_{i},-\alpha_{i}-\mathrm{d} \alpha_{i-1}).
\end{gather}
By Lemma \ref{lem:mod-p-product} and \eqref{eq:quasiIso-J(r,n)->I(r)-to-p(r)-cone-differential}, an element in $H^i\Big(\mathrm{Cone}\big(J(r,n)\Omega_{X_{\centerdot}}^{\bullet}\rightarrow I(r)\Omega_{X_{\centerdot}}^{\bullet}\big)[-1]\otimes^{\mathbf{L}}\mathbb{Z}/p \mathbb{Z}\Big)$ is represented by
\begin{gather*}
(\alpha_{i},\alpha_{i-1})\oplus (\beta_{i+1},\beta_{i})\in J(r,n)\Omega_{X_{\centerdot}}^{i}\oplus I(r)\Omega_{X_{\centerdot}}^{i-1}\oplus J(r,n)\Omega_{X_{\centerdot}}^{i+1}\oplus I(r)\Omega_{X_{\centerdot}}^{i}
\end{gather*}
satisfying
\begin{gather*}
\mathrm{d} \alpha_{i}+p \beta_{i+1}=0,\ -\alpha_{i}-\mathrm{d} \alpha_{i-1}+p\beta_{i}=0,\ 
\mathrm{d} \beta_{i+1}=0,\ \beta_{i+1}+  \mathrm{d}\beta_{i}=0.
\end{gather*}
The solutions to this system of equations are generated by the following \eqref{eq:relative-motivic-complex-homology-mod-p-generators-1} and \eqref{eq:relative-motivic-complex-homology-mod-p-generators-2}.
\begin{lemma}\label{lem:relative-motivic-complex-homology-mod-p-generators}
For $\alpha\in  \Omega_{X_{\centerdot}}^i$, let $\overline{\alpha} \in \Omega_{X_1/\Bbbk}^i$ be $\alpha \mod p$. Then
via the isomorphism \eqref{eq:quasiIso-J(r,n)->I(r)-to-p(r)-cohomology},
\begin{enumerate}[(i)]
  \item for $1\leq i\leq r$ and  a local section $\alpha$ of $\Omega_{X_{\centerdot}}^{i-1}$, the  element
\begin{gather}\label{eq:relative-motivic-complex-homology-mod-p-generators-1}
(0,p^{r-i+1}\alpha)\oplus (0,p^{r-i}\mathrm{d} \alpha)
\end{gather}
represents a cohomology class on the left-hand side of \eqref{eq:rel-HC-nonreduced-modP-components-sheafified} (with $M=1$) and 
maps to the element $\overline{\alpha}\in\Omega_{X_1/\Bbbk}^{i-1}$ on the right-hand side of \eqref{eq:rel-HC-nonreduced-modP-components-sheafified}, and
\item for $0\leq i\leq r-1$ and a local section $\beta$ of $\Omega_{X_{\centerdot}}^{i}$, the element
\begin{gather}\label{eq:relative-motivic-complex-homology-mod-p-generators-2}
(p^{(r-i)n}\beta,0)\oplus (-  p^{(r-i)n-1}\mathrm{d}\beta, p^{(r-i)n-1}\beta)
\end{gather}
represents a cohomology class on the left-hand side of \eqref{eq:rel-HC-nonreduced-modP-components-sheafified} (with $M=1$) and 
maps to the element  $\overline{\beta}\in\Omega_{X_1/\Bbbk}^{i}$ on the right-hand side of \eqref{eq:rel-HC-nonreduced-modP-components-sheafified}.
\end{enumerate}
\end{lemma}
\begin{proof}
Identifying the resolutions \eqref{eq:quasiIso-J(r,n)->I(r)-to-p(r)} with \eqref{eq-infinitesimalMOtivicComplex-flatResolution}, 
the conclusion follows from Lemma \ref{lem:mod-p-representative-flatResolution}.
\end{proof}
Recall the definition \eqref{eq:infinitesimal-syntomicComplex-Nis} of $\mathfrak{S}_{X_{\centerdot}}(r,n)$ and the map $p(r)\Omega^{\bullet}_{X_{n}}[-1]\rightarrow \mathfrak{S}_{X_{\centerdot}}(r,n)$ in \eqref{eq:fundamentalTriangle-Nis}.

\begin{lemma}\label{lem:infinitesimal-syntomic-cohomology-mod-p-generators}
Let $n\geq 2$ be an integer.
\begin{enumerate}[(i)]
  \item If $1\leq i\leq r$ and $\alpha$ is a local section $\Omega_{X_{\centerdot}}^{i-1}$, then 
  \begin{gather*}
  \big(0,(1-f_r)p^{r-i+1}\alpha\big)\oplus\big(0,(1-f_r)p^{r-i}\mathrm{d}\alpha\big)
   \end{gather*}
  represents an element\footnote{We remind the reader that $f_r (p^{r-i+1}\alpha)$ and $f_r \mathrm{d}(p^{r-i}\alpha)$ are defined because $p^{r-i+1}\alpha \in p^{r-i+1}\Omega_{X_{\centerdot}}^{i-1}$ and $\mathrm{d}(p^{r-i}\alpha) \in p^{r-i}\Omega_{X_{\centerdot}}^{i}$, while $f_r(p^{r-i}\alpha)$ is not.} of $\mathcal{H}^{i}\big(\mathfrak{S}_{X_{\centerdot}}(r,n)\otimes^{\mathbf{L}}\mathbb{Z}/p\big)$ and is the image of \eqref{eq:relative-motivic-complex-homology-mod-p-generators-1}.
  \item If $0\leq i\leq r-1$ and  $\beta$ is a  local section of $\Omega_{X_{\centerdot}}^{i}$, then 
   \begin{gather*}
   \big(p^{(r-i)n}\beta,0\big)\oplus\big(-p^{(r-i)n-1}\mathrm{d}\beta,(1-f_r)p^{(r-i)n-1}\beta\big)
   \end{gather*} 
  represents an element of $\mathcal{H}^i\big(\mathfrak{S}_{X_{\centerdot}}(r,n)\otimes^{\mathbf{L}}\mathbb{Z}/p\big)$ and is the image of \eqref{eq:relative-motivic-complex-homology-mod-p-generators-2}.
\end{enumerate}
\end{lemma}
\begin{proof}
This follows from the commutative diagram
\begin{gather*}
\xymatrix{
 & J(r,n)\Omega_{X_{\centerdot}}^{\bullet} \ar[r] \ar@{=}[d] & I(r)\Omega_{X_{\centerdot}}^{\bullet} \ar[d]^{1-f_r} \\
 & J(r,n)\Omega_{X_{\centerdot}}^{\bullet} \ar[r]^<<<<<{1-f_r} & \Omega_{X_{\centerdot}}^{\bullet} & .
}
\end{gather*}
\end{proof}

\begin{lemma}\label{lem:Kato-product-mod-p}
Let $n\geq 2$ and $r,s\geq 1$ be integers.  Let  $0\leq i\leq r-1$ be an integer, and $\gamma$ be a local section of $p^{(r-i)n-1}\Omega_{X_{\centerdot}}^{i}$.
\begin{enumerate}[(i)]
  \item If $1\leq j\leq s$ is an integer and $\alpha$ is a local section of $p^{s-j}\Omega^{j-1}_{X_{\centerdot}}$, then  
\begin{gather*}
\Big(\big(p\gamma,0\big)\oplus \big(-\mathrm{d}\gamma,(1-f_{r})\gamma\big)\Big)\cdot\Big(\big(0,(1-f_s)p\alpha\big)\oplus \big(0,(1-f_s)\mathrm{d}\alpha\big)\Big)=0
\end{gather*}
in $\mathcal{H}^{i+j}\big(\mathfrak{S}_{X_{\centerdot}}(r+s,n)\otimes^{\mathbf{L}}\mathbb{Z}/p\big)$.
  \item If  $0\leq j\leq s-1$ and $\beta$ is a local section of $p^{(s-j)n-1}\Omega_{X_{\centerdot}}^{j}$, then
\begin{gather*}
\Big(\big(p\gamma,0\big)\oplus\big(-\mathrm{d}\gamma,(1-f_r)\gamma\big)\Big)\cdot \Big(\big(p\beta,0\big)\oplus\big(-\mathrm{d}\beta,(1-f_{s})\beta\big)\Big)\\
= \big(p^2 \gamma\beta,0\big)\oplus \big(-p \mathrm{d}(\gamma \beta),p(1-f_{s+r})(\gamma\beta)\big)
\end{gather*}
in $\mathcal{H}^{i+j}\big(\mathfrak{S}_{X_{\centerdot}}(r+s,n)\otimes^{\mathbf{L}}\mathbb{Z}/p\big)$, 
\end{enumerate}
\end{lemma}
\begin{proof}
By \eqref{eq:mod-p-product} and \eqref{eq:Kato-product} and by using $\mathrm{d}\circ f_r=f_r\circ \mathrm{d}$, we have
\begin{align*}
&\Big(\big(p\gamma,0\big)\oplus \big(-\mathrm{d}\gamma,(1-f_{r})\gamma\big)\Big)\cdot\Big(\big(0,(1-f_s)p\alpha\big)\oplus \big(0,(1-f_s)\mathrm{d}\alpha\big)\Big)\\
\overset{\eqref{eq:mod-p-product}}{=}{}&
\Big(\big(p\gamma,0\big)\cdot \big(0,(1-f_s)p\alpha\big)\Big)\oplus\Big(\big(-\mathrm{d}\gamma,(1-f_{r})\gamma\big)\cdot \big(0,(1-f_s)p\alpha\big)
+(-1)^{i}(p\gamma,0)\cdot \big(0,(1-f_s)\mathrm{d}\alpha\big)
\Big)\\
\overset{\eqref{eq:Kato-product}}{=}{}&\big(0,(-1)^{i}f_r(p\gamma)(1-f_s)p\alpha\big)\oplus\Big(\big(0,(-1)^{i}f_r(\mathrm{d}\gamma)\cdot(1-f_s)p\alpha\big)
+(-1)^{i}(0,(-1)^{i}f_r(p\gamma)\cdot (1-f_s)\mathrm{d}\alpha)
\Big)\\
={}&\big(0,(-1)^{i}f_r(p\gamma)\cdot (1-f_s)p\alpha\big)\oplus\big(0,(-1)^{i}\mathrm{d}f_r\gamma\cdot (1-f_s)p\alpha+pf_r\gamma\cdot (1-f_s)\mathrm{d}\alpha\big)\\
={}&\big(0,(-1)^{i}pf_r\gamma\cdot (1-f_s)p\alpha\big)\oplus\Big(0,(-1)^{i}\mathrm{d}\big(f_r\gamma\cdot (1-f_s)p\alpha\big)\Big).
\end{align*}
But by \eqref{eq:Kato-differential} and \eqref{eq:mod-p-differential}, we have
\begin{align*}
&\big(0,pf_r\gamma\cdot (1-f_s)p\alpha\big)\oplus\Big(0,\mathrm{d}\big(f_r\gamma\cdot (1-f_s)p\alpha\big)\Big)\\
\overset{\eqref{eq:Kato-differential}}{=}{}&\big(0,pf_r\gamma\cdot (1-f_s)p\alpha\big)\oplus \mathrm{d}\big(0,-f_r\gamma\cdot (1-f_s)p\alpha\big)\\
\overset{\eqref{eq:mod-p-differential}}{=}{}&\mathrm{d}\Big(\big(0,0)\oplus \big(0,f_r\gamma\cdot (1-f_s)p\alpha\big)\Big),
\end{align*}
which completes the proof of (i).
By \eqref{eq:mod-p-product} and \eqref{eq:Kato-product} again, we have
\begin{align*}
& \big((p\gamma,0)\oplus(-\mathrm{d}\gamma,(1-f_r)\gamma)\big)\cdot \big((p\beta,0)\oplus(-\mathrm{d}\beta,(1-f_{s})\beta)\big)\\
\overset{\eqref{eq:mod-p-product}}{=}{}& (p^2 \gamma\beta,0)\oplus\big((-\mathrm{d}\gamma,(1-f_r)\gamma)\cdot (p\beta,0)+(-1)^j(p\gamma,0)\cdot (-\mathrm{d}\beta,(1-f_{s})\beta)\big)\\
\overset{\eqref{eq:Kato-product}}{=}{}& (p^2 \gamma\beta,0)\oplus\Big(\big(-\mathrm{d}\gamma\cdot p\beta,(1-f_r)\gamma\cdot p\beta\big)+(-1)^j\big(-p\gamma \mathrm{d}\beta,(-1)^j f_r(p\gamma)(1-f_{s})\beta\big)\Big)\\
={}& (p^2 \gamma\beta,0)\oplus (-p \mathrm{d}(\gamma \beta),p\gamma\beta-pf_r\gamma\cdot f_{s}\beta)\\
={}& (p^2 \gamma\beta,0)\oplus \big(-p \mathrm{d}(\gamma \beta),p(1-f_{r+s})(\gamma\beta)\big)
\end{align*}
which is (ii).
\end{proof}

\begin{lemma}\label{lem:KatoRep-c1(a)}
The image of $a\in \mathcal{O}_{X_{\centerdot}}^{\times}$ via the composition $$\mathcal{O}_{X_{\centerdot}}^{\times}=\mathcal{K}_1(X_{\centerdot}) \xrightarrow{c_1} \mathcal{H}^1(\mathbb{Z}_{X_{\centerdot}}(1)) \rightarrow \mathcal{H}^1(\mathfrak{S}_{X_{\centerdot}}(1))$$ is represented by $\big(\mathrm{d}\log a,p^{-1}\log (f(a)a^{-p})\big)$.
\end{lemma}
\begin{proof}
By Proposition \ref{prop:c1-K1}, $c_1$ is given by the identification $\mathcal{K}_1(X_{\centerdot})\cong \mathcal{O}_{X_{\centerdot}}^{\times}\cong \mathcal{H}^1(\mathbb{Z}_{X_{\centerdot}}(1))$. By the proof of Proposition \ref{prop:properties-Zn(r)}(v), the composition of this identification with the map $\mathcal{H}^1(\mathbb{Z}_{X_{\centerdot}}(1)) \rightarrow \mathcal{H}^1(\mathfrak{S}_{X_{\centerdot}}(1))$ is given by the formula (\cite[Page 216]{Kat87}) of Kato's symbol\footnote{
See Remark \ref{rem:Kato-BEK-sign}.}:
\begin{gather*}
\mathcal{K}^M_{X_{\centerdot},1} \rightarrow   \mathcal{H}^1(\mathfrak{S}_{X_{\centerdot}}(1)),\quad
a\mapsto \big(\mathrm{d}\log a,p^{-1}\log (f(a)a^{-p})\big).
\end{gather*} 
\end{proof}

\begin{lemma}\label{exa:Bott-element-syntomic-complex}
Let $n\geq 2$ be an integer.
Let $\zeta=1+p^{n-1}\in W_n(\Bbbk)$. Let $A_n$ be a local ring of $X_n$, and let $\beta_{\zeta}$ be  the Bott element of $K_2(A_n;\mathbb{Z}/p \mathbb{Z})$ associated with $\zeta$. Then $c_1(\beta_{\zeta})\in H^0(\mathfrak{S}_{X_{\centerdot}}(1,n)\otimes^{\mathbf{L}}\mathbb{Z}/p \mathbb{Z})$  is represented by
\begin{gather*}
\big(p^n,0\big)\oplus \big(0,(1-f_1)p^{n-1}\big)\\
=(p^n,0)\oplus (0,p^{n-1}-p^{n-2}).
\end{gather*}
\end{lemma}
\begin{proof}
We have a commutative diagram
\begin{gather*}
\xymatrix{
  0 \ar[r] & K_2(A_n)/p \ar[r] \ar[d]_{c_1} & K_2(A; \mathbb{Z}/p) \ar[r] \ar[d]_{c_1} & \leftidx{_p}K_1(A_n) \ar[r] \ar[d]_{c_1} & 0 \\
  0 \ar[r] & H^0(\mathbb{Z}_{A_n}(1))/p \ar[r] & H^0(\mathbb{Z}_{A_n}(1)\otimes^{\mathbf{L}}\mathbb{Z}/p) \ar[r] & \leftidx{_p}H^1(\mathbb{Z}_{A_n}(1)) \ar[r] & 0.
}
\end{gather*}
By Proposition \ref{prop:properties-Zn(r)}(iv), $H^0(\mathbb{Z}_{A_n}(1))=0$, and $H^1(\mathbb{Z}_{A_n}(1))=A_n^{\times}$. The Chern class map $c_1$ identifies $K_1(A_n)=A_n^{\times}=H^1(\mathbb{Z}_{A_n}(1))$. It suffices to compute the image of $\zeta$ under the composition of maps
\begin{gather*}
\leftidx{_p}H^1(\mathbb{Z}_{A_n}(1)) \cong H^0(\mathbb{Z}_{A_n}(1)\otimes^{\mathbf{L}}\mathbb{Z}/p) \rightarrow 
H^0(\mathfrak{S}_{X_{\centerdot}}(1,n)\otimes^{\mathbf{L}}\mathbb{Z}/p \mathbb{Z})\quad.
\end{gather*}
 By the commutative diagrams and the map $\mathrm{Exp}$ in the proof of Proposition \ref{prop:properties-Zn(r)} in the case $r=1$, $c_1(\beta_{\zeta})$ 
 is the image of  $1\in \mathcal{O}_{X_1}$ via the case $i=0$ of  the isomorphism \eqref{eq:rel-HC-nonreduced-modP-components-sheafified}. Then 
 the conclusion follows from
 Lemma \ref{lem:relative-motivic-complex-homology-mod-p-generators}(ii) and Lemma \ref{lem:infinitesimal-syntomic-cohomology-mod-p-generators}(ii).
\end{proof}

\begin{lemma}\label{lem:sheafCoh-p(r)-to-syntomic-mono}
For any $i\geq 0$, the homomorphism 
\begin{gather*}
\mathcal{H}^i\big(p(r)\Omega_{X_{n}}^{\bullet}[-1]\otimes^{\mathbf{L}}\mathbb{Z}/p\big) \rightarrow \mathcal{H}^i\big(\mathfrak{G}_{X_{\centerdot}}(r,n)\otimes^{\mathbf{L}}\mathbb{Z}/p\big)
\end{gather*}
is a monomorphism.
\end{lemma}
\begin{proof}
There is a long exact sequence
\begin{gather*}
 \dots \rightarrow \mathcal{H}^i\big(p(r)\Omega_{X_{n}}^{\bullet}[-1]\otimes^{\mathbf{L}}\mathbb{Z}/p\big) \rightarrow \mathcal{H}^i\big(\mathfrak{G}_{X_{\centerdot}}(r,n)\otimes^{\mathbf{L}}\mathbb{Z}/p\big) \rightarrow \mathcal{H}^i\big(W_{\centerdot}\Omega^r_{X_1,\log}[-r]\otimes^{\mathbf{L}}\mathbb{Z}/p\big)\\
  \rightarrow \mathcal{H}^{i+1}\big(p(r)\Omega_{X_{n}}^{\bullet}[-1]\otimes^{\mathbf{L}}\mathbb{Z}/p\big) \rightarrow \cdots
 \end{gather*} 
Since $W_{\centerdot}\Omega^r_{X_1,\log}$ is a subsheaf of $W_{\centerdot}\Omega^r_{X_1}$, it has no $p$-torsion. Thus $\mathcal{H}^i\big(W_{\centerdot}\Omega^r_{X_1,\log}[-r]\otimes^{\mathbf{L}}\mathbb{Z}/p\big)$ is zero if $i\neq r$ and is equal to $W_{\centerdot}\Omega^r_{X_1,\log}\otimes\mathbb{Z}/p$ if $i=r$. By Proposition \ref{prop:properties-Zn(r)}(v), both homomorphisms in
\begin{gather*}
\mathcal{H}^r\big(\mathbb{Z}_{X_n}(r)\big) \rightarrow \mathcal{H}^r\big(\mathbb{Z}_{X_1}(r)\big) \rightarrow  \mathcal{H}^r\big(W_{\centerdot}\Omega^r_{X_1,\log}[-r]\big)
\end{gather*}
are epimorphisms. Then by a diagram chase on \eqref{eq:motivicFundamentalTriangle-to-syntomicFundamentalTriangle}, the map $\mathcal{H}^r\big(\mathfrak{G}_{X_{\centerdot}}(r,n)\big) \rightarrow \mathcal{H}^r\big(W_{\centerdot}\Omega^r_{X_1,\log}[-r]\big)$, and thus 
$\mathcal{H}^r\big(\mathfrak{G}_{X_{\centerdot}}(r,n)\otimes^{\mathbf{L}}\mathbb{Z}/p\big) \rightarrow \mathcal{H}^r\big(W_{\centerdot}\Omega^r_{X_1,\log}[-r]\otimes^{\mathbf{L}}\mathbb{Z}/p\big)$, are epimorphisms.  Hence the conclusion follows.
\end{proof}

\begin{proposition}\label{prop:K1-multiplication-motivic-complex}
Let $L>M\geq 1$ be integers. Let $i,r$ be integers and $0\leq i\leq r\leq p-2$.
The multiplication 
\begin{gather*}
\mathcal{K}_1(X_L)\times  \mathcal{H}^i\big(p^{r,M}_{r,L}\Omega^{\bullet}_{X_{\centerdot}}[-1]\otimes_{\mathbb{Z}}^{\mathbf{L}}\mathbb{Z}/p\big) 
\rightarrow \mathcal{H}^{i+1}\big(p^{r+1,M}_{r+1,L}\Omega^{\bullet}_{X_{\centerdot}}[-1]\otimes_{\mathbb{Z}}^{\mathbf{L}}\mathbb{Z}/p\big)
\end{gather*}
which is the composition $\mathcal{K}_1(X_L) \rightarrow \mathcal{K}_1(X_L;\mathbb{Z}/p)$ and \eqref{eq:c1-product-relativeMotivicComplex}, preserves the decomposition \eqref{eq:rel-HC-nonreduced-modP-components-sheafified}, and the multiplication on the summands is given by
\begin{gather*}
\xymatrix@R=1pc@C=4pc{
 \mathcal{H}^i\big(p^{r,M}_{r,L}\Omega^{\bullet}_{X_{\centerdot}}[-1]\otimes_{\mathbb{Z}}^{\mathbf{L}}\mathbb{Z}/p\big)  \ar[r]^<<<<<<<<{c_1(a)\times} & \mathcal{H}^{i+1}\big(p^{r+1,M}_{r+1,L}\Omega^{\bullet}_{X_{\centerdot}}[-1]\otimes_{\mathbb{Z}}^{\mathbf{L}}\mathbb{Z}/p\big) \\
  \Omega_{X_1/\Bbbk}^{i-1} \ar[r]^{-(\mathrm{d}\log \bar{a})\land} \ar@{}[u]|-{\rotatebox{90}{$\subset$}} &  
  \Omega_{X_1/\Bbbk}^{i} \ar@{}[u]|-{\rotatebox{90}{$\subset$}}
}
\end{gather*}
and
\begin{gather*}
\xymatrix@R=1pc@C=4pc{
 \mathcal{H}^i\big(p^{r,M}_{r,L}\Omega^{\bullet}_{X_{\centerdot}}[-1]\otimes_{\mathbb{Z}}^{\mathbf{L}}\mathbb{Z}/p\big)  \ar[r]^<<<<<<<<{c_1(a)\times} & \mathcal{H}^{i+1}\big(p^{r+1,M}_{r+1,L}\Omega^{\bullet}_{X_{\centerdot}}[-1]\otimes_{\mathbb{Z}}^{\mathbf{L}}\mathbb{Z}/p\big) \\
  \Omega_{X_1/\Bbbk}^{i} \ar[r]^{(\mathrm{d}\log \bar{a})\land} \ar@{}[u]|-{\rotatebox{90}{$\subset$}} &  
  \Omega_{X_1/\Bbbk}^{i+1} \ar@{}[u]|-{\rotatebox{90}{$\subset$}}
}
\end{gather*}
where $a\in \mathcal{O}_{X_L}^{\times}=\mathcal{K}_1(\mathcal{O}_{X_L})$ and $\bar{a}=a \mod p\in \mathcal{O}_{X_1}^{\times}$.
\end{proposition}
\begin{proof}
First we consider the case $M=1$.
By Proposition \ref{prop:definition-infinitesimal-ChernCharacter}, we can compute the multiplication by $K$-theory by the multiplication as the first row of the diagram
\begin{gather*}
\xymatrix{
\mathbb{Z}_{X_L}(s)\times  p^{r,1}_{r,L}\Omega_{X_{\centerdot}}^{\bullet}[-1] \ar[r] \ar[d] & p^{r+s,1}_{r+s,L}\Omega_{X_{\centerdot}}^{\bullet}[-1] \ar[d]\\
\mathfrak{G}_{X_{\centerdot}}(s,L)\times p^{r,1}_{r,L}\Omega_{X_{\centerdot}}^{\bullet}[-1] \ar[r] & p^{r+s,1}_{r+s,L}\Omega_{X_{\centerdot}}^{\bullet}[-1].
}
\end{gather*}
By Proposition \ref{prop:infinitesimalMotComp-product}(iv), this diagram is commutative. Namely, we can compute the multiplication of $\mathbb{Z}_{X_L}(s)$ on $p^{r,1}_{r,L}\Omega_{X_{\centerdot}}^{\bullet}$ by that of $\mathfrak{G}_{X_{\centerdot}}(s,L)$ on $p^{r,1}_{r,L}\Omega_{X_{\centerdot}}^{\bullet}$. The same holds for the products mod $p$.
By Lemma \ref{lem:sheafCoh-p(r)-to-syntomic-mono}, we can compute the multiplication 
\begin{gather*}
\mathcal{H}^j\big(\mathfrak{G}_{X_{\centerdot}}(s,L)\otimes^{\mathbf{L}}\mathbb{Z}/p\big) \times \mathcal{H}^i\big(p^{r,1}_{r,L}\Omega_{X_{\centerdot}}^{\bullet}[-1]\otimes^{\mathbf{L}}\mathbb{Z}/p\big)
\rightarrow \mathcal{H}^{i+j}\big(p^{r+s,1}_{r+s,L}\Omega_{X_{\centerdot}}^{\bullet}[-1]\otimes^{\mathbf{L}}\mathbb{Z}/p\big)
\end{gather*}
by computing the multiplication of $\mathcal{H}^j\big(\mathfrak{G}_{X_{\centerdot}}(s,\centerdot)\otimes^{\mathbf{L}}\mathbb{Z}/p\big)$ on the image of 
\begin{gather*}
\mathcal{H}^i\big(p^{r,1}_{r,L}\Omega_{X_{\centerdot}}^{\bullet}[-1]\otimes^{\mathbf{L}}\mathbb{Z}/p\big) \rightarrow 
\mathcal{H}^i\big(\mathfrak{G}_{X_{\centerdot}}(r,L)\otimes^{\mathbf{L}}\mathbb{Z}/p\big).
\end{gather*}

In the case $M=1$, the elements  in the decomposition on the RHS of \eqref{eq:rel-HC-nonreduced-modP-components-sheafified} are represented by the elements in Lemma \ref{lem:relative-motivic-complex-homology-mod-p-generators}, whose images in $\mathcal{H}^i\big(\mathfrak{G}_{X_{\centerdot}}(r,L)\otimes^{\mathbf{L}}\mathbb{Z}/p\big)$ are given in Lemma \ref{lem:infinitesimal-syntomic-cohomology-mod-p-generators}.
Then by  \eqref{eq:mod-p-product},  \eqref{eq:Kato-product}, and Lemma \ref{lem:KatoRep-c1(a)}, we  do the following computations: for $\alpha\in p^{r-i}\Omega^{i-1}_{X_{\centerdot}}$,
\begin{align*}
& \Big(\big(\mathrm{d}\log a,p^{-1}\log (f(a)a^{-p})\big)\oplus\big(0,0\big)\Big)\cdot \Big(\big(0,(1-f_r)p\alpha\big)\oplus\big(0,(1-f_r)\mathrm{d}\alpha\big)\Big)\\
\overset{\eqref{eq:mod-p-product}}{=}{}&
\Big(\big(\mathrm{d}\log a,p^{-1}\log (f(a)a^{-p})\big)\cdot \big(0,(1-f_r)p\alpha\big)\Big)\oplus
\Big(-\big(\mathrm{d}\log a,p^{-1}\log (f(a)a^{-p})\big)\cdot\big(0,(1-f_r)\mathrm{d}\alpha\big)\Big)\\
\overset{\eqref{eq:Kato-product}}{=}{}&
\big(0,-f_1\mathrm{d}\log a\cdot(1-f_r)p\alpha\big)\oplus \big(0,f_1\mathrm{d}\log a\cdot(1-f_r)\mathrm{d}\alpha\big)\\
={}&\big(0,-\mathrm{d}\log a\cdot p\alpha+f_1\mathrm{d}\log a\cdot f_r(p\alpha)\big)\oplus
 \big(0,\mathrm{d}\log a \cdot\mathrm{d}\alpha- f_1\mathrm{d}\log a\cdot  f_r\mathrm{d}\alpha\big)\\
 &+\big(0,(1-f_1)\mathrm{d}\log a\cdot p\alpha\big)\oplus
 \big(0,(f_1- 1)\mathrm{d}\log a\cdot  \mathrm{d}\alpha\big)\\
={}&\big(0,(1-f_{r+1})(-p\mathrm{d}\log a\cdot \alpha)\big)\oplus
 \big(0,(1-f_{r+1})\mathrm{d}(-\mathrm{d}\log a\cdot \alpha)\big)\\
 &+\uwave{\big(0,(1-f_1)\mathrm{d}\log a\cdot p\alpha\big)\oplus
 \big(0,(f_1-1)\mathrm{d}\log a\cdot  \mathrm{d}\alpha\big)},
\end{align*}
and for $\beta\in p^{(r-i)n-1}\Omega^i_{X_{\centerdot}}$,
\begin{align*}
& \Big(\big(\mathrm{d}\log a,p^{-1}\log (f(a)a^{-p})\big)\oplus\big(0,0\big)\Big)\cdot \Big(\big(p\beta,0\big)\oplus \big(-\mathrm{d}\beta,(1-f_r)\beta\big)\Big)\\
\overset{\eqref{eq:mod-p-product}}{=}{}& \Big(\big(\mathrm{d}\log a,p^{-1}\log (f(a)a^{-p})\big)\cdot \big(p\beta,0\big)\Big)\oplus
\Big(-\big(\mathrm{d}\log a,p^{-1}\log (f(a)a^{-p})\big)\cdot\big(-\mathrm{d}\beta,(1-f_{r})\beta\big)\Big)\\
\overset{\eqref{eq:Kato-product}}{=}{}&
\Big(p\mathrm{d}\log a\cdot \beta,p^{-1}\log \big(f(a)a^{-p}\big)\cdot p\beta\Big)\oplus\\
&\Big(\mathrm{d}\log a \cdot \mathrm{d}\beta,f_1\mathrm{d}\log a\cdot(1-f_{r})\beta+p^{-1}\log\big(f(a)a^{-p}\big)\cdot \mathrm{d}\beta\Big)\\
={}&\big(p\mathrm{d}\log a\cdot \beta,0\big)\oplus\big(\mathrm{d}\log a \cdot \mathrm{d}\beta,\mathrm{d}\log a\cdot \beta -f_1 \mathrm{d}\log a\cdot f_r\beta\big)\\
&+\Big(0,\log\big(f(a)a^{-p}\big)\cdot \beta\Big)\oplus\Big(0,f_1\mathrm{d}\log a\cdot \beta-\mathrm{d}\log a\cdot \beta+p^{-1}\log\big(f(a)a^{-p}\big)\cdot \mathrm{d}\beta\Big)\\
={}&\big(p\mathrm{d}\log a\cdot \beta,0\big)\oplus
\big(-\mathrm{d}(\mathrm{d}\log a \cdot \beta),(1-f_{r+1})(\mathrm{d}\log a\cdot \beta)\big)\\
&+\uwave{\Big(0,\log\big(f(a)a^{-p}\big)\beta\Big)\oplus\Big(0,p^{-1}\mathrm{d}\big(\log (f(a)a^{-p})\beta\big)\Big)}.
\end{align*}
But there exists $b\in \mathcal{O}_{X_{\centerdot}}$ such that $\log (f(a)a^{-p})=pb$, therefore by \eqref{eq:formula-f1}, \eqref{eq:Kato-differential}, and \eqref{eq:mod-p-differential} we have
\begin{gather*}
\big(0,(1-f_1)\mathrm{d}\log a\cdot p\alpha\big)\oplus
 \big(0,(f_1- 1)\mathrm{d}\log a\cdot  \mathrm{d}\alpha\big)
 \overset{\eqref{eq:formula-f1}}{=}
 \big(0,-\mathrm{d}b\cdot p\alpha\big)\oplus
 \big(0,\mathrm{d}b\cdot  \mathrm{d}\alpha\big)\\
 \overset{\eqref{eq:Kato-differential}}{=}
 \big(0,-\mathrm{d}b\cdot p\alpha\big)\oplus
 \mathrm{d}\big(0,\mathrm{d}b\cdot \alpha\big)
 \overset{\eqref{eq:mod-p-differential}}{=}
 \mathrm{d}\big((0,0)\oplus(0,-\mathrm{d}\mathrm{b}\cdot \alpha)\big)
\end{gather*}
 and
\begin{gather*}
\Big(0,\log\big(f(a)a^{-p}\big)\beta\Big)\oplus\Big(0,p^{-1}\mathrm{d}\big(\log (f(a)a^{-p})\beta\big)\Big)
 \overset{\eqref{eq:formula-f1}}{=}
\big(0,pb\beta\big)\oplus\big(0, \mathrm{d}(bf_r\beta)\big)\\
\overset{\eqref{eq:Kato-differential}}{=}
(0,pb\beta)\oplus \mathrm{d}(0,-b\beta)
\overset{\eqref{eq:mod-p-differential}}{=}\mathrm{d}\big((0,0)\oplus(0,b\beta)\big).
\end{gather*}
Thus the underwaved terms in the above are 0 in $\mathcal{H}^{i+1}\big(\mathfrak{G}_{X_{\centerdot}}(r+1,L)\otimes^{\mathbf{L}}\mathbb{Z}/p\big)$. This completes the proof for the case $M=1$.
Now suppose $M\geq 2$. There is a map of exact triangles in $\mathrm{D}_{\mathrm{Nis}}(X_1)$:
\begin{equation}\label{eq:map-triangles-ZXLM-prML}
\begin{gathered}
\xymatrix{
  p^{r,M}_{r,L}\Omega_{X_{\centerdot}}^{\bullet}[-1] \ar[r] \ar@{=}[d] & \mathbb{Z}_{X_L}(r) \ar[r]  & \mathbb{Z}_{X_M}(r) \ar[r]  &  p^{r,M}_{r,L}\Omega_{X_{\centerdot}}^{\bullet} \ar@{=}[d]\\
  p^{r,M}_{r,L}\Omega_{X_{\centerdot}}^{\bullet}[-1] \ar[r]  & p^{r,1}_{r,L}\Omega_{X_{\centerdot}}^{\bullet}[-1]   \ar[r] \ar[u] \ar@{=}[d] & p^{r,1}_{r,M}\Omega_{X_{\centerdot}}^{\bullet}[-1]  \ar[r] \ar[u] \ar@{=}[d] & p^{r,M}_{r,L}\Omega_{X_{\centerdot}}^{\bullet} \\
  & p(r)\Omega_{X_{L}}^{\bullet}[-1] & p(r)\Omega_{X_{M}}^{\bullet}[-1]
  }
\end{gathered}
\end{equation}
Recall  the resolution \eqref{eq-infinitesimalMOtivicComplex-flatResolution}.
Applying $\mathcal{H}^i(?\otimes^{\mathbf{L}}\mathbb{Z}/p)$ to the row \eqref{eq:map-triangles-ZXLM-prML}, in terms of the decomposition \eqref{eq:rel-HC-nonreduced-modP-components-sheafified} the induced homomorphisms are given by
\begin{gather*}
\xymatrix{
  \Omega^{i-1}_{X_1}\oplus \Omega^{i}_{X_1} \ar[r]^{\begin{pmatrix} 0 & 0 \\ 0 & 1 \end{pmatrix}} & 
  \Omega^{i-1}_{X_1}\oplus \Omega^{i}_{X_1} \ar[r]^{\begin{pmatrix} 1 & 0 \\ 0 & 0 \end{pmatrix}} & 
  \Omega^{i-1}_{X_1}\oplus \Omega^{i}_{X_1} \ar[r]^{\begin{pmatrix} 0
   & 0 \\ 1 & 0 \end{pmatrix}} & 
  \Omega^i_{X_1}\oplus \Omega^{i+1}_{X_1} 
}
\end{gather*}
where the maps are indicated by the right multiplications by the displayed matrices. Then by the compatibility of the Chern character (see Proposition \ref{prop:infinitesimalChernClass}), we can compute the multiplication 
\begin{gather*}
\xymatrix@R=0.8pc{
\mathcal{K}_1(X_L) \times \Omega^{i}_{X_1} \ar[r] \ar@{}[d]|-{\mathrm{id}\times \cap} & \Omega^{i+1}_{X_1} \ar@{}[d]|-{\cap} \\
\mathcal{K}_1(X_L) \times \mathcal{H}^i\big(p^{r,M}_{r,L}\Omega_{X_{\centerdot}}^{\bullet}[-1]\otimes^{\mathbf{L}}\mathbb{Z}/p\big) \ar[r] & \mathcal{H}^{i+1}(p^{r+1,M}_{r+1,L}\Omega_{X_{\centerdot}}^{\bullet}[-1]\otimes^{\mathbf{L}}\mathbb{Z}/p)
}
\end{gather*}
by computing the multiplication 
\begin{gather*}
\xymatrix@R=0.8pc{
\mathcal{K}_1(X_L) \times \Omega^{i}_{X_1} \ar[r] \ar@{}[d]|-{\mathrm{id}\times\cap} & \Omega^{i+1}_{X_1} \ar@{}[d]|-{\cap} \\
\mathcal{K}_1(X_L) \times \mathcal{H}^i\big(p^{r,1}_{r,L}\Omega_{X_{\centerdot}}^{\bullet}[-1]\otimes^{\mathbf{L}}\mathbb{Z}/p\big) \ar[r] & \mathcal{H}^{i+1}\big(p^{r+1,1}_{r+1,L}\Omega_{X_{\centerdot}}^{\bullet}[-1]\otimes^{\mathbf{L}}\mathbb{Z}/p\big),
}
\end{gather*}
and compute the multiplication 
\begin{gather*}
\xymatrix@R=0.8pc{
\mathcal{K}_1(X_L) \times \Omega^{i-1}_{X_1} \ar[r] \ar@{}[d]|-{\mathrm{id}\times \cap} & \Omega^{i}_{X_1} \ar@{}[d]|-{\cap} \\
\mathcal{K}_1(X_L) \times \mathcal{H}^i\big(p^{r,M}_{r,L}\Omega_{X_{\centerdot}}^{\bullet}[-1]\otimes^{\mathbf{L}}\mathbb{Z}/p\big) \ar[r] & \mathcal{H}^{i+1}\big(p^{r+1,M}_{r+1,L}\Omega_{X_{\centerdot}}^{\bullet}[-1]\otimes^{\mathbf{L}}\mathbb{Z}/p\big)
}
\end{gather*}
by  computing the multiplication 
\begin{gather*}
\xymatrix@R=0.8pc{
\mathcal{K}_1(X_M) \times \Omega^{i-1}_{X_1} \ar[r] \ar@{}[d]|-{\mathrm{id}\times\cap} & \Omega^{i}_{X_1} \ar@{}[d]|-{\cap} \\
\mathcal{K}_1(X_M) \times \mathcal{H}^i\big(p^{r,1}_{r,M}\Omega_{X_{\centerdot}}^{\bullet}[-2]\otimes^{\mathbf{L}}\mathbb{Z}/p\big) \ar[r] & \mathcal{H}^{i+1}\big(p^{r+1,1}_{r+1,M}\Omega_{X_{\centerdot}}^{\bullet}[-2]\otimes^{\mathbf{L}}\mathbb{Z}/p\big)
}
\end{gather*}
and by using the map $\mathcal{K}_1(X_L) \rightarrow \mathcal{K}_1(X_M)$ and noticing the sign change due to  exchanging the boundary operator
$\delta: \mathcal{H}^{i-1}\big(p^{r,1}_{r,M}\Omega_{X_{\centerdot}}^{\bullet}\big) \rightarrow  
\mathcal{H}^i\big(p^{r,M}_{r,L}\Omega_{X_{\centerdot}}^{\bullet}\otimes^{\mathbf{L}}\mathbb{Z}/p\big)$ and the multiplication by $c_1(a)$. Topologically this is due to the exchange of multiplication with the smash product with $S^1$.
Hence the conclusion follows from the computations in the case $M=1$.
\end{proof}

\begin{proposition}\label{prop:multiplication-Bott-element}
Let $L>M\geq 1$ be integers. Let $i,r$ be integers and $0\leq i\leq r\leq p-2$.
Let $\beta_{\zeta}\in \mathcal{K}_2(\mathcal{O}_{X_L}; \mathbb{Z}/p)$ be the Bott element associated with $\zeta=1+p^{L-1}\in W_L$. Then the multiplication 
\begin{gather*}
\mathcal{H}^i\big(p^{r,M}_{r,L}\Omega^{\bullet}_{X_{\centerdot}}[-1]\otimes_{\mathbb{Z}}^{\mathbf{L}}\mathbb{Z}/p\big) 
\rightarrow \mathcal{H}^{i}\big(p^{r+1,M}_{r+1,L}\Omega^{\bullet}_{X_{\centerdot}}[-1]\otimes_{\mathbb{Z}}^{\mathbf{L}}\mathbb{Z}/p\big)
\end{gather*}
by $c_1(\beta_{\zeta})$ preserves the decomposition \eqref{eq:rel-HC-nonreduced-modP-components-sheafified}, and its multiplication on the summands is given by
\begin{gather*}
\xymatrix@R=1pc@C=4pc{
 \mathcal{H}^i\big(p^{r,M}_{r,L}\Omega^{\bullet}_{X_{\centerdot}}[-1]\otimes_{\mathbb{Z}}^{\mathbf{L}}\mathbb{Z}/p\big)  \ar[r]^{c_1(\beta_{\zeta})\times} & \mathcal{H}^{i}\big(p^{r+1,M}_{r+1,L}\Omega^{\bullet}_{X_{\centerdot}}[-1]\otimes_{\mathbb{Z}}^{\mathbf{L}}\mathbb{Z}/p\big) \\
  \Omega_{X_1/\Bbbk}^{i-1} \ar[r]^{0} \ar@{}[u]|-{\rotatebox{90}{$\subset$}} &  
  \Omega_{X_1/\Bbbk}^{i-1} \ar@{}[u]|-{\rotatebox{90}{$\subset$}}
}
\end{gather*}
and
\begin{gather*}
\xymatrix@R=1pc@C=4pc{
 \mathcal{H}^i\big(p^{r,M}_{r,L}\Omega^{\bullet}_{X_{\centerdot}}[-1]\otimes_{\mathbb{Z}}^{\mathbf{L}}\mathbb{Z}/p\big)  \ar[r]^{c_1(\beta_{\zeta})\times} & \mathcal{H}^{i}\big(p^{r+1,M}_{r+1,L}\Omega^{\bullet}_{X_{\centerdot}}[-1]\otimes_{\mathbb{Z}}^{\mathbf{L}}\mathbb{Z}/p\big) \\
  \Omega_{X_1/\Bbbk}^{i} \ar@{=}[r] \ar@{}[u]|-{\rotatebox{90}{$\subset$}} &  
  \Omega_{X_1/\Bbbk}^{i} \ar@{}[u]|-{\rotatebox{90}{$\subset$}}.
}
\end{gather*}
\end{proposition}
\begin{proof}
In the case $M=1$, by Lemma \ref{lem:sheafCoh-p(r)-to-syntomic-mono} as in the proof of Proposition \ref{prop:K1-multiplication-motivic-complex} we can compute the products in the cohomology sheaves $\mathcal{H}^i\big(\mathfrak{G}_{X_{\centerdot}}(r,L)\otimes^{\mathbf{L}}\mathbb{Z}/p\big)$. The image of the Bott element of $K_2(X_L;\mathbb{Z}/p \mathbb{Z})$ in $\mathcal{H}^0\big(\mathfrak{G}_{X_{\centerdot}}(1,L)\otimes^{\mathbf{L}}\mathbb{Z}/p\big)$ is computed in Lemma \ref{exa:Bott-element-syntomic-complex}.   Then the conclusion  follows by applying  Lemma \ref{lem:Kato-product-mod-p} to $\gamma=p^{n-1}$.

The case $M\geq 2$ follows from the $M=1$ case as in the  proof of Proposition \ref{prop:K1-multiplication-motivic-complex}; we need only replace $\mathcal{K}_1(X_L)$ (resp. $\mathcal{K}_1(X_M)$) at there with $\mathcal{K}_2(X_L;\mathbb{Z}/p)$ (resp. $\mathcal{K}_2(X_M;\mathbb{Z}/p)$), and notice that the map $\mathcal{K}_2(X_L;\mathbb{Z}/p) \rightarrow \mathcal{K}_2(X_M;\mathbb{Z}/p)$ sends the Bott element in $\mathcal{K}_2(X_L;\mathbb{Z}/p)$ to 0, because $W_L(\Bbbk)\ni\zeta=1+p^{L-1}\mapsto 0 \in W_{M-1}(\Bbbk)$.
\end{proof}

By Propositions \ref{prop:brun-iso-product}, \ref{prop:K1-multiplication-motivic-complex} and \ref{prop:multiplication-Bott-element}, the multiplication of $(\mathcal{K}/p)_{X_n,1}$  and the Bott element a on $(\mathcal{K}/p)_{X_n,X_{n-1},i}$ for $i\leq p-4$ is illustrated by Figure \ref{fig:multTable} in the introduction.
We see that the summand $\mathcal{O}_{X_1}$ in $(\mathcal{K}/p)_{X_n,X_{n-1},i}$ for odd $i\leq p-3$ and for $i=2$ is not generated by such multiplications from lower degree elements, which also holds for $(\mathcal{K}/p)_{X_n,X_{n-1},i}$, where $n>m\geq 1$, by Theorem \ref{thm:relative-comparison-local} which we shall prove. This motivates the following definition.

\begin{definition}\label{def:basicElements}
Let $n>m\geq 1$ be integers. Let $1\leq i\leq p-3$ be an odd integer or $i=2$. We call the summand $\mathcal{O}_{X_1}\subset (\mathcal{K}/p)_{X_n,X_{m},i}$  in the decomposition induced by Theorem \ref{thm:relative-comparison-local} and \eqref{eq:rel-HC-nonreduced-modP-components-sheafified} the \emph{basic summand}, and call the elements in  the basic summand the  \emph{basic elements}.
\end{definition}

\subsection{Functoriality resolves ambiguity} 
\label{sub:functoriality_resolves_ambiguity}
This subsection is devoted to rescue the not-being-generated basic summand
by  exploring some consequences of the functoriality of the infinitesimal Chern character.

\begin{definition}
For a commutative ring  $R$, let $\mathrm{SmAlg}_{R}$ be the category whose objects are the smooth $R$-algebras of finite type, and morphisms are the  $R$-algebra maps. Let $\mathbf{Ab}$ be the category of abelian groups. 
For a perfect field $\Bbbk$ of characteristic $p$, let $\mathfrak{r}: \mathrm{SmAlg}_{W(\Bbbk)} \rightarrow \mathrm{SmAlg}_{\Bbbk}$ be the reduction functor sending $A$ to $A/p$.
Let $\mathfrak{a}:\mathrm{SmAlg}_{\Bbbk} \rightarrow \mathbf{Ab}$ be the forgetful functor sending a smooth $\Bbbk$-algebra to its underlying abelian group. An \emph{additive endomorphism of $\mathfrak{r}$} is an endomorphism of the composition functor $\mathfrak{a}\circ \mathfrak{r}$. 
\end{definition}
In other words, an additive endomorphism of $\mathfrak{r}$ is to attach to every object $A$ of $\mathrm{SmAlg}_{W(\Bbbk)}$ an additive map $F_{A}:A/p\rightarrow A/p$, such that for every morphism $\alpha:A \rightarrow A'$ in 
$\mathrm{SmAlg}_{W(\Bbbk)}$ the diagram
\begin{gather*}
\xymatrix{
A/p \ar[r]^{F_{A}} \ar[d]_{\overline{\alpha}} & A/p \ar[d]^{\overline{\alpha}} \\
A'/p \ar[r]^{F_{A'}} & A'/p
}
\end{gather*}
commutes, where $\overline{\alpha}=\alpha\mod p$. 

\begin{lemma}\label{lem:universal-maps-A1-to-A1}
Let $\Bbbk$ be a perfect field with $\mathrm{char}(\Bbbk)=p$.
 Then any additive endomorphism $F$ of $\mathfrak{r}$ has the following form: there exist $n\in \mathbb{N}\cup\{0\}$ and $c_i\in \Bbbk$ for $0\leq i\leq n$ such that
\begin{gather*}
F_A(\bar{a})=\sum_{i=0}^n c_i \bar{a}^{p^i}
\end{gather*}
for any $A\in \mathrm{ob}(\mathrm{SmAlg}_{W(\Bbbk)})$ and any $a\in A$, where $\bar{a}=a\mod p$.
\end{lemma}
\begin{proof}
Let $W=W(\Bbbk)$.
Let $f=F_{W[T]}(T)\in \Bbbk[T]$. For any $A\in \mathrm{ob}(\mathrm{SmAlg}_{W})$ and any $a\in A$, there is a unique 
$W$-algebra map $\alpha:W[T]\rightarrow A$, $T\mapsto a$.  By  functoriality,  we have
\begin{gather}\label{eq:universal-maps-to-A1}
F_A(\bar{a})=\overline{\alpha}(f)=f(\bar{a}).
\end{gather}
Fix an algebraic closure $\overline{\Bbbk}$ of $\Bbbk$. For any $u,v\in \overline{\Bbbk}$, there is a unique $\Bbbk$-algebra map $\beta:\Bbbk[T_1,T_2]\rightarrow \overline{\Bbbk}$, such that $\beta(T_1)=u$ and $\beta(T_2)=v$. Let $E$ be the subfield $\Bbbk(u,v)$ of $\overline{\Bbbk}$. Since $\Bbbk$ is perfect, $E\in \mathrm{ob}(\mathrm{SmAlg}_{\Bbbk})$. 
Let $\widetilde{E}\in \mathrm{ob}(\mathrm{SmAlg}_W)$ be a lifting of $E$, which exists by \cite[Théorème 6]{Elk73} (see also \cite[Theorem 1.3.1]{Ara01}), and $\tilde{\beta}:W[T_1,T_2]\rightarrow\widetilde{E}$ a lifting of $\beta$.
Then the commutative diagram
\begin{gather*}
\xymatrix@C=6pc{
\Bbbk[T_1,T_2] \ar[r]^{F_{W[T_1,T_2]}} \ar[d]_{\overline{\beta}} & \Bbbk[T_1,T_2] \ar[d]^{\overline{\beta}} \\
E \ar[r]^{F_{\widetilde{E}}} & E
}
\end{gather*}
and \eqref{eq:universal-maps-to-A1} and the additivity of $F$ imply
\begin{gather*}
f(u+v)=F_{\widetilde{E}}(u+v)=F_{\widetilde{E}}\circ \overline{\beta}(T_1+T_2)=\overline{\beta}\circ F_{W[T_1,T_2]}(T_1+T_2)\\
=\overline{\beta}\circ F_{W[T_1,T_2]}(T_1)+\overline{\beta}\circ F_{W[T_1,T_2]}(T_2)
=F_{\widetilde{E}}\circ \overline{\beta}(T_1)+F_{\widetilde{E}}\circ \overline{\beta}(T_2)\\
=F_{\widetilde{E}}(u)+F_{\widetilde{E}}(v)
=f(u)+f(v).
\end{gather*}
Namely, $f$ is an \emph{absolutely additive polynomial} in the sense of \cite[\S 1.1]{Gos96}. Then by \cite[Cor. 1.1.6]{Gos96}, there exist $c_0,\dots,c_n$ such that $f=\sum_{i=0}^n c_iT^{p^i}$. By \eqref{eq:universal-maps-to-A1}, the proof is complete.
\end{proof}

\begin{lemma}\label{lem:universal-iso-A-to-A}
Let $\Bbbk$ be a field with $\mathrm{char}(\Bbbk)=p$, and $n\in \mathbb{N}\cup\{0\}$.
Let $c_i\in \Bbbk$ for $0\leq i\leq n$ and let $\mathbf{c}:\Bbbk[T]\rightarrow \Bbbk[T]$ be given by $\mathbf{c}(a)=\sum_{i=0}^n c_i a^{p^i}$ for $a\in \Bbbk[T]$. Then
\begin{enumerate}[(i)]
  \item  $\mathbf{c}$ is an isomorphism iff $\mathbf{c}$ is surjective iff  $c_0\in \Bbbk^{\times}$ and $c_i=0$ for $i>0$;
  \item if there exists $m\in \mathbb{N}$ such that the restriction of $\mathbf{c}$ to the $\Bbbk$-subspace $U_m$ spanned by $T^j$ with $v_p(j)\geq m$ is a surjective map  onto $U_m$ itself, then $\mathbf{c}$ is an isomorphism.
\end{enumerate}
\end{lemma}
\begin{proof}
(i)
Suppose $\mathbf{c}(T)=\sum_{i=0}^n c_i T^{p^i}$ is surjective. If there exists  $c_i\neq 0$ with $i\geq 1$, then by considering the nonzero highest degree terms  we see that there is no $f\in C[T]$ such that $\mathbf{c}(f)=T$. Thus $\mathbf{c}(T)=c_0T$ for  $c_0\in k$ and then $c_0\in \Bbbk^{\times}$.

(ii) The assumption implies that $\mathbf{c}:\Bbbk[T^{p^m}] \rightarrow \Bbbk[T^{p^m}]$ is surjective, hence the conclusion follows from (i).
\end{proof}

\begin{proposition}\label{prop:functoriality-determine-universal-transform}
Let  $L>M\geq 1$ and $r\geq 1$ be integers.
For a smooth $W(\Bbbk)$-algebra $A$ and $n\in \mathbb{N}$,  let $A_n=A/p^n A$.
Let $F: \mathrm{SmAlg}_{W(\Bbbk)} \rightarrow \mathrm{D}(\mathbf{Ab})$ be the functor  $F(A)=p^{r,M}_{r,L}\Omega^{\bullet}_{A}$.
By the decomposition \eqref{eq:rel-HC-nonreduced-modP-components}, we embed the reduction functor $\mathfrak{r}$ into the functor $H^0(F\otimes^{\mathbf{L}} \mathbb{Z}/p \mathbb{Z}): \mathrm{SmAlg}_{W(\Bbbk)} \rightarrow \mathbf{Ab}$ as a direct summand, namely
\begin{gather*}
\mathfrak{r}(A)=A_1\hookrightarrow A_1\oplus \Omega^{1}_{A_1/\Bbbk} \overset{\eqref{eq:rel-HC-nonreduced-modP-components}}{=} H^0(F(A)\otimes^{\mathbf{L}} \mathbb{Z}/p \mathbb{Z}).
\end{gather*}
Let $f:F \rightarrow F$ be an endomorphism of functors, and $g:\mathfrak{r} \rightarrow \mathfrak{r}$ be the induced endomorphism, namely, for $A\in \mathrm{ob}(\mathrm{SmAlg}_{W(\Bbbk)})$, $g_A: A_1 \rightarrow A_1$ is the  unique map which makes the diagram
\begin{gather*}
\xymatrix@C=10pc{
  A_1 \ar[r]^{g_A} \ar@{^{(}->}[d] & A_1 \ar@{^{(}->}[d] \\
  H^0(F(A)\otimes^{\mathbf{L}} \mathbb{Z}/p \mathbb{Z}) \ar[r]^{H^0(f_A\otimes^{\mathbf{L}} \mathbb{Z}/p \mathbb{Z})} & H^0(F(A)\otimes^{\mathbf{L}} \mathbb{Z}/p \mathbb{Z})
}
\end{gather*}
commute. 
Suppose now $H^{-1}(f\otimes^{\mathbf{L}}\mathbb{Z}/p \mathbb{Z})$ is an isomorphism.
Then $g$ is also an isomorphism.
\end{proposition}
\begin{proof}[Proof for the case $\Bbbk=\mathbb{F}_p$]
By assumption, $g$ is an additive endomorphism of $\mathfrak{r}$. 
By Lemma \ref{lem:universal-maps-A1-to-A1}, there exist $n\in \mathbb{N}$ and  $c_i\in \mathbb{F}_p$ for $0\leq i\leq n$, such that  $g_{A}(\bar{a})=\sum_{i=0}^n c_i(\bar{a})^{p^i}$ for all $A\in \mathrm{ob}(\mathcal{V})$ and $a\in A$.

By \eqref{eq:rel-HC-nonreduced-modP-components}, $H^{-1}(f\otimes^{\mathbf{L}}\mathbb{Z}/p \mathbb{Z})$ is also an additive endomorphism of $\mathfrak{r}$. Thus by  Lemma \ref{lem:universal-maps-A1-to-A1} and Lemma \ref{lem:universal-iso-A-to-A}(i), it is a scalar multiplication by an element of $\mathbb{F}_p^{\times}$.
To determine the $c_i$'s, we study the case $A=\mathbb{Z}_p[T]$ to find a relation between $H^{-1}(f\otimes^{\mathbf{L}}\mathbb{Z}/p \mathbb{Z})$ and $g$. 

For any morphism of $\mathbb{Z}$-modules $\tau:M \rightarrow N$, denote by $\overline{\tau}$ the induced map $M\otimes \mathbb{Z}/p \rightarrow N\otimes \mathbb{Z}/p$.

\noindent \textbf{Step 1: Free resolutions.}
We replace \eqref{eq-infinitesimalMOtivicComplex-flatResolution} with the following resolution by  \emph{free $\mathbb{Z}$-modules} (instead of free $\mathbb{Z}_p$-modules)
\begin{flalign}\label{eq:infinitesimalMOtivicComplex-flatResolution-Zp[T]}
\xymatrix{
\mathbb{Z}[T]\ar[r]^{p^L \mathrm{d}} \ar[d]_{p^{r(L-M)}} &  \Omega^1_{\mathbb{Z}[T]} \ar[d]^{p^{(r-1)(L-M)}}   \\
 \mathbb{Z}[T]\ar[r]^{p^M \mathrm{d}} \ar[d]_{p^{rM}}
 &  \Omega^1_{\mathbb{Z}[T]} \ar[d]^{p^{(r-1)M}}\\
  p^{rM}\mathbb{Z}[T]/p^{rL}\mathbb{Z}[T]\ar[r]^<<<<{\mathrm{d}} &  p^{(r-1)M}\Omega^1_{\mathbb{Z}[T]/p^{(r-1)L}}
}
\end{flalign}
where $\Omega^i_{R}:=\Omega^i_{R/\mathbb{Z}}$ for any commutative ring $R$. 
One finds that 
\begin{gather*}
H^0F(\mathbb{Z}_p[T])=\mathrm{Ker}\left(p^{rM}\mathbb{Z}[T]/p^{rL}\mathbb{Z}[T]\xrightarrow{\mathrm{d}} p^{(r-1)M}\Omega^1_{\mathbb{Z}[T]/p^{(r-1)L}}\right)
\end{gather*}
is spanned by
\begin{gather*}
\left\{p^{\max\{(r-1)L-v_p(k),rM\}}T^k: k\in \mathbb{N}\cup\{0\}\right\}.
\end{gather*}
Then there is  a resolution of $H^0F(\mathbb{Z}_p[T])$ by  free $\mathbb{Z}$-modules
\begin{gather}\label{eq:resolution-H0F(Zp[T])}
\xymatrix@R=0.5pc@C=0.8pc{
  0 \ar[r] & \mathbb{Z}[T] \ar[r] & \mathbb{Z}[T] \ar[r] & H^0F(\mathbb{Z}_p[T]) \ar[r] & 0\\
  & &  T^k \ar@{|->}[r] & p^{\max\{(r-1)L-v_p(k),rM\}}T^k &  \\
           &   T^k \ar@{|->}[r] & p^{rL-\max\{(r-1)L-v_p(k),rM\}}T^k & &.
}
\end{gather}
Define additive maps
\begin{align*}
&\alpha_1: \mathbb{Z}[T] \rightarrow \mathbb{Z}[T],\quad T^k\mapsto p^{\max\{(r-1)L-rM-v_p(k),0\}}T^k,\\
&\alpha_2: \mathbb{Z}[T] \rightarrow \Omega^1_{\mathbb{Z}[T]},\quad T^k\mapsto p^{\max\{-(r-1)L+rM+v_p(k),0\}} p^{-v_p(k)}\mathrm{d}T^k.
\end{align*}
We have a  morphism of complexes (the bottom arrows are left actions of matrices of operators)
\begin{gather}\label{eq:morphismOfComplexes-kernel-to-complex}
\xymatrix@R=4pc@C=10pc{
  \mathbb{Z}[T] \ar@{=}[d] \ar[r]^{T^k\mapsto p^{rL-\max\{(r-1)L-v_p(k),rM \}}T^k} & 
  \mathbb{Z}[T] \ar[d]^{(\alpha_1,\alpha_2)} 
  \ar[r]  & 0 \ar[d]\\
        \mathbb{Z}[T] \ar[r]^{\scriptscriptstyle\begin{pmatrix}p^{r(L-M)}\\ p^L \mathrm{d} \end{pmatrix}}   &  \mathbb{Z}[T]\oplus \Omega^1_{\mathbb{Z}[T]}\ar[r]^{\scriptstyle{\begin{pmatrix} p^M \mathrm{d} & -p^{(r-1)(L-M)}\end{pmatrix}}} & \Omega^1_{\mathbb{Z}[T]}
}
\end{gather}
which covers the map $\theta$ of complexes (in the vertical direction)
\begin{gather*}
\xymatrix@C=4pc{
\big(H^0F(\mathbb{Z}_p[T])\big)[0] \ar@{}[dr]|-{\mbox{, namely}} \ar[d]_{\theta} &   H^0F(\mathbb{Z}_p[T]) \ar[r] \ar[d] & 0 \ar[d] \\
F(\mathbb{Z}_p[T]) &  p^{rM}\mathbb{Z}[T]/p^{rL}\mathbb{Z}[T] \ar[r]^<<<<<<<{\mathrm{d}}  & p^{(r-1)M}\Omega^1_{\mathbb{Z}[T]/p^{(r-1)L}}
}
\end{gather*}
via the resolutions \eqref{eq:infinitesimalMOtivicComplex-flatResolution-Zp[T]} and \eqref{eq:resolution-H0F(Zp[T])}. 
Let $U$ be the subspace of $\mathbb{F}_p[T]$ spanned by 
\begin{gather*}
\{T^k: v_p(k)> (r-1)L-rM\}.
\end{gather*}
 From \eqref{eq:morphismOfComplexes-kernel-to-complex}  we have several observations:
\begin{align}\label{eq:observation-on-subspace-U-1}
&&\text{The restriction of  $\overline{\alpha_1}$ on $U$ is the identity and the restriction of  $\overline{\alpha_2}$ on $U$ is 0. }
\end{align}
\begin{align}\label{eq:observation-on-subspace-U-2}
&\mbox{
The map on $H^{-1}(?\otimes^{\mathbf{L}}\mathbb{Z}/p)$ induced by $\theta$ is the identification 
$\mathbb{F}_p[T] = \mathbb{F}_p[T]$.
}
\end{align}
\begin{gather}\label{eq:observation-on-subspace-U-3}
\text{
The restriction to $U$, of the map on $H^{0}(?\otimes^{\mathbf{L}}\mathbb{Z}/p)$ induced by $\theta$,
is the inclusion }\\
U= U\oplus 0 \hookrightarrow \mathbb{F}_p[T]\oplus \Omega^1_{\mathbb{F}_p[T]}.\nn
\end{gather}

\noindent \textbf{Step 2: The maps induced by  the endomorphism $f$ of the functor $F$.}

The endomorphism $f_{\mathbb{Z}_p[T]}$ of $F(\mathbb{Z}_p[T])=p^{r,M}_{r,L}\Omega^{\bullet}_{\mathbb{Z}_p[T]}$ in $\mathrm{D}(\mathbf{Ab})$ induces an endomomorphism $H^0 f_{\mathbb{Z}_p[T]}$ of the 0-th cohomology group $H^0 F(\mathbb{Z}_p[T])$, and there is a commutative diagram
\begin{gather}\label{eq:diag-endo-f-induced-cohomology}
\xymatrix@C=4pc{
  \big(H^0 F(\mathbb{Z}_p[T])\big)[0]  \ar[r]^{H^0 f_{\mathbb{Z}_p[T]}} \ar[d]_{\theta} & 
  \big(H^0 F(\mathbb{Z}_p[T])\big)[0]\ar[d]^{\theta} \\
  F(\mathbb{Z}_p[T]) \ar[r]^{ f_{\mathbb{Z}_p[T]}} &  F(\mathbb{Z}_p[T])
}
\end{gather}
in $\mathrm{D}(\mathbf{Ab})$.

Since \eqref{eq:infinitesimalMOtivicComplex-flatResolution-Zp[T]} is a resolution by free abelian groups, the endomorphism $f_{\mathbb{Z}_p[T]}$ of $F(\mathbb{Z}_p[T])$ can be represented by a map of complexes
\begin{gather*}
\xymatrix@R=0.2pc@C=4pc{
   \mathbb{Z}[T] \ar[r] \ar[ddddd]_{\beta}   &  \mathbb{Z}[T]\oplus \Omega^1_{\mathbb{Z}[T]} \ar[r] \ar[ddddd]^{\begin{pmatrix}\beta_{11} & \beta_{12}\\ \beta_{21} & \beta_{22}\end{pmatrix}} & \Omega^1_{\mathbb{Z}[T]} \ar[ddddd] \\
   &&\\
   &&\\
   &&\\
   &&\\
      \mathbb{Z}[T] \ar[r]    &  \mathbb{Z}[T]\oplus \Omega^1_{\mathbb{Z}[T]} \ar[r]  & \Omega^1_{\mathbb{Z}[T]}  \\
    (\deg\ -1) & (\deg\ 0) & (\deg\ 1)
}
\end{gather*}
where both rows are the bottom row of \eqref{eq:morphismOfComplexes-kernel-to-complex}, and the middle vertical arrow is the left multiplication of a matrix of additive maps.
Since each of the horizontal arrows is  multiplication by a positive $p$-power,  we have
\begin{equation}\label{eq:Hminus1=beta-mod-p}
\begin{gathered}
H^{-1}(f_{\mathbb{Z}_p[T]}\otimes^{\mathbf{L}}\mathbb{Z}/p \mathbb{Z})=\overline{\beta},\\
H^{0}(f_{\mathbb{Z}_p[T]}\otimes^{\mathbf{L}}\mathbb{Z}/p \mathbb{Z})=\begin{pmatrix}\beta_{11} & \beta_{12}\\ \beta_{21} & \beta_{22}\end{pmatrix}\mod p,
\end{gathered}
\end{equation}
and by our setup of the proposition,
\begin{gather}\label{eq:g=beta11-mod-p}
g_{\mathbb{Z}_p[T]}=\overline{\beta_{11}}.
\end{gather}

Since  \eqref{eq:resolution-H0F(Zp[T])} is a resolution by free abelian groups, the endomorphism $H^0f_{\mathbb{Z}_p[T]}$ of $H^0F(\mathbb{Z}_p[T])$ can be represented by a map of complexes
\begin{gather}\label{eq:resolution-H0F(Zp[T])-endomorphism-f}
\xymatrix@R=0.2pc{
  \mathbb{Z}[T] \ar[r] \ar[dddd]_{\gamma_{-1}} & \mathbb{Z}[T]  \ar[dddd]^{\gamma_0} \\
  &\\ &\\ &\\ 
   \mathbb{Z}[T] \ar[r] & \mathbb{Z}[T] \\
   (\deg\ -1) & (\deg\ 0)
}
\end{gather}
where both rows are the top row of \eqref{eq:morphismOfComplexes-kernel-to-complex}. Again, since the horizontal arrows  are multiplications by positive $p$-powers  on each $T^k$, we have
\begin{gather*}
H^{-1}(H^0f_{\mathbb{Z}_p[T]}\otimes^{\mathbf{L}}\mathbb{Z}/p \mathbb{Z})=\overline{\gamma_{-1}},\\
H^{0}(H^0f_{\mathbb{Z}_p[T]}\otimes^{\mathbf{L}}\mathbb{Z}/p \mathbb{Z})=\overline{\gamma_0}.
\end{gather*}
Moreover, the commutativity of \eqref{eq:resolution-H0F(Zp[T])-endomorphism-f} implies
\begin{gather}\label{eq:gamma(-1)=gamma0}
\gamma_{-1}=\gamma_0.
\end{gather}

\noindent\textbf{Step 3: Comparison of maps.}

Now the commutativity of \eqref{eq:diag-endo-f-induced-cohomology} combined with the observation \eqref{eq:observation-on-subspace-U-2} yields the commutative diagram
\begin{gather*}
\xymatrix{
  \mathbb{F}_p[T] \ar[r]^{\overline{\gamma_{-1}}} \ar@{=}[d] & \mathbb{F}_p[T] \ar@{=}[d] \\
  \mathbb{F}_p[T] \ar[r]^{\overline{\beta}} & \mathbb{F}_p[T]
}
\end{gather*}
and thus we have
\begin{gather}\label{eq:gamma(-1)Bar=betaBar}
\overline{\gamma_{-1}}=\overline{\beta}.
\end{gather}
The commutativity of \eqref{eq:diag-endo-f-induced-cohomology}  yields the commutative diagram
\begin{gather*}
\xymatrix@R=4pc@C=4pc{
 \mathbb{F}_p[T] \ar[r]^{\overline{\gamma_0}} \ar[d]_{(\overline{\alpha_1},\overline{\alpha_2})} & \mathbb{F}_p[T] \ar[d]^{(\overline{\alpha_1},\overline{\alpha_2})} \\
  \mathbb{F}_p[T]\oplus \Omega^1_{\mathbb{F}_p[T]} \ar[r]^{\overline{\begin{pmatrix}\beta_{11} & \beta_{12}\\ \beta_{21} & \beta_{22}\end{pmatrix}}}  
  & \mathbb{F}_p[T]\oplus \Omega^1_{\mathbb{F}_p[T]}
}
\end{gather*}
which combined with the observation \eqref{eq:observation-on-subspace-U-3} implies
\begin{gather*}
\overline{\gamma_{0}}=\overline{\beta_{11}}\quad \mbox{on}\ U.
\end{gather*}
Combining this  with \eqref{eq:Hminus1=beta-mod-p}, \eqref{eq:g=beta11-mod-p}, \eqref{eq:gamma(-1)=gamma0}, and \eqref{eq:gamma(-1)Bar=betaBar}, we get
\begin{gather*}
g_{\mathbb{Z}_p[T]}=H^{-1}(f_{\mathbb{Z}_p[T]}\otimes^{\mathbf{L}}\mathbb{Z}/p \mathbb{Z})\ \mbox{on}\  U.
\end{gather*}
As we have seen in the beginning of the proof, $H^{-1}(f_{\mathbb{Z}_p[T]}\otimes^{\mathbf{L}}\mathbb{Z}/p \mathbb{Z})$ is a nonzero scalar multiplication on $U$, thus so is  the restriction of $g_{\mathbb{Z}_{p}[T]}$ on $U$.
Then by applying  Lemma  \ref{lem:universal-iso-A-to-A}(ii) to  $g_{\mathbb{Z}_{p}[T]}$ and $U$, we deduce that $g_{\mathbb{Z}_{p}[T]}$ is an isomorphism on $\mathbb{F}_p[T]$ and, by Lemma  \ref{lem:universal-iso-A-to-A}(i), is thus a scalar multiplication by some element in $\mathbb{F}_p^{\times}$. This means that  $c_i=0$ for $i>0$ and $c_0\in \mathbb{F}_p^{\times}$, and the proof is complete.
\end{proof}

Perception upon this proof reveals that what we used essentially is that $W_n(\mathbb{F}_p)$ is a \emph{free} $\mathbb{Z}/p^n \mathbb{Z}$-module. This is true for all perfect $\Bbbk$ by the following lemma.

\begin{lemma}\label{lem:freeness-basis-of-Wn(k)}
Let $\Bbbk$ be a perfect field of characteristic $p$.
There exists a free $\mathbb{Z}$-module $\mathcal{W}$  and additive maps $\phi_i:\mathcal{W} \rightarrow W_n(\Bbbk)$ which induce isomorphisms $ \mathcal{W}/p^n \mathcal{W}\xrightarrow{\cong}W_n(\Bbbk)$ and  satisfy $R_n\circ \phi_n=\phi_{n-1}$, where $R_n:W_n(\Bbbk) \rightarrow W_{n-1}(\Bbbk)$ is the restriction homomorphism.
\end{lemma}
\begin{proof}
Let $\{x_{i,1}\}_{i\in I}$ be an $\mathbb{F}_p$-basis of $\Bbbk$. Choose $x_{i,n}\in W_n(\Bbbk)$ such that $R_n x_{i,n}=x_{i,n-1}$  for $n\geq 1$ and $i\in I$. Let $\mathcal{W}=\bigoplus_{i\in I}\mathbb{Z}e_i$, and let $\phi_n: \mathcal{W} \rightarrow W_n(\Bbbk)$ be the homomorphism such that $\phi_n e_i=x_{i,n}$. Since $\Bbbk$ is perfect, the bottom row of
\begin{gather*}
\xymatrix{
  0 \ar[r] & \mathcal{W}/p^{n-1} \ar[r]^{\times p} \ar[d]_{\phi_{n-1}} & \mathcal{W}/p^n \ar[r] \ar[d]_{\phi_n} & \mathcal{W}/p \ar[r] \ar[d]_{\phi_1} & 0 \\
  0 \ar[r] & W_{n-1}(\Bbbk) \ar[r]^{\times p} & W_n(\Bbbk) \ar[r]^{R} & W_1(\Bbbk) \ar[r] & 0
}
\end{gather*}
is exact. Then the conclusion follows by induction on $n$.
\end{proof}

\begin{proof}[Proof of Proposition \ref{prop:functoriality-determine-universal-transform} for all perfect fields $\Bbbk$ of characteristic $p$] 
Let $\mathcal{W}$ and $\phi_i$ be as in Lemma \ref{lem:freeness-basis-of-Wn(k)}. 
By the compatibility of $\phi_i$'s, we denote all $\phi_i$ by $\phi$.
We replace the $\mathbb{F}_p$'s in the preceding proof with $\Bbbk$, $\mathbb{Z}_p$ with $W(\Bbbk)$, and replace the resolution \eqref{eq:infinitesimalMOtivicComplex-flatResolution-Zp[T]} with
\begin{flalign}\label{eq-infinitesimalMOtivicComplex-flatResolution-Zp[T]-replaced}
\xymatrix{
\mathcal{W}\otimes\mathbb{Z}[T]\ar[r]^{1\otimes p^L \mathrm{d}} \ar[d]_{p^{r(L-M)}} &  \mathcal{W}\otimes\Omega^1_{\mathbb{Z}[T]} \ar[d]^{p^{(r-1)(L-M)}}   \\
 \mathcal{W}\otimes\mathbb{Z}[T]\ar[r]^{1\otimes p^M \mathrm{d}} \ar[d]_{\phi\otimes p^{rM}}
 &  \mathcal{W}\otimes\Omega^1_{\mathbb{Z}[T]} \ar[d]^{\phi\otimes p^{(r-1)M}}\\
  p^{rM}W(\Bbbk)[T]/p^{rL}W(\Bbbk)[T]\ar[r]^<<<<{\mathrm{d}} &  p^{(r-1)M}\Omega^1_{W_{(r-1)L}(\Bbbk)[T]/W_{(r-1)L}(\Bbbk)}
}
\end{flalign}
and replace \eqref{eq:resolution-H0F(Zp[T])} with
\begin{gather}\label{eq:resolution-H0F(Zp[T])-replaced}
\xymatrix@R=0.5pc@C=0.8pc{
  0 \ar[r] & \mathcal{W}\otimes\mathbb{Z}[T] \ar[r] & \mathcal{W}\otimes \mathbb{Z}[T] \ar[r] & H^0F(W(\Bbbk)[T]) \ar[r] & 0\\
  & &  w\otimes T^k \ar@{|->}[r] & \phi(w)p^{\max\{(r-1)L-v_p(k),rM\}}T^k &  \\
           &   w\otimes T^k \ar@{|->}[r] & w\otimes p^{rL-\max\{(r-1)L-v_p(k),rM\}}T^k & &
  }
\end{gather}
and replace also the subsequent appearances of (terms of) these resolutions accordingly. Then \eqref{eq-infinitesimalMOtivicComplex-flatResolution-Zp[T]-replaced} and \eqref{eq:resolution-H0F(Zp[T])-replaced} are  resolutions by free abelian groups, and then the proof works for $\Bbbk$.
\end{proof}

\begin{remark}\label{rem:wider-functoriality}
There is a slightly different approach to Proposition \ref{prop:functoriality-determine-universal-transform}: first show a
 further functoriality 
 relating the infinitesimal motivic complex over $W_n(\mathbb{F}_p)$ to that over  $W_n(\Bbbk)$, then show that there is a canonical construction of Chern classes respecting such functoriality, which reduces the determination of the universal coefficients $c_i$ to the problem over $\mathbb{F}_p$ (taking care of the modified Milnor $K$-theory à la Kerz \cite{Ker10}).
 We leave the details to the reader.
 \end{remark}



\subsection{Completion of the proof of Proposition \ref{prop:relative-comparison-m=n-1-modP}} 
\label{sub:completion_of_the_proof_of_relative-comparison}
With the preparations in the previous sections, we are ready to show  Proposition \ref{prop:relative-comparison-m=n-1-modP} and thus complete the proof of Theorem \ref{thm:relative-comparison-local}. Fix $n\geq 2$. Let $\zeta=1+p^{n-1}$ and  $\beta_\zeta$ be the  Bott element in Proposition \ref{prop:multiplication-Bott-element}.
\begin{lemma}\label{lem:comparison-chernClass-chernCharacter}
(i) For $a\in \mathcal{O}_{X_{\centerdot}}^{\times}$, $c_1(a)=\mathrm{ch}(a)$. (ii)  $c_1(\beta_{\zeta})=\mathrm{ch}(\beta_{\zeta})$. 
\end{lemma}
\begin{proof}
(i) $\mathrm{ch}_r(a)\in \mathcal{H}^{2r-1}\big(\mathbb{Z}_{X_n}(r)\big)[\frac{1}{r!}]$. By Proposition \ref{prop:properties-Zn(r)}(i), $\mathbb{Z}_{X_n}(r)$ is supported in degree $\leq r$. Thus $\mathrm{ch}_r(a)=0$ for $r\geq 2$, and $\mathrm{ch}(a)=c_1(a)$.

(ii)  $\mathrm{ch}_r(\beta_{\zeta})\in \mathcal{H}^{2r-2}\big(\mathbb{Z}_{X_n}(r)\big)[\frac{1}{r!}]$ lies in the image of $\mathcal{H}^{2r-2}\big(\mathbb{Z}_{\operatorname{Spec}W_n(\Bbbk)}(r)\big)[\frac{1}{r!}]$ by functoriality of the Chern character (Proposition \ref{prop:infinitesimalChernClass}(i)). By the exact sequence \eqref{eq:motivicFundamentalTriangle}, for $r\geq 2$ we have $\mathcal{H}^{2r-2}\big(\mathbb{Z}_{\operatorname{Spec}W_n(\Bbbk)}(r)\big)=\mathcal{H}^{2r-2}\big(\mathbb{Z}_{\operatorname{Spec}(\Bbbk)}(r)\big)$. Since $\zeta\mapsto 1\in \Bbbk$, by the compatibility of Chern character (Proposition \ref{prop:infinitesimalChernClass}(i) again) $\mathrm{ch}_r(\beta_{\zeta})=0$ for $r\geq 2$.
\end{proof}

\begin{lemma}\label{lem:relative-comparison-modP-i=1,2}
The Nisnevich sheafified Chern character induces isomorphisms
\begin{align*}
\mathcal{K}_{X_n,X_{n-1},1} &\xrightarrow{\cong} p^{n-1}\mathcal{O}_{X_{n}},\quad \mbox{and}\\
(\mathcal{K}/p)_{X_n,X_{n-1},2} &\xrightarrow{\cong} \mathcal{H}^1(p^{2,n-1}_{2,n}\Omega^{\bullet}_{X_{\centerdot}}\otimes^{\mathbf{L}}\mathbb{Z}/p\big)\quad \mbox{if}\ p\geq 5.
\end{align*}
\end{lemma}
\begin{proof}
By the construction in \S\ref{sec:infinitesimal_chern_classes_and_chern_characters}, the Chern character in the relevant range is induced by a sheafified Chern character
\begin{gather*}
\mathcal{K}_{X_n,X_{n-1}} \xrightarrow{\mathrm{ch}} \bigoplus_{r=1}^{p-1}\mathrm{K}\big(p^{r,n-1}_{r,n}\Omega^{\bullet}_{X_{\centerdot}}[2r-1]\big)
\end{gather*}
where  we are using the notation in \S\ref{sub:notations}(iv) and (v). 
In particular, we have the maps
\begin{gather*}
\mathcal{K}_{X_n,X_{n-1},1} \xrightarrow{\mathrm{ch}} \bigoplus_{r=1}^{p-1}\mathcal{H}^{-1}\big(p^{r,n-1}_{r,n}\Omega^{\bullet}_{X_{\centerdot}}[2r-1]\big)=\bigoplus_{r=1}^{p-1}\mathcal{H}^{0}\big(p^{r,n-1}_{r,n}\Omega^{\bullet}_{X_{\centerdot}}[2r-2]\big)\\
=\mathcal{H}^{0}\big(p^{1,n-1}_{1,n}\Omega^{\bullet}_{X_{\centerdot}}\big)=p^{n-1}\mathcal{O}_{X_{n}}
\end{gather*}
and
\begin{gather*}
\mathcal{K}_{X_n,X_{n-1},2} \xrightarrow{\mathrm{ch}} \bigoplus_{r=1}^{p-1}\mathcal{H}^{-2}\big(p^{r,n-1}_{r,n}\Omega^{\bullet}_{X_{\centerdot}}[2r-1]\big)=\bigoplus_{r=1}^{p-1}\mathcal{H}^{0}\big(p^{r,n-1}_{r,n}\Omega^{\bullet}_{X_{\centerdot}}[2r-3]\big)\\
=\mathcal{H}^1(p^{2,n-1}_{2,n}\Omega^{\bullet}_{X_{\centerdot}})=p^{n-1}\Omega^1_{X_n/W_n}/d\big(p^{2(n-1)}\mathcal{O}_{X_n}\big).
\end{gather*}
The mod $p$ Chern character $\mathrm{ch}:(\mathcal{K}/p)_{X_n,X_{n-1},2} \rightarrow  \bigoplus_{r=1}^{p-1}\mathcal{H}^{-2}\big(p^{r,n-1}_{r,n}\Omega^{\bullet}_{X_{\centerdot}}[2r-1]\otimes^{\mathbf{L}}\mathbb{Z}/p\big)$ fits into the maps  of short exact sequences
\begin{gather}\label{prop:relative-comparison-i=2}
\xymatrix{
  0 \ar[r] &  (\mathcal{K}_{X_n,X_{n-1},2})/p \ar[r] \ar[d]_{\mathrm{ch}} & (\mathcal{K}/p)_{X_n,X_{n-1},2} \ar[r] \ar[d]_{\mathrm{ch}} & \leftidx{_p}(\mathcal{K}_{X_n,X_{n-1},1}) \ar[r] \ar[d]_{\mathrm{ch}} & 0 \\
  0 \ar[r] & \mathcal{H}^1(p^{2,n-1}_{2,n}\Omega^{\bullet}_{X_{\centerdot}})/p \ar[r] \ar@{=}[d] &  \bigoplus_{r=1}^{p-1}\mathcal{H}^{-2}\big(p^{r,n-1}_{r,n}\Omega^{\bullet}_{X_{\centerdot}}[2r-1]\otimes^{\mathbf{L}}\mathbb{Z}/p\big) \ar[r] \ar@{=}[d] & \leftidx{_p}(p^{n-1}\mathcal{O}_{X_{n}}) \ar[r] \ar@{=}[d] & 0\\
  0 \ar[r] & \Omega^1_{X_1} \ar[r] & \mathcal{O}_{X_1}\oplus \Omega^1_{X_1} \ar[r] & \mathcal{O}_{X_1} \ar[r] & 0
}
\end{gather}
where the bottom row is induced by the obvious splitting of $\mathcal{O}_{X_1}\oplus \Omega^1_{X_1}$ and the vertical identifications are given by \eqref{eq:rel-HC-nonreduced-modP-components-sheafified}.
By Proposition \ref{prop:c1-K1}, $\mathrm{ch}=c_1:\mathcal{K}_{X_n,X_{n-1},1} \rightarrow p^{n-1}\mathcal{O}_{X_{n}}$ is an isomorphism. 
In particular, the induced map  $\leftidx{_p}(\mathcal{K}_{X_n,X_{n-1},1}) \rightarrow \leftidx{_p}(p^{n-1}\mathcal{O}_{X_{n}})$ is an isomorphism. Thus it remains to show that the upper left vertical arrow ch in \eqref{prop:relative-comparison-i=2} is an isomorphism. We do this by comparing the Brun map and the Chern character as follows.

By Proposition \ref{prop:definition-infinitesimal-ChernCharacter}, the Chern character preserves products. Thus by Lemma \ref{lem:comparison-chernClass-chernCharacter} and Proposition \ref{prop:K1-multiplication-motivic-complex} we have the following commutative diagram:
\begin{gather}\label{eq:relative-comparison-modP-i=1,2-proof-graph-1}
\xymatrix@C=5pc{
  \mathcal{K}_{X_n,1}\times \mathcal{K}_{X_n,X_{n-1},1} \ar[r] \ar[d]_{c_1\times \mathrm{ch}} & \mathcal{K}_{X_n,X_{n-1},2}/p \ar[d]^{\mathrm{ch}} \\
  \mathcal{O}_{X_n}^{\times} \times p^{n-1}\mathcal{O}_{X_n} \ar[r]^{(x,\alpha)\mapsto -\mathrm{dlog}x\cdot \alpha} & p^{n-1}\Omega^1_{X_n/W_n}
}
\end{gather}
On the other hand, by Proposition \ref{prop:brun-iso-product}, the Brun map preserves products. Thus by Proposition \ref{prop:K1-multiplication-on-rel-HC}, we have the following commutative diagram:
\begin{gather}\label{eq:relative-comparison-modP-i=1,2-proof-graph-2}
\xymatrix@C=5pc{
  \mathcal{K}_{X_n,1}\times \mathcal{K}_{X_n,X_{n-1},1} \ar[r] \ar[d]_{c_1\times (\Psi^{-1}\circ\mathrm{br})} & \mathcal{K}_{X_n,X_{n-1},2}/p \ar[d]^{(\Psi^{-1}\circ \mathrm{br}) \mod p} \\
  \mathcal{O}_{X_n}^{\times} \times p^{n-1}\mathcal{O}_{X_n} \ar[r]^{(x,\alpha)\mapsto -\mathrm{dlog}x\cdot \alpha} & p^{n-1}\Omega^1_{X_n/W_n}
}
\end{gather}
where the right-hand vertical arrow is the composition
\begin{gather*}
\mathcal{K}_{X_n,X_{n-1},2} \xrightarrow{\mathrm{br}}\mathcal{HC}_{X_n,X_{n-1},1} \xrightarrow[\cong]{\Psi^{-1}}\mathcal{H}^1(p^{2,n-1}_{2,n}\Omega^{\bullet}_{X_{\centerdot}})
\end{gather*}
modulo $p$ and $\Psi$ is the map induced by the quasi-isomorphism $\Psi$ in Theorem \ref{thm:rel-HC-homotopyEquiv}, and similarly for the left-hand vertical $\Psi^{-1}\circ\mathrm{br}$. By Brun's Theorem \ref{thm:Brun}, the vertical arrows in \eqref{eq:relative-comparison-modP-i=1,2-proof-graph-1} are isomorphisms. The lower horizontal arrow in \eqref{eq:relative-comparison-modP-i=1,2-proof-graph-2} is clearly surjective, thus so is the upper horizontal arrow. But the horizontal arrows in \eqref{eq:relative-comparison-modP-i=1,2-proof-graph-1} and \eqref{eq:relative-comparison-modP-i=1,2-proof-graph-2} and identical, respectively. Thus the right-hand vertical arrow in \eqref{eq:relative-comparison-modP-i=1,2-proof-graph-1} is determined by the composition 
\begin{gather*}
\mathcal{K}_{X_n,1}\times \mathcal{K}_{X_n,X_{n-1},1} \xrightarrow{c_1\times \mathrm{ch}}\mathcal{O}_{X_n}^{\times} \times p^{n-1}\mathcal{O}_{X_n} \xrightarrow{(x,\alpha)\mapsto \mathrm{dlog}x\cdot \alpha}  p^{n-1}\Omega^1_{X_n/W_n}.
\end{gather*}
By Proposition \ref{prop:Brun-map-K1}, the map $\Psi^{-1}\circ \mathrm{br}:\mathcal{K}_{X_n,X_{n-1},1} \rightarrow p^{n-1}\mathcal{O}_{X_{n}}$ differs from  the Chern character only by a universal unit of $\mathbb{Z}_p$. 
It follows that the right-hand vertical arrow in \eqref{eq:relative-comparison-modP-i=1,2-proof-graph-1} differs from that  in \eqref{eq:relative-comparison-modP-i=1,2-proof-graph-2} only by a unit in $\mathbb{F}_p$, hence is also an isomorphism. This completes the proof of this lemma.
\end{proof}

\begin{proof}[Proof of Proposition \ref{prop:relative-comparison-m=n-1-modP}]
 We are going to show that the mod $p$ relative Chern character \eqref{eq:relative-comparison-m=n-1-modP} is an isomorphism by induction on $i\leq p-3$. The case  $i=1$ or $2$ is proven by Lemma \ref{lem:relative-comparison-modP-i=1,2}.
We define an isomorphism
\begin{gather}\label{eq:identification-br/p}
\mathfrak{br}: (\mathcal{K}/p)_{X_n,X_{n-1},i}\cong \bigoplus_{j=0}^{i-1}\Omega^j_{X_1/\Bbbk}
\end{gather}
to be the composition
\begin{gather*}
\xymatrix@C=6pc{
  (\mathcal{K}/p)_{X_n,X_{n-1},i} \ar[r]^{\mathrm{br}}_{\simeq} \ar[dr]_{\mathfrak{br}} &  (\mathcal{HC}_{X_n,X_{n-1}}/p)_{i-1} \ar[d]^{(\Psi/p)^{-1}}_{\simeq} \\
  & \bigoplus_{j=0}^{i-1}\Omega^j_{X_1/\Bbbk}
}
\end{gather*}
where $\mathrm{br}$ is Brun's isomorphism in Remark \ref{rem:BrunIsom-withCoeff}, $\Psi$ is the quasi-isomorphism in Theorem \ref{thm:rel-HC-homotopyEquiv}, and $\Psi/p$ is the isomorphism of homology with coefficients $\mathbb{Z}/p$ given by Corollary \ref{cor:decomposition-HH-HC}.
  It remains to show the following 
\begin{claim}\label{claim:relative-comparison-m=n-1-modP-induction}
Under the identification $\mathfrak{br}$, for $3\leq i\leq p-3$, the mod $p$ relative Chern character $\mathrm{ch}:\bigoplus_{j=0}^{i-1}\Omega^j_{X_1} \rightarrow \bigoplus_{j=0}^{i-1}\Omega^j_{X_1}$ is an upper-triangular automorphism. Here upper-triangular means that the induced map $\Omega^j_{X_1} \rightarrow \Omega^k_{X_1}$ on components is 0 if $k<j$.
\end{claim}
\begin{proof}[Proof of Claim \ref{claim:relative-comparison-m=n-1-modP-induction}]
Since the Chern character preserves products, we have, by Lemma \ref{lem:comparison-chernClass-chernCharacter} and Propositions \ref{prop:K1-multiplication-motivic-complex} and \ref{prop:multiplication-Bott-element}, the following commutative diagrams:
\begin{gather}\label{eq:relative-comparison-m=n-1-modP-induction-graph-1}
  \xymatrix@C=5pc{
  \mathcal{K}_{X_n,1}\times (\mathcal{K}_{X_n,X_{n-1}}/p)_{i-1} \ar[r] \ar[d]_{c_1\times \mathrm{ch}} & (\mathcal{K}_{X_n,X_{n-1}}/p)_{i} \ar[d]^{\mathrm{ch}} \\
  \mathcal{O}_{X_n}^{\times} \times \bigoplus_{j=0}^{i-2}\Omega^j_{X_1/\Bbbk} \ar[r]^{(x,\alpha)\mapsto \pm\mathrm{dlog}x\land \alpha} &  \bigoplus_{j=0}^{i-1}\Omega^j_{X_1/\Bbbk}
}
\end{gather}
and
\begin{gather}\label{eq:relative-comparison-m=n-1-modP-induction-graph-2}
\xymatrix{
  (\mathcal{K}_{X_n,X_{n-1}}/p)_{i-2} \ar[r]^{\beta_{\zeta}\times } \ar[d]_{\mathrm{ch}} & (\mathcal{K}_{X_n,X_{n-1}}/p)_{i} \ar[d]^{\mathrm{ch}} \\
  \bigoplus_{j=0}^{i-3}\Omega^j_{X_1/\Bbbk} \ar[r]^{\boldsymbol{\beta}} &  \bigoplus_{j=0}^{i-1}\Omega^j_{X_1/\Bbbk}
}
\end{gather}
where (looking at Figure \ref{fig:multTable} in the introduction may be helpful)
\begin{gather*}
\boldsymbol{\beta}\left|_{\Omega^j_{X_1/\Bbbk}}\right.=\begin{cases}
0 & \mbox{if}\ i-j\ \mbox{is odd} \\
\mathrm{id} & \mbox{if}\ i-j\ \mbox{is even}
\end{cases}.
\end{gather*}
On the other hand, by Propositions \ref{prop:brun-iso-product}, \ref{prop:K1-multiplication-on-rel-HC}, and \ref{prop:BottElement-multiplication-on-rel-HC},  the two preceding diagrams \eqref{eq:relative-comparison-m=n-1-modP-induction-graph-1} and \eqref{eq:relative-comparison-m=n-1-modP-induction-graph-2} remain commutative when we replace all $\mathrm{ch}$ with $\mathrm{br}$.

If $i$ is even, the images of the two bottom arrows in \eqref{eq:relative-comparison-m=n-1-modP-induction-graph-1} and \eqref{eq:relative-comparison-m=n-1-modP-induction-graph-2} span $\bigoplus_{j=0}^{i-1}\Omega^j_{X_1/\Bbbk}$, hence  the claim follows from the case for $i-1$ and $i-2$.

If $i$ is odd, the same argument implies that the restriction of $\mathrm{ch}$ on $\bigoplus_{j=1}^{i-1}\Omega^j_{X_1} \rightarrow \bigoplus_{j=1}^{i-1}\Omega^j_{X_1}$ is an upper-triangular automorphism. It remains to show that the induced map on the basic summand
\begin{gather*}
\mathcal{O}_{X_1}\subset \mathcal{H}^0\left(p^{\frac{i+1}{2},n-1}_{\frac{i+1}{2},n}\Omega^{\bullet}_{X_{\centerdot}}\otimes^{\mathbf{L}}\mathbb{Z}/p\right)
\subset \mathcal{H}^{1-i}\Big(\bigoplus_{r=1}^{N}p^{r,n-1}_{r,n}\Omega^{\bullet}_{X_{\centerdot}}[2r-2]\otimes^{\mathbf{L}}\mathbb{Z}/p\Big).
\end{gather*}
is an automorphism. The problem is local, and via the identification \eqref{eq:identification-br/p}, the relative Chern character yields an endomorphism of the functor $X_{\centerdot}\mapsto p^{\frac{i+1}{2},n-1}_{\frac{i+1}{2},n}\Omega^{\bullet}_{X_{\centerdot}}$. Hence the last assertion follows from Proposition \ref{prop:functoriality-determine-universal-transform}.
\end{proof}
The proof of \eqref{eq:relative-comparison-m=n-1-modP}  being an isomorphism is completed.
\end{proof}

\begin{corollary}\label{cor:generatingElements-KXn}
The Nisnevich sheaf of rings $\bigoplus_{i\leq p-3}(\mathcal{K}/p)_{X_n,i}$ is generated by $(\mathcal{K}/p)_{X_n,1}$, a Bott element, and the image of the basic elements.
\end{corollary}
\begin{proof}
This is shown in the proof of Claim \ref{claim:relative-comparison-m=n-1-modP-induction}, especially the diagrams \eqref{eq:relative-comparison-m=n-1-modP-induction-graph-1} and \eqref{eq:relative-comparison-m=n-1-modP-induction-graph-2}.
\end{proof}



\section{Consequences in \texorpdfstring{$K$}{K}-theory and deformation theory} 
\label{sec:consequences_in_k_theory_and_deformation_theory}

In this section, we draw two direct consequences of Theorem \ref{thm:relative-comparison-local}.
Both are motivated by \cite{BEK14}.
\subsection{Infinitesimal deformation in algebraic \texorpdfstring{$K$}{K}-theory} 
\label{sub:deformation_in_k_theory}
\begin{proposition}\label{prop:deformationInK-theory}
Let $\Bbbk$ be a perfect field of characteristic $p$.  Let $X_{\centerdot}\in \operatorname{ob}\mathrm{Sm}_{W_{\centerdot}(\Bbbk)}$.
Let $m,n\in \mathbb{Z}$ and $n>m\geq 1$.
\begin{enumerate}[(i)]
  \item If $d\leq p-6$, then  a class $\xi\in K_0(X_m)$ lifts to $K_0(X_n)$ if and only if it maps to 0 by the composition, which we regard as a Hodge obstruction map,
  \begin{gather*}
  \mathrm{ob}:
  K_0(X_m) \xrightarrow{\mathrm{ch}} \bigoplus_{r=0}^{p-1}H^{2r}(X_1,\mathbb{Z}_m(r))[\frac{1}{r!}] \rightarrow \bigoplus_{r=1}^{d-1}\mathbb{H}^{2r}(X_1,p^{r,m}_{r,n}\Omega^{\bullet}_{X_{\centerdot}})\ .
  \end{gather*}
  \item  Let $i\in \mathbb{Z}$ and $1\leq i\leq p-d-4$. Then a class $\xi\in K_i(X_m)$ lifts to $K_i(X_n)$ if and only if it maps to 0 by the composition
  \begin{gather*}
  \mathrm{ob}:
  K_i(X_m) \xrightarrow{\mathrm{ch}} \bigoplus_{r= 0}^{p-1}H^{2r-i}(X_1,\mathbb{Z}_m(r))[\frac{1}{r!}] \rightarrow \bigoplus_{r=1}^{d+i-1}\mathbb{H}^{2r-i}(X_1,p^{r,m}_{r,n}\Omega^{\bullet}_{X_{\centerdot}})\ .
  \end{gather*}
\end{enumerate}
\end{proposition}
\begin{proof}
First consider the case (ii). By the compatibility of Chern characters (Proposition \ref{prop:infinitesimalChernClass}), we have the following commutative diagram
\begin{gather*}
\xymatrix{
  K_i(X_n) \ar[r] \ar[d]_{\mathrm{ch}}  & K_i(X_m) \ar[d]_{\mathrm{ch}} \ar[r] & K_{i-1}(X_n,X_m) \ar[d]^{\mathrm{ch}} \\
  \bigoplus\limits_{r=0}^{p-1}H^{2r-i}(X_1,\mathbb{Z}_n(r))[\frac{1}{r!}] \ar[r] & \bigoplus\limits_{r=0}^{p-1}H^{2r-i}(X_1,\mathbb{Z}_m(r))[\frac{1}{r!}] \ar[r] & 
  \bigoplus\limits_{r=0}^{p-1}\mathbb{H}^{2r-i}(X_1,p^{r,m}_{r,n}\Omega^{\bullet}_{X_{\centerdot}})[\frac{1}{r!}]
}
\end{gather*}
where  the horizontal sequences are exact. For $r\geq d+i$, $H^{2r-i-j}(X_1,p^{r,m}_{r,n}\Omega^{j}_{X_{\centerdot}})$ vanishes because $2r-i-j\geq r+1-i\geq d+1$. 
Thus we have
\begin{gather*}
\bigoplus_{r=0}^{p-1}\mathbb{H}^{2r-i}(X_1,p^{r,m}_{r,n}\Omega^{\bullet}_{X_{\centerdot}})[\frac{1}{r!}]
=\bigoplus_{r=1}^{d+i-1}\mathbb{H}^{2r-i}(X_1,p^{r,m}_{r,n}\Omega^{\bullet}_{X_{\centerdot}})
\end{gather*}
since the RHS is a $p$-primary group. Hence the rightmost $\mathrm{ch}$ is an isomorphism by Corollary \ref{cor:smAlg-relKTheory}(i), and (ii) follows.
As  the argument in \cite[Page 706]{BEK14}, the case $i=0$ follows from the case $i=1$ by the following lemma.
\end{proof}

\begin{lemma}\label{lem:splitting-K0-to-K1}
Let 
$Y_{n}=X_{n}\times \mathbb{G}_{\mathrm{m}}=X_{n}\times_{\operatorname{Spec}W_n(\Bbbk)} \operatorname{Spec}W_n(\Bbbk)[T,T^{-1}]$.
There are diagrams 
\begin{gather*}
\xymatrix{
K_1(Y_{n}) \ar@<1ex>[d]^{\partial} \ar[r]^<<<<{\mathrm{ch}} & \bigoplus_{r=0}^{p-1}
H^{2r-1}(Y_1,\mathbb{Z}_{Y_{n}}(r))[\frac{1}{r!}] \ar@<1ex>[d]^{\partial} \\
K_0(X_{n}) \ar[u]^{\{T\}} \ar[r]^>>>>>{\mathrm{ch}} & \bigoplus_{r=0}^{p-1}
H^{2r}(X_1,\mathbb{Z}_{X_{n}}(r))[\frac{1}{r!}] \ar[u]^{\{T\}}
}
\end{gather*}
and
\begin{gather*}
\xymatrix{
K_1(Y_{n},Y_m) \ar@<1ex>[d]^{\partial} \ar[r]^<<<<{\mathrm{ch}} & \bigoplus_{r=0}^{p-1}
H^{2r-2}(Y_1,p^{r,m}_{r,n}\Omega^{\bullet}_{Y_{\centerdot}})[\frac{1}{r!}] \ar@<1ex>[d]^{\partial} \\
K_0(X_{n},X_m) \ar[u]^{\{T\}} \ar[r]^>>>>>{\mathrm{ch}} & \bigoplus_{r=0}^{p-1}
H^{2r-1}(X_1,p^{r,m}_{r,n}\Omega^{\bullet}_{Y_{\centerdot}})[\frac{1}{r!}] \ar[u]^{\{T\}}
}
\end{gather*}
where the vertical maps $\{T\}$ on the left are the exterior multiplication by $T\in K_1(\mathbb{Z}[T,T^{-1}])$, and on the right are the exterior multiplications by 
$c_1(T)=T\in H^0(\mathbb{G}_{\mathrm{m},W_n(\Bbbk)}, \mathbb{G}_{\mathrm{m}})= H^1\big(\mathbb{G}_{\mathrm{m},\Bbbk},\mathbb{Z}_{\mathbb{G}_{\mathrm{m},W_n(\Bbbk)}}(1)\big)$, satsifying the following properties:
\begin{enumerate}[(i)]
  \item Commutativity: $\partial\circ \mathrm{ch}=\mathrm{ch}\circ \partial,\ \mathrm{ch}\circ \{T\}=\{T\}\circ \mathrm{ch}$.
  \item Splitting:  $\partial\circ \{T\}=\mathrm{id}$.
\end{enumerate} 
Moreover, the maps $\{T\}$ and $\partial$ are compatible with the long exact sequence relating the relative $K$-theory and the absolute $K$-theory.
\end{lemma}
\begin{proof}
Let $U_{n}=X_{n}\times_{\operatorname{Spec}W_n(\Bbbk)} \operatorname{Spec}W_n(\Bbbk)[T]$ and 
$V_{n}=X_{n}\times_{\operatorname{Spec}W_n(\Bbbk)} \operatorname{Spec}W_n(\Bbbk)[T^{-1}]$.
The map $\partial$ is induced by the Mayer-Vietoris sequence associated with the covering $U_n\cup V_n$ of $X\times \mathbb{P}^1_{W_n(\Bbbk)}$, and see \cite[Theorem 6.1]{ThT90} for the splitting property. The properties of $\partial$ then follow.

By the multiplicativity of the  Chern character (Proposition \ref{prop:definition-infinitesimal-ChernCharacter}), to show the properties of the maps $\{T\}$ it suffices to show $H^{2r-1}\big(\mathbb{G}_{\mathrm{m},\Bbbk},\mathbb{Z}_{\mathbb{G}_{\mathrm{m},W_n(\Bbbk)}}(r)\big)=0$ for $r\geq 2$, which implies $\mathrm{ch}(T)=c_1(T)$. By the exact sequence
\begin{gather*}
H^{2r-2}\big(\mathbb{G}_{\mathrm{m},\Bbbk},p^{r,1}_{r,n}\Omega^{\bullet}_{\mathbb{G}_{\mathrm{m},W_n(\Bbbk)}}(r)\big) \rightarrow
H^{2r-1}\big(\mathbb{G}_{\mathrm{m},\Bbbk},\mathbb{Z}_{\mathbb{G}_{\mathrm{m},W_n(\Bbbk)}}(r)\big) \rightarrow
H^{2r-1}\big(\mathbb{G}_{\mathrm{m},\Bbbk},\mathbb{Z}(r)\big),
\end{gather*}
it remains to show that the groups at both ends vanish. Since $\dim \mathbb{G}_{\mathrm{m},\Bbbk}=1$ and is affine, $H^{2r-2}\big(\mathbb{G}_{\mathrm{m},\Bbbk},p^{r,1}_{r,n}\Omega^{\bullet}_{\mathbb{G}_{\mathrm{m},W_n(\Bbbk)}}(r)\big)=0$ for $r\geq 2$. There is an isomorphism $H^{2r-1}\big(\mathbb{G}_{\mathrm{m},\Bbbk},\mathbb{Z}(r)\big)\cong \mathrm{CH}^{r}(\mathbb{G}_{\mathrm{m},\Bbbk},1)$ with higher Chow groups, which vanish for $r\geq 3$ for the dimension reason, and for  $r=2$ by the localization exact sequence $\mathrm{CH}^{2}(\mathbb{A}^1,1)\rightarrow \mathrm{CH}^{2}(\mathbb{G}_m,1)\rightarrow \mathrm{CH}^{1}(\operatorname{Spec} \Bbbk,0)$ and the $\mathbb{A}^1$-invariance.
\end{proof}


\subsection{Continuity of algebraic  \texorpdfstring{$K$}{K}-theory} 
\label{sub:continuous_k_theory}

\begin{lemma}\label{lem:contCohomology_of_relMotComplex}
Let $\Bbbk$ be a perfect field of characteristic $p$.
Let $X_{\centerdot}\in \mathrm{Sm}_{W_{\centerdot}}$ be an object associated with a smooth projective scheme $X$ over $W(\Bbbk)$. Then 
\begin{align*}
H_{\mathrm{cont}}^{i}\big(X_{\centerdot},p(r)\Omega^{\bullet}_{X_{\centerdot}}\big)
\cong \varprojlim_{n}H^i\big(X_{1},p(r)\Omega^{\bullet}_{X_{n}}\big)\ .
\end{align*}
\end{lemma}
\begin{proof}
We need only to show  $\varprojlim_{n}^1H^i\big(X_{1},p(r)\Omega^{\bullet}_{X_{n}}\big)=0$.
Denote  $W_j(\Bbbk)$ by $W_j$ for $j\geq 0$. We have the Hodge-to-de Rham spectral sequence
\begin{gather*}
E_1^{i,j}(n)=H^i(X_1,p^{r-j}\Omega^j_{X_{(r-j)n}/W_{(r-j)n}})\Longrightarrow H^{i+j}\big(X_{1},p(r)\Omega^{\bullet}_{X_{n}}\big)\ .
\end{gather*}
It suffices to show that for any fixed $n$ the images of the composition of homomorphisms of $W(\Bbbk)$-modules
\begin{gather*}
E_{\infty}^{i,j}(n+k)\rightarrow E_{\infty}^{i,j}(n+k-1)\rightarrow \dots\rightarrow
E_{\infty}^{i,j}(n+1)\rightarrow E_{\infty}^{i,j}(n)
\end{gather*}
stabilize for $k\gg 0$.
 To show this, first we note that this spectral sequence is constructed by filtrations of complexes of $W(\Bbbk)$-modules, therefore 
  $E_{\infty}^{i,j}(n)$ is a subquotient-$W_{(r-j)n}$-module of $E_1^{i,j}(n)$, and for every $k\geq 0$, the image of the homomorphism $E_{\infty}^{i,j}(n+k)\rightarrow E_{\infty}^{i,j}(n)$ is a $W_{(r-j)n}$-submodule of $E_{\infty}^{i,j}(n)$.
   Secondly, since $X$ is projective over $W(\Bbbk)$,  $E_1^{i,j}(n)\cong H^i\big(X_{(r-j)(n-1)},\Omega^j_{X_{(r-j)(n-1)}/W_{(r-j)(n-1)}}\big)$ is a finite $W_{(r-j)(n-1)}(\Bbbk)$-module, thus so is $E_{\infty}^{i,j}(n)$. 
   Hence $E_{\infty}^{i,j}(n)$  has finite length and the proof is completed.
\end{proof}

\begin{proposition}\label{prop:continuity-alg-K-theory}
Let $\Bbbk$ be a perfect field of characteristic $p$.
Let $X_{\centerdot}\in \mathrm{Sm}_{W_{\centerdot}}$ be an object associated with a smooth projective scheme $X$ over $W(\Bbbk)$. Then for $0\leq i\leq p-5-d$,
\begin{align*}
K_{\mathrm{cont},i}(X_{\centerdot})\xrightarrow{\sim} \varprojlim_{n} K_i(X_n)\quad.
\end{align*}
\end{proposition}
\begin{proof}
This follows from Corollary \ref{cor:smAlg-relKTheory}(i) and Lemma \ref{lem:contCohomology_of_relMotComplex}.
\end{proof}






\begin{thebibliography}{000}

\bibitem[Ada84]{Ada84}
Adams, J. F. 
Prerequisites (on equivariant stable homotopy) for Carlsson's lecture. Algebraic topology, Aarhus 1982 (Aarhus, 1982), 483--532, Lecture Notes in Math., 1051, Springer, Berlin, 1984.

\bibitem[Ang11]{Ang11}
Angeltveit, Vigleik. 
On the algebraic $K$-theory of Witt vectors of finite length. 
arXiv:1101.1866 (2011).

\bibitem[AMMN22]{AMMN22}
Antieau, Benjamin; Mathew, Akhil; Morrow, Matthew; Nikolaus, Thomas. 
On the Beilinson fiber square. Duke Math. J. 171 (2022), no. 18, 3707--3806.

\bibitem[AKN24]{AKN24}
Antieau, Benjamin, Achim Krause, and Thomas Nikolaus. 
On the $K$-theory of $\mathbb{Z}/p^n$. arXiv:2405.04329 (2024).

\bibitem[Ara01]{Ara01}
Arabia, Alberto. 
Relèvements des algèbres lisses et de leurs morphismes.  Comment. Math. Helv. 76 (2001), no. 4, 607--639.

\bibitem[AGV71]{AGV71}
Artin, Michael; Grothendieck, Alexander; Verdier, Jean-Louis.
Theorie de topos et cohomologie etale des schemas III, Lecture Notes in Mathematics, vol. 305, Springer, 1971.

\bibitem[ArM69]{ArM69}
Artin, M.; Mazur, B.
Etale homotopy. Lecture Notes in Mathematics, No. 100. Springer-Verlag, Berlin-New York, 1969. 


\bibitem[Bei14]{Bei14}
Beilinson, Alexander.
Relative continuous K-theory and cyclic homology.
Münster J. Math. 7 (2014), no. 1, 51–81.


\bibitem[BeO78]{BeO78}
Berthelot, Pierre; Ogus, Arthur.
Notes on crystalline cohomology.
Princeton University Press, Princeton, NJ; University of Tokyo Press, Tokyo, 1978.

\bibitem[Blo72]{Blo72}
Bloch, Spencer. 
Semi-regularity and deRham cohomology. Invent. Math. 17 (1972), 51--66.

\bibitem[BEK14]{BEK14} 
Bloch, Spencer; Esnault, Hélène; Kerz, Moritz.
$p$-adic deformation of algebraic cycle classes. 
Invent. Math. 195 (2014), no. 3, 673–722.

\bibitem[BEK14b]{BEK14b}
Bloch, Spencer; Esnault, Hélène; Kerz, Moritz. Deformation of algebraic cycle classes in characteristic zero. 
Algebr. Geom. 1 (2014), no. 3, 290--310.


\bibitem[BlK86]{BlK86}
Bloch, Spencer; Kato, Kazuya. 
$p$-adic étale cohomology. Inst. Hautes Études Sci. Publ. Math. No. 63 (1986), 107–152.



\bibitem[Bou24]{Bou24}
Bouis, Tess. 
Motivic cohomology of mixed characteristic schemes. arXiv:2412.06635 (2024).

\bibitem[Bou03]{Bou03}
Bourbaki, Nicolas. Algebra II. Chapters 4–7. Translated from the 1981 French edition by P. M. Cohn and J. Howie. Reprint of the 1990 English edition. Elements of Mathematics (Berlin). Springer-Verlag, Berlin, 2003.


\bibitem[BrG73]{BrG73}
Brown, Kenneth S.; Gersten, Stephen M. Algebraic $K$-theory as generalized sheaf cohomology. Algebraic $K$-theory, I: Higher $K$-theories (Proc. Conf., Battelle Memorial Inst., Seattle, Wash., 1972), pp. 266--292, Lecture Notes in Math., Vol. 341, Springer, Berlin-New York, 1973. 


\bibitem[Bru01]{Bru01}
Brun, Morten.
Filtered topological cyclic homology and relative $K$-theory of nilpotent ideals. 
Algebr. Geom. Topol. 1 (2001), 201–230.

\bibitem[BuF03]{BuF03}
Buchweitz, Ragnar-Olaf; Flenner, Hubert. A semiregularity map for modules and applications to deformations. Compositio Math. 137 (2003), no. 2, 135--210.

\bibitem[BuV88]{BuV88}
Burghelea, Dan; Vigué-Poirrier, Micheline. 
Cyclic homology of commutative algebras. I. Algebraic topology—rational homotopy (Louvain-la-Neuve, 1986), 51--72, Lecture Notes in Math., 1318, Springer, Berlin, 1988.

\bibitem[BuF86]{BuF86}
Burghelea, D.; Fiedorowicz, Z. 
Cyclic homology and algebraic $K$-theory of spaces. II. Topology 25 (1986), no. 3, 303--317.

\bibitem[ChS14]{ChS14}
Charles, François; Schnell, Christian. 
Notes on absolute Hodge classes. Hodge theory, 469--530, Math. Notes, 49, Princeton Univ. Press, Princeton, NJ, 2014.

\bibitem[CiD19]{CiD19}
Cisinski, Denis-Charles; Déglise, Frédéric. Triangulated categories of mixed motives. Springer Monographs in Mathematics. Springer, Cham, 2019.

\bibitem[Dun97]{Dun97}
Dundas, Bjørn Ian. 
Relative $K$-theory and topological cyclic homology. Acta Math. 179 (1997), no. 2, 223--242. 

\bibitem[DGM13]{DGM13}
Dundas, Bjørn Ian; Goodwillie, Thomas G.; McCarthy, Randy.
The local structure of algebraic $K$-theory.
Algebra and Applications, 18. Springer-Verlag London, Ltd., London, 2013.

\bibitem[EGA I]{EGA I}
Grothendieck, A. 
Éléments de géométrie algébrique. I. Le langage des schémas. 
Inst. Hautes Études Sci. Publ. Math. No. 4 (1960), 



\bibitem[EGA IV-4]{EGA IV-4}
Grothendieck, A.
Éléments de géométrie algébrique. IV. Étude locale des schémas et des morphismes de schémas. IV.
Inst. Hautes Études Sci. Publ. Math. No. 32 (1967).

\bibitem[ElM02]{ElM02}
Elbaz-Vincent, Philippe; Müller-Stach, Stefan. Milnor $K$-theory of rings, higher Chow groups and applications. Invent. Math. 148 (2002), no. 1, 177--206.

\bibitem[Elk73]{Elk73}
Elkik, Renée. 
Solutions d'équations à coefficients dans un anneau hensélien.  Ann. Sci. École Norm. Sup. (4) 6 (1973), 553–603 (1974).

\bibitem[ElM23]{ElM23}
Elmanto, Elden;  Morrow,  Matthew. Motivic cohomology of equicharacteristic schemes. arXiv:2309.08463 (2023).

\bibitem[Eme97]{Eme97}
Emerton, Matthew. 
A $p$-adic variational Hodge conjecture and modular forms with complex multiplication. preprint (1997).


\bibitem[FoM87]{FoM87}
Fontaine, Jean-Marc; Messing, William. 
$p$-adic periods and $p$-adic étale cohomology. Current trends in arithmetical algebraic geometry (Arcata, Calif., 1985), 179--207, Contemp. Math., 67, Amer. Math. Soc., Providence, RI, 1987.


\bibitem[GeH99]{GeH99}
Geisser, Thomas; Hesselholt, Lars.
Topological cyclic homology of schemes. Algebraic $K$-theory (Seattle, WA, 1997), 41–87,
Proc. Sympos. Pure Math., 67, Amer. Math. Soc., Providence, RI, 1999.

\bibitem[GeH06]{GeH06}
Geisser, Thomas; Hesselholt, Lars. The de Rham-Witt complex and $p$-adic vanishing cycles. J. Amer. Math. Soc. 19 (2006), no. 1, 1--36. 


\bibitem[GeH11]{GeH11}
Geisser, Thomas; Hesselholt, Lars.
On relative and bi-relative algebraic K-theory of rings of finite characteristic.
J. Amer. Math. Soc. 24 (2011), no. 1, 29–49.

\bibitem[GeL00]{GeL00}
Geisser, Thomas; Levine, Marc. 
The $K$-theory of fields in characteristic $p$. Invent. Math. 139 (2000), no. 3, 459--493. 

\bibitem[Gil81]{Gil81}
Gillet, Henri. 
Riemann-Roch theorems for higher algebraic $K$-theory. Adv. in Math. 40 (1981), no. 3, 203--289. 


\bibitem[Goo85]{Goo85}
Goodwillie, Thomas G. 
Cyclic homology, derivations, and the free loopspace. Topology 24 (1985), no. 2, 187--215. 

\bibitem[Gos96]{Gos96}
Goss, David. Basic structures of function field arithmetic. Ergebnisse der Mathematik und ihrer Grenzgebiete (3), 35. Springer-Verlag, Berlin, 1996.

\bibitem[GrM95]{GrM95}
Greenlees, J. P. C.; May, J. P. Generalized Tate cohomology. Mem. Amer. Math. Soc. 113 (1995), no. 543.

\bibitem[GrL21]{GrL21}
Gregory, Oliver;  Langer, Andreas. A log-motivic cohomology for semistable varieties and its $p$-adic deformation theory. arXiv:2108.02845 (2021).

\bibitem[Gro66]{Gro66}
Grothendieck, A. On the de Rham cohomology of algebraic varieties. Inst. Hautes Études Sci. Publ. Math. No. 29 (1966), 95--103.


\bibitem[HeM97]{HeM97}
Hesselholt, Lars; Madsen, Ib.
On the $K$-theory of finite algebras over Witt vectors of perfect fields. 
Topology 36 (1997), no. 1, 29–101.

\bibitem[HeM04]{HeM04}
Hesselholt, Lars; Madsen, Ib. 
On the De Rham-Witt complex in mixed characteristic. Ann. Sci. École Norm. Sup. (4) 37 (2004), no. 1, 1--43.


\bibitem[HoJ87]{HoJ87}
Hood, Christine E.; Jones, John D. S. 
Some algebraic properties of cyclic homology groups. $K$-Theory 1 (1987), no. 4, 361--384.


\bibitem[Hou66]{Hou66}
Houzel, Christian. Morphisme de Frobenius et rationalité des fonctions L. In Séminaire de Géométrie Algébrique du Bois-Marie 1965–66 SGA 5: dirigé par A. Grothendieck avec la collaboration de I. Bucur, C. Houzel, L. Illusie, J.-P. Jouanolou et J.-P. Serre, pp. 442-480. Berlin, Heidelberg: Springer Berlin Heidelberg, 1977.


\bibitem[Hov99]{Hov99}
Hovey, Mark.
Model categories.
Math. Surveys Monogr., 63.
American Mathematical Society, Providence, RI, 1999. 


\bibitem[Ill79]{Ill79}
Illusie, Luc.
Complexe de de Rham-Witt et cohomologie cristalline. 
Ann. Sci. École Norm. Sup. (4) 12 (1979), no. 4, 501–661.

\bibitem[Isa04]{Isa04}
Isaksen, Daniel C. 
Strict model structures for pro-categories. Categorical decomposition techniques in algebraic topology (Isle of Skye, 2001), 179--198, Progr. Math., 215, Birkhäuser, Basel, 2004. 

\bibitem[Jar15]{Jar15}
Jardine, John F. 
Local homotopy theory. Springer Monographs in Mathematics. Springer, New York, 2015.

\bibitem[Jon87]{Jon87}
Jones, John D. S. 
Cyclic homology and equivariant homology. Invent. Math. 87 (1987), no. 2, 403--423.

\bibitem[Kat82]{Kat82}
Kato, Kazuya. 
Galois cohomology of complete discrete valuation fields. Algebraic $K$-theory, Part II (Oberwolfach, 1980), pp. 215--238, Lecture Notes in Math., 967, Springer, Berlin-New York, 1982.

\bibitem[Kat87]{Kat87}
Kato, Kazuya. On $p$-adic vanishing cycles (application of ideas of Fontaine-Messing). Algebraic geometry, Sendai, 1985, 207–251, Adv. Stud. Pure Math., 10, North-Holland, Amsterdam, 1987.

\bibitem[KeS24]{KeS24}
Kelly, Shane; Saito, Shuji. A procdh topology. arXiv:2401.02699 (2024).

\bibitem[Ker09]{Ker09}
Kerz, Moritz. The Gersten conjecture for Milnor $K$-theory. Invent. Math. 175 (2009), no. 1, 1--33. 

\bibitem[Ker10]{Ker10}
Kerz, Moritz. 
Milnor $K$-theory of local rings with finite residue fields. J. Algebraic Geom. 19 (2010), no. 1, 173--191. 

\bibitem[Kur87]{Kur87}
Kurihara, Masato.
On two types of complete discrete valuation fields.
Compositio Math. 63 (1987), no. 2, 237–257.

\bibitem[Kur88]{Kur88}
Kurihara, Masato.
Abelian extensions of an absolutely unramified local field with general residue field.
Invent. Math. 93 (1988), no. 2, 451–480.


\bibitem[LaT19]{LaT19}
Land, Markus; Tamme, Georg. 
On the $K$-theory of pullbacks. Ann. of Math. (2) 190 (2019), no. 3, 877--930.


\bibitem[LaZ04]{LaZ04}
Langer, Andreas, and Thomas Zink. De Rham–Witt cohomology for a proper and smooth morphism. Journal of the Institute of Mathematics of Jussieu 3, no. 2 (2004): 231-314.

\bibitem[Lan18]{Lan18}
Langer, Andreas. 
$p$-adic deformation of motivic Chow groups. Doc. Math. 23 (2018), 1863--1894. 

\bibitem[LMS86]{LMS86}
Lewis, L. G., Jr.; May, J. P.; Steinberger, M.; McClure, J. E. 
Equivariant stable homotopy theory. With contributions by J. E. McClure. 
Lecture Notes in Mathematics, 1213. Springer-Verlag, Berlin, 1986.

\bibitem[Lod76]{Lod76}
Loday, Jean-Louis. 
$K$-théorie algébrique et représentations de groupes.  Ann. Sci. École Norm. Sup. (4) 9 (1976), no. 3, 309--377.


\bibitem[Lod98]{Lod98}
Loday, Jean-Louis.
Cyclic homology. Appendix E by María O. Ronco. Second edition. Chapter 13 by the author in collaboration with Teimuraz Pirashvili. Grundlehren der mathematischen Wissenschaften, 301. Springer-Verlag, Berlin, 1998.

\bibitem[McL63]{McL63}
Mac Lane, Saunders.
Homology.
Die Grundlehren der mathematischen Wissenschaften, Band 114
Springer-Verlag, Berlin-Göttingen-Heidelberg; Academic Press, Inc., Publishers, New York, 1963. 

\bibitem[Mad94]{Mad94}
Madsen, Ib. 
Algebraic $K$-theory and traces. Current developments in mathematics, 1995 (Cambridge, MA), 191--321, Int. Press, Cambridge, MA, 1994. 

\bibitem[Mar25]{Mar25}
Markman, Eyal. Cycles on abelian $2n$-folds of Weil type from secant sheaves on abelian $n$-folds. arXiv:2502.03415 (2025).


\bibitem[MaP12]{MaP12}
Maulik, Davesh; Poonen, Bjorn. 
Néron-Severi groups under specialization. Duke Math. J. 161 (2012), no. 11, 2167--2206.


\bibitem[May67]{May67}
May, J. Peter. 
Simplicial objects in algebraic topology. Van Nostrand Mathematical Studies, No. 11. D. Van Nostrand Co., Inc., Princeton, N.J.-Toronto, Ont.-London, 1967. 

\bibitem[MaPo12]{MaPo12}
May, J. P.; Ponto, K. 
More concise algebraic topology. Localization, completion, and model categories. Chicago Lectures in Mathematics. University of Chicago Press, Chicago, IL, 2012. 

\bibitem[MVW06]{MVW06}
Mazza, Carlo; Voevodsky, Vladimir; Weibel, Charles. 
Lecture notes on motivic cohomology. Clay Mathematics Monographs, 2. American Mathematical Society, Providence, RI; Clay Mathematics Institute, Cambridge, MA, 2006. 

\bibitem[McC96]{McC96}
McCarthy, Randy. 
The cyclic homology of an exact category. J. Pure Appl. Algebra 93 (1994), no. 3, 251--296.

\bibitem[McC97]{McC97}
McCarthy, Randy.
Relative algebraic $K$-theory and topological cyclic homology. Acta Math. 179 (1997), no. 2, 197–222.

\bibitem[Mor14]{Mor14}
Morrow, Matthew. 
A case of the deformational Hodge conjecture via a pro Hochschild-Kostant-Rosenberg theorem. C. R. Math. Acad. Sci. Paris 352 (2014), no. 3, 173--177. 

\bibitem[Mor19]{Mor19}
Morrow, Matthew. 
Topological Hochschild homology in arithmetic geometry. Arizona winter school note (2019).

\bibitem[Mor19b]{Mor19b}
Morrow, Matthew. 
A variational Tate conjecture in crystalline cohomology. J. Eur. Math. Soc.  21 (2019), no. 11, 3467--3511. 


\bibitem[Nei10]{Nei10}
Neisendorfer, Joseph. 
Algebraic methods in unstable homotopy theory. New Mathematical Monographs, 12. Cambridge University Press, Cambridge, 2010. 

\bibitem[NiS18]{NiS18}
Nikolaus, Thomas; Scholze, Peter.
On topological cyclic homology. 
Acta Math. 221 (2018), no. 2, 203–409.

\bibitem[Nis89]{Nis89}
Nisnevich, Ye. A. The completely decomposed topology on schemes and associated descent spectral sequences in algebraic $K$-theory. Algebraic $K$-theory: connections with geometry and topology (Lake Louise, AB, 1987), 241--342, NATO Adv. Sci. Inst. Ser. C: Math. Phys. Sci., 279, Kluwer Acad. Publ., Dordrecht, 1989.


\bibitem[PiW92]{PiW92}
Pirashvili, Teimuraz; Waldhausen, Friedhelm. Mac Lane homology and topological Hochschild homology. J. Pure Appl. Algebra 82 (1992), no. 1, 81--98. 

\bibitem[Pus04]{Pus04}
Pushin, Oleg. Higher Chern classes and Steenrod operations in motivic cohomology. $K$-Theory 31 (2004), no. 4, 307--321. 



\bibitem[ReV16]{ReV16}
Reich, Holger; Varisco, Marco. 
On the Adams isomorphism for equivariant orthogonal spectra. Algebr. Geom. Topol. 16 (2016), no. 3, 1493--1566.


\bibitem[ScS03]{ScS03}
Schwede, Stefan; Shipley, Brooke. 
Equivalences of monoidal model categories. Algebr. Geom. Topol. 3 (2003), 287--334.


\bibitem[SGA6]{SGA6}
Grothendieck, Alexander et al.
Théorie des intersections et théorème de Riemann-Roch.  Séminaire de Géométrie Algébrique du Bois-Marie 1966–1967 (SGA 6). Dirigé par P. Berthelot, A. Grothendieck et L. Illusie. Avec la collaboration de D. Ferrand, J. P. Jouanolou, O. Jussila, S. Kleiman, M. Raynaud et J. P. Serre. Lecture Notes in Mathematics, Vol. 225. Springer-Verlag, Berlin-New York, 1971. 


\bibitem[ThT90]{ThT90}
Thomason, R. W.; Trobaugh, Thomas.
Higher algebraic K-theory of schemes and of derived categories. The Grothendieck Festschrift, Vol. III, 247–435,
Progr. Math., 88, Birkhäuser Boston, Boston, MA, 1990.

\bibitem[Tot18]{Tot18}
Totaro, Burt. 
Hodge theory of classifying stacks. Duke Math. J. 167 (2018), no. 8, 1573--1621. 


\bibitem[Wei91]{Wei91}
Weibel, Charles. 
Étale Chern classes at the prime $2$. Algebraic $K$-theory and algebraic topology (Lake Louise, AB, 1991), 249--286, NATO Adv. Sci. Inst. Ser. C: Math. Phys. Sci., 407, Kluwer Acad. Publ., Dordrecht, 1993. 



\bibitem[Whi62]{Whi62}
Whitehead, George W. 
Generalized homology theories. Trans. Amer. Math. Soc. 102 (1962), 227--283. 

\end{thebibliography}
\end{document}